\documentclass[11pt]{amsart}
\usepackage{amscd,amssymb,amsthm,amsmath,amssymb,textcomp,supertabular,longtable,enumerate,rotating,mathrsfs,mathtools,hyperref,latexsym,amscd,amsbsy,mathrsfs}
\usepackage{pdflscape}
\usepackage{graphicx}% http://ctan.org/pkg/graphicx
\usepackage{array}% http://ctan.org/pkg/array
\usepackage[matrix,arrow,curve]{xy}
\usepackage{xcolor}
\usepackage{longtable}
\usepackage{enumitem}
\usepackage{MnSymbol}
\usepackage{multicol}
\usepackage{tabulary}
\usepackage{colortbl}

\usepackage[
    maxbibnames=999,    % Show all names in bibliography
    maxcitenames=1,     % Show only first name + "et al." in citations
    mincitenames=1,     % Minimum names before "et al." in citations
    style=apa,    % or whatever style you're using
    url=false,
    doi=true,
]{biblatex}
\usepackage{geometry}
\newtheorem{lemma}[equation]{Lemma}
\newtheorem{corollary}[equation]{Corollary}

\newtheorem*{calabiproblem*}{Calabi Problem}
\newtheorem{remark}[equation]{Remark}

\theoremstyle{definition}

\makeatletter\@addtoreset{equation}{section} \makeatother

\swapnumbers

\newtheoremstyle{dotless}{}{}{\rm}{}{\sc}{}{ }{}

\theoremstyle{dotless}

\newcommand{\DR}{\mathbb{R}} %real numbers
\newcommand{\DC}{\mathbb{C}}

\newcommand{\DP}{\mathbb{P}}%projective space
\newcommand{\DA}{\mathbb{A}}

\newcommand{\MO}{\mathcal{O}}

\newcommand{\Aut}{\mathrm{Aut}}

\newtheorem*{theorem*}{Theorem}
\newtheorem*{maintheorem*}{Main Theorem.}

\author{Elena Denisova}

\title{$\delta$-invariants of Du Val del Pezzo surfaces of degree $2$}

\address{\emph{Elena Denisova}
\newline
\textnormal{School of Mathematics, The University of Edinburgh, Edinburgh EH9 3JZ, UK.}
\newline
\textnormal{\texttt{e.denisova@sms.ed.ac.uk}}}

\begin{document}

\maketitle
\begin{abstract}
In this article, we compute $\delta$-invariants of Du Val del Pezzo surfaces of degree $2$.
\end{abstract}
%\tableofcontents
\section{Introduction}
\subsection{History and Results.} 
It is known that a smooth Fano variety admits a~K\"ahler--Einstein metric if and only if it is $K$-polystable.\index{K\"ahler--Einstein metric}\index{$K$-polystable}
For two-dimensional Fano varieties (del Pezzo surfaces) Tian and Yau proved that a~smooth del Pezzo surface is $K$-polystable if and only if it is not a~blow up of $\mathbb{P}^2$ in one or two points (see \cite{TianYau1987,Ti90}).
 A lot of research was done for threefolds (see 
 \cite{AbbanZhuangSeshadri, Fano21, Liu23, CheltsovFujitaKishimotoPark23, LiuZhao24, GuerreiroGiovenzanaViswanathan23, Malbon24, CheltsovPark22, BelousovLoginov24, BelousovLoginov23, CheltsovFujitaKishimotoOkada23, Denisova24, Denisova23, CheltsovDenisovaFujita24}). 
 However, not everything is known for Fano varieties of higher dimensions and threefolds showed that often the problem can be reduced to computing $\delta$-invariant of (possibly singular) del Pezzo surfaces (see \cite{Fano21,CheltsovDenisovaFujita24,CheltsovFujitaKishimotoOkada23} etc). For smooth del Pezzo surfaces  $\delta$-invariants were computed in \cite{Fano21} where it was proven that $\delta(X)\ge \frac{9}{5}$ if $X$ is a smooth del Pezzo surface of degree $2$. $\delta$-invariants of Du Val del Pezzo surfaces of degrees $\ge 3$  were computed in previous parts of this series. In this article, we compute  $\delta$-invariants  of Du Val del Pezzo surfaces of degree $2$. We prove that:
\begin{maintheorem*}
Let $X$ be a singular Du Val del Pezzo surface of degree $2$. Then  $\delta$-invariant of $X$ is uniquely determined by the degree of $X$, the number of lines on $X$, and the type of singularities on $X$ which is given in the following table:\\
\begin{minipage}{6cm}
 \renewcommand{\arraystretch}{1.4}
  \begin{longtable}{ | c | c | c | c | }
   \hline
   $K_X^2$ & $\#$ lines & $\mathrm{Sing}(X)$ & $\delta(X)$\\
  \hline
\endhead 
\hline
 $2$ & $44$ & $\DA_1$ & $\frac{3}{2}$\\
\hline
 $2$ & $34$ & $2\DA_1$ & $\frac{3}{2}$\\
\hline
 $2$ & $26$ & $3\DA_1$ & $\frac{3}{2}$\\
\hline
$2$ & $25$ & $3\DA_1$ & $\frac{3}{2}$\\
\hline
$2$ & $20$ & $4\DA_1$ & $\frac{3}{2}$\\
\hline
$2$ & $19$ & $4\DA_1$ & $\frac{3}{2}$\\
\hline
$2$ & $14$ & $5\DA_1$ & $\frac{3}{2}$\\
\hline
$2$ & $10$ & $6\DA_1$ & $\frac{3}{2}$\\
\hline
 $2$ & $31$ & $\DA_2$ & $\frac{6}{5}$\\
\hline
 $2$ & $20$ & $\DA_2+\DA_1$ & $\frac{6}{5}$\\
\hline
$2$ & $18$ & $\DA_2+2\DA_1$ & $\frac{6}{5}$\\
\hline
$2$ & $13$ & $\DA_2+3\DA_1$ & $\frac{6}{5}$\\
\hline
$2$ & $16$ & $2\DA_2$ & $\frac{6}{5}$\\
\hline
$2$ & $12$ & $2\DA_2+\DA_1$ & $\frac{6}{5}$\\
\hline
$2$ & $8$ & $3\DA_2$ & $\frac{6}{5}$\\
\hline

  \end{longtable}
  \end{minipage}
  \begin{minipage}{6.5cm}
 \renewcommand{\arraystretch}{1.4}
  \begin{longtable}{ | c | c | c | c | }
   \hline
   $K_X^2$ & $\#$ lines & $\mathrm{Sing}(X)$ & $\delta(X)$\\
  \hline
\endhead 
\hline
  $2$ & $22$ & $\DA_3$ & $1$\\
  \hline
   $2$ & $16$ & $\DA_3+\DA_1$ & $1$\\
\hline
$2$ & $15$ & $\DA_3+\DA_1$ & $1$\\
\hline
$2$ & $12$ & $\DA_3+2\DA_1$ & $1$\\
\hline
$2$ & $11$ & $\DA_3+2\DA_1$ & $1$\\
\hline
$2$ & $8$ & $\DA_3+3\DA_1$ & $1$\\
\hline
$2$ & $10$ & $\DA_3+\DA_2$ & $1$\\
\hline
  $2$ & $7$ & $\DA_3+\DA_2+\DA_1$ & $1$\\
\hline
$2$ & $6$ & $2\DA_3$ & $1$\\
\hline
$2$ & $4$ & $2\DA_3+\DA_1$ & $1$\\
\hline
$2$ & $14$ & $\DA_4$ & $\frac{12}{13}$\\
\hline
$2$ & $10$ & $\DA_4+\DA_1$ & $\frac{12}{13}$\\
\hline
$2$ & $6$ & $\DA_4+\DA_2$ & $\frac{12}{13}$\\
\hline
$2$ & $8$ & $\DA_5$ & $\frac{6}{7}$\\
\hline
$2$ & $7$ & $\DA_5$ & $\frac{3}{4}$\\
\hline
  \end{longtable}
  \end{minipage}
  \begin{minipage}{6cm}
 \renewcommand{\arraystretch}{1.4}
    \begin{longtable}{ | c | c | c | c | }
   \hline
   $K_X^2$ & $\#$ lines & $\mathrm{Sing}(X)$ & $\delta(X)$\\
  \hline\hline
  \endhead 

$2$ & $6$ & $\DA_5+\DA_1$ & $\frac{6}{7}$\\
\hline
$2$ & $5$ & $\DA_5+\DA_1$ & $\frac{3}{4}$\\
\hline
$2$ & $3$ & $\DA_5+\DA_2$ & $\frac{3}{4}$\\
\hline
$2$ & $4$ & $\DA_6$ & $\frac{4}{5}$\\
\hline
$2$ & $2$ & $\DA_7$ & $\frac{3}{4}$\\
\hline
$2$ & $14$ & $\mathbb{D}_4$ & $\frac{3}{4}$\\
\hline
$2$ & $9$ & $\mathbb{D}_4+\DA_1$ & $\frac{3}{4}$\\
\hline
$2$ & $6$ & $\mathbb{D}_4+2\DA_1$ & $\frac{3}{4}$\\
\hline
$2$ & $4$ & $\mathbb{D}_4+3\DA_1$ & $\frac{3}{4}$\\
\hline
$2$ & $8$ & $\mathbb{D}_5$ & $\frac{3}{5}$\\
\hline
$2$ & $5$ & $\mathbb{D}_5+\DA_1$ & $\frac{3}{5}$\\
\hline
$2$ & $3$ & $\mathbb{D}_6$ & $\frac{1}{2}$\\
\hline
$2$ & $2$ & $\mathbb{D}_6+\DA_1$ & $\frac{1}{2}$\\
\hline
$2$ & $4$ & $\mathbb{E}_6$ & $\frac{3}{7}$\\
\hline
$2$ & $1$ & $\mathbb{E}_7$ & $\frac{3}{10}$\\
\hline
  \end{longtable}
  \end{minipage}
  \end{maintheorem*}
\noindent {\bf Acknowledgments:} I am grateful to my supervisor Professor Ivan Cheltsov for the introduction to the topic and continuous support.  
\subsection{Applications.}  Let  $X$ be a del Pezzo surface of degree $2$ with at  most Du Val singularities. Let $S$ be a weak resolution of $X$. We will call an image on $X$ of a $(-1)$-curve in $S$  {\bf a line} as was done in \cite{CheltsovProkhorov21}. The immediate corollaries from Main Theorem are:
\begin{corollary}
Let $X$ be a Du Val del Pezzo surface of degree $2$ with at most $\DA_3$ singularities then $X$ is $K$-semi-stable.
\end{corollary}
\begin{proof}
    For such $X$ have $\delta(X)\ge 1$. Thus, $X$ is $K$-semi-stable by \cite[Theorem 1.59]{Fano21}.
\end{proof}

\begin{corollary}[\cite{OdakaSpottiSun16, GhigiKollár07}]
Let $X$ be a Du Val del Pezzo surface of degree $2$ with at most $\DA_2$ singularities then $X$ is $K$-stable.
\end{corollary}

\begin{proof}
    For such $X$ have $\delta(X)> 1$. Thus, $X$ is $K$-stable. This was studied in \cite{GhigiKollár07}.
\end{proof}

\begin{corollary}[\cite{CheltsovProkhorov21}]
Let $X$ be a Du Val del Pezzo surface of degree $2$ with at most $\DA_2$ singularities then $\Aut(X)$ is finite.
\end{corollary}

\begin{proof}
    $X$ is $K$-stable so by \cite[Corollary 1.3]{BlumXu19} $\Aut(X)$ is finite.
\end{proof}

\begin{corollary}[\cite{OdakaSpottiSun16}]
Let $X$ be a Du Val del Pezzo surface of degree $2$ with at least one $\DA_3$ singularity. Then $X$ is $K$-polystable if it has two $\DA_3$ singularities and $X$ is strictly $K$-semi-stable otherwise.
\end{corollary}

\begin{proof}
 Suppose $X$ has at least one $\DA_3$ singularity. By \cite{CheltsovProkhorov21} only surfaces of types $2\DA_3$ and $2\DA_3+\DA_1$ have infinite automorphism groups. Suppose $X$ has exactly two $\DA_3$ singularities, $G=\Aut(X)$. Then there are no $G$-invariant points that belong to lines on $X$ and for all $G$-invariant  points we have $\delta_P(X)>1$ so by \cite[Corollary 4.14]{Zhuang} $X$ is $K$-polystable. Any $X$ with exactly one $\DA_3$-singularity deforms isotrivially to this case. Thus, if $X$ has exactly one $\DA_3$-singularity it is strictly $K$-semi-stable by   \cite[Corollary 1.13]{Fano21}.  Suppose now that $X$ is of type $2\DA_3\DA_1$ and $G=\Aut(X)$ then the only $G$-invariant points that belong to lines on $X$ is a unique $\DA_1$ point so for all $G$-invariant points we have $\delta_P(X)>1$ and by \cite[Corollary 4.14]{Zhuang} $X$ is $K$-polystable. 
\end{proof}
\noindent There are also some applications in the case of threefolds. Smooth Fano threefolds over  $\DC$ were classified in \cite{Is77,Is78,MoMu81,MoMu03} into $105$ families. The detailed description of these families can be found in \cite{Fano21} where the  problem  to find all K-polystable smooth Fano threefolds in each family was posed. The output of this paper, give some alternative proofs for this problem as well as some proofs in case of singular Fano threefolds.
\subsubsection{Family 2.2}
Let $R$ be a  surface with isolated singularities of degree $(2, 4)$ in $\DP^1 \times\DP^2$, let $\pi: \mathbf{X}\to \DP^1 \times\DP^2$ be a double cover ramified over the surface $R$.  Let $pr_1:\DP^1 \times\DP^2\to \DP^1$ be the projection on the first factor. Set $p_1=pr_1\circ \pi$. Then $p_1$ is a fibration into del Pezzo surfaces of degree $2$. 
\begin{remark}
If $R$ is smooth then $\mathbf{X}$ is a smooth Fano threefold in Family  2.2. and all smooth Fano threefolds in this family can be obtained this way. Every smooth element in this family is known to be $K$-stable  \cite{CheltsovDenisovaFujita24}.
\end{remark}
\noindent Main Theorem gives the following corollary:
\begin{corollary}
 If every fiber $X$ of $p_1$ has at most $\DA_3$ singularities, then $\mathbf{X}$ is $K$-stable.
\end{corollary}
\begin{proof}
If $X$ is an irreducible fiber of $p_1$ then arguing as in the proof of  \cite[Theorem  10.1]{Fujita16} we have $S_{\mathbf{X}}(X)<1$.
We know (\cite{Fujita19,Li17}) that the Fano threefold $\mathbf{X}$ is $K$-stable if and only if for every prime divisor $\mathbf{E}$ over $\mathbf{X}$ we have
$$
\beta(\mathbf{E})=A_{\mathbf{X}}(\mathbf{E})-S_{\mathbf{X}}(\mathbf{E})>0
$$
 where $A_{\mathbf{X}}(\mathbf{E})$ is the~log discrepancy of the~divisor $\mathbf{E}$ and
$S_{\mathbf{X}}\big(\mathbf{E}\big)=\frac{1}{(-K_{\mathbf{X}})^3}\int\limits_0^{\infty}\mathrm{vol}\big(-K_{\mathbf{X}}-u\mathbf{E}\big)du.$
To show this, we fix a prime divisor $\mathbf{E}$ over~$\mathbf{X}$.
Then we set $Z=C_{\mathbf{X}}(\mathbf{E})$. Let $Q$ be a general point in $Z$. 
Following \cite{AbbanZhuang,Fano21} denote
$$
\delta_Q\big(X,W^{X}_{\bullet,\bullet}\big)=\inf_{\substack{F/X\\Q\in C_{X}(F)}}\frac{A_X(F)}{S\big(W^{X}_{\bullet,\bullet};F\big)}\text{ and }\delta_Q\big( \mathbf{X}\big)=\inf_{\substack{\mathbf{F}/\mathbf{X}\\Q\in C_{\mathbf{X}}(\mathbf{F})}}\frac{A_{\mathbf{X}}(\mathbf{F})}{S_{\mathbf{X}}(\mathbf{F})}
$$
where the first infimum is taken by all prime divisors $F$ over the surface $X$ whose center on $X$ contains $Q$ and the second infimum is taken by all prime divisors $\mathbf{F}$ over the threefold $\mathbf{X}$ whose center on $\mathbf{X}$ contains $Q$.
Then it follows from \cite{AbbanZhuang,Fano21} that
\begin{equation*}
\label{equation:AZ}%\tag{$\bigstar$}
\frac{A_{\mathbf{X}}(\mathbf{E})}{S_{\mathbf{X}}(\mathbf{E})}\geqslant\delta_Q(\mathbf{X}) \geqslant\min\Bigg\{\frac{1}{S_{\mathbf{X}}(X)},\delta_Q\big(X,W^{X}_{\bullet,\bullet}\big)\Bigg\}.
\end{equation*}
If $\beta(\mathbf{E})\leqslant 0$,  then
 $\delta_Q(X,W^{X}_{\bullet,\bullet})\leqslant 1$.
\\The divisor $-K_{\mathbf{X}}-uX$ is nef if and only if $u\le 1$ and the Zariski Decomposition is given by $P(u)=-K_{\mathbf{X}}-uX$ and $N(u)=0$ for $u\in[0,1]$.
By \cite[Corollary 1.110]{Fano21} for any divisor $F$ such that $Q\in C_{X}(F)$ over $X$ we get:
\begin{align*}
S\big(W^{X}_{\bullet,\bullet}&;F\big)=\frac{3}{(-K_{\mathbf{X}})^3}\Bigg(\int_0^\tau\big(P(u)^{2}\cdot X\big)\cdot\mathrm{ord}_{Q}\Big(N(u)\big\vert_{X}\Big)du+\int_0^\tau\int_0^\infty \mathrm{vol}\big(P(u)\big\vert_{X}-vF\big)dvdu\Bigg)=\\
&= \frac{3}{6}\int_0^\tau\int_0^\infty \mathrm{vol}\big(P(u)\big\vert_{X}-vF\big)dvdu=
\frac{3}{6}\int_0^1\int_0^\infty \mathrm{vol}\big(-K_{X}-vF\big)dvdu=\\
&= \frac{3}{6}\int_0^\infty \mathrm{vol}\big(-K_{X}-vF\big)dv=\frac{1}{2}\int_0^\infty \mathrm{vol}\big(-K_{X}-vF\big)dv= S_{X}(F)\le  \frac{A_{X}(F)}{\delta_Q(X)}
\end{align*}
We see that $\delta_Q(\mathbf{X})\ge \delta_Q(X)$. Thus, by Main Theorem if every fiber of $p_1$ has at most $\DA_3$ singularities, i.e. all the fibers are $K$-semi-stable, then $\delta(\mathbf{X})>1$ since if $Q$ is a singular point and $X$ is a fiber containing $Q$, such that  $\delta_Q(X)= 1$ then  $\frac{A_{\mathbf{X}}(\mathbf{E})}{S_{\mathbf{X}}(\mathbf{E})}>\min\Big\{\frac{1}{S_{\mathbf{X}}(X)},\delta_Q\big(X,W^{X}_{\bullet,\bullet}\big)\Big\}$ from \cite[Corollary 1.108.]{Fano21} and otherwise $\delta_Q(X)>1$  and the result follows.  
\end{proof}
\subsubsection{Family 1.12 (Del Pezzo Threefold of degree $2$)} 
Let $\psi:\mathbf{V}\to \DP^3$ be the double cover branched along the reduced possibly reducible quartic surface $R$. Set $H = \psi^*(\MO_{\DP^3} (1))$.
Then $\mathbf{V}$ is a del Pezzo threefold of degree $2$.  One can show that a general surface in $|H|$ is a smooth del Pezzo surface of degree $2$. 
\begin{remark}
If $R$ is smooth then $\mathbf{V}$ is a smooth Fano threefold in Family  1.12. and all smooth Fano threefolds in this family can be obtained this way. Every smooth element in this family is known to be $K$-stable  \cite{Fano21, Dervan16}.
\end{remark}
\noindent Singular Del Pezzo Threefolds of degree $2$ were studied in \cite{AscherDeVlemingLiu23}. It follows from  \cite{AscherDeVlemingLiu23, Shah81} that $\mathbf{V}$ is $K$-polystable if and only if the quartic surface $R$ is $GIT$-polystable with respect to natural action $PGL(4)$  except for those of the form $(x_0x_2+x_1^2+x_3^2)^2+a\cdot x_3^4$ for $a\in\DC$. Main Theorem gives the following (slightly weaker) corollary:
\begin{corollary}
If $R$ has $\mathbb{A}_n$-singulalrities, then $\mathbf{V}$ is $K$-stable.
\end{corollary}

\begin{proof}
Suppose $X$ is an irreducible element of $|H|$ then $S_{\mathbf{V}}(X)<1$. As for Family 2.2
we fix a prime divisor $\mathbf{E}$ over~$\mathbf{V}$.
Then we set $Z=C_{\mathbf{V}}(\mathbf{E})$ and  if $\beta(\mathbf{E})\leqslant 0$,  then
 $\delta_Q(X,W^{X}_{\bullet,\bullet})\leqslant 1$. Let $Q$ be a general point in $Z$, Let $X$ be the general element of $|H|$ that contains $Q$.  The divisor $-K_{\mathbf{V}}-uX$ is nef if and only if $u\le 2$ and the Zariski Decomposition is given by by $P(u)=-K_{\mathbf{V}}-uX\sim(2-u)X$ and $N(u)=0$ for $u\in[0,2]$.
By \cite[Corollary 1.110]{Fano21} for any divisor $F$ such that $Q\in C_{X}(F)$ over $X$ we get:
\begin{align*}
S\big(W^{X}_{\bullet,\bullet}&;F\big)=\frac{3}{(-K_{\mathbf{V}})^3}\Bigg(\int_0^\tau\big(P(u)^{2}\cdot X\big)\cdot\mathrm{ord}_{Q}\Big(N(u)\big\vert_{X}\Big)du+\int_0^\tau\int_0^\infty \mathrm{vol}\big(P(u)\big\vert_{X}-vF\big)dvdu\Bigg)=\\
&= \frac{3}{16}\int_0^\tau\int_0^\infty \mathrm{vol}\big(P(u)\big\vert_{X}-vF\big)dvdu=
\frac{3}{16}\int_0^2(2-u)^3\int_0^\infty\mathrm{vol}\big(-K_{X}-wF\big)dwdu=\\
&= \frac{3}{8}\int_0^2(2-u)^3\Big(\frac{1}{2}\int_0^\infty\mathrm{vol}\big(-K_{X}-wF\big)dw\Big)du=\frac{3}{8}\int_0^2(2-u)^3S_{X}(F)du= \frac{3}{2}S_{X}(F)\le \frac{3}{2} \frac{A_{X}(F)}{\delta_Q(X)}
\end{align*}
We get that $\delta_Q(\mathbf{V})\ge \frac{2}{3}\delta_Q(X)$. For $X$ with at most $\DA_1$-singularities we have $\delta_Q(X)\ge \frac{3}{2}$.  If $Q$ is a singular point and there exists an element $X$ of $|H|$ with  $\delta_Q(X)= \frac{3}{2}$ then  $\frac{A_{\mathbf{X}}(\mathbf{E})}{S_{\mathbf{X}}(\mathbf{E})}>\min\Big\{\frac{1}{S_{\mathbf{X}}(X)},\delta_Q\big(X,W^{X}_{\bullet,\bullet}\big)\Big\}$ from \cite[Corollary 1.108.]{Fano21} and otherwise we choose $X$ with   $\delta_Q(X)> \frac{3}{2}$ so $\delta_Q(\mathbf{V})>1$ if  $X$ has at most $\DA_1$-singularities. This is the case when $R$ has $\mathbb{A}_n$-singularities and the result follows.
\end{proof}

\subsubsection{Family 2.3} 
Let $\psi:\mathbf{V}\to \DP^3$ be the double cover branched along the reduced possibly reducible quartic surface $R$. Set $H = \psi^*(\MO_{\DP^3} (1))$. 
Let $S_1$ and $S_2$ be two distinct surfaces in the linear system $|H|$, and let $\mathcal{C} = S_1 \cap S_2$. Suppose that the curve $\mathcal{C}$ is smooth. Then $\mathcal{C}$ is an elliptic curve by the adjunction formula. Let $\pi : \mathbf{X}\to \mathbf{V}$ be the blow up of the curve $\mathcal{C}$, and let $E$ be the $\pi$-exceptional surface. We have the following commutative diagram:
$$\xymatrix{
&\mathbf{X}\ar[dl]_{\pi} \ar[dr]^{\phi}&\\
\mathbf{V}\ar@{-->}[rr]& &\DP^1
}$$
Where $\mathbf{V} \dashedrightarrow \DP^1$
is the rational map given by the pencil that is generated by $S_1$ and $S_2$,
and $\phi$ is a fibration into del Pezzo surfaces of degree $2$. 
\begin{remark}
If $R$ is smooth then $\mathbf{X}$ is a smooth Fano threefold in Family  2.3. and all smooth Fano threefolds in this family can be obtained this way. Every smooth  threefold in this family is known to be $K$-stable  \cite{CheltsovDenisovaFujita24}.
\end{remark}
\noindent Main Theorem gives the following corollary:
\begin{corollary}
    If  every fiber $X$ of $\phi$   at most $\DA_3$ singularities, then $\mathbf{X}$ is $K$-stable.
\end{corollary}
\begin{proof}
If $X$ is an irreducible fiber of $p_1$ then arguing as in the proof of  \cite[Theorem  10.1]{Fujita16} we have $S_{\mathbf{X}}(X)<1$. As for Family 2.2
we fix a prime divisor $\mathbf{E}$ over~$\mathbf{X}$.
Then we set $Z=C_{\mathbf{X}}(\mathbf{E})$. Let $Q$ be the point on $Z$. let $X$ be the fiber of $\phi$ that passes through $Q$.  Then $-K_{\mathbf{X}}-uX$ is nef if and only if $u\le 2$ and the Zariski Decomposition is given by 
$$P(u)=
\begin{cases}
    -K_{\mathbf{X}}-uX\sim (2-u)X+E\text{ if }u\in[0,1],\\
    -K_{\mathbf{X}}-uX-(u-1)E\sim(2-u)\pi^*(H)\text{ if }u\in[1,2],
\end{cases}
\text{ and }
N(u)=
\begin{cases}
    0\text{ if }u\in[0,1],\\
    (u-1)E\text{ if }u\in[1,2],
\end{cases}
$$
We apply Abban-Zhuang method to prove that $Q\not \in E\cong \mathcal{C}\times \DP^1$.
By \cite[Corollary 1.110]{Fano21} for any divisor $F$ such that $Q\in C_{X}(F)$ over $X$  we get:
\begin{align*}
S\big(W^{X}_{\bullet,\bullet}&;F\big)=\frac{3}{(-K_{\mathbf{X}})^3}\Bigg(\int_0^\tau\big(P(u)^{2}\cdot X\big)\cdot\mathrm{ord}_{Q}\Big(N(u)\big\vert_{X}\Big)du+\int_0^\tau\int_0^\infty \mathrm{vol}\big(P(u)\big\vert_{X}-vF\big)dvdu\Bigg)=\\
&= \frac{3}{8}\int_0^\tau\int_0^\infty \mathrm{vol}\big(P(u)\big\vert_{X}-vF\big)dvdu=\\
&= \frac{3}{8}\Bigg(\int_0^1\int_0^\infty \mathrm{vol}\big(-K_{X}-vF\big)dvdu+\int_1^2\int_0^\infty \mathrm{vol}\big(-K_{X}-(u-1)E|_X-vF\big)dvdu\Bigg)=\\
&= \frac{3}{8}\Bigg(\int_0^\infty \mathrm{vol}\big(-K_{X}-vF\big)dv+\int_1^2 (2-u)^3 \int_0^\infty \mathrm{vol}\big(-K_X-(u-1)E|_X-vF\big)dv\Bigg)=\\
&= \frac{3}{8}\Bigg(\int_0^\infty \mathrm{vol}\big(-K_{X}-vF\big)dv+
\int_1^2 (2-u)^3\int_0^\infty \mathrm{vol}\big(-K_{X}-wF\big)dw du\Bigg)=\\
&= \frac{3}{4}\Bigg(\frac{1}{2}\int_0^\infty \mathrm{vol}\big(-K_{X}-vF\big)dv + \int_1^2 (2-u)^3
\cdot \frac{1}{2} \int_0^\infty \mathrm{vol}\big(-K_{X}-wF\big)dw du\Bigg)=\\
&=\frac{3}{4}\Bigg(S_{X}(F) +  \frac{1}{4}\cdot S_{X}(F)\Bigg)=\frac{15}{16}S_{X}(F)\le \frac{15}{16}\cdot \frac{A_{X}(F)}{\delta_Q(X)}
\end{align*}
We see that $\delta_Q(\mathbf{X})\ge \frac{16}{15}\delta_Q(X)$.  
Thus, by Main Theorem if every fiber of $p_1$ has at most $\DA_3$ singularities the result follows.
\end{proof}

\section{Proof of Main Theorem via  Kento Fujita’s formulas}
\noindent In this work in order to find  $\delta$-invariants of Du Val del Pezzo surfaces we apply Abban--Zhuang theory and use Kento Fujita’s formulas. Let $X$ be a Du Val del Pezzo surface, and let $S$ be a minimal resolution of $X$. For a birational morphism  $f\colon\widetilde{X}\to X$ and $E$ be a prime divisor in $\widetilde{X}$ we say that $E$ is a prime divisor over $X$.
If~$E$~is $f$-exceptional, we say that $E$ is an~exceptional  prime divisor over~$X$.
We will denote the~subvariety $f(E)$ by $C_X(E)$. 
Let \index{$S_X(E)$}
$$
S_X(E)=\frac{1}{(-K_X)^2}\int_{0}^{\tau}\mathrm{vol}(f^*(-K_X)-vE)dv\text{ and }A_X (E) = 1 + \mathrm{ord}_E(K_{\widetilde{X}} - f^*(K_X)),
$$
where $\tau=\tau(E)$ is the~pseudo-effective threshold of $E$ with respect to $-K_X$.
% i.e. we have \index{pseudo-effective threshold}
% $$
% \tau(E)=\mathrm{sup}\Big\{ v \in \mathbb{Q}_{>0}\ \big\vert \ f^*(-K_X)-vE\ \text{is pseudo-effective}\Big\}.
% $$
% Let 
% $$A_X (E) = 1 + \mathrm{ord}_E(K_{\widetilde{X}} - f^*(K_X)).$$
% \index{$\beta$-invariant}
Let $Q$ be a point in $X$. We can define a local $\delta$-invariant and a global $\delta$-invariant of $X$ as
$$
\delta_Q(X)=\inf_{\substack{E/X\\ Q\in C_X(E)}}\frac{A_X(E)}{S_X(E)}\text{ and }\delta(X)=\inf_{Q\in X}\delta_Q(X)
$$
where the~infimum runs over all prime divisors $E$ over the surface $X$ such that $Q\in C_X(E)$. Similarly, for the  surface $S$ and a point $P$ on $S$ we define local $\delta$-invariant and a global $\delta$-invariant of $S$ as
$$
\delta_P(S)=\inf_{\substack{F/S\\ P\in C_S(F)}}\frac{A_S(F)}{S_S(F)}
\text{
and }\delta(S)=\inf_{P\in S}\delta_P(S)$$
where $S_S(F)$ and $A_S(F)$ are defined as $S_X(E)$ and $A_X(E)$ above. It is clear that
$$\delta(X)=\delta(S)\text{ and }\delta_Q(X)=\inf_{P: Q=f(P)}\delta_P(S)$$
We now fix a point $P$ on $S$ and choose a smooth curve $A$ on $S$ containing $P$. 
Set
$$
\tau(A)=\mathrm{sup}\Big\{v\in\mathbb{R}_{\geqslant 0}\ \big\vert\ \text{the divisor  $-K_S-vA$ is pseudo-effective}\Big\}.
$$
For~$v\in[0,\tau]$, let $P(v)$ be the~positive part of the~Zariski decomposition of the~divisor $-K_S-vA$,
and let $N(v)$ be its negative part. 
Then we set $$
S\big(W^A_{\bullet,\bullet};P\big)=\frac{2}{K_S^2}\int_0^{\tau(A)} h(v) dv,
\text{ where }
h(v)=\big(P(v)\cdot A\big)\times\big(N(v)\cdot A\big)_P+\frac{\big(P(v)\cdot A\big)^2}{2}.
$$
It follows from {\cite[Theorem 1.7.1]{Fano21}} that:
\begin{equation}\label{estimation1}
    \delta_P(S)\geqslant\mathrm{min}\Bigg\{\frac{1}{S_S(A)},\frac{1}{S(W_{\bullet,\bullet}^A,P)}\Bigg\}.
\end{equation}
Unfortunately, this approach does not always give us a good estimation. If this is the case, we apply the generalization of this method. Let $\sigma: \widehat{S}\to S$ be a weighted blowup of the point $P$ on $S$. Suppose, in addition, that $\widehat{S}$ is a Mori Dream Space Then
\begin{itemize}
\item the~$\sigma$-exceptional curve $E_P\cong \DP^1$ such that $\sigma(E_P)=P$,
\item the~log pair $(\widehat{S},E_P)$ has purely log terminal singularities.
\end{itemize}
We write
$$
K_{E_P}+\Delta_{E_P}=\big(K_{\widehat{S}}+E_P\big)\big\vert_{E_P},
$$
where $\Delta_{E_P}$ is an~effective $\mathbb{Q}$-divisor on $E_P$ known as the~different of the~log pair $(\widehat{S},E_P)$.
Note that the~log pair $(E_P,\Delta_{E_P})$ has at most Kawamata log terminal singularities, and the~divisor $-(K_{E_P}+\Delta_{E_P})$ is $\sigma\vert_{E_P}$-ample.
\\Let $O$ be a point on $E_P$. 
Set
$$
\tau(E_P)=\mathrm{sup}\Big\{v\in\mathbb{R}_{\geqslant 0}\ \big\vert\ \text{the divisor  $\sigma^*(-K_S)-vE_P$ is pseudo-effective}\Big\}.
$$
For~$v\in[0,\tau]$, let $\widehat{P}(v)$ be the~positive part of the~Zariski decomposition of the~divisor $\sigma^*(-K_S)-vE_P$,
and let $\widehat{N}(v)$ be its negative part. 
Then we set $$
S\big(W^{E_P}_{\bullet,\bullet};O\big)=\frac{2}{K_{\widehat{S}}^2}\int_0^{\tau(E_P)} \widehat{h}(v) dv,
\text{ where }
\widehat{h}(v)=\big(\widehat{P}(v)\cdot E_P\big)\times\big(\widehat{N}(v)\cdot E_P\big)_O+\frac{\big(\widehat{P}(v)\cdot E_P\big)^2}{2}.
$$
Let
$A_{E_P,\Delta_{E_P}}(O)=1-\mathrm{ord}_{\Delta_{E_P}}(O)$.
It follows from {\cite[Theorem 1.7.9]{Fano21}} and {\cite[Corollary 1.7.12]{Fano21}} that
\begin{equation}
\label{estimation2}
\delta_P(S)\geqslant\mathrm{min}\Bigg\{\frac{A_S(E_P)}{S_S(E_P)},\inf_{O\in E_P}\frac{A_{E_P,\Delta_{E_P}}(O)}{S\big(W^{E_P}_{\bullet,\bullet};O\big)}\Bigg\},
\end{equation}
where the~infimum is taken over all points $O\in E_P$. Now for all the points $P$ on $S$ we know either values of local $\delta$-invariants or estimations of them. Taking the minimum we compute $\delta(S)$ --- the  global  $\delta$-invariant of $S$ and thus, $\delta(X)=\delta(S)$ --- the  global  $\delta$-invariant of $X$. We apply this method to minimal resolutions of Du Val del Pezzo surfaces to prove the Main Theorem. Throughout this work small circles correspond to a $(-1)$-curves and large circles correspond to a $(-2)$-curves on dual graphs.
\section{Du Val del Pezzo Surfaces of Degree $2$}
\noindent In \cite[Lemma 2.15]{Fano21} it was proven that $\delta(X)=\frac{9}{5}$ when $X$ is a smooth del Pezzo surface of degree $2$ and $|-K_X|$ contains a tacnodal curve, and $\delta(X)=\frac{15}{8}$ when $X$ is a smooth del Pezzo surface of degree $2$ and $|-K_X|$ does not contain a tacnodal curve.  In this section we compute  $\delta$-invariants  of Du Val del Pezzo surfaces of degree $2$.  
\begin{maintheorem*}
Let $X$ be the Du Val del Pezzo surface of degree $2$. Then the  $\delta$-invariant of $X$ is uniquely determined by the degree of $X$, the number of lines on $X$, and the type of singularities on $X$ which is given in the following table:\\
{\begin{minipage}{5.7cm}
 \renewcommand{\arraystretch}{1.2}
  \begin{longtable}{ | c | c | c | c | }
   \hline
   $K_X^2$ & $\#$ lines & $\mathrm{Sing}(X)$ & $\delta(X)$\\
  \hline
\endhead 
\hline
 $2$ & $44$ & $\DA_1$ & $\frac{3}{2}$\\
\hline
 $2$ & $34$ & $2\DA_1$ & $\frac{3}{2}$\\
\hline
 $2$ & $26$ & $3\DA_1$ & $\frac{3}{2}$\\
\hline
$2$ & $25$ & $3\DA_1$ & $\frac{3}{2}$\\
\hline
$2$ & $20$ & $4\DA_1$ & $\frac{3}{2}$\\
\hline
$2$ & $19$ & $4\DA_1$ & $\frac{3}{2}$\\
\hline
$2$ & $14$ & $5\DA_1$ & $\frac{3}{2}$\\
\hline
$2$ & $10$ & $6\DA_1$ & $\frac{3}{2}$\\
\hline
 $2$ & $31$ & $\DA_2$ & $\frac{6}{5}$\\
\hline
 $2$ & $20$ & $\DA_2+\DA_1$ & $\frac{6}{5}$\\
\hline
$2$ & $18$ & $\DA_2+2\DA_1$ & $\frac{6}{5}$\\
\hline
$2$ & $13$ & $\DA_2+3\DA_1$ & $\frac{6}{5}$\\
\hline
$2$ & $16$ & $2\DA_2$ & $\frac{6}{5}$\\
\hline
$2$ & $12$ & $2\DA_2+\DA_1$ & $\frac{6}{5}$\\
\hline
$2$ & $8$ & $3\DA_2$ & $\frac{6}{5}$\\
\hline
  \end{longtable}
  \end{minipage}
  \begin{minipage}{6.7cm}
 \renewcommand{\arraystretch}{1.2}
  \begin{longtable}{ | c | c | c | c | }
   \hline
   $K_X^2$ & $\#$ lines & $\mathrm{Sing}(X)$ & $\delta(X)$\\
  \hline
\endhead 
\hline
  $2$ & $22$ & $\DA_3$ & $1$\\
  \hline
   $2$ & $16$ & $\DA_3+\DA_1$ & $1$\\
\hline
$2$ & $15$ & $\DA_3+\DA_1$ & $1$\\
\hline
$2$ & $12$ & $\DA_3+2\DA_1$ & $1$\\
\hline
$2$ & $11$ & $\DA_3+2\DA_1$ & $1$\\
\hline
$2$ & $8$ & $\DA_3+3\DA_1$ & $1$\\
\hline
$2$ & $10$ & $\DA_3+\DA_2$ & $1$\\
\hline
  $2$ & $7$ & $\DA_3+\DA_2+\DA_1$ & $1$\\
\hline
$2$ & $6$ & $2\DA_3$ & $1$\\
\hline
$2$ & $4$ & $2\DA_3+\DA_1$ & $1$\\
\hline
$2$ & $14$ & $\DA_4$ & $\frac{12}{13}$\\
\hline
$2$ & $10$ & $\DA_4+\DA_1$ & $\frac{12}{13}$\\
\hline
$2$ & $6$ & $\DA_4+\DA_2$ & $\frac{12}{13}$\\
\hline
$2$ & $8$ & $\DA_5$ & $\frac{6}{7}$\\
\hline
$2$ & $7$ & $\DA_5$ & $\frac{3}{4}$\\
\hline
  \end{longtable}
  \end{minipage}
  \begin{minipage}{3.5cm}
 \renewcommand{\arraystretch}{1.2}
    \begin{longtable}{ | c | c | c | c | }
   \hline
   $K_X^2$ & $\#$ lines & $\mathrm{Sing}(X)$ & $\delta(X)$\\
  \hline\hline
  \endhead 
$2$ & $6$ & $\DA_5+\DA_1$ & $\frac{6}{7}$\\
\hline
$2$ & $5$ & $\DA_5+\DA_1$ & $\frac{3}{4}$\\
\hline
$2$ & $3$ & $\DA_5+\DA_2$ & $\frac{3}{4}$\\
\hline
$2$ & $4$ & $\DA_6$ & $\frac{4}{5}$\\
\hline
$2$ & $2$ & $\DA_7$ & $\frac{3}{4}$\\
\hline
$2$ & $14$ & $\mathbb{D}_4$ & $\frac{3}{4}$\\
\hline
$2$ & $9$ & $\mathbb{D}_4+\DA_1$ & $\frac{3}{4}$\\
\hline
$2$ & $6$ & $\mathbb{D}_4+2\DA_1$ & $\frac{3}{4}$\\
\hline
$2$ & $4$ & $\mathbb{D}_4+3\DA_1$ & $\frac{3}{4}$\\
\hline
$2$ & $8$ & $\mathbb{D}_5$ & $\frac{3}{5}$\\
\hline
$2$ & $5$ & $\mathbb{D}_5+\DA_1$ & $\frac{3}{5}$\\
\hline
$2$ & $3$ & $\mathbb{D}_6$ & $\frac{1}{2}$\\
\hline
$2$ & $2$ & $\mathbb{D}_6+\DA_1$ & $\frac{1}{2}$\\
\hline
$2$ & $4$ & $\mathbb{E}_6$ & $\frac{3}{7}$\\
\hline
$2$ & $1$ & $\mathbb{E}_7$ & $\frac{3}{10}$\\
\hline
  \end{longtable}
  \end{minipage}}
  \\
  \end{maintheorem*}
\subsection{General results for degree $2$}
\noindent Let  $X$ be a del Pezzo surface of degree $2$ with at most Du Val singularities. Let $S$ be a weak resolution of $X$. We will call an image of a $(-1)$-curve in $S$ on $X$ {\bf a line} as was done in \cite{CheltsovProkhorov21}.
\begin{lemma}\label{deg2-genpoint}
    Assume that the point $Q$ is not contained in any line that passes through a singular point of $X$. Then  $\delta_{Q}(X)\ge \frac{9}{5}$.
\end{lemma}
\begin{proof}
Follows from  Remark {\cite[2.14]{Fano21}} and proof   of Lemma {\cite[2.15]{Fano21}}.
\end{proof}
% \noindent Now we consider a curve $A$ on $S$. Small circles correspond to a $(-1)$-curves and large circles correspond to a $(-2)$-curves on dual graphs.
\begin{lemma}\label{deg2-4526-nearA1A5}
Suppose $P$ belongs to a $(-1)$-curve $A$ and there exist $(-1)$-curves and $(-2)$-curves   which form the following dual graph:
\begin{figure}[h!]
    \centering
\includegraphics[width=7cm]{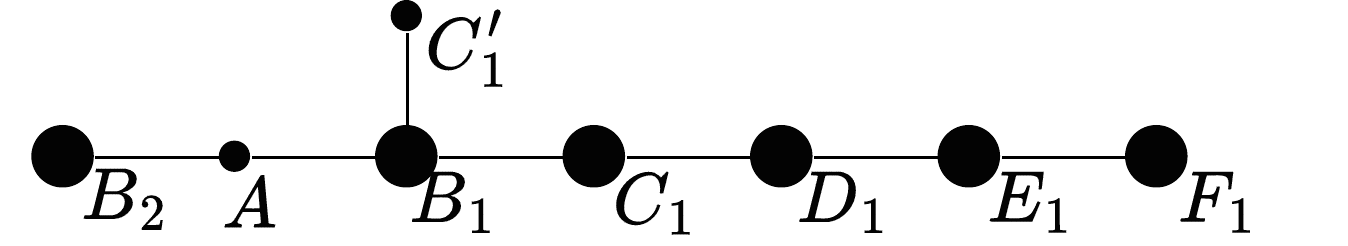}
    \caption{Dual graph: $(-K_S)^2=2$ and $\delta_P(S)=\frac{45}{16}$}
\end{figure}
\par 
Then $\tau(A)=\frac{4}{3}$ and the Zariski Decomposition of the divisor $-K_S-vA$ is given by:
{{\allowdisplaybreaks\begin{align*}
\hspace*{-0.5cm}&&P(v)=\begin{cases}
-K_S-vA-\frac{v}{6}(5B_1+4C_1+3D_1+2E_1+F_1+3B_2)\text{ if }v\in\big[0,\frac{6}{5}\big],\\
-K_S-vA-\frac{v}{2}B_2-(v-1)(5B_1+4C_1+3D_1+2E_1+F_1)-(3v-4)C_1'\text{ if }v\in\big[\frac{6}{5},\frac{4}{3}\big].
\end{cases}\\
\hspace*{-0.5cm}&&N(v)=\begin{cases}\frac{v}{6}(5B_1+4C_1+3D_1+2E_1+F_1+3B_2)\text{ if }v\in\big[0,\frac{6}{5}\big],\\
\frac{v}{2}B_2+(v-1)(5B_1+4C_1+3D_1+2E_1+F_1)+(3v-4)C_1'\text{ if }v\in\big[\frac{6}{5},\frac{4}{3}\big].
\end{cases}
\end{align*}}}
Moreover, 
$$(P(v))^2=\begin{cases}
2 - 2v+ \frac{v^2}{3}\text{ if }v\in\big[0,\frac{6}{5}\big],\\
\frac{(4-3v)^2}{2}\text{ if }v\in\big[\frac{6}{5},\frac{4}{3}\big].
\end{cases}P(v)\cdot A=\begin{cases}
1-\frac{v}{3}\text{ if }v\in\big[0,\frac{6}{5}\big],\\
3(2-\frac{3v}{2})\text{ if }v\in\big[\frac{6}{5},\frac{4}{3}\big].
\end{cases}$$
In this case $\delta_P(S)=\frac{45}{26}\text{ if }P\in A\backslash (B_1\cup B_2)$.
\end{lemma}

\begin{proof}
The Zariski Decomposition  follows from 
 $$-K_S-vA\sim_{\DR} \Big(\frac{4}{3}-v\Big)A+\frac{1}{3}\Big(2B_2+5B_1+4C_1+3D_1+2E_1+F_1+2C_1'\Big).$$
We have $S_S(A)=\frac{26}{45}$. Thus, $\delta_P(S)\le \frac{45}{26}$ for $P\in E_2$. Note that for $P\in A\backslash (B_1\cup B_2)$ we have:
$$h(v)\le \begin{cases} 
 \frac{(3 - v)^2}{18}\text{ if }v\in\big[0,\frac{6}{5}\big],\\
\frac{9 (4 - 3 v)^2}{8}\text{ if }v\in\big[\frac{6}{5},\frac{4}{3}\big].
\end{cases}$$
So 
$S(W_{\bullet,\bullet}^{A};P)\le\frac{2}{5}<\frac{26}{45}$. Thus, $\delta_P(S)=\frac{45}{26}$ if $P\in A\backslash (B_1\cup B_2)$.
\end{proof}
\begin{lemma}\label{deg2-2-near2A1points}
Suppose $P$ belongs to a $(-1)$-curve $A$ and there exist $(-1)$-curves and $(-2)$-curves   which form the following dual graph and no other $(-2)$-curves intersect $A$:
\begin{figure}[h!]
    \centering
\hspace*{-0.5cm}\includegraphics[width=17cm]{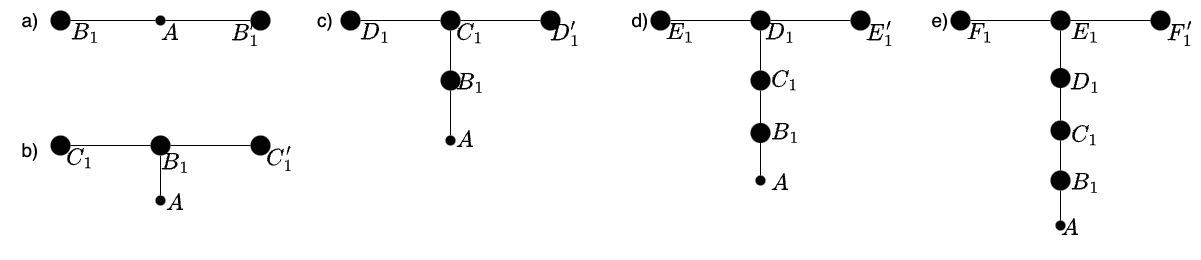}
    \caption{Dual graph: $(-K_S)^2=2$ and $\delta_P(S)=2$}
\end{figure}
\par Then $\tau(A)=1$ and the Zariski Decomposition of the divisor $-K_S-vA$ is given by:
{\allowdisplaybreaks\begin{align*}
&{\text{\bf a). }}& P(v)=-K_S-vA-\frac{v}{2}(B_1+B_1')\text{ if }v\in[0,1].\\&&N(v)=\frac{v}{2}(B_1+B_1')\text{ if }v\in[0,1].\\
&{\text{\bf b). }}&P(v)=-K_S-vA-\frac{v}{2}(2B_1+C_1+C_1')\text{ if }v\in[0,1].\\&&N(v)=\frac{v}{2}(2B_1+C_1+C_1')\text{ if }v\in[0,1].\\
&{\text{\bf c). }}&P(v)=-K_S-vA-\frac{v}{2}(2B_1+2C_1+D_1+D_1')\text{ if }v\in[0,1].\\&&N(v)=\frac{v}{2}(2B_1+2C_1+D_1+D_1')\text{ if }v\in[0,1].\\
&{\text{\bf d). }}&P(v)=-K_S-vA-\frac{v}{2}(2B_1+2C_1+2D_1+E_1+E_1')\text{ if }v\in[0,1].\\&&N(v)=\frac{v}{2}(2B_1+2C_1+2D_1+E_1+E_1')\text{ if }v\in[0,1].\\
&{\text{\bf e). }}&P(v)=-K_S-vA-\frac{v}{2}(2B_1+2C_1+2D_1+2E_1+F_1+F_1')\text{ if }v\in[0,1].\\&&N(v)=\frac{v}{2}(2B_1+2C_1+2D_1+2E_1+F_1+F_1')\text{ if }v\in[0,1].
\end{align*}}
Moreover, 
$$(P(v))^2=2(1 - v )\text{ and }P(v)\cdot A=1 \text{ if }v\in[0,1]. $$
In this case: $\delta_P(S)=2\text{ if }P\in A\backslash (B_1\cup B_1').$
\end{lemma}
\begin{proof}
 The Zariski Decomposition in part a). follows from $-K_S-vA\sim (1-v)A+A'+B_1+B_2$ where $A'$ is the image of $A$ under Geiser involution. Similar statement holds in other parts. We have $S_S(A)=\frac{1}{2}.$ Thus, $\delta_P(S)\le 2$ for $P\in A$. Note that for $P\in A\backslash (B_1\cup B_1')$ we have
$h(v) = 1/2\text{ if }v\in[0,1].$
So for these points  
$S(W_{\bullet,\bullet}^{A};P)= \frac{1}{2}.$ Thus, $\delta_P(S)=2$ if $P\in A\backslash (B_1\cup B_1')$.
\end{proof}
\begin{lemma}\label{deg2-158-nearA1A2points}
Suppose $P$ belongs to a $(-1)$-curve $A$ and there exist $(-1)$-curves and $(-2)$-curves   which form the following dual graph:
\begin{figure}[h!]
    \centering
\includegraphics[width=12cm]{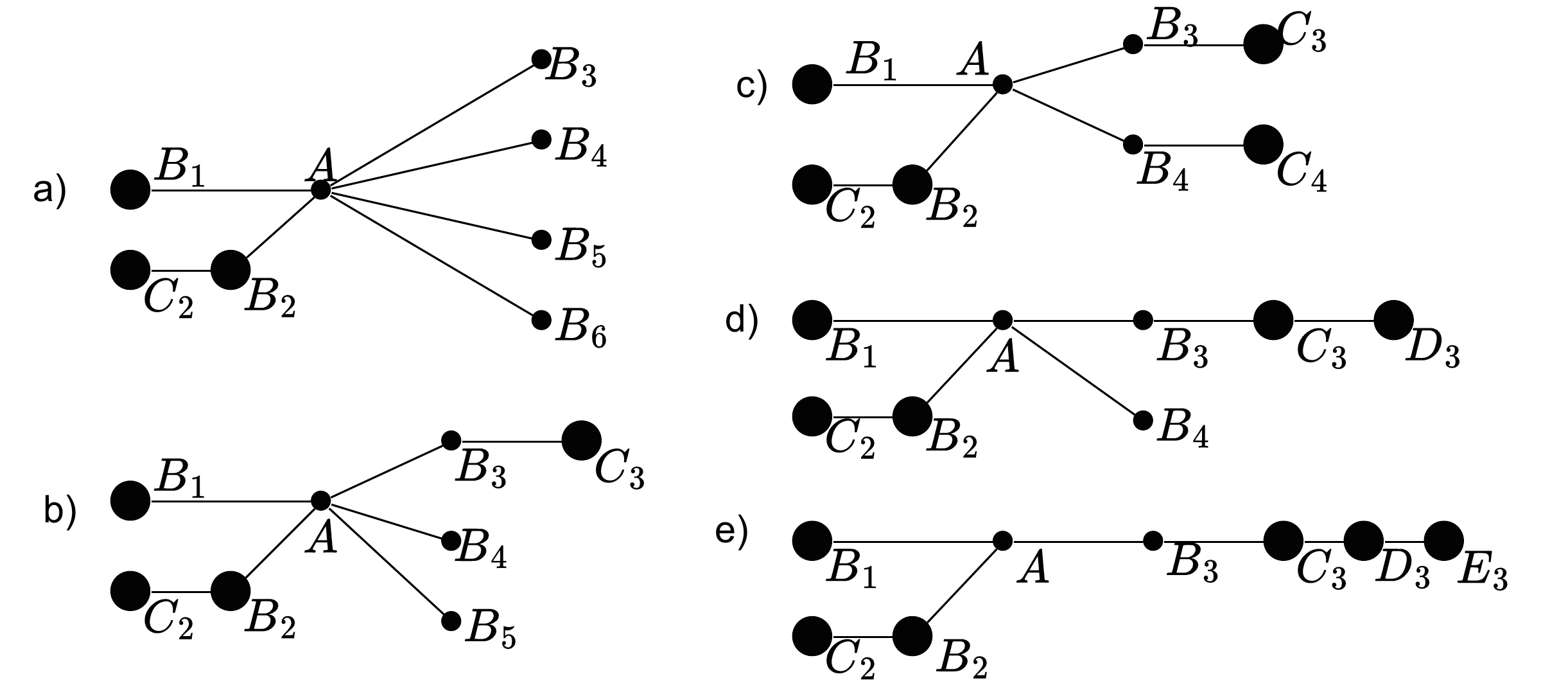}
    \caption{Dual graph: $(-K_S)^2=2$ and $\delta_P(S)=\frac{15}{8}$}
\end{figure}
\par Then $\tau(A)=\frac{6}{5}$ and the Zariski Decomposition of the divisor $-K_S-vA$ is given by:
{
{\allowdisplaybreaks\begin{align*}   
&{\text{\bf a). }}&P(v)=\begin{cases}
-K_S-vA-\frac{v}{6}(3B_1+4B_2+2C_2)\text{ if }v\in[0,1],\\
-K_S-vA-\frac{v}{6}(3B_1+4B_2+2C_2)-(v-1)(B_3+B_4+B_5+B_6)\text{ if }v\in\big[1, \frac{6}{5}\big].
\end{cases}\\
&&N(v)=\begin{cases}
\frac{v}{6}(3B_1+4B_2+2C_2)\text{ if }v\in[0,1],\\
\frac{v}{6}(3B_1+4B_2+2C_2)+(v-1)(B_3+B_4+B_5+B_6)\text{ if }v\in\big[1, \frac{6}{5}\big].
\end{cases}\\
&{\text{\bf b). }}&P(v)=\begin{cases}
-K_S-vA-\frac{v}{6}(3B_1+4B_2+2C_2)\text{ if }v\in[0,1],\\
-K_S-vA-\frac{v}{6}(3B_1+4B_2+2C_2)-(v-1)(2B_3+C_3+B_4+B_5)\text{ if }v\in\big[1, \frac{6}{5}\big].
\end{cases}\\
&&N(v)=\begin{cases}
\frac{v}{6}(3B_1+4B_2+2C_2)\text{ if }v\in[0,1],\\
\frac{v}{6}(3B_1+4B_2+2C_2)+(v-1)(2B_3+C_3+B_4+B_5)\text{ if }v\in\big[1, \frac{6}{5}\big].
\end{cases}\\
&{\text{\bf c). }}&P(v)=\begin{cases}
-K_S-vA-\frac{v}{6}(3B_1+4B_2+2C_2)\text{ if }v\in[0,1],\\
-K_S-vA-\frac{v}{6}(3B_1+4B_2+2C_2)-(v-1)(2B_3+C_3+2B_4+C_4)\text{ if }v\in \big[1,\frac{6}{5}\big].
\end{cases}\\
&&N(v)=\begin{cases}
\frac{v}{6}(3B_1+4B_2+2C_2)\text{ if }v\in[0,1],\\
\frac{v}{6}(3B_1+4B_2+2C_2)+(v-1)(2B_3+C_3+2B_4+C_4)\text{ if }v\in \big[1,\frac{6}{5}\big].
\end{cases}\\
&{\text{\bf d). }}&P(v)=\begin{cases}
-K_S-vA-\frac{v}{6}(3B_1+4B_2+2C_2)\text{ if }v\in[0,1],\\
-K_S-vA-\frac{v}{6}(3B_1+4B_2+2C_2)-(v-1)(3B_3+2C_3+D_3+B_4)\text{ if }v\in \big[1,\frac{6}{5}\big].
\end{cases}\\
&&N(v)=\begin{cases}
\frac{v}{6}(3B_1+4B_2+2C_2)\text{ if }v\in[0,1],\\
\frac{v}{6}(3B_1+4B_2+2C_2)+(v-1)(3B_3+2C_3+D_3+B_4)\text{ if }v\in \big[1,\frac{6}{5}\big].
\end{cases}\\
&{\text{\bf e). }}&P(v)=\begin{cases}
-K_S-vA-\frac{v}{6}(3B_1+4B_2+2C_2)\text{ if }v\in[0,1],\\
-K_S-vA-\frac{v}{6}(3B_1+4B_2+2C_2)-(v-1)(4B_3+3C_3+2D_3+E_3)\text{ if }v\in \big[1,\frac{6}{5}\big].
\end{cases}\\
&&N(v)=\begin{cases}
\frac{v}{6}(3B_1+4B_2+2C_2)\text{ if }v\in[0,1],\\
\frac{v}{6}(3B_1+4B_2+2C_2)+(v-1)(4B_3+3C_3+2D_3+E_3)\text{ if }v\in \big[1,\frac{6}{5}\big].
\end{cases}\\
\end{align*}}
}
Moreover, 
$$(P(v))^2=\begin{cases}
2 - 2v + \frac{v^2}{6} \text{ if }v\in [0,1],\\
\frac{(5v - 6)^2}{6}\text{ if }v\in\big[1, \frac{6}{5}\big].
\end{cases}P(v)\cdot A=\begin{cases}1-\frac{v}{6}\text{ if }v\in[0,1],\\
5(1 - \frac{5v}{6})\text{ if }v\in\big[1, \frac{6}{5}\big].
\end{cases}$$
In this case for $P\in A\backslash (B_1\cup B_2)$ we have: 
$\delta_P(S)=\frac{15}{8}.$
\end{lemma}
\begin{proof}
 The Zariski Decomposition in part a). follows from $$-K_S-vA\sim_{\DR} \Big(\frac{6}{5}-v\Big)A+\frac{3}{5}B_1+\frac{2}{5}\Big(2B_2+C_2\Big)+ \frac{1}{5}\Big(B_3+B_4+B_5+B_6\Big).$$ A similar statement holds in other parts.  We have $S_S(A)=\frac{8}{15}.$ Thus, $\delta_P(S)\le \frac{15}{8}$ for $P\in A$. Note that for $P\in A\backslash (B_1\cup B_2)$ we have:
$$h(v)\le \begin{cases}\frac{(6-v)^2}{72}\text{ if }v\in[0,1],\\
\frac{5(6 -5v)(23v - 18)}{72}\text{ if }v\in\big[1, \frac{6}{5}\big].
\end{cases}$$
So we have 
$S(W_{\bullet,\bullet}^{A};P)\le\frac{7}{15}<\frac{8}{15}$. Thus, $\delta_P(S)=\frac{15}{8}$ if $P\in A\backslash (B_1\cup B_2)$.
\end{proof}
\begin{lemma}\label{deg2-2413-nearA4points}
Suppose $P$ belongs to a $(-1)$-curve $A$ and there exist $(-1)$-curves and $(-2)$-curves   which form the following dual graph:
\begin{figure}[h!]
    \centering
\includegraphics[width=16cm]{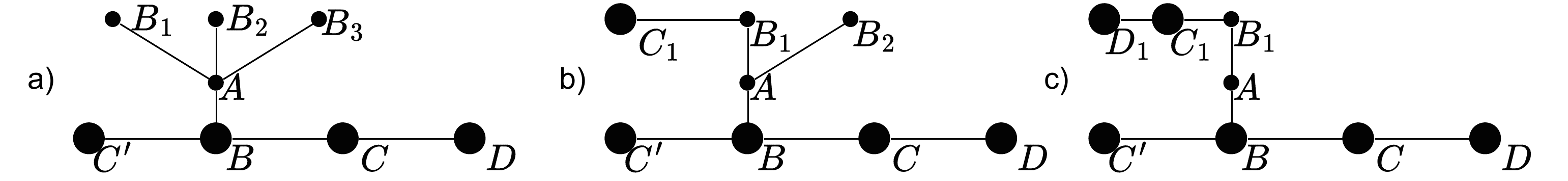}
    \caption{Dual graph: $(-K_S)^2=2$ and $\delta_P(S)=\frac{24}{13}$}
\end{figure}
\par  Then $\tau(A)=\frac{5}{4}$ and the Zariski Decomposition of the divisor $-K_S-vA$ is given by:
{ {\allowdisplaybreaks\begin{align*}
&{\text{\bf a). }}&P(v)=\begin{cases}
-K_S-vA-\frac{v}{5}(3C'+6B+4C+2D)\text{ if }v\in[0,1],\\
-K_S-vA-\frac{v}{5}(3C'+6B+4C+2D)-(v-1)(B_1+B_2+B_3)\text{ if }v\in\big[1,\frac{5}{4}\big].
\end{cases}\\
&&
N(v)=\begin{cases}
\frac{v}{5}(3C'+6B+4C+2D)\text{ if }v\in[0,1],\\
\frac{v}{5}(3C'+6B+4C+2D)+(v-1)(B_1+B_2+B_3)\text{ if }v\in\big[1,\frac{5}{4}\big].
\end{cases}\\
&{\text{\bf b). }}&P(v)=\begin{cases}
-K_S-vA-\frac{v}{5}(3C'+6B+4C+2D)\text{ if }v\in[0,1],\\
-K_S-vA-\frac{v}{5}(3C'+6B+4C+2D)-(v-1)(2B_1+C_1+B_2)\text{ if }v\in\big[1,\frac{5}{4}\big].
\end{cases}\\
&&
N(v)=\begin{cases}
\frac{v}{5}(3C'+6B+4C+2D)\text{ if }v\in[0,1],\\
\frac{v}{5}(3C'+6B+4C+2D)+(v-1)(2B_1+C_1+B_2)\text{ if }v\in\big[1,\frac{5}{4}\big].
\end{cases}\\
&{\text{\bf c). }}&P(v)=\begin{cases}
-K_S-vA-\frac{v}{5}(3C'+6B+4C+2D)\text{ if }v\in[0,1],\\
-K_S-vA-\frac{v}{5}(3C'+6B+4C+2D)-(v-1)(3B_1+2C_1+D_1)\text{ if }v\in\big[1,\frac{5}{4}\big].
\end{cases}\\
&&
N(v)=\begin{cases}
\frac{v}{5}(3C'+6B+4C+2D)\text{ if }v\in[0,1],\\
\frac{v}{5}(3C'+6B+4C+2D)+(v-1)(3B_1+2C_1+D_1)\text{ if }v\in\big[1,\frac{5}{4}\big].
\end{cases}
\end{align*}}}
Moreover, 
$$(P(v))^2=\begin{cases}
2 - 2v +\frac{v^2}{5} \text{ if }v\in[0,1],\\
\frac{(5-4v)^2}{5}\text{ if }v\in\big[1,\frac{5}{4}\big].
\end{cases}P(v)\cdot A=\begin{cases}1-\frac{v}{5}\text{ if }v\in[0,1],\\
4(1-\frac{4v}{5})\text{ if }v\in\big[1,\frac{5}{4}\big].
\end{cases}$$
In this case $\delta_P(S)=\frac{24}{13}\text{ if }P\in A\backslash B.$
\end{lemma}

\begin{proof}
In part a). the Zariski Decomposition  follows from 
 $$-K_S-vA\sim_{\DR} \Big(\frac{5}{4}-v\Big)A+\frac{1}{4}\Big(3C'+6B+4C+2D+B_1+B_2+B_3\Big).$$
A similar statement holds in other parts. We have
$S_S(A)=\frac{13}{24}.$
Thus, $\delta_P(S)\le \frac{24}{13}$ for $P\in A$. Note that for $P\in  A\backslash B$ we have:
$$h(v)\le \begin{cases}
\frac{ (5 - v)^2}{50}\text{ if }v\in[0,1],\\
 \frac{4 (5 - 4 v) (7v-5)}{25}\text{ if }v\in\big[1,\frac{5}{4}\big].
\end{cases}$$
So $S(W_{\bullet,\bullet}^{A};P)\le\frac{11}{24}<\frac{13}{24}.$
Thus, $\delta_P(S)=\frac{24}{13}$ if $P\in A\backslash B$.
\end{proof}
\begin{lemma}\label{deg2-95-nearA1A3}
Suppose $P$ belongs to a $(-1)$-curve $A$ and there exist $(-1)$-curves and $(-2)$-curves   which form the following dual graph:
\begin{figure}[h!]
    \centering
\includegraphics[width=8cm]{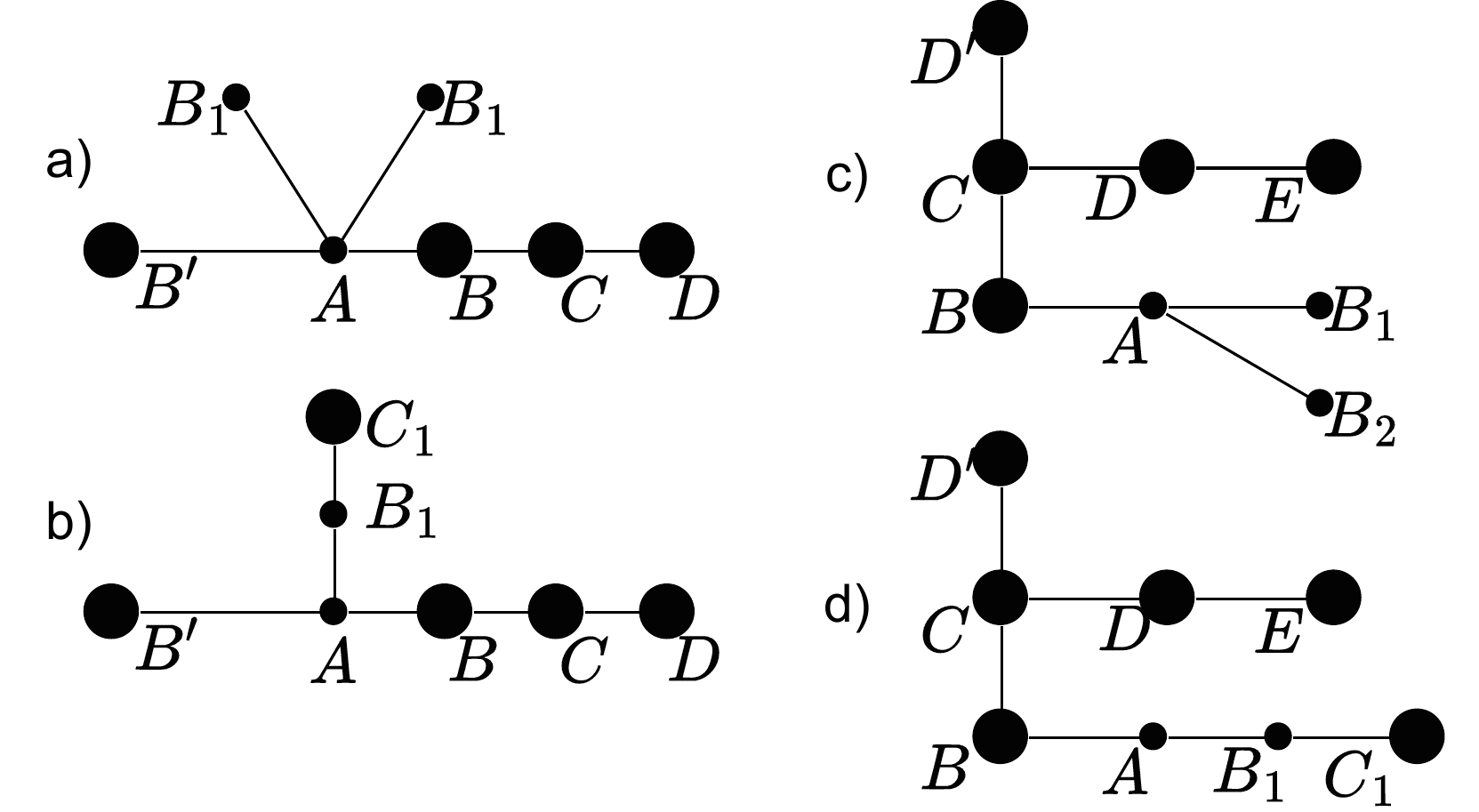}
    \caption{Dual graph: $(-K_S)^2=2$ and $\delta_P(S)=\frac{9}{5}$}
\end{figure}
\par Then $\tau(A)=\frac{4}{3}$ and the Zariski Decomposition of the divisor $-K_S-vA$ is given by:
{{\allowdisplaybreaks\begin{align*}
&{\text{\bf a). }}&
P(v)=\begin{cases}
-K_S-vA-\frac{v}{4}(3B_2+2C_2+D_2+2B_1)\text{ if }v\in[0,1],\\
-K_S-vA-\frac{v}{4}(3B_2+2C_2+D_2+2B_1)-(v-1)(B_1+B_2)\text{ if }v\in\big[1,\frac{4}{3}\big].
\end{cases}\\
&&N(v)=\begin{cases}
\frac{v}{4}(3B_2+2C_2+D_2+2B_1)\text{ if }v\in[0,1],\\
\frac{v}{4}(3B_2+2C_2+D_2+2B_1)+(v-1)(B_1+B_2)\text{ if }v\in\big[1,\frac{4}{3}\big].
\end{cases}\\
&{\text{\bf b). }}&
P(v)=\begin{cases}-K_S-vA-\frac{v}{4}(3B_2+2C_2+D_2+2B_1)\text{ if }v\in[0,1],\\
-K_S-vA-\frac{v}{4}(3B_2+2C_2+D_2+2B_1)-(v-1)(2B_1+C_1)\text{ if }v\in\big[1,\frac{4}{3}\big].
\end{cases}\\
&&N(v)=\begin{cases}
\frac{v}{4}(3B_2+2C_2+D_2+2B_1)\text{ if }v\in[0,1],\\
\frac{v}{4}(3B_2+2C_2+D_2+2B_1)+(v-1)(2B_1+C_1)\text{ if }v\in\big[1,\frac{4}{3}\big].
\end{cases}\\
&{\text{\bf c). }}&
P(v)=\begin{cases}
-K_S-vA-\frac{v}{4}(5B + 6C + 4D + 2E + 3D')\text{ if }v\in[0,1],\\
-K_S-vA-\frac{v}{4}(5B + 6C + 4D + 2E + 3D')-(v-1)(B_1+B_2)\text{ if }v\in\big[1,\frac{4}{3}\big].
\end{cases}\\
&&N(v)=\begin{cases}
\frac{v}{4}(5B + 6C + 4D + 2E + 3D')\text{ if }v\in[0,1],\\
\frac{v}{4}(5B + 6C + 4D + 2E + 3D')+(v-1)(B_1+B_2)\text{ if }v\in\big[1,\frac{4}{3}\big].
\end{cases}\\
&{\text{\bf d). }}&
P(v)=\begin{cases}-K_S-vA-\frac{v}{4}(5B + 6C + 4D + 2E + 3D')\text{ if }v\in[0,1],\\
-K_S-vA-\frac{v}{4}(5B + 6C + 4D + 2E + 3D')-(v-1)(B_1+C_1)\text{ if }v\in\big[1,\frac{4}{3}\big].
\end{cases}\\
&&N(v)=\begin{cases}
\frac{v}{4}(5B + 6C + 4D + 2E + 3D')\text{ if }v\in[0,1],\\
\frac{v}{4}(5B + 6C + 4D + 2E + 3D')+(v-1)(2B_1+C_1)\text{ if }v\in\big[1,\frac{4}{3}\big].
\end{cases}
\end{align*}}}
Moreover, 
$$(P(v))^2=\begin{cases}
2-2 v+\frac{v^2}{4} \text{ if }v\in[0,1],\\
\frac{(4-3v)^2}{4}\text{ if }v\in\big[1,\frac{4}{3}\big].
\end{cases}P(v)\cdot A=\begin{cases}
1+\frac{v}{4}\text{ if }v\in[0,1],\\
3(1-\frac{3v}{4})\text{ if }v\in\big[1,\frac{4}{3}\big].
\end{cases}$$
In this case $\delta_P(S)=\frac{9}{5}\text{ if }P\in A\backslash (B\cup B').$
\end{lemma}
\begin{proof}
The Zariski Decomposition in part a). follows from 
 $$-K_S-vA\sim_{\DR} \Big(\frac{4}{3}-v\Big)A+\frac{1}{3}\Big(2B'+3B+2C+D+B_1+B_2\Big).$$ 
 A similar statement holds in other parts.
We have
$S_S(A)=\frac{5}{9}.$
Thus, $\delta_P(S)\le \frac{9}{5}$ for $P\in A$. Note that for $P\in A\backslash (B\cup B')$ we have:
$$h(v)\le \begin{cases} \frac{(4+v)^2}{32} \text{ if }v\in[0,1],\\
\frac{3(4 - 3v)(7v - 4)}{32}\text{ if }v\in\big[1,\frac{4}{3}\big].
\end{cases}$$
So 
$S(W_{\bullet,\bullet}^{A};P)\le\frac{4}{9}<\frac{5}{9}.$
Thus, $\delta_P(S)=\frac{9}{5}$ if $P\in A\backslash(B_1\cup B_2)$.
\end{proof}
\begin{lemma}\label{deg2-7241-nearA1A4}
Suppose $P$ belongs to a $(-1)$-curve $A$ and there exist $(-1)$-curves and $(-2)$-curves   which form the following dual graph:
\begin{figure}[h!]
    \centering
\includegraphics[width=6.5cm]{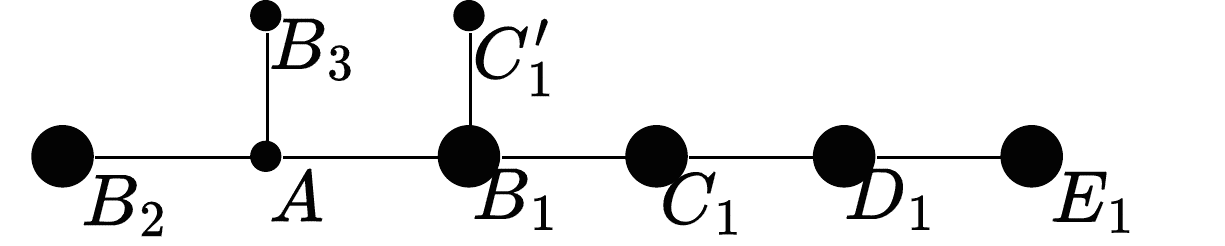}
    \caption{Dual graph: $(-K_S)^2=2$ and $\delta_P(S)=\frac{72}{41}$}
\end{figure}
\par
 Then $\tau(A)=\frac{4}{3}$ and the Zariski Decomposition of the divisor $-K_S-vA$ is given by:
{ {\allowdisplaybreaks\begin{align*}
\hspace*{-0.5cm}&&P(v)=\begin{cases}
-K_S-vA-\frac{v}{5}(4B_1+3C_1+2D_1+E_1)-\frac{v}{2}B_2\text{ if }v\in[0,1],\\
-K_S-vA-\frac{v}{5}(4B_1+3C_1+2D_1+E_1)-\frac{v}{2}B_2-(v-1)B_3\text{ if }v\in\big[1,\frac{5}{4}\big],\\
-K_S-vA-(v-1)(4B_1+3C_1+2D_1+E_1+B_3)-\frac{v}{2}B_2-(4v-5)C_1'\text{ if }v\in\big[\frac{5}{4},\frac{4}{3}\big].
\end{cases}\\
\text{}
\hspace*{-0.5cm}&&N(v)=\begin{cases}
\frac{v}{5}(4B_1+3C_1+2D_1+E_1)+\frac{v}{2}B_2\text{ if }v\in[0,1],\\
\frac{v}{5}(4B_1+3C_1+2D_1+E_1)+\frac{v}{2}B_2+(v-1)B_3\text{ if }v\in\big[1,\frac{5}{4}\big],\\
(v-1)(4B_1+3C_1+2D_1+E_1+B_3)+\frac{v}{2}B_2+(4v-5)C_1'\text{ if }v\in\big[\frac{5}{4},\frac{4}{3}\big].
\end{cases}
\end{align*}}}
Moreover, 
$$(P(v))^2=\begin{cases}
\frac{3v^2}{10}-2v+2\text{ if }v\in[0,1],\\
\frac{13v^2}{10}-4v+3 \text{ if }v\in\big[1,\frac{5}{4}\big],\\
\frac{(4-3v)^2}{2}\text{ if }v\in\big[\frac{5}{4},\frac{4}{3}\big].
\end{cases}P(v)\cdot A=\begin{cases}
1-\frac{3v}{10}\text{ if }v\in[0,1],\\
2 - \frac{13v}{10}\text{ if }v\in\big[1,\frac{5}{4}\big],\\
3(2-\frac{3v}{2})\text{ if }v\in\big[\frac{5}{4},\frac{4}{3}\big].
\end{cases}$$
In this case $\delta_P(S)=\frac{72}{41}\text{ if }P\in A\backslash (B_1\cup B_2).$
\end{lemma}
\begin{proof}
In part a). the Zariski Decomposition  follows from 
 $$-K_S-vA\sim_{\DR} \Big(\frac{4}{3}-v\Big)A+\frac{1}{3}\Big(4B_1+3C_1+2D_1+E_1+B_3+2B_2+C_1'\Big).$$
A similar statement holds in other parts.
We have $S_S(A)=\frac{41}{72}.$
Thus, $\delta_P(S)\le \frac{72}{11}$ for $P\in A$. Note that for $P\in  A\backslash (B_1\cup B_2)$ we have:
$$h(v)\le 
\begin{cases}
 \frac{(10 - 3 v)^2}{200}\text{ if }v\in[0,1],\\
\frac{ 7 (20 - 13 v) v}{200}\text{ if }v\in\big[1,\frac{5}{4}\big],\\
\frac{ 3 (4 - 3 v) (8 - 5 v)}{8}\text{ if }v\in\big[\frac{5}{4},\frac{4}{3}\big].
\end{cases}$$
So 
$S(W_{\bullet,\bullet}^{A};P)\le\frac{61}{144}<\frac{41}{72}$.
Thus, $\delta_P(S)=\frac{72}{41}$ if $ P\in  A\backslash (B_1\cup B_2)$.
\end{proof}
\begin{lemma}\label{deg2-127-near2A2point}\label{deg2-127-nearA5}
Suppose $P$ belongs to a $(-1)$-curve $A$ and there exist $(-1)$-curves and $(-2)$-curves   which form the following dual graph:
\begin{figure}[h!]
    \centering
\includegraphics[width=12.5cm]{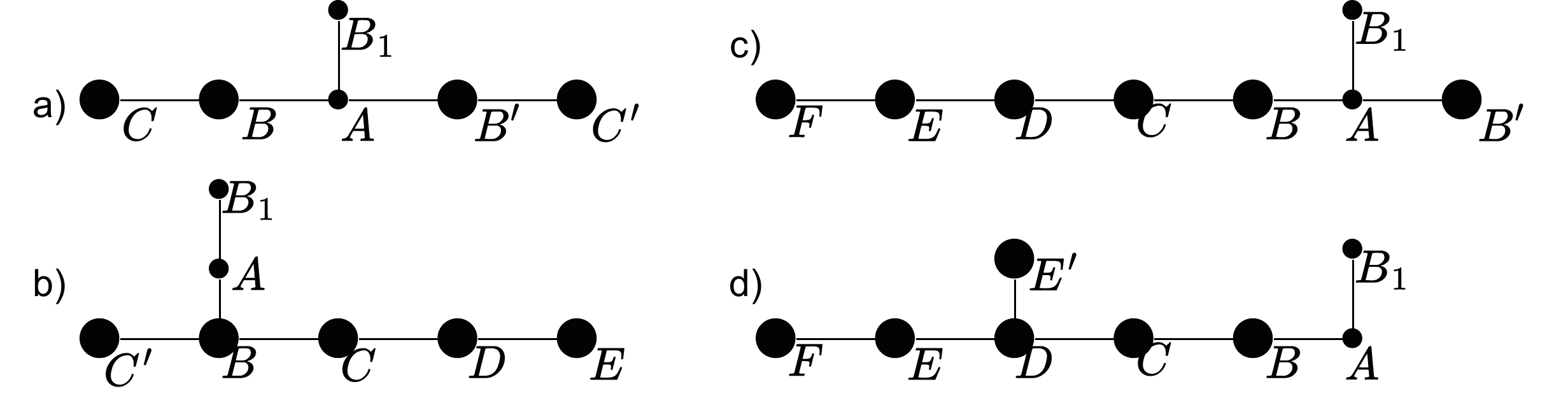}
    \caption{Dual graph: $(-K_S)^2=2$ and $\delta_P(S)=\frac{12}{7}$}
\end{figure}
\par Then $\tau(A)=\frac{3}{2}$ and the Zariski Decomposition of the divisor $-K_S-vA$ is given by:
{ {\allowdisplaybreaks\begin{align*}
&{\text{\bf a). }}&P(v)=\begin{cases}
-K_S-vA-\frac{v}{3}(C+2B+2B'+C')\text{ if }v\in[0,1],\\
-K_S-vA-\frac{v}{3}(C+2B+2B'+C')-(v-1)B_1\text{ if }v\in\big[1,\frac{3}{2}\big].
\end{cases}\\
\text{}
&&N(v)=\begin{cases}
\frac{v}{3}(C+2B+2B'+C')\text{ if }v\in[0,1],\\
\frac{v}{3}(C+2B+2B'+C')+(v-1)B_1\text{ if }v\in\big[1,\frac{3}{2}\big].
\end{cases}\\
&{\text{\bf b). }}&P(v)=\begin{cases}
-K_S-vA-\frac{v}{3}(2C'+4B+3C+2D+E)\text{ if }v\in[0,1],\\
-K_S-vA-\frac{v}{3}(2C'+4B+3C+2D+E)-(v-1)B_1\text{ if }v\in\big[1,\frac{3}{2}\big].
\end{cases}\\
\text{}
&&N(v)=\begin{cases}
\frac{v}{3}(2C'+4B+3C+2D+E)\text{ if }v\in[0,1],\\
\frac{v}{3}(2C'+4B+3C+2D+E)+(v-1)B_1\text{ if }v\in\big[1,\frac{3}{2}\big].
\end{cases}\\
&{\text{\bf c). }}&P(v)=\begin{cases}
-K_S-vA-\frac{v}{6}(5B+4C+3D+2E+F+3B')\text{ if }v\in[0,1],\\
-K_S-vA-\frac{v}{6}(5B+4C+3D+2E+F+3B')-(v-1)B_1\text{ if }v\in\big[1,\frac{3}{2}\big].
\end{cases}\\
\text{}
&&N(v)=\begin{cases}
\frac{v}{6}(5B+4C+3D+2E+F+3B')\text{ if }v\in[0,1],\\
\frac{v}{6}(5B+4C+3D+2E+F+3B')+(v-1)B_1\text{ if }v\in\big[1,\frac{3}{2}\big].
\end{cases}\\
&{\text{\bf d). }}&P(v)=\begin{cases}
-K_S-vA-\frac{v}{3}(2F + 4E + 6D + 5C + 4B + 3E')\text{ if }v\in[0,1],\\
-K_S-vA-\frac{v}{3}(2F + 4E + 6D + 5C + 4B + 3E')-(v-1)B_1\text{ if }v\in\big[1,\frac{3}{2}\big].
\end{cases}\\
\text{}
&&N(v)=\begin{cases}
\frac{v}{3}(2F + 4E + 6D + 5C + 4B + 3E')\text{ if }v\in[0,1],\\
\frac{v}{6}(2F + 4E + 6D + 5C + 4B + 3E')+(v-1)B_1\text{ if }v\in\big[1,\frac{3}{2}\big].
\end{cases}
\end{align*}}}
Moreover, 
$$(P(v))^2=\begin{cases}
2 - 2v+ \frac{v^2}{3}\text{ if }v\in[0,1],\\
\frac{(3-2v)^2}{3}\text{ if }v\in\big[1,\frac{3}{2}\big].
\end{cases}P(v)\cdot A=\begin{cases}
1-\frac{v}{3}\text{ if }v\in[0,1],\\
2(1-\frac{2v}{3})\text{ if }v\in\big[1,\frac{3}{2}\big].
\end{cases}$$
In this case $\delta_P(S)=\frac{12}{7}\text{ if }P\in A\backslash (B\cup B').$
\end{lemma}
\begin{proof}
In part a). the Zariski Decomposition  follows from 
 $$-K_S-vA\sim_{\DR} \Big(\frac{3}{2}-v\Big)A+\frac{1}{2}\Big(C + 2B + 2B' + C'+B_1\Big).$$
A similar statement holds in other parts.
We have
$S_S(A)=\frac{7}{12}$. Thus, $\delta_P(S)\le \frac{12}{7}$ for $P\in A$. Note that for $P\in A\backslash (B\cup B')$ we have:
$$h(v)\le \begin{cases} 
 \frac{(3 - v)^2}{18}\text{ if }v\in[0,1],\\
\frac{2 v (3-2 v)}{9}\text{ if }v\in\big[1, \frac{3}{2}\big].
\end{cases}$$
So 
$S(W_{\bullet,\bullet}^{A};P)\le\frac{5}{12}<\frac{7}{12}$. Thus, $\delta_P(S)=\frac{12}{7}$ if $P\in A\backslash (B\cup B')$.
\end{proof}
\begin{lemma}\label{deg2-1811-nearA2A3}
Suppose $P$ belongs to a $(-1)$-curve $A$ and there exist $(-1)$-curves and $(-2)$-curves   which form the following dual graph:
\begin{figure}[h!]
    \centering
\includegraphics[width=6cm]{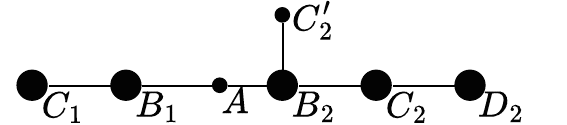}
    \caption{Dual graph: $(-K_S)^2=2$ and $\delta_P(S)=\frac{18}{11}$}
\end{figure}
\par Then $\tau(A)=\frac{3}{2}$ and the Zariski Decomposition of the divisor $-K_S-vA$ is given by:
{ {\allowdisplaybreaks\begin{align*}
&&P(v)=\begin{cases}-K_S-vA-\frac{v}{3}(2B_1+C_1)-\frac{v}{4}(3B_2+2C_2+D_2)\text{ if }v\in\big[0,\frac{4}{3}\big],\\
-K_S-vA-\frac{v}{3}(2B_1+C_1)-(v-1)(3B_2+2C_2+D_2)-(3v-4)C_2'\text{ if }v\in\big[\frac{4}{3},\frac{3}{2}\big].
\end{cases}\\
&&
N(v)=\begin{cases}
\frac{v}{3}(2B_1+C_1)+\frac{v}{4}(3B_2+2C_2+D_2)\text{ if }v\in\big[0,\frac{4}{3}\big],\\
\frac{v}{3}(2B_1+C_1)+(v-1)(3B_2+2C_2+D_2)+(3v-4)C_2'\text{ if }v\in\big[\frac{4}{3},\frac{3}{2}\big].
\end{cases}
\end{align*}}}
Moreover, 
$$(P(v))^2=\begin{cases}
2-2v+\frac{5v^2}{12},\text{ if }v\in\big[0,\frac{4}{3}\big],\\
\frac{2(3 - 2v)^2}{3}\text{ if }v\in\big[\frac{4}{3},\frac{3}{2}\big].
\end{cases}P(v)\cdot A=\begin{cases}
1 - \frac{5v}{12}\text{ if }v\in\big[0,\frac{4}{3}\big],\\
4(1 - \frac{2v}{3})\text{ if }v\in\big[\frac{4}{3},\frac{3}{2}\big].
\end{cases}$$
In this case $\delta_P(S)=\frac{18}{11}\text{ if }P\in A\backslash (B_1\cup B_2).$
\end{lemma}
\begin{proof}
  The Zariski Decomposition  follows from 
 $$-K_S-vA\sim_{\DR} \Big(\frac{3}{2}-v\Big)A+\frac{1}{2}\Big(2B_1+C_1+3B_2+2C_2+D_2+C_2'\Big).$$
We have
$S_S(A)=\frac{11}{18}.$
Thus, $\delta_P(S)\le \frac{18}{11}$ for $P\in A$. Note that for $P\in A\backslash (B_1\cup B_2)$ we have:
$$h(v)\le \begin{cases}
\frac{ (12 - 5 v)^2}{288}\text{ if }v\in\big[0,\frac{4}{3}\big],\\
\frac{8 (3 - 2 v)^2}{9}\text{ if }v\in\big[\frac{4}{3},\frac{3}{2}\big].
\end{cases}$$
So we have 
$S(W_{\bullet,\bullet}^{A};P)=\frac{10}{27}<\frac{11}{18}.$
Thus, $\delta_P(S)=\frac{18}{11}$ if $P\in A\backslash (B_1\cup B_2)$.
\end{proof}
\begin{lemma}\label{deg2-6037-nearA6points}
Suppose $P$ belongs to a $(-1)$-curve $A$ and there exist $(-1)$-curves and $(-2)$-curves   which form the following dual graph:
\begin{figure}[h!]
    \centering
\includegraphics[width=7cm]{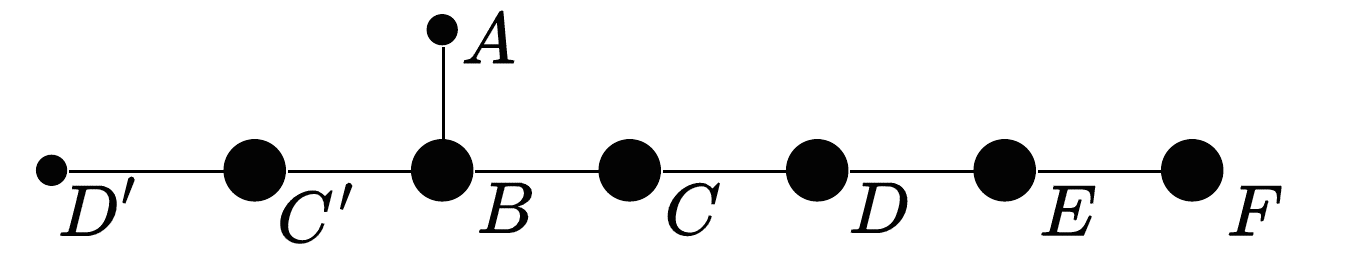}
    \caption{Dual graph: $(-K_S)^2=2$ and $\delta_P(S)=\frac{60}{37}$}
\end{figure}
\par
Then $\tau(A)=\frac{3}{2}$ and the Zariski Decomposition of the divisor $-K_S-vA$ is given by:
{{\allowdisplaybreaks\begin{align*}
\hspace*{-0.5cm}&&P(v)=\begin{cases}
-K_S-vA-\frac{v}{7}(2F+4E+6D+8C+10B+5C')\text{ if }v\in\big[0,\frac{7}{5}\big],\\
-K_S-vA-(v-1)(2F+4E+6D+8C+10B)-(5v-6)C'-(5v - 7)D'\text{ if }v\in\big[\frac{7}{5},\frac{3}{2}\big].
\end{cases}\\
\hspace*{-0.5cm}&&N(v)=\begin{cases}\frac{v}{7}(2F+4E+6D+8C+10B+5C')\text{ if }v\in\big[0,\frac{7}{5}\big],\\
(v-1)(2F+4E+6D+8C+10B)+(5v-6)C'+(5v - 7)D'\text{ if }v\in\big[\frac{7}{5},\frac{3}{2}\big].
\end{cases}
\end{align*}}}
Moreover, 
$$(P(v))^2=\begin{cases}
2 - 2v+ \frac{3v^2}{7}\text{ if }v\in\big[0,\frac{7}{5}\big],\\
(3-2v)^2\text{ if }v\in\big[\frac{7}{5},\frac{3}{2}\big].
\end{cases}P(v)\cdot A=\begin{cases}
1-\frac{3v}{7}\text{ if }v\in\big[0,\frac{7}{5}\big],\\
2(3-2v)\text{ if }v\in\big[\frac{7}{5},\frac{3}{2}\big].
\end{cases}$$
In this case $\delta_P(S)=\frac{60}{37}\text{ if }P\in A\backslash B$.
\end{lemma}
\begin{proof}
The Zariski Decomposition  follows from 
 $$-K_S-vA\sim_{\DR} \Big(\frac{3}{2}-v\Big)A+\frac{1}{2}\Big(F+2E+3D+4C+5B+3C'+D'\Big).$$
We have
$S_S(A)=\frac{37}{60}$. Thus, $\delta_P(S)\le \frac{60}{37}$ for $P\in A$. Note that for $P\in A\backslash B$ we have:
$$h(v)\le \begin{cases} 
 \frac{(7 - 3 v)^2}{98}\text{ if }v\in\big[0,\frac{7}{5}\big],\\
2(3-2v)^2\text{ if }v\in\big[\frac{7}{5},\frac{3}{2}\big].
\end{cases}$$
So  $S(W_{\bullet,\bullet}^{A};P)\le\frac{11}{30}<\frac{37}{60}$.
Thus, $\delta_P(S)=\frac{60}{37}$ if $P\in A\backslash B$.
\end{proof}
\begin{lemma}\label{deg2-3623-nearA2A4}
Suppose $P$ belongs to a $(-1)$-curve $A$ and there exist $(-1)$-curves and $(-2)$-curves   which form the following dual graph:
\begin{figure}[h!]
    \centering
\includegraphics[width=7cm]{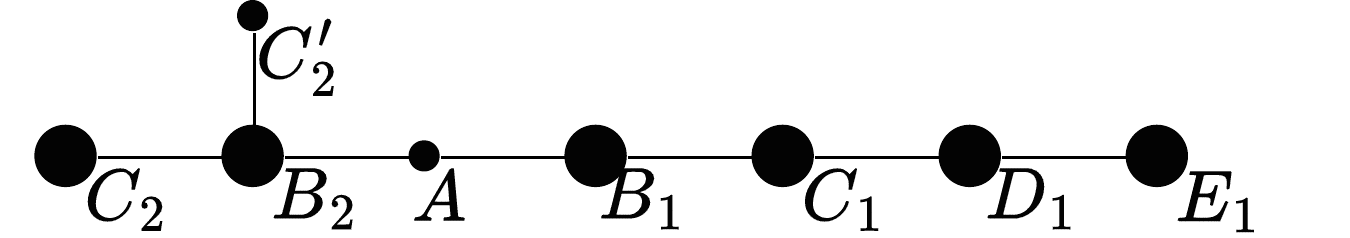}
    \caption{Dual graph: $(-K_S)^2=2$ and $\delta_P(S)=\frac{36}{23}$}
\end{figure}
\par Then the Zariski Decomposition of the divisor $-K_S-vA$ is given by:
{ {\allowdisplaybreaks\begin{align*}
\hspace*{-0.5cm}&&P(v)=
\begin{cases}
-K_S-vA-\frac{v}{5}(4B_1+3C_1+2D_1+E_1)-\frac{v}{2}(2B_2+C_2)\text{ if }v\in\big[0,\frac{3}{2}\big],\\
-K_S-vA-\frac{v}{5}(4B_1+3C_1+2D_1+E_1)-(v-1)(2B_2+C_2)-(2v-3)C_2'\text{ if }v\in\big[\frac{3}{2},\frac{5}{3}\big].
\end{cases}\\
\text{}
\hspace*{-0.5cm}&&N(v)=\begin{cases}
\frac{v}{5}(4B_1+3C_1+2D_1+E_1)+\frac{v}{2}(2B_2+C_2)\text{ if }v\in\big[0,\frac{3}{2}\big],\\
\frac{v}{5}(4B_1+3C_1+2D_1+E_1)+(v-1)(2B_2+C_2)+(2v-3)C_2'\text{ if }v\in\big[\frac{3}{2},\frac{5}{3}\big].
\end{cases}
\end{align*}}}
Moreover, 
$$(P(v))^2=\begin{cases}
2 - 2v +\frac{7v^2}{15} \text{ if }v\in\big[0,\frac{3}{2}\big],\\
\frac{(5-3v)^2}{5}\text{ if }v\in\big[\frac{3}{2},\frac{5}{3}\big].
\end{cases}P(v)\cdot A=\begin{cases}1-\frac{7v}{15}\text{ if }v\in\big[0,\frac{3}{2}\big],\\
3(1-\frac{3v}{5})\text{ if }v\in\big[\frac{3}{2},\frac{5}{3}\big].
\end{cases}$$
In this case $\delta_P(S)=\frac{36}{23}\text{ if }P\in A\backslash (B_1\cup B_2).$
\end{lemma}
\begin{proof}
The Zariski Decomposition  follows from 
 $$-K_S-vA\sim_{\DR} \Big(\frac{5}{3}-v\Big)A+\frac{1}{3}\Big(4B_1+3C_1+2D_1+E_1+4B_2+2C_2+C_2'\Big).$$
 We have
$S_S(A)=\frac{23}{36}.$ Thus, $\delta_P(S)\le \frac{36}{23}$ for $P\in A$. Note that for $P\in A\backslash (B_1\cup B_2)$ we have:
$$h(v)\le \begin{cases}
 \frac{(15 - 7 v)^2}{450}\text{ if }v\in\big[0,\frac{3}{2}\big],\\
\frac{9 (5 - 3 v)^2}{50}\text{ if }v\in\big[\frac{3}{2},\frac{5}{3}\big].
\end{cases}$$
So 
$S(W_{\bullet,\bullet}^{A};P)\le\frac{7}{20}<\frac{23}{36}.$
Thus, $\delta_P(S)=\frac{36}{23}$ if $P\in A\backslash (B_1\cup B_2)$.
\end{proof}
\begin{lemma}\label{deg2-32-A1points}
Suppose $P\in A$ where is a $(-2)$-curve disjoint from other $(-2)$-curves then $\tau(A)=1$ and the Zariski decomposition of the divisor $-K_S-vA$  given by:
$$P(v)=-K_S-vA\text{ and }N(v)=0\text{ if }v\in[0,1].$$
Moreover, 
$$(P(v))^2=2(1 - v )(v + 1)\text{ and }P(v)\cdot A=2v \text{ if }v\in[0,1]. $$
In this case: $\delta_P(S)=\frac{3}{2}\text{ if }P\in A.$
\end{lemma}
\begin{proof}  The Zariski Decomposition follows from $-K_S-vA\sim_{\DR} L+(1-v)A$ where $L$ is a strict transform of an element $|-K_X|$ passing through a singular point which is the image of $A$ on $X$. We have $S_S(A)=\frac{2}{3}.$ Thus, $\delta_P(S)\le \frac{3}{2}$ for $P\in A$. Note that for $P\in A$ we have $h(v) = 2v^2\text{ if }v\in[0,1].$
So $S(W_{\bullet,\bullet}^{A};P)= \frac{2}{3}.$
Thus, $\delta_P(S)=\frac{3}{2}$ if $P\in A$.    
\end{proof}
\begin{lemma}\label{deg2-32-near3A1points}\label{deg2-32-nearA1middleA3}
Suppose $P$ belongs to a $(-1)$-curve $A$ and there exist $(-1)$-curves and $(-2)$-curves   which form the following dual graph:
\begin{figure}[h!]
    \centering
\hspace*{-0.3cm}\includegraphics[width=17cm]{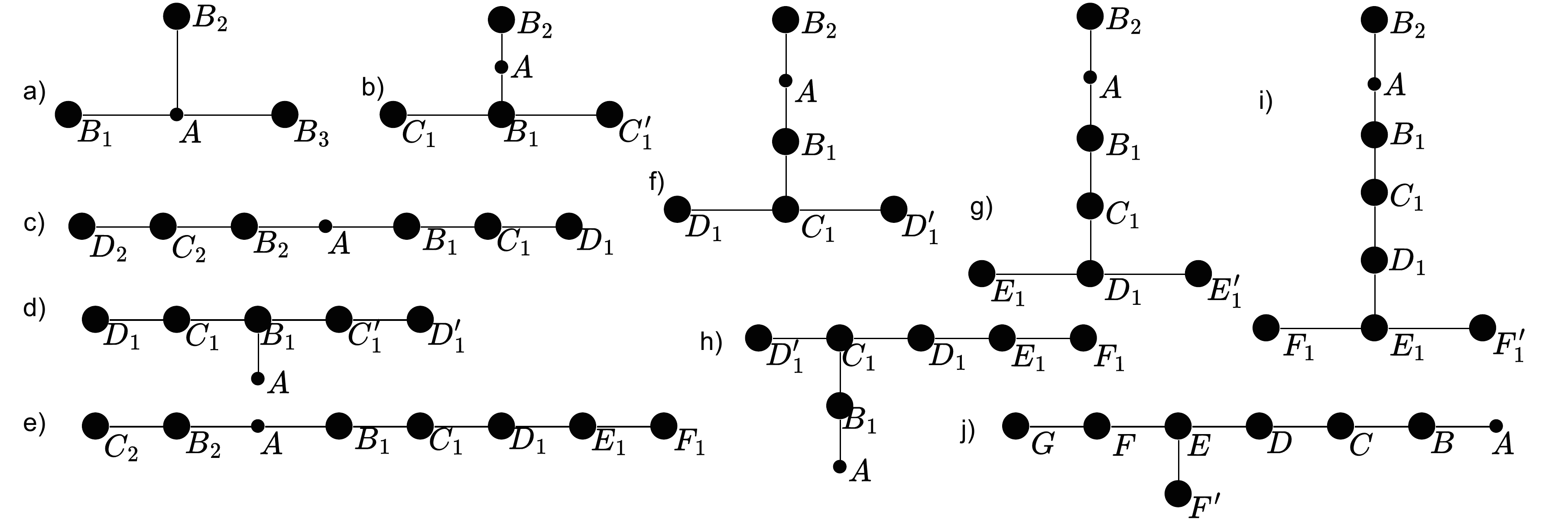}
    \caption{Dual graph: $(-K_S)^2=2$ and $\delta_P(S)=\frac{3}{2}$}
\end{figure}
\par Then $\tau(A)=2$ and the Zariski Decomposition of the divisor $-K_S-vA$ is given by:
{ {\allowdisplaybreaks\begin{align*}
&{\text{\bf a). }}&P(v)=-K_S-vA-\frac{v}{2}(B_1+B_2+B_3)\text{ if }v\in[0,2].\\&&N(v)=\frac{v}{2}(B_1+B_2+B_3)\text{ if }v\in[0,2].\\
&{\text{\bf b). }}&P(v)=-K_S-vA-\frac{v}{2}(2B_1+C_1+C_1'+B_2)\text{ if }v\in[0,2].\\&&N(v)=\frac{v}{2}(2B_1+C_1+C_1'+B_2)\text{ if }v\in[0,2].\\
&{\text{\bf c). }}&P(v)=-K_S-vA-\frac{v}{4}(D_1+2C_2+3B_1+3B_2+2C_2+D_2)\text{ if }v\in[0,2].\\&&N(v)=\frac{v}{4}(D_1+2C_2+3B_1+3B_2+2C_2+D_2)\text{ if }v\in[0,2].\\
&{\text{\bf d). }}&P(v)=-K_S-vA-\frac{v}{2}(D_1+2C_2+3B_1+2C_2+D_2)\text{ if }v\in[0,2].\\&& N(v)=\frac{v}{2}(D_1+2C_2+3B_1+2C_2+D_2)\text{ if }v\in[0,2].\\
&{\text{\bf e). }}&P(v)=-K_S-vA-\frac{v}{6}(2C_2+4B_2+5B_1+4C_1+3D_1+2E_1+F_1)\text{ if }v\in[0,2].\\&&N(v)=\frac{v}{6}(2C_2+4B_2+5B_1+4C_1+3D_1+2E_1+F_1)\text{ if }v\in[0,2].
\\
&{\text{\bf f). }}&P(v)=-K_S-vA-\frac{v}{2}(2B_1+2C_1+D_1+D_1'+B_2)\text{ if }v\in[0,2].\\&&N(v)=\frac{v}{2}(2B_1+2C_1+D_1+D_1'+B_2)\text{ if }v\in[0,2].\\
&{\text{\bf g). }}&P(v)=-K_S-vA-\frac{v}{2}(2B_1+2C_1+2D_1+E_1+E_1'+B_2)\text{ if }v\in[0,2].\\&&N(v)=\frac{v}{2}(2B_1+2C_1+2D_1+E_1+E_1'+B_2)\text{ if }v\in[0,2].\\
&{\text{\bf h). }}&P(v)=-K_S-vA-\frac{v}{2}(3B+2D'+4C+3D+2E+F)\text{ if }v\in[0,2].\\&&N(v)=\frac{v}{2}(3B+2D'+4C+3D+2E+F)\text{ if }v\in[0,2].\\
&{\text{\bf i). }}&P(v)=-K_S-vA-\frac{v}{2}(2B_1+2C_1+2D_1+2E_1+F_1+F_1'+B_2)\text{ if }v\in[0,2].\\&&N(v)=\frac{v}{2}(2B_1+2C_1+2D_1+2E_1+F_1+F_1'+B_2)\text{ if }v\in[0,2],\\
&{\text{\bf j). }}&P(v)=-K_S-vA-\frac{v}{2}(2G+4F+6E+5D+4C+3B+3F')\text{ if }v\in[0,2].\\&&N(v)=\frac{v}{2}(2G+4F+6E+5D+4C+3B+3F')\text{ if }v\in[0,2].
\end{align*}}}
Moreover, 
$$(P(v))^2=\frac{(2 - v )^2}{2}P(v)\cdot A= 1-\frac{v}{2} \text{ if }v\in[0,2]. $$
In this case $\delta_P(S)=\frac{3}{2}\text{ if }P\in A\backslash (B_1\cup B_2\cup B_3).$
\end{lemma}

\begin{proof}
 The Zariski Decomposition in part a). follows from $-K_S-vA\sim_{\DR} (2-v)A+B_1+B_2+B_3$. A similar statement holds in other parts.
 We have
$S_S(A)=\frac{2}{3}.$ Thus, $\delta_P(S)\le \frac{3}{2}$ for $P\in A$. Note that for $P\in A\backslash (B_1\cup B_2\cup B_3)$
$h(v) =  \frac{(2-v)^2}{8} \text{ if }v\in[0,2].$
So we have 
$S(W_{\bullet,\bullet}^{A};P)=\frac{1}{3}<\frac{2}{3}.$
Thus, $\delta_P(S)=\frac{3}{2}$ if $P\in A\backslash (B_1\cup B_2\cup B_3)$.
\end{proof}
\begin{lemma}\label{deg2-97-A2points}
Suppose $P$ belongs to a $(-1)$-curve $A$ and there exist $(-1)$-curves and $(-2)$-curves   which form the following dual graph:
\begin{figure}[h!]
    \centering
\hspace*{-0.3cm}\includegraphics[width=16cm]{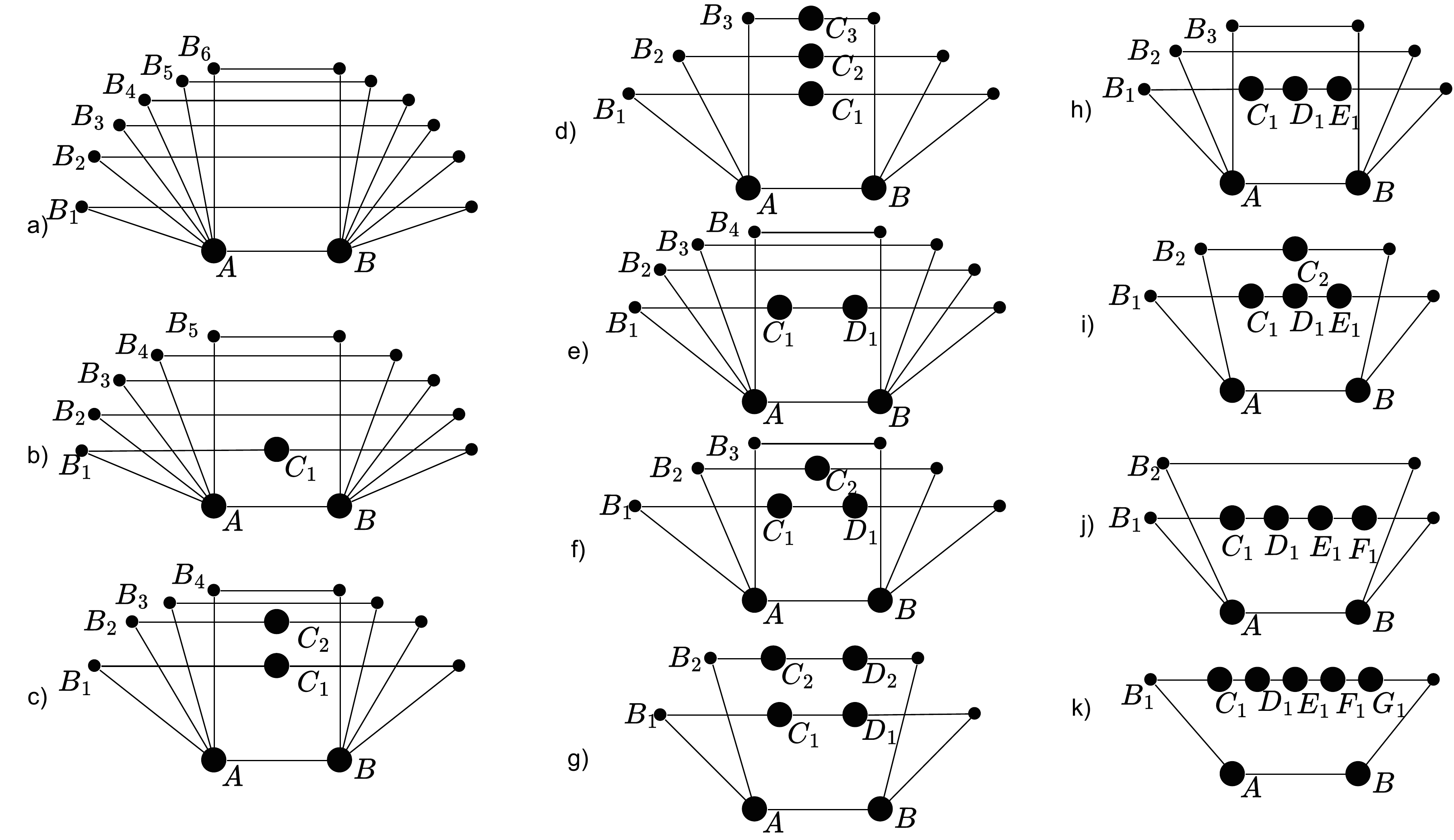}
    \caption{Dual graph: $(-K_S)^2=2$ and $\delta_P(S)=\frac{9}{7}$}
\end{figure}
\par  Then $\tau(A)=\frac{4}{3}$ and the Zariski Decomposition of the divisor $-K_S-vA$ is given by:
{ 
{\allowdisplaybreaks\begin{align*}
&{\text{\bf a). }}&P(v)=\begin{cases}
-K_S-vA-\frac{v}{2}B\text{ if }v\in[0,1],\\
-K_S-vA-\frac{v}{2}B-(v-1)(B_1+B_2+B_3+B_4+B_5+B_6)\text{ if }v\in\big[1,\frac{4}{3}\big].
\end{cases}\\&&N(v)=\begin{cases}
\frac{v}{2}B\text{ if }v\in[0,1],\\
\frac{v}{2}B+(v-1)(B_1+B_2+B_3+B_4+B_5+B_6)\text{ if }v\in\big[1,\frac{4}{3}\big].
\end{cases}\\
&{\text{\bf b). }}&P(v)=\begin{cases}
-K_S-vA-\frac{v}{2}B\text{ if }v\in[0,1],\\
-K_S-vA-\frac{v}{2}B-(v-1)(2B_1+C_1+B_2+B_3+B_4+B_5)\text{ if }v\in\big[1,\frac{4}{3}\big].
\end{cases}\\&&N(v)=\begin{cases}
\frac{v}{2}B\text{ if }v\in[0,1],\\
\frac{v}{2}B+(v-1)(2B_1+C_1+B_2+B_3+B_4+B_5)\text{ if }v\in\big[1,\frac{4}{3}\big].
\end{cases}\\
&{\text{\bf c). }}&P(v)=\begin{cases}
-K_S-vA-\frac{v}{2}B\text{ if }v\in[0,1],\\
-K_S-vA-\frac{v}{2}B-(v-1)(2B_1+C_1+2B_2+C_2+B_3+B_4)\text{ if }v\in\big[1,\frac{4}{3}\big].
\end{cases}\\&&N(v)=\begin{cases}
\frac{v}{2}B\text{ if }v\in[0,1],\\
\frac{v}{2}B+(v-1)(2B_1+C_1+2B_2+C_2+B_3+B_4)\text{ if }v\in\big[1,\frac{4}{3}\big].
\end{cases}\\
&{\text{\bf d). }}&P(v)=\begin{cases}
-K_S-vA-\frac{v}{2}B\text{ if }v\in[0,1],\\
-K_S-vA-\frac{v}{2}B-(v-1)(2B_1+C_1+2B_2+C_2+2B_3+C_3)\text{ if }v\in\big[1,\frac{4}{3}\big].
\end{cases}\\&&N(v)=\begin{cases}
\frac{v}{2}B\text{ if }v\in[0,1],\\
\frac{v}{2}B+(v-1)(2B_1+C_1+2B_2+C_2+2B_3+C_3)\text{ if }v\in\big[1,\frac{4}{3}\big].
\end{cases}\\
&{\text{\bf e). }}&P(v)=\begin{cases}
-K_S-vA-\frac{v}{2}B\text{ if }v\in[0,1],\\
-K_S-vA-\frac{v}{2}B-(v-1)(3B_1+2C_1+D_1+B_2+B_3+B_4)\text{ if }v\in\big[1,\frac{4}{3}\big].
\end{cases}\\&&N(v)=\begin{cases}
\frac{v}{2}B\text{ if }v\in[0,1],\\
\frac{v}{2}B+(v-1)(3B_1+2C_1+D_1+B_2+B_3+B_4)\text{ if }v\in\big[1,\frac{4}{3}\big].
\end{cases}\\
&{\text{\bf f). }}&P(v)=\begin{cases}
-K_S-vA-\frac{v}{2}B\text{ if }v\in[0,1],\\
-K_S-vA-\frac{v}{2}B-(v-1)(3B_1+2C_1+D_1+2B_2+C_2+B_3)\text{ if }v\in\big[1,\frac{4}{3}\big].
\end{cases}\\&&N(v)=\begin{cases}
\frac{v}{2}B\text{ if }v\in[0,1],\\
\frac{v}{2}B+(v-1)(3B_1+2C_1+D_1+2B_2+C_2+B_3)\text{ if }v\in\big[1,\frac{4}{3}\big].
\end{cases}\\
&{\text{\bf g). }}&P(v)=\begin{cases}
-K_S-vA-\frac{v}{2}B\text{ if }v\in[0,1],\\
-K_S-vA-\frac{v}{2}B-(v-1)(3B_1+2C_1+D_1+3B_2+2C_2+D_2)\text{ if }v\in\big[1,\frac{4}{3}\big].
\end{cases}\\&&N(v)=\begin{cases}
\frac{v}{2}B\text{ if }v\in[0,1],\\
\frac{v}{2}B+(v-1)(3B_1+2C_1+D_1+3B_2+2C_2+D_2)\text{ if }v\in\big[1,\frac{4}{3}\big].
\end{cases}\\
&{\text{\bf h). }}&P(v)=\begin{cases}
-K_S-vA-\frac{v}{2}B\text{ if }v\in[0,1],\\
-K_S-vA-\frac{v}{2}B-(v-1)(4B_1+3C_1+2D_1+E_1+B_2+B_3)\text{ if }v\in\big[1,\frac{4}{3}\big].
\end{cases}\\&&N(v)=\begin{cases}
\frac{v}{2}B\text{ if }v\in[0,1],\\
\frac{v}{2}B+(v-1)(4B_1+3C_1+2D_1+E_1+B_2+B_3)\text{ if }v\in\big[1,\frac{4}{3}\big].
\end{cases}\\
&{\text{\bf i). }}&P(v)=\begin{cases}
-K_S-vA-\frac{v}{2}B\text{ if }v\in[0,1],\\
-K_S-vA-\frac{v}{2}B-(v-1)(4B_1+3C_1+2D_1+E_1+2B_2+C_2)\text{ if }v\in\big[1,\frac{4}{3}\big].
\end{cases}\\&&N(v)=\begin{cases}
\frac{v}{2}B\text{ if }v\in[0,1],\\
\frac{v}{2}B+(v-1)(4B_1+3C_1+2D_1+E_1+2B_2+C_2)\text{ if }v\in\big[1,\frac{4}{3}\big].
\end{cases}\\
&{\text{\bf j). }}&P(v)=\begin{cases}
-K_S-vA-\frac{v}{2}B\text{ if }v\in[0,1],\\
-K_S-vA-\frac{v}{2}B-(v-1)(5B_1+4C_1+3D_1+2E_1+F_1+B_2)\text{ if }v\in\big[1,\frac{4}{3}\big].
\end{cases}\\&&N(v)=\begin{cases}
\frac{v}{2}B\text{ if }v\in[0,1],\\
\frac{v}{2}B+(v-1)(5B_1+4C_1+3D_1+2E_1+F_1+B_2)\text{ if }v\in\big[1,\frac{4}{3}\big].
\end{cases}\\
&{\text{\bf k). }}&P(v)=\begin{cases}
-K_S-vA-\frac{v}{2}B\text{ if }v\in[0,1],\\
-K_S-vA-\frac{v}{2}B-(v-1)(6B_1+5C_1+4D_1+3E_1+2F_1+G_1)\text{ if }v\in\big[1,\frac{4}{3}\big].
\end{cases}\\&&N(v)=\begin{cases}
\frac{v}{2}B\text{ if }v\in[0,1],\\
\frac{v}{2}B+(v-1)(6B_1+5C_1+4D_1+3E_1+2F_1+G_1)\text{ if }v\in\big[1,\frac{4}{3}\big].
\end{cases}
\end{align*}}
}
Moreover, 
$$(P(v))^2=\begin{cases}2 - \frac{3v^2}{2} \text{ if }v\in[0,1],\\
\frac{(4-3v)^2}{2}\text{ if }v\in\big[1,\frac{4}{3}\big].
\end{cases}P(v)\cdot A=\begin{cases}\frac{3v}{2}\text{ if }v\in[0,1],\\
3(2-\frac{3v}{2})\text{ if }v\in\big[1,\frac{4}{3}\big].
\end{cases}$$
In this case $\delta_P(S)=\frac{9}{7}\text{ if }P\in A\backslash B.$
\end{lemma}

\begin{proof}
 The Zariski Decomposition in part a). follows from $$-K_S-vA\sim_{\DR} \Big(\frac{4}{3}-v\Big)A+\frac{1}{3}\Big(2B+B_1+B_2+B_3+B_4+B_5+B_6\Big).$$ A similar statement holds in other parts. 
We have $S_S(A)=\frac{7}{9}.$ Thus, $\delta_P(S)\le \frac{9}{7}$ for $P\in A$. Note that for $P\in A\backslash B$ we have:
$$h(v)\le \begin{cases}\frac{9v^2}{8}\text{ if }v\in[0,1],\\
\frac{9(4-3v)(5v - 4)}{8}\text{ if }v\in\big[1, \frac{4}{3}\big].
\end{cases}$$
So we have 
$S(W_{\bullet,\bullet}^{A};P)\le\frac{2}{3}<\frac{7}{9}.$ Thus, $\delta_P(S)=\frac{9}{7}$ if $P\in A\backslash B$.
\end{proof}
\begin{lemma}\label{deg2-65-A2points}
Suppose $P= A_1\cap A_2$ where $A_1$ and $A_2$ are $(-2)$-curves disjoint from other $(-2)$-curves, and $B$ is a unique $(0)$-curve containing $P$. Consider the blowup $\sigma:\widetilde{S}\to S$ of $S$ at $P$ with the exceptional divisor $E_P$.  
\begin{figure}[h!]
    \centering
\includegraphics[width=8cm]{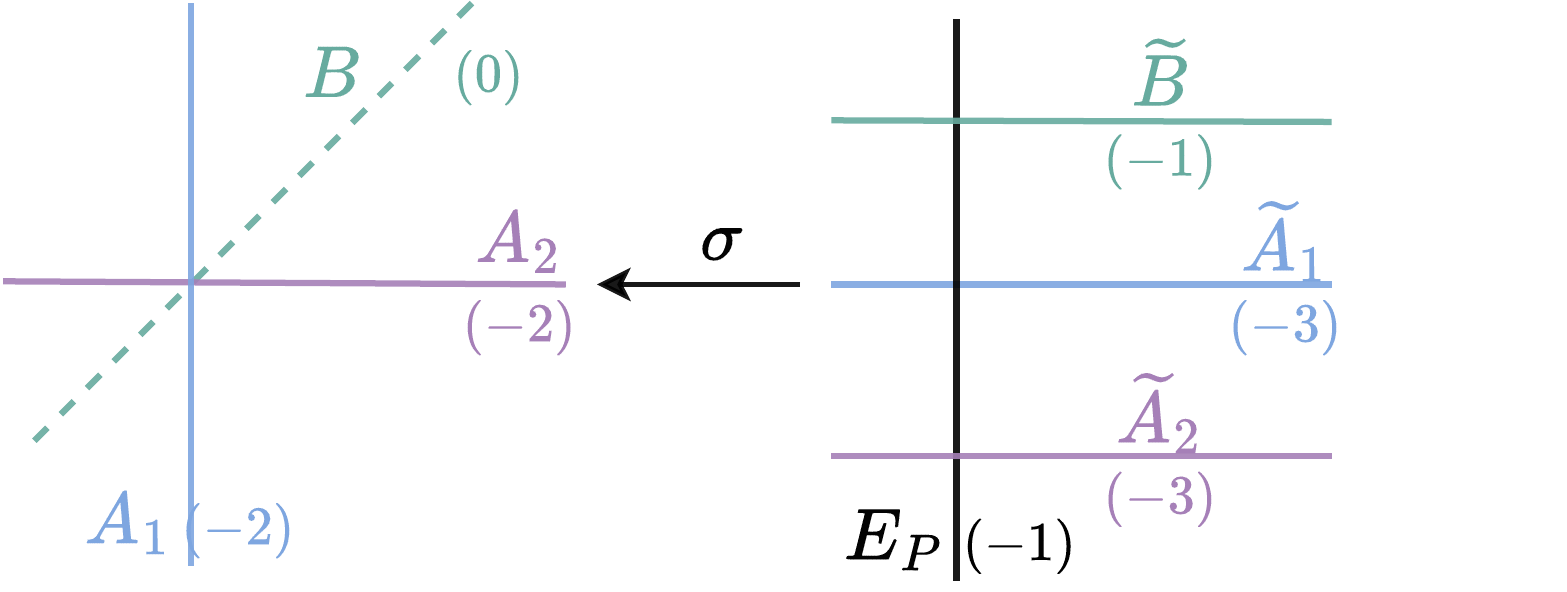}
    \caption{Picture: $(-K_S)^2=2$ and $\delta_P(S)=\frac{6}{5}$ (intersection of two $(-2)$-curves)}
\end{figure}
\par    Then $\tau(A)=3$ and the Zariski decomposition of the divisor $\sigma^*(-K_S)-vE_P$ is given by:
{  \allowdisplaybreaks 
\begin{align*}
&&P(v)=\begin{cases}
\sigma^*(-K_S)-vE_P-\frac{v}{3}(\widetilde{A}_{1}+\widetilde{A}_2)\text{ if }v\in[0,2],\\
\sigma^*(-K_S)-vE_P-\frac{v}{3}(\widetilde{A}_{1}+\widetilde{A}_2)-(v-2)\widetilde{B}\text{ if }v\in[2,3].
\end{cases}
\\&&N(v)=\begin{cases}
\frac{v}{3}(\widetilde{A}_{1}+\widetilde{A}_2)\text{ if }v\in[0,2],\\
\frac{v}{3}(\widetilde{A}_{1}+\widetilde{A}_2)+(v-2)\widetilde{B}\text{ if }v\in[2,3].
\end{cases}
\end{align*}
}
Moreover, 
$$P(v)^2=\begin{cases}
2-\frac{v^2}{3}\text{ if }v\in[0,2],\\
\frac{2(3-v)^2}{3}\text{ if }v\in[2,3].
\end{cases}
P(v)\cdot E_P=
\begin{cases}
\frac{v}{3}\text{ if }v\in[0,2],\\
2(1-\frac{v}{3})\text{ if }v\in[2,3].
\end{cases}$$
In this case $\delta_P(S)=\frac{6}{5}$ for  $P=A_1\cap A_2$.
\end{lemma}
\begin{proof}
 The Zariski Decomposition follows from $\sigma^*(-K_S)-vE_P\sim_{\DR}  (3 - v)E_P+\widetilde{A}_{1}+\widetilde{A}_{2}+\widetilde{B}$ where $\widetilde{A}_{1}$, $\widetilde{A}_{1}$ and $\widetilde{B}$ is are strict transforms of $A_1$, $A_2$ and $B$  respectively and $E_P$ is the exceptional divisor. 
We have
$S_S(E_P)=\frac{5}{3}$. Thus, $\delta_P(S)\le \frac{2}{5/3}=\frac{6}{5}$. Moreover if $O\in E_P\backslash (\widetilde{A}_1\cup\widetilde{A}_2)$ or if $O\in E_P\cap (\widetilde{A}_1\cup\widetilde{A}_2)$:
$$h(v)\le \begin{cases}
 \frac{v^2}{18} \text{ if }v\in[0,2],\\
 \frac{2(3 - v) (2 v - 3)}{9} \text{ if }v\in[2,3].
\end{cases}
\text{ or }
h(v)=\begin{cases}
\frac{v^2}{6}\text{ if }v\in[0,2],\\
 \frac{(3 - v) (v + 6)}{9} \text{ if }v\in[2,3].
\end{cases}
$$
Thus, $S(W_{\bullet,\bullet}^{E_P};O)= \frac{1}{3}\le \frac{5}{6}$
or
$S(W_{\bullet,\bullet}^{E_P};O)= \frac{7}{9}\le \frac{5}{6}$.
We get that $\delta_P(S)=\frac{6}{5}$ for  $P=A_1\cap A_2$.
\end{proof}
\begin{lemma}\label{deg2-65-A3points}
Suppose $P$ belongs to a $(-2)$-curve $A$ and there exist $(-1)$-curves and $(-2)$-curves   which form the following dual graph:
\begin{figure}[h!]
    \centering
\includegraphics[width=14cm]{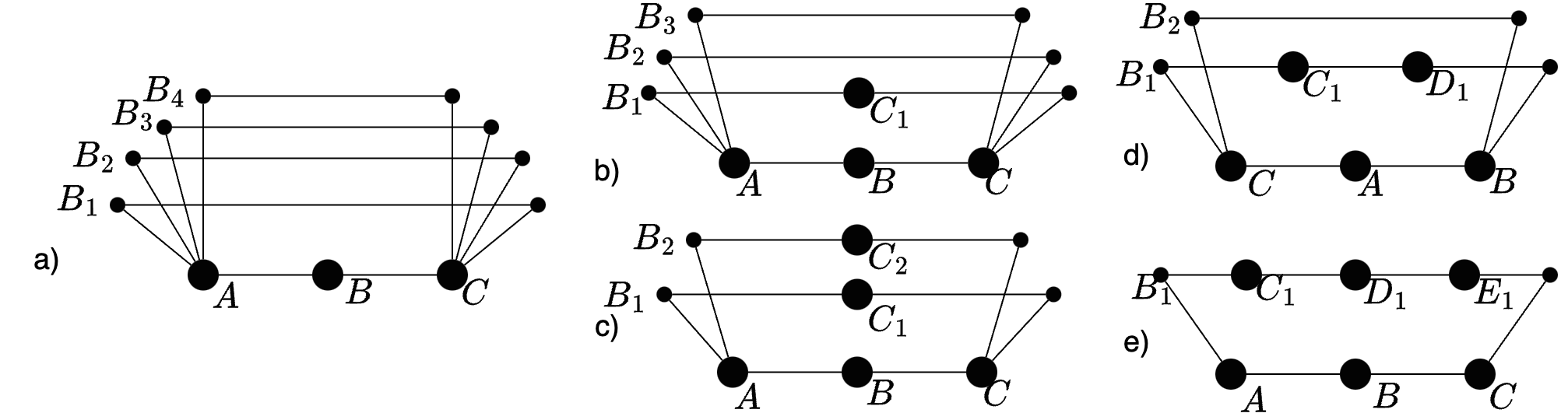}
    \caption{Dual graph: $(-K_S)^2=2$ and $\delta_P(S)=\frac{6}{5}$}
\end{figure}
\par Then $\tau(A)=\frac{3}{2}$ and the Zariski Decomposition of the divisor $-K_S-vA$ is given by:
{{\allowdisplaybreaks\begin{align*}
&{\text{\bf a). }}&P(v)=
\begin{cases}
-K_S-vA-\frac{v}{3}(2B+C)\text{ if }v\in[0,1],\\
-K_S-vA-\frac{v}{3}(2B+C)-(v-1)(B_1+B_2+B_3+B_4)\text{ if }v\in\big[1,\frac{3}{2}\big].
\end{cases}\\
&&N(v)=\begin{cases}
\frac{v}{3}(2B+C)\text{ if }v\in[0,1],\\
\frac{v}{3}(2B+C)+(v-1)(B_1+B_2+B_3+B_4)\text{ if }v\in\big[1,\frac{3}{2}\big].
\end{cases}\\
&{\text{\bf b). }}&P(v)=
\begin{cases}
-K_S-vA-\frac{v}{3}(2B+C)\text{ if }v\in[0,1],\\
-K_S-vA-\frac{v}{3}(2B+C)-(v-1)(2B_1+C_1+B_2+B_3)\text{ if }v\in\big[1,\frac{3}{2}\big].
\end{cases}\\
&&N(v)=\begin{cases}
\frac{v}{3}(2B+C)\text{ if }v\in[0,1],\\
\frac{v}{3}(2B+C)+(v-1)(2B_1+C_1+B_2+B_3)\text{ if }v\in\big[1,\frac{3}{2}\big].
\end{cases}\\
&{\text{\bf c). }}&P(v)=
\begin{cases}
-K_S-vA-\frac{v}{3}(2B+C)\text{ if }v\in[0,1],\\
-K_S-vA-\frac{v}{3}(2B+C)-(v-1)(2B_1+C_1+2B_2+C_2)\text{ if }v\in\big[1,\frac{3}{2}\big].
\end{cases}\\
&&N(v)=\begin{cases}
\frac{v}{3}(2B+C)\text{ if }v\in[0,1],\\
\frac{v}{3}(2B+C)+(v-1)(2B_1+C_1+2B_2+C_2)\text{ if }v\in\big[1,\frac{3}{2}\big].
\end{cases}\\
&{\text{\bf d). }}&P(v)=
\begin{cases}
-K_S-vA-\frac{v}{3}(2B+C)\text{ if }v\in[0,1],\\
-K_S-vA-\frac{v}{3}(2B+C)-(v-1)(3B_1+2C_1+D_1+B_2)\text{ if }v\in\big[1,\frac{3}{2}\big].
\end{cases}\\
&&N(v)=\begin{cases}
\frac{v}{3}(2B+C)\text{ if }v\in[0,1],\\
\frac{v}{3}(2B+C)+(v-1)(3B_1+2C_1+D_1+B_2)\text{ if }v\in\big[1,\frac{3}{2}\big].
\end{cases}\\
&{\text{\bf e). }}&P(v)=
\begin{cases}
-K_S-vA-\frac{v}{3}(2B+C)\text{ if }v\in[0,1],\\
-K_S-vA-\frac{v}{3}(2B+C)-(v-1)(4B_1+3C_1+2D_1+E_1)\text{ if }v\in\big[1,\frac{3}{2}\big].
\end{cases}\\
&&N(v)=\begin{cases}
\frac{v}{3}(2B+C)\text{ if }v\in[0,1],\\
\frac{v}{3}(2B+C)+(v-1)(4B_1+3C_1+2D_1+E_1)\text{ if }v\in\big[1,\frac{3}{2}\big].
\end{cases}
\end{align*}}}
Moreover, 
$$(P(v))^2=\begin{cases}
2 - \frac{4v^2}{3} \text{ if }v\in[0,1],\\
\frac{2(3-2v)^2}{3}\text{ if }v\in\big[1,\frac{3}{2}\big].
\end{cases}P(v)\cdot A=\begin{cases}
\frac{4v}{3}\text{ if }v\in[0,1],\\
4(1-\frac{2v}{3})\text{ if }v\in\big[1,\frac{3}{2}\big].
\end{cases}$$
In this case $\delta_P(S)=\frac{6}{5}\text{ if }P\in A\backslash B.$
\end{lemma}

\begin{proof}
 The Zariski Decomposition in part a). follows from 
$$-K_S-vA\sim_{\DR} \Big(\frac{3}{2}-v\Big)A+\frac{1}{2}\Big(2B+C+B_1+B_2+B_3+B_4\Big).$$ 
A similar statement holds in other parts.  We have $S_S(A)=\frac{5}{6}.$
Thus, $\delta_P(S)\le \frac{6}{5}$ for $P\in A$. Note that for $P\in A\backslash B$ we have:
$$h(v)\le \begin{cases}\frac{8v^2}{9}\text{ if }v\in[0,1],\\
\frac{4 (3 - 2 v) (3 - v)}{9}\text{ if }v\in\big[1, \frac{3}{2}\big].
\end{cases}$$
So we have 
$S(W_{\bullet,\bullet}^{A};P)\le\frac{1}{2}<\frac{5}{6}.$ Thus, $\delta_P(S)=\frac{6}{5}$ if $P\in A\backslash B$.
\end{proof}
\begin{lemma}\label{deg2-3631-A4points}
Suppose $P$ belongs to a $(-2)$-curve $A$ and there exist $(-1)$-curves and $(-2)$-curves   which form the following dual graph:
\begin{figure}[h!]
    \centering
\includegraphics[width=15cm]{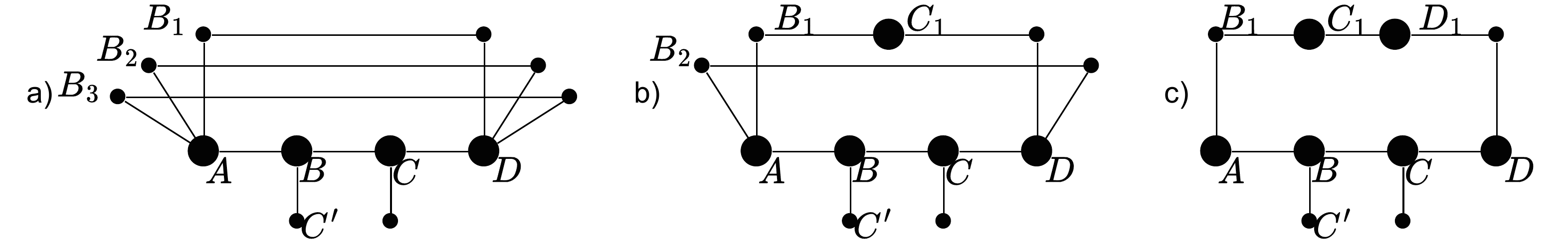}
    \caption{Dual graph: $(-K_S)^2=2$ and $\delta_P(S)=\frac{36}{31}$}
\end{figure}
\par Then $\tau(A)=\frac{3}{2}$ and the Zariski Decomposition of the divisor $-K_S-vA$ is given by:
{{\allowdisplaybreaks\begin{align*}
&{\text{\bf a). }}&
P(v)=
\begin{cases}
-K_S-vA-\frac{v}{4}(3B+2C+D)\text{ if }v\in[0,1],\\
-K_S-vA-\frac{v}{4}(3B+2C+D)-(v-1)(B_1+B_2+B_3)\text{ if }v\in\big[1,\frac{4}{3}\big],\\
-K_S-vA-(v-1)(3B+2C+D+B_1+B_2+B_3)-(3v-4)C'\text{ if }v\in\big[\frac{4}{3},\frac{3}{2}\big].
\end{cases}
\\
&&N(v)=\begin{cases}
\frac{v}{4}(3B+2C+D)\text{ if }v\in[0,1],\\
\frac{v}{4}(3B+2C+D)+(v-1)(B_1+B_2+B_3)\text{ if }v\in\big[1,\frac{4}{3}\big],\\
(v-1)(3B+2C+D+B_1+B_2+B_3)+(3v-4)C'\text{ if }v\in\big[\frac{4}{3},\frac{3}{2}\big].
\end{cases}\\
&{\text{\bf b). }}&
P(v)=
\begin{cases}
-K_S-vA-\frac{v}{4}(3B+2C+D)\text{ if }v\in[0,1],\\
-K_S-vA-\frac{v}{4}(3B+2C+D)-(v-1)(2B_1+C_1+B_2)\text{ if }v\in\big[1,\frac{4}{3}\big],\\
-K_S-vA-(v-1)(3B+2C+D+2B_1+C_1+B_2)-(3v-4)C'\text{ if }v\in\big[\frac{4}{3},\frac{3}{2}\big].
\end{cases}
\\
&&N(v)=\begin{cases}
\frac{v}{4}(3B+2C+D)\text{ if }v\in[0,1],\\
\frac{v}{4}(3B+2C+D)+(v-1)(2B_1+C_1+B_2)\text{ if }v\in\big[1,\frac{4}{3}\big],\\
(v-1)(3B+2C+D+2B_1+C_1+B_2)+(3v-4)C'\text{ if }v\in\big[\frac{4}{3},\frac{3}{2}\big].
\end{cases}\\
&{\text{\bf c). }}&
P(v)=
\begin{cases}
-K_S-vA-\frac{v}{4}(3B+2C+D)\text{ if }v\in[0,1],\\
-K_S-vA-\frac{v}{4}(3B+2C+D)-(v-1)(3B_1+2C_1+D_1)\text{ if }v\in\big[1,\frac{4}{3}\big],\\
-K_S-vA-(v-1)(3B+2C+D+3B_1+2C_1+D_1)-(3v-4)C'\text{ if }v\in\big[\frac{4}{3},\frac{3}{2}\big].
\end{cases}
\\
&&N(v)=\begin{cases}
\frac{v}{4}(3B+2C+D)\text{ if }v\in[0,1],\\
\frac{v}{4}(3B+2C+D)+(v-1)(3B_1+2C_1+D_1)\text{ if }v\in\big[1,\frac{4}{3}\big],\\
(v-1)(3B+2C+D+3B_1+2C_1+D_1)+(3v-4)C'\text{ if }v\in\big[\frac{4}{3},\frac{3}{2}\big].
\end{cases}
\end{align*}}}
Moreover, 
$$(P(v))^2=\begin{cases}2-\frac{5v^2}{4}\text{ if }v\in[0,1],\\
\frac{(10 - 7v)(2 - v)}{4} \text{ if }v\in\big[1,\frac{4}{3}\big],\\
(3-2v)^2\text{ if }v\in\big[\frac{4}{3},\frac{3}{2}\big].
\end{cases}P(v)\cdot A=\begin{cases}
\frac{5v}{4}\text{ if }v\in[0,1],\\
3 - \frac{7v}{4}\text{ if }v\in\big[1,\frac{4}{3}\big],\\
 2(3-2v)\text{ if }v\in\big[\frac{4}{3},\frac{3}{2}\big].
\end{cases}$$
In this case $\delta_P(S)=\frac{36}{31}\text{ if }P\in A\backslash B.$
\end{lemma}

\begin{proof}
 In part a). the Zariski Decomposition  follows from 
 $$-K_S-vA\sim_{\DR} \Big(\frac{3}{2}-v\Big)A+\frac{1}{2}\Big(3B+2C+D+B_1+B_2+B_3+C'\Big).$$
A similar statement holds in other parts. We have
$S_S(A)=\frac{31}{36}.$
Thus, $\delta_P(S)\le \frac{36}{31}$ for $P\in A$. Note that for $P\in A\backslash B$ we have:
$$h(v)\le 
\begin{cases}
\frac{25v^2}{32}\text{ if }v\in[0,1],\\
\frac{ (12 - 7 v) (17v - 12)}{32}\text{ if }v\in\big[1,\frac{4}{3}\big],\\
2 (3 - 2 v)v\text{ if }v\in\big[\frac{4}{3},\frac{3}{2}\big]
\end{cases}$$
So 
$S(W_{\bullet,\bullet}^{A};P)\le\frac{26}{36}<\frac{31}{36}$. Thus, $\delta_P(S)=\frac{36}{31}$ if $ P\in A\backslash B$.
\end{proof}
\begin{lemma}\label{deg2-87-A5points}
Suppose $P$ belongs to a $(-2)$-curve $A$ and there exist $(-1)$-curves and $(-2)$-curves   which form the following dual graph:
\begin{figure}[h!]
    \centering
\includegraphics[width=11cm]{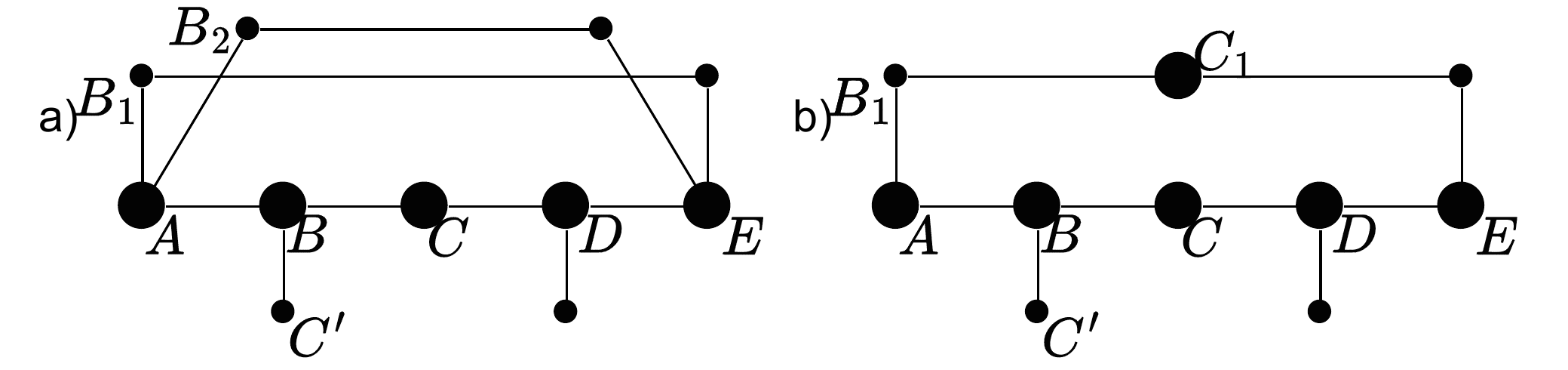}
    \caption{Dual graph: $(-K_S)^2=2$ and $\delta_P(S)=\frac{8}{7}$}
\end{figure}
\par  Then $\tau(A)=\frac{3}{2}$ and the Zariski Decomposition of the divisor $-K_S-vA$ is given by:
{ {\allowdisplaybreaks\begin{align*}
&{\text{\bf a). }}&P(v)=
\begin{cases}
-K_S-vA-\frac{v}{5}(4B+3C+2D+E)\text{ if }v\in[0,1],\\
-K_S-vA-\frac{v}{5}(4B+3C+2D+E)-(v-1)(B_1+B_2)\text{ if }v\in\big[1,\frac{5}{4}\big],\\
-K_S-vA-(v-1)(4B+3C+2D+E+B_1+B_2)-(4v-5)C'\text{ if }v\in\big[\frac{5}{4},\frac{3}{2}\big].
\end{cases}\\
\text{}
&&N(v)=\begin{cases}
\frac{v}{5}(4B+3C+2D+E)\text{ if }v\in[0,1],\\
\frac{v}{5}(4B+3C+2D+E)+(v-1)(B_1+B_2)\text{ if }v\in\big[1,\frac{5}{4}\big],\\
(v-1)(4B+3C+2D+E+B_1+B_2)+(4v-5)C'\text{ if }v\in\big[\frac{5}{4},\frac{3}{2}\big].
\end{cases}\\
&{\text{\bf b). }}&P(v)=
\begin{cases}
-K_S-vA-\frac{v}{5}(4B+3C+2D+E)\text{ if }v\in[0,1],\\
-K_S-vA-\frac{v}{5}(4B+3C+2D+E)-(v-1)(2B_1+C_1)\text{ if }v\in\big[1,\frac{5}{4}\big],\\
-K_S-vA-(v-1)(4B+3C+2D+E+2B_1+C_1)-(4v-5)C'\text{ if }v\in\big[\frac{5}{4},\frac{3}{2}\big].
\end{cases}\\
\text{}
&&N(v)=\begin{cases}
\frac{v}{5}(4B+3C+2D+E)\text{ if }v\in[0,1],\\
\frac{v}{5}(4B+3C+2D+E)+(v-1)(2B_1+C_1)\text{ if }v\in\big[1,\frac{5}{4}\big],\\
(v-1)(4B+3C+2D+E+2B_1+C_1)+(4v-5)C'\text{ if }v\in\big[\frac{5}{4},\frac{3}{2}\big].
\end{cases}
\end{align*}}}
$$(P(v))^2=\begin{cases}2-\frac{6v^2}{5}\text{ if }v\in[0,1],\\
4-4v+\frac{4v^2}{5} \text{ if }v\in\big[1,\frac{5}{4}\big],\\
(3-2v)^2\text{ if }v\in\big[\frac{5}{4},\frac{3}{2}\big].
\end{cases}
P(v)\cdot A=
\begin{cases}
\frac{6v}{5}\text{ if }v\in[0,1],\\
2(1 - \frac{2v}{5})\text{ if }v\in\big[1,\frac{5}{4}\big],\\
2(3 - 2v)\text{ if }v\in\big[\frac{5}{4},\frac{3}{2}\big].
\end{cases}$$
In this case $\delta_P(S)=\frac{8}{7}\text{ if }P\in A\backslash B.$
\end{lemma}
\begin{proof}
In part a). the Zariski Decomposition  follows from 
 $$-K_S-vA\sim_{\DR} \Big(\frac{3}{2}-v\Big)A+\frac{1}{2}\Big(4B + 3C + 2D + E + B_1 + B_2+2C'\Big).$$
A similar statement holds in other parts.
We have
$S_S(A)=\frac{7}{8}.$
Thus, $\delta_P(S)\le \frac{8}{7}$ for $P\in A$. Note that for $P\in A\backslash B$ we have:
$$h(v)\le \begin{cases}
 \frac{18v^2}{25}\text{ if }v\in[0,1],\\
\frac{ 2(5 - 2v)(8v - 5)}{25}\text{ if }v\in\big[1,\frac{5}{4}\big],\\
2 (3 - 2 v)\text{ if }v\in\big[\frac{5}{4},\frac{3}{2}\big].
\end{cases}$$
So 
$S(W_{\bullet,\bullet}^{A};P)\le\frac{7}{12}<\frac{7}{8}$.
Thus, $\delta_P(S)=\frac{8}{7}$ if $P\in A\backslash B$.
\end{proof}
\begin{lemma}\label{deg2-6053-A6points}
Suppose $P$ belongs to a $(-2)$-curve $A$ and there exist $(-1)$-curves and $(-2)$-curves   which form the following dual graph:
\begin{figure}[h!]
    \centering
\includegraphics[width=7cm]{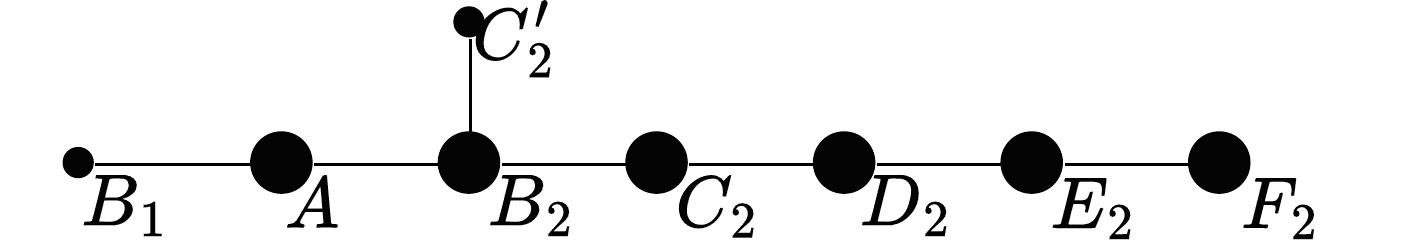}
    \caption{Dual graph: $(-K_S)^2=2$ and $\delta_P(S)=\frac{60}{53}$}
\end{figure}
\par
Then $\tau(A)=\frac{3}{2}$ and the Zariski Decomposition of the divisor $-K_S-vA$ is given by:
{{\allowdisplaybreaks\begin{align*}
\hspace*{-0.5cm}&&P(v)=\begin{cases}
-K_S-vA-\frac{v}{6}(5B_2+4C_2+3D_2+2E_2+F_2)\text{ if }v\in[0,1],\\
-K_S-vA-\frac{v}{6}(5B_2+4C_2+3D_2+2E_2+F_2)-(v-1)B_1\text{ if }v\in\big[1, \frac{6}{5}\big],\\
-K_S-vA-(v-1)(5B_2+4C_2+3D_2+2E_2+F_2+B_1)-(5v-6)C_2'\text{ if }v\in\big[\frac{6}{5},\frac{3}{2}\big].
\end{cases}\\
\hspace*{-0.5cm}&&N(v)=\begin{cases}
\frac{v}{6}(5B_2+4C_2+3D_2+2E_2+F_2)\text{ if }v\in[0,1],\\
\frac{v}{6}(5B_2+4C_2+3D_2+2E_2+F_2)+(v-1)B_1\text{ if }v\in\big[1, \frac{6}{5}\big],\\
(v-1)(5B_2+4C_2+3D_2+2E_2+F_2+B_1)+(5v-6)C_2'\text{ if }v\in\big[\frac{6}{5},\frac{3}{2}\big].
\end{cases}
\end{align*}}}
Moreover, 
$$(P(v))^2=\begin{cases}2-\frac{7v^2}{6}\text{ if }v\in[0,1],\\
3-2v-\frac{v^2}{6} \text{ if }v\in\big[1, \frac{6}{5}\big],\\
(3-2v)^2\text{ if }v\in\big[\frac{6}{5},\frac{3}{2}\big].
\end{cases}P(v)\cdot A=\begin{cases}
\frac{7v}{6}\text{ if }v\in[0,1],\\
1 + \frac{v}{6}\text{ if }v\in\big[1, \frac{6}{5}\big],\\
2(3 - 2v)\text{ if }v\in\big[\frac{6}{5},\frac{3}{2}\big].
\end{cases}$$
In this case $\delta_P(S)=\frac{60}{53}\text{ if }P\in A\backslash B_2$.
\end{lemma}

\begin{proof}
The Zariski Decomposition  follows from 
 $$-K_S-vA\sim_{\DR} \Big(\frac{3}{2}-v\Big)A+\frac{1}{2}\Big(5B_2+4C_2+3D_2+2E_2+F_2+B_1+3C_2'\Big).$$
We have
$S_S(A)=\frac{53}{60}$.
Thus, $\delta_P(S)\le \frac{60}{53}$ for $P\in A$. Note that for $P\in A\backslash B_2$ we have:
$$h(v)\le \begin{cases}
 \frac{49v^2}{72}\text{ if }v\in[0,1],\\
\frac{(v + 6) (13 v - 6)}{72}\text{ if }v\in\big[1, \frac{6}{5}\big],\\
 2 (3 - 2 v) (2 - v)\text{ if }v\in\big[\frac{6}{5},\frac{3}{2}\big].
\end{cases}$$
So
$S(W_{\bullet,\bullet}^{A};P)\le\frac{31}{60}<\frac{53}{60}$. Thus, $\delta_P(S)=\frac{60}{53}$ if $P\in  A\backslash B_2$.
\end{proof}
\begin{lemma}\label{deg2-98-A5points}
Suppose $P$ belongs to a $(-2)$-curve $A$ and there exist $(-1)$-curves and $(-2)$-curves   which form the following dual graph:
\begin{figure}[h!]
    \centering
\includegraphics[width=16cm]{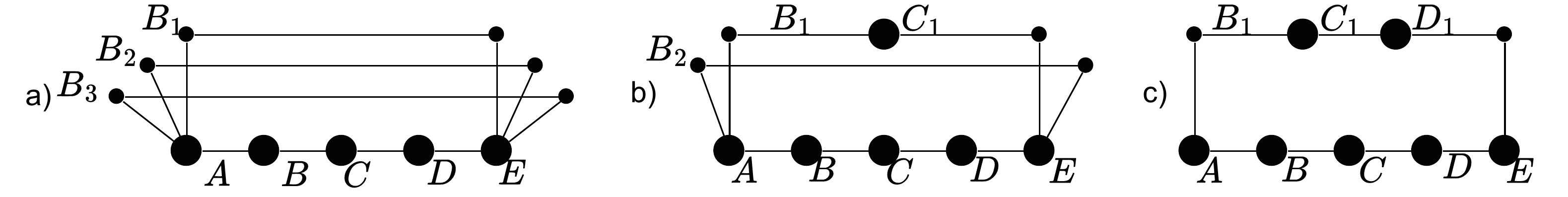}
    \caption{Dual graph: $(-K_S)^2=2$ and $\delta_P(S)=\frac{9}{8}$ with $\tau(A)=\frac{5}{3}$}
\end{figure}
\par 
Then $\tau(A)=\frac{5}{3}$ and the Zariski Decomposition of the divisor $-K_S-vA$ is given by:
{ {\allowdisplaybreaks\begin{align*}
&{\text{\bf a). }}&P(v)=\begin{cases}
-K_S-vA-\frac{v}{5}(4B+3C+2D+E)\text{ if }v\in[0,1],\\
-K_S-vA-\frac{v}{5}(4B+3C+2D+E)-(v-1)(B_1+B_2+B_3)\text{ if }v\in\big[1,\frac{5}{3}\big].
\end{cases}\\
\text{}
&&N(v)=\begin{cases}
\frac{v}{5}(4B+3C+2D+E)\text{ if }v\in[0,1],\\
\frac{v}{5}(4B+3C+2D+E)+(v-1)(B_1+B_2+B_3)\text{ if }v\in\big[1,\frac{5}{3}\big].
\end{cases}\\
&{\text{\bf b). }}&P(v)=\begin{cases}
-K_S-vA-\frac{v}{5}(4B+3C+2D+E)\text{ if }v\in[0,1],\\
-K_S-vA-\frac{v}{5}(4B+3C+2D+E)-(v-1)(2B_1+C_1+B_2)\text{ if }v\in\big[1,\frac{5}{3}\big].
\end{cases}\\
\text{}
&&N(v)=\begin{cases}
\frac{v}{5}(4B+3C+2D+E)\text{ if }v\in[0,1],\\
\frac{v}{5}(4B+3C+2D+E)+(v-1)(2B_1+C_1+B_2)\text{ if }v\in\big[1,\frac{5}{3}\big].
\end{cases}\\
&{\text{\bf c). }}&P(v)=\begin{cases}
-K_S-vA-\frac{v}{5}(4B+3C+2D+E)\text{ if }v\in[0,1],\\
-K_S-vA-\frac{v}{5}(4B+3C+2D+E)-(v-1)(3B_1+2C_1+D_1)\text{ if }v\in\big[1,\frac{5}{3}\big].
\end{cases}\\
\text{}
&&N(v)=\begin{cases}
\frac{v}{5}(4B+3C+2D+E)\text{ if }v\in[0,1],\\
\frac{v}{5}(4B+3C+2D+E)+(v-1)(3B_1+2C_1+D_1)\text{ if }v\in\big[1,\frac{5}{3}\big].
\end{cases}
\end{align*}}}
Moreover, 
$$(P(v))^2=
\begin{cases}
2 - \frac{6v^2}{5}\text{ if }v\in[0,1],\\
\frac{(5-3v)^2}{5}\text{ if }v\in\big[1,\frac{5}{3}\big].
\end{cases}P(v)\cdot A=\begin{cases}\frac{6v}{5}\text{ if }v\in[0,1],\\
3(1-\frac{3v}{5})\text{ if }v\in\big[1,\frac{5}{3}\big].
\end{cases}$$
In this case $\delta_P(S)=\frac{9}{8}\text{ if }P\in A\backslash B$.
\end{lemma}

\begin{proof}
In part a). the Zariski Decomposition  follows from 
 $$-K_S-vA\sim_{\DR} \Big(\frac{5}{3}-v\Big)A+\frac{1}{3}\Big(4B + 3C + 2D + E+2B_1+2B_2+2B_3\Big).$$
A similar statement holds in other parts.
We have $S_S(A)=\frac{8}{9}$. Thus, $\delta_P(S)\le \frac{9}{8}$ for $P\in A$. Note that for $P\in A\backslash B$ we have:
$$h(v)\le \begin{cases} 
\frac{18v^2}{25}\text{ if }v\in[0,1],\\
 \frac{9 (5 - 3 v) (7v - 5)}{50}\text{ if }v\in\big[1,\frac{5}{3}\big].
\end{cases}$$
So we have 
$S(W_{\bullet,\bullet}^{A};P)\le\frac{2}{3}<\frac{8}{9}$.
Thus, $\delta_P(S)=\frac{9}{8}$ if $P\in A\backslash B$.
\end{proof}
\begin{lemma}\label{deg2-98-A7points}
Suppose $P$ belongs to a $(-2)$-curve $A$ and there exist $(-1)$-curves and $(-2)$-curves   which form the following dual graph:
\begin{figure}[h!]
    \centering
\includegraphics[width=7cm]{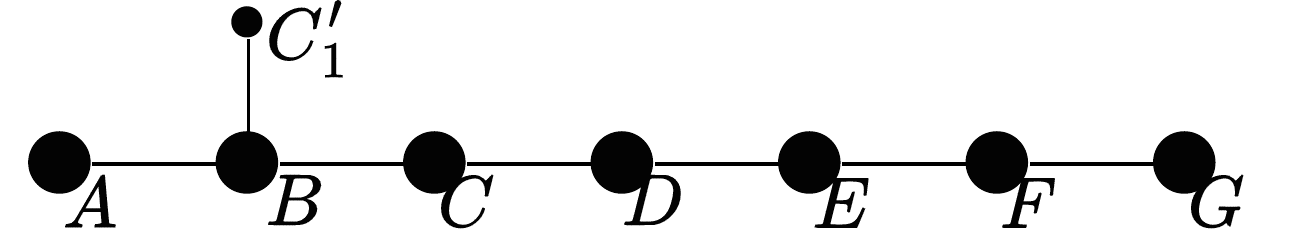}
    \caption{Dual graph: $(-K_S)^2=2$ and $\delta_P(S)=\frac{9}{8}$ with $\tau(A)=\frac{3}{2}$}
\end{figure}
\par
Then $\tau(A)=\frac{3}{2}$ and the Zariski Decomposition of the divisor $-K_S-vA$ is given by:
{{\allowdisplaybreaks\begin{align*}
&&P(v)=\begin{cases}
-K_S-vA-\frac{v}{7}(6B+5C+4D+3E+2F+G)\text{ if }v\in\big[0,\frac{7}{6}\big],\\
-K_S-vA-(v-1)(6B+5C+4D+3E+2F+G)-(6v - 7)C'\text{ if }v\in\big[\frac{7}{6},\frac{3}{2}\big].
\end{cases}\\
&&N(v)=\begin{cases}\frac{v}{7}(6B+5C+4D+3E+2F+G)\text{ if }v\in\big[0,\frac{7}{6}\big],\\
(v-1)(6B+5C+4D+3E+2F+G)+(6v - 7)C'\text{ if }v\in\big[\frac{7}{6},\frac{3}{2}\big].
\end{cases}
\end{align*}}}
Moreover, 
$$(P(v))^2=\begin{cases}2 - \frac{8v^2}{7}\text{ if }v\in\big[0,\frac{7}{6}\big],\\
(3-2v)^2\text{ if }v\in\big[\frac{7}{6},\frac{3}{2}\big].
\end{cases}P(v)\cdot A=\begin{cases}
\frac{8v}{7}\text{ if }v\in\big[0,\frac{7}{6}\big],\\
2(3-2v)\text{ if }v\in\big[\frac{7}{6},\frac{3}{2}\big].
\end{cases}$$
In this case $\delta_P(S)=\frac{9}{8}\text{ if }P\in A\backslash B$.
\end{lemma}

\begin{proof}
The Zariski Decomposition  follows from 
 $$-K_S-vA\sim_{\DR} \Big(\frac{3}{2}-v\Big)A+\frac{1}{2}\Big(6B+5C+4D+3E+2F+G+4C'\Big).$$
We have
$S_S(A)=\frac{8}{9}$.
Thus, $\delta_P(S)\le \frac{9}{8}$ for $P\in A$. Note that for $P\in  A\backslash B$ we have:
$$h(v)\le \begin{cases} 
 \frac{32v^2}{49}\text{ if }v\in\big[0,\frac{7}{6}\big],\\
2(3-2v)^2\text{ if }v\in\big[\frac{7}{6},\frac{3}{2}\big].
\end{cases}$$
So  $S(W_{\bullet,\bullet}^{A};P)\le\frac{4}{9}<\frac{8}{9}$.
Thus, $\delta_P(S)=\frac{9}{8}$ if $P\in  A\backslash B$.
\end{proof}
\begin{lemma}\label{deg2-1-middleA3}
Suppose $P$ belongs to a $(-2)$-curve $A$ and there exist $(-1)$-curves and $(-2)$-curves   which form the following dual graph:
\begin{figure}[h!]
    \centering
\includegraphics[width=14cm]{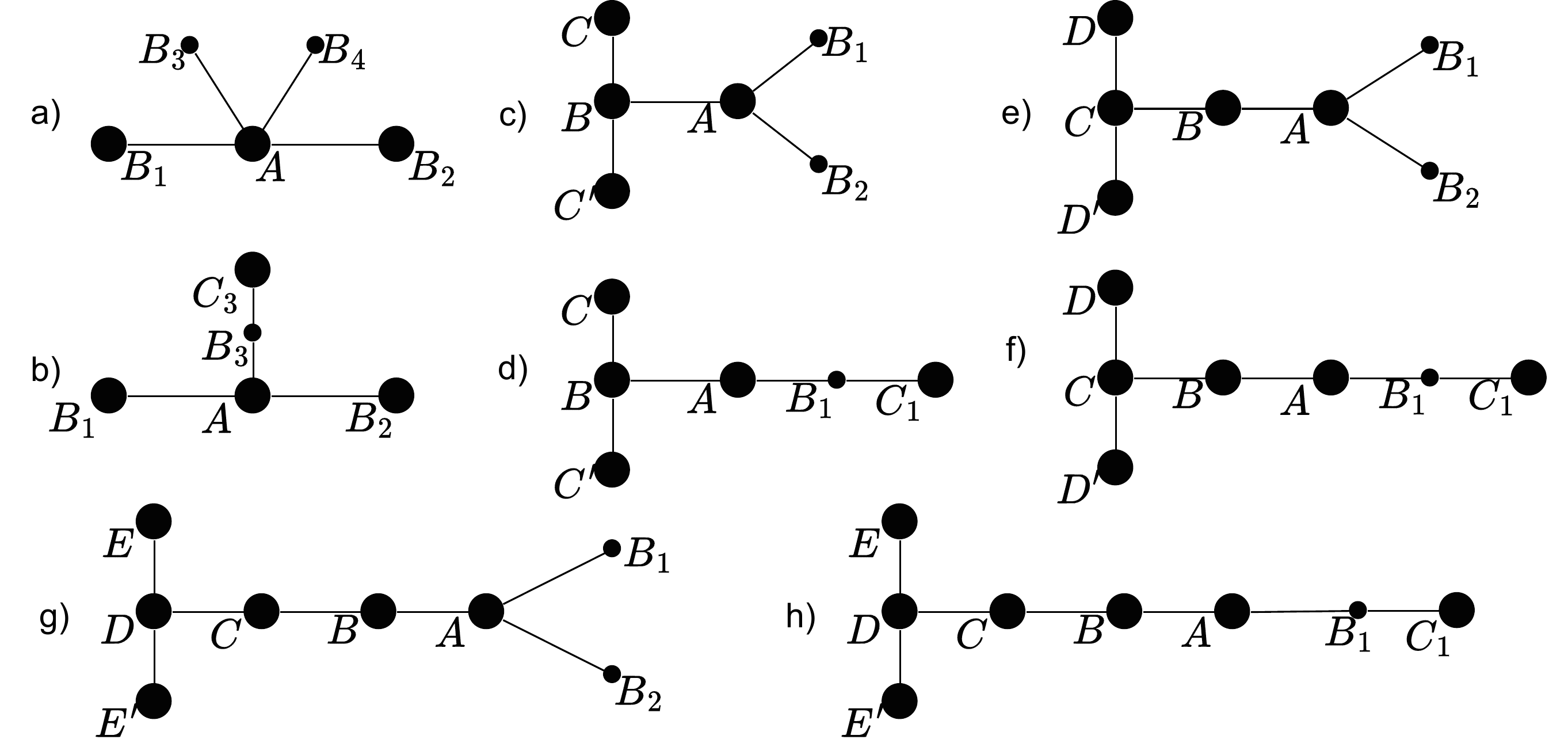}
    \caption{Dual graph: $(-K_S)^2=2$ and $\delta_P(S)=1$}
\end{figure}
\par Then $\tau(A)=2$ and the Zariski Decomposition of the divisor $-K_S-vA$ is given by:
{ {\allowdisplaybreaks\begin{align*}
&{\text{\bf a). }}&P(v)=\begin{cases}-K_S-vA-\frac{v}{2}(B_1+B_2)\text{ if }v\in[0,1],\\
-K_S-vA-\frac{v}{2}(B_1+B_2)-(v-1)(B_3+B_4)\text{ if }v\in[1,2].
\end{cases}\\&&
N(v)=\begin{cases}\frac{v}{2}(B_1+B_2)\text{ if }v\in[0,1],\\
\frac{v}{2}(B_1+B_2)+(v-1)(B_3+B_4)\text{ if }v\in[1,2].
\end{cases}\\
&{\text{\bf b). }}&P(v)=\begin{cases}-K_S-vA-\frac{v}{2}(B_1+B_2)\text{ if }v\in[0,1],\\
-K_S-vA-\frac{v}{2}(B_1+B_2)-(v-1)(2B_3+C_3)\text{ if }v\in[1,2].
\end{cases}\\&&
N(v)=\begin{cases}\frac{v}{2}(B_1+B_2)\text{ if }v\in[0,1],\\
\frac{v}{2}(B_1+B_2)+(v-1)(2B_3+C_3)\text{ if }v\in[1,2].
\end{cases}\\
&{\text{\bf c). }}&P(v)=\begin{cases}-K_S-vA-\frac{v}{2}(2B+C+C')\text{ if }v\in[0,1],\\
-K_S-vA-\frac{v}{2}(2B+C+C')-(v-1)(B_1+B_2)\text{ if }v\in[1,2].
\end{cases}\\&&
N(v)=\begin{cases}\frac{v}{2}(2B+C+C')\text{ if }v\in[0,1],\\
\frac{v}{2}(2B+C+C')+(v-1)(B_1+B_2)\text{ if }v\in[1,2].
\end{cases}\\
&{\text{\bf d). }}&P(v)=\begin{cases}-K_S-vA-\frac{v}{2}(2B+C+C')\text{ if }v\in[0,1],\\
-K_S-vA-\frac{v}{2}(2B+C+C')-(v-1)(2B_1+C_1)\text{ if }v\in[1,2].
\end{cases}\\&&
N(v)=\begin{cases}\frac{v}{2}(2B+C+C')\text{ if }v\in[0,1],\\
\frac{v}{2}(2B+C+C')+(v-1)(2B_1+C_1)\text{ if }v\in[1,2].
\end{cases}\\
&{\text{\bf e). }}&P(v)=\begin{cases}-K_S-vA-\frac{v}{2}(2B+2C+D+D')\text{ if }v\in[0,1],\\
-K_S-vA-\frac{v}{2}(2B+2C+D+D')-(v-1)(B_1+B_2)\text{ if }v\in[1,2].
\end{cases}\\&&
N(v)=\begin{cases}\frac{v}{2}(2B+2C+D+D')\text{ if }v\in[0,1],\\
\frac{v}{2}(2B+2C+D+D')+(v-1)(B_1+B_2)\text{ if }v\in[1,2].
\end{cases}\\
&{\text{\bf f). }}&P(v)=\begin{cases}-K_S-vA-\frac{v}{2}(2B+2C+D+D')\text{ if }v\in[0,1],\\
-K_S-vA-\frac{v}{2}(2B+2C+D+D')-(v-1)(2B_1+C_1)\text{ if }v\in[1,2].
\end{cases}\\&&
N(v)=\begin{cases}\frac{v}{2}(2B+2C+D+D')\text{ if }v\in[0,1],\\
\frac{v}{2}(2B+2C+D+D')+(v-1)(2B_1+C_1)\text{ if }v\in[1,2].
\end{cases}\\
&{\text{\bf g). }}&P(v)=\begin{cases}-K_S-vA-\frac{v}{2}(2B+2C+2D+E+E')\text{ if }v\in[0,1],\\
-K_S-vA-\frac{v}{2}(2B+2C+2D+E+E')-(v-1)(B_1+B_2)\text{ if }v\in[1,2].
\end{cases}\\&&
N(v)=\begin{cases}\frac{v}{2}(2B+2C+2D+E+E')\text{ if }v\in[0,1],\\
\frac{v}{2}(2B+2C+2D+E+E')+(v-1)(B_1+B_2)\text{ if }v\in[1,2].
\end{cases}\\
&{\text{\bf h). }}&P(v)=\begin{cases}-K_S-vA-\frac{v}{2}(2B+2C+2D+E+E')\text{ if }v\in[0,1],\\
-K_S-vA-\frac{v}{2}(2B+2C+2D+E+E')-(v-1)(2B_1+C_1)\text{ if }v\in[1,2].
\end{cases}\\&&
N(v)=\begin{cases}\frac{v}{2}(2B+2C+2D+E+E')\text{ if }v\in[0,1],\\
\frac{v}{2}(2B+2C+2D+E+E')+(v-1)(2B_1+C_1)\text{ if }v\in[1,2].
\end{cases}
\end{align*}}}
Moreover, 
$$(P(v))^2=\begin{cases}2 - v^2\text{ if }v\in[0,1],\\
(2-v)^2\text{ if }v\in[1,2].
\end{cases}P(v)\cdot A=\begin{cases}v\text{ if }v\in[0,1],\\
2-v\text{ if }v\in[1,2].
\end{cases}$$
In this case $\delta_P(S)=1\text{ if }P\in A\backslash B.$
\end{lemma}

\begin{proof}
The Zariski Decomposition in part a). follows from
$-K_S-vA\sim_{\DR} \big(2-v\big)A+B_1+B_2+B_3+B_4$.
A similar statement holds in other parts. We have
$S_S(A)=1.$ Thus, $\delta_P(S)\le 1$ for $P\in A$. Note that for $P\in A\backslash (B_1\cup B_2\cup B)$ or $P\in A\cap (B_1\cup B_2)$ we have:
$$h(v)\le \begin{cases}v^2\text{ if }v\in[0,1],\\
2 - v\text{ if }v\in[1,2].
\end{cases}
\text{ or }
h(v)\le \begin{cases}
\frac{v^2}{2}\text{ if }v\in[0,1],\\
\frac{(2 - v)v}{2}\text{ if }v\in[1,2].
\end{cases}$$
So 
$S(W_{\bullet,\bullet}^{A};P)\le\frac{5}{6}<1\text{ or }S(W_{\bullet,\bullet}^{A};P)\le\frac{1}{2}<1.$
Thus, $\delta_P(S)=1$ if $P\in A$.
\end{proof}
\begin{lemma}\label{deg2-1213-A4points}
Suppose $P$ belongs to a $(-2)$-curve $A$ and there exist $(-1)$-curves and $(-2)$-curves   which form the following dual graph:
\begin{figure}[h!]
    \centering
\includegraphics[width=5cm]{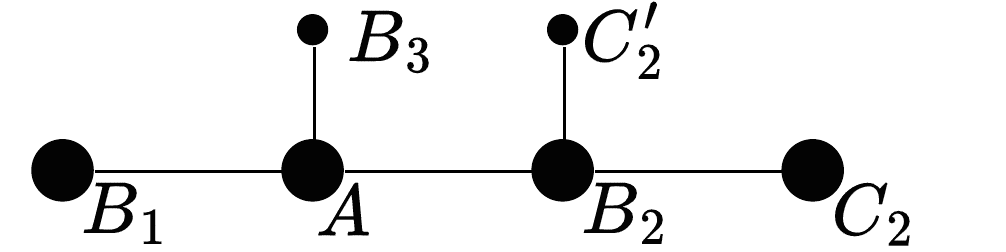}
    \caption{Dual graph: $(-K_S)^2=2$ and $\delta_P(S)=\frac{12}{13}$}
\end{figure}
\par Then $\tau(A)=2$ and the Zariski Decomposition of the divisor $-K_S-vA$ is given by:
{ {\allowdisplaybreaks\begin{align*}
&&P(v)=\begin{cases}
-K_S-vA-\frac{v}{6}(3B_1+4B_2+2C_2)\text{ if }v\in[0,1],\\
-K_S-vA-\frac{v}{6}(3B_1+4B_2+2C_2)-(v-1)B_3\text{ if }v\in\big[1,\frac{3}{2}\big],\\
-K_S-vA-(v-1)(2B_2+C_2+B_3)-\frac{v}{2}B_1-(2v-3)C_2'\text{ if }v\in\big[\frac{3}{2},2\big].
\end{cases}\\
\text{}
&&N(v)=\begin{cases}
\frac{v}{6}(3B_1+4B_2+2C_2)\text{ if }v\in[0,1],\\
\frac{v}{6}(3B_1+4B_2+2C_2)+(v-1)B_3\text{ if }v\in\big[1,\frac{3}{2}\big],\\
(v-1)(2B_2+C_2+B_3)+\frac{v}{2}B_1+(2v-3)C_2'\text{ if }v\in\big[\frac{3}{2},2\big].
\end{cases}
\end{align*}}}
Moreover, 
$$(P(v))^2=\begin{cases}2-\frac{5v^2}{6}\text{ if }v\in[0,1],\\
3-2v+\frac{v^2}{6} \text{ if }v\in\big[1,\frac{3}{2}\big],\\
\frac{3(2 - v)^2}{2}\text{ if }v\in\big[\frac{3}{2},2\big].
\end{cases}P(v)\cdot A=\begin{cases}
\frac{5v}{6}\text{ if }v\in[0,1],\\
1 - \frac{v}{6}\text{ if }v\in\big[1,\frac{3}{2}\big],\\
3(1 - \frac{v}{2})\text{ if }v\in\big[\frac{3}{2},2\big].
\end{cases}$$
In this case $\delta_P(S)=\frac{12}{13}\text{ if }P\in A.$
\end{lemma}
\begin{proof}
 The Zariski Decomposition  follows from 
 $-K_S-vA\sim_{\DR} (2-v)A+B_1+2B_2+C_2+C_2'+B_3$.
We have
$S_S(A)=\frac{13}{12}.$
Thus, $\delta_P(S)\le \frac{12}{13}$ for $P\in A$. Note that if $P\in A\cap B_2$ or if $P\in A\backslash B_2$  we have:
$$h(v)\le \begin{cases}
\frac{65v^2}{72}\text{ if }v\in[0,1],\\
 \frac{(6 - v) (7 v + 6)}{72}\text{ if }v\in\big[1,\frac{3}{2}\big],\\
 \frac{ 3 (2 - v) (5 v - 2)}{8}\text{ if }v\in\big[\frac{3}{2},2\big].
\end{cases}
\text{ or }
h(v)=\begin{cases}
 \frac{25v^2}{72}\text{ if }v\in[0,1],\\
 \frac{(6 - v) (11 v - 6)}{72}\text{ if }v\in\big[1,\frac{3}{2}\big],\\
\frac{ 3 (2- v ) (v + 2)}{8}\text{ if }v\in\big[\frac{3}{2},2\big]
\end{cases}$$
So  
$S(W_{\bullet,\bullet}^{A};P)=\frac{13}{12}$
or
$S(W_{\bullet,\bullet}^{A};P) =\frac{13}{24}<\frac{13}{12}.$
Thus, $\delta_P(S)=\frac{12}{13}$ if $P\in  A$.
\end{proof}
\begin{lemma}\label{deg2-910-A5points}
Suppose $P$ belongs to a $(-2)$-curve $A$ and there exist $(-1)$-curves and $(-2)$-curves   which form the following dual graph:
\begin{figure}[h!]
    \centering
\includegraphics[width=5.5cm]{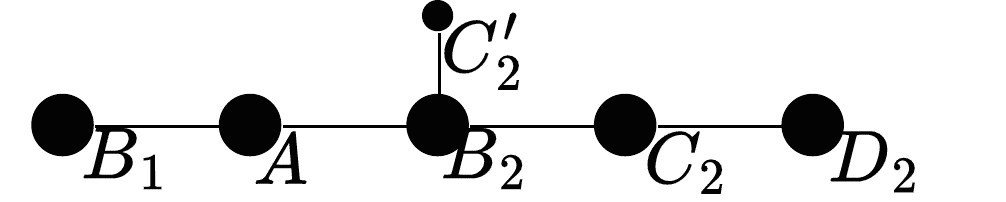}
    \caption{Dual graph: $(-K_S)^2=2$ and $\delta_P(S)=\frac{9}{10}$ with  $-K_S-vA$ nef on $\big[0,\frac{4}{3}\big]$}
\end{figure}
\par Then $\tau(A)=2$ and the Zariski Decomposition of the divisor $-K_S-vA$ is given by:
{ {\allowdisplaybreaks\begin{align*}
&&P(v)=\begin{cases}
-K_S-vA-\frac{v}{4}(2B_1+3B_2+2C_2+D_2)\text{ if }v\in\big[0,\frac{4}{3}\big],\\
-K_S-vA-\frac{v}{2}B_1-(v-1)(3B_2+2C_2+D_2)-(3v-4)C_2'\text{ if }v\in\big[\frac{4}{3},2\big].
\end{cases}\\
&&N(v)=\begin{cases}
\frac{v}{4}(2B_1+3B_2+2C_2+D_2)\text{ if }v\in\big[0,\frac{4}{3}\big],\\
\frac{v}{2}B_1+(v-1)(3B_2+2C_2+D_2)+(3v-4)C_2'\text{ if }v\in\big[\frac{4}{3},2\big].
\end{cases}
\end{align*}}}
Moreover, 
$$(P(v))^2=\begin{cases}2 - \frac{3v^2}{4}\text{ if }v\in\big[0,\frac{4}{3}\big],\\
\frac{3(2-v)^2}{2}\text{ if }v\in\big[\frac{4}{3},2\big].
\end{cases}P(v)\cdot A=\begin{cases}\frac{3v}{4}\text{ if }v\in\big[0,\frac{4}{3}\big],\\
3(1-\frac{v}{2})\text{ if }v\in\big[\frac{4}{3},2\big].
\end{cases}$$
In this case $\delta_P(S)=\frac{9}{10}\text{ if }P\in A\backslash B_2$.
\end{lemma}
\begin{proof}
The Zariski Decomposition  follows from 
 $-K_S-vA\sim_{\DR} (2-v)A+B_1+3B_2+2C_2+D_2+2C_2'$.
We have
$S_S(A)=\frac{10}{9}$.
Thus, $\delta_P(S)\le \frac{9}{10}$ for $P\in A$. Note that for $P\in A\backslash B_2$ we have:
$$h(v)\le \begin{cases} 
 \frac{21v^2}{32}\text{ if }v\in\big[0,\frac{4}{3}\big],\\
 \frac{ 3 (2 - v) (6 - v)}{8}\text{ if }v\in\big[\frac{4}{3},2\big].
\end{cases}$$
So  
$S(W_{\bullet,\bullet}^{A};P)\le\frac{8}{9}<\frac{10}{9}$.
Thus, $\delta_P(S)=\frac{9}{10}$ if $P\in A\backslash B_2$.
\end{proof}
\begin{lemma}\label{deg2-910-D5point}
Suppose $P$ belongs to a $(-2)$-curve $A$ and there exist $(-1)$-curves and $(-2)$-curves   which form the following dual graph:
\begin{figure}[h!]
    \centering
\includegraphics[width=6cm]{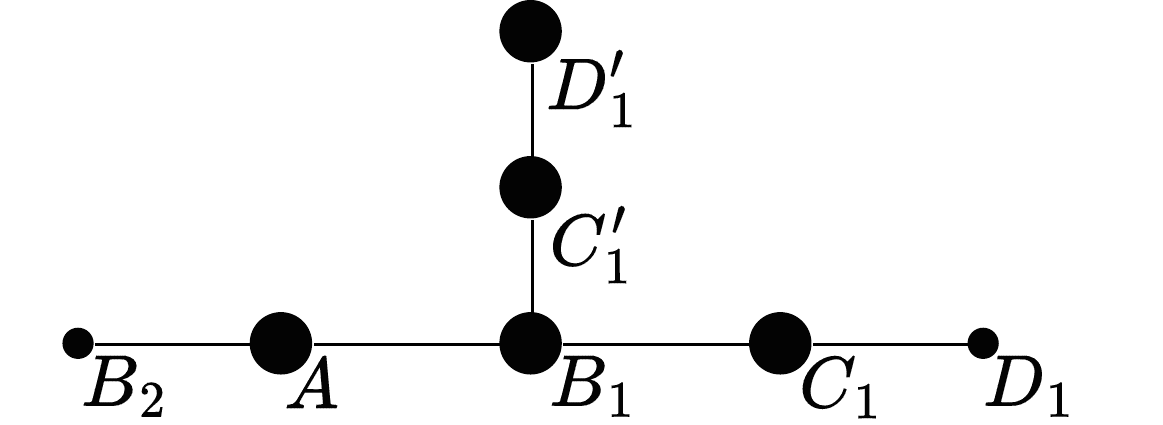}
    \caption{Dual graph: $(-K_S)^2=2$ and $\delta_P(S)=\frac{9}{10}$ with  $-K_S-vA$ nef on $[0,1]$}
\end{figure}
\par 
Then $\tau(A)=2$ and the Zariski Decomposition of the divisor $-K_S-vA$ is given by:
{ {\allowdisplaybreaks\begin{align*}
\hspace*{-0.5cm}&&P(v)=\begin{cases}
-K_S-vA-\frac{v}{5}(3C_1+6B_1+4C_1'+2D_1')\text{ if }v\in[0,1],\\
-K_S-vA-\frac{v}{5}(3C_1+6B_1+4C_1'+2D_1')-(v-1)B_2\text{ if }v\in\big[1,\frac{5}{3}\big],\\
-K_S-vA-(v-1)(3B_1+2C_1'+D_1'+B_2)-(3v-4)C_1-(3v-5)D_1\text{ if }v\in\big[\frac{5}{3},2\big].
\end{cases}\\
\hspace*{-0.5cm}&&N(v)=\begin{cases}
\frac{v}{5}(3C_1+6B_1+4C_1'+2D_1')\text{ if }v\in[0,1],\\
\frac{v}{5}(3C_1+6B_1+4C_1'+2D_1')+(v-1)B_2\text{ if }v\in\big[1,\frac{5}{3}\big],\\
(v-1)(3B_1+2C_1'+D_1'+B_2)+(3v-4)C_1+(3v-5)D_1\text{ if }v\in\big[\frac{5}{3},2\big].
\end{cases}
\end{align*}}}
Moreover, 
$$(P(v))^2=\begin{cases}
2-\frac{4v^2}{5}\text{ if }v\in[0,1],\\
3 - 2v + \frac{v^2}{5} \text{ if }v\in\big[1,\frac{5}{3}\big],\\
2(2-v)^2\text{ if }v\in\big[\frac{5}{3},2\big].
\end{cases}P(v)\cdot A=\begin{cases}
\frac{4v}{5}\text{ if }v\in[0,1],\\
1 - \frac{v}{5}\text{ if }v\in\big[1,\frac{5}{3}\big],\\
 2(2-v)\text{ if }v\in\big[\frac{5}{3},2\big].
\end{cases}$$
In this case $\delta_P(S)=\frac{9}{10}\text{ if }P\in A\backslash B_1$.
\end{lemma}
\begin{proof}
The Zariski Decomposition  follows from 
 $-K_S-vA\sim_{\DR} (2-v)A+3B_1+2C_1'+D_1+2C_1+D_1'+B_2$.
We have $S_S(A)=\frac{10}{9}$.
Thus, $\delta_P(S)\le \frac{9}{10}$ for $P\in A$. Note that for $P\in A\backslash B_1$ we have:
$$h(v)\le 
\begin{cases}
\frac{8v^2}{25}\text{ if }v\in[0,1],\\
\frac{(5 - v)( 9v - 5)}{50}\text{ if }v\in\big[1,\frac{5}{3}\big],\\
2  (2 - v)\text{ if }v\in\big[\frac{5}{3},2\big]
\end{cases}$$
So  $S(W_{\bullet,\bullet}^{A};P)\le\frac{5}{9}<\frac{10}{9}$.
Thus, $\delta_P(S)=\frac{9}{10}$ if $ P\in A\backslash B_1$.
\end{proof}
\begin{lemma}\label{deg2-67-themiddleA5points}
Suppose $P$ belongs to a $(-2)$-curve $A$ and there exist $(-1)$-curves and $(-2)$-curves   which form the following dual graph:
\begin{figure}[h!]
    \centering
\includegraphics[width=11cm]{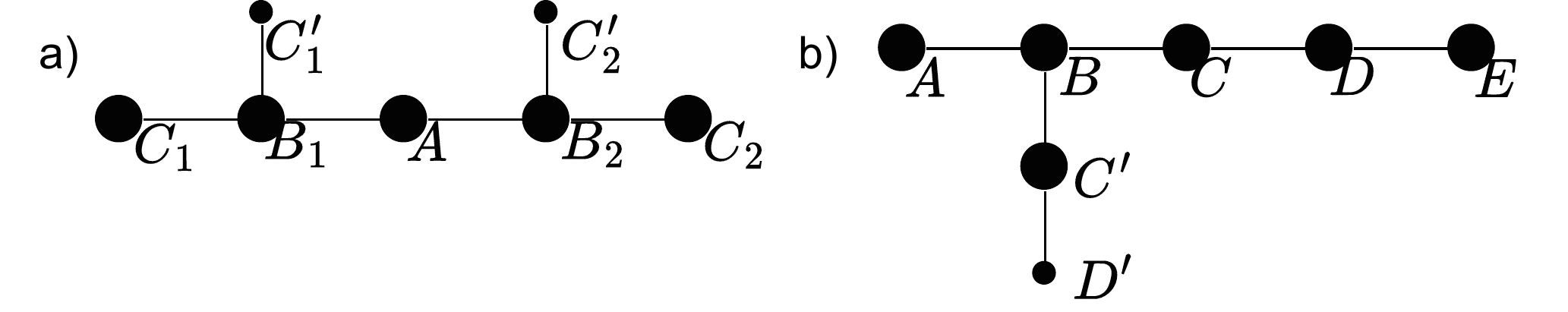}
    \caption{Dual graph: $(-K_S)^2=2$ and $\delta_P(S)=\frac{6}{7}$ with  $-K_S-vA$ nef on $\big[\frac{3}{2},2\big]$}
\end{figure}
\par Then $\tau(A)=2$ and the Zariski Decomposition of the divisor $-K_S-vA$ is given by:
{ {\allowdisplaybreaks\begin{align*}
&{\text{\bf a). }}&P(v)=
\begin{cases}
-K_S-vA-\frac{v}{3}(C_1+2B_1+2B_2+C_2)\text{ if }v\in\big[0,\frac{3}{2}\big],\\
-K_S-vA-(v-1)(C_1+2B_1+2B_2+C_2)-(2v-3)(C_1'+C_2')\text{ if }v\in\big[\frac{3}{2},2\big].
\end{cases}\\
\text{}
&&N(v)=\begin{cases}
\frac{v}{3}(C_1+2B_1+2B_2+C_2)\text{ if }v\in\big[0,\frac{3}{2}\big],\\
(v-1)(C_1+2B_1+2B_2+C_2)+(2v-3)(C_1'+C_2')\text{ if }v\in\big[\frac{3}{2},2\big].
\end{cases}\\
&{\text{\bf b). }}&P(v)=
\begin{cases}
-K_S-vA-\frac{v}{3}(2C'+4B+3C+2D+E)\text{ if }v\in\big[0,\frac{3}{2}\big],\\
-K_S-vA-(v-1)(4B+3C+2D+E)-(4v-5)C_1'-(4v-6)D'\text{ if }v\in\big[\frac{3}{2},2\big].
\end{cases}\\
\text{}
&&N(v)=\begin{cases}
\frac{v}{3}(2C'+4B+3C+2D+E)\text{ if }v\in\big[0,\frac{3}{2}\big],\\
(v-1)(4B+3C+2D+E)+(4v-5)C_1'+(4v-6)D'\text{ if }v\in\big[\frac{3}{2},2\big].
\end{cases}
\end{align*}}}

$$(P(v))^2=\begin{cases}
2 - \frac{2v^2}{3}\text{ if }v\in\big[0,\frac{3}{2}\big],\\
2(2-v)^2\text{ if }v\in\big[\frac{3}{2},2\big].
\end{cases}P(v)\cdot A=\begin{cases}\frac{2v}{3}\text{ if }v\in\big[0,\frac{3}{2}\big],\\
2(2-v)\text{ if }v\in\big[\frac{3}{2},2\big].
\end{cases}$$
In this case $\delta_P(S)=\frac{6}{7}\text{ if }P\in A.$
\end{lemma}
\begin{proof}
In part a). the Zariski Decomposition  follows from 
 $-K_S-vA\sim_{\DR} (2-v)A+2B_1+C_1+C_1'+2B_2+C_2 +C_2'$.
A similar statement holds in other parts.
We have
$S_S(A)=\frac{7}{6}.$
Thus, $\delta_P(S)\le \frac{6}{7}$ for $P\in A$. Note that for $P\in A$ we have:
$$h(v)\le \begin{cases} 
 \frac{2v^2}{3}\text{ if }v\in\big[0,\frac{3}{2}\big],\\
2(2-v)\text{ if }v\in\big[\frac{3}{2},2\big].
\end{cases}$$
So  
$S(W_{\bullet,\bullet}^{A};P)\le\frac{7}{6}.$
Thus, $\delta_P(S)=\frac{6}{7}$ if $P\in A$.
\end{proof}
\begin{lemma}\label{deg2-67-A5points}
Suppose $P$ belongs to a $(-2)$-curve $A$ and there exist $(-1)$-curves and $(-2)$-curves   which form the following dual graph:
\begin{figure}[h!]
    \centering
\includegraphics[width=13cm]{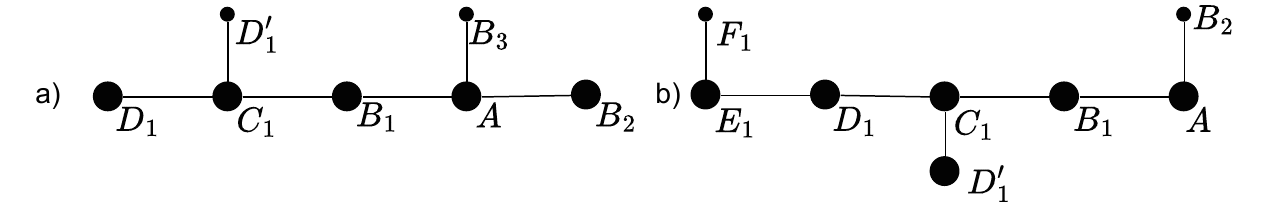}
    \caption{Dual graph: $(-K_S)^2=2$ and $\delta_P(S)=\frac{6}{7}$  with  $-K_S-vA$ nef on $[0,1]$}
\end{figure}
\par 
Then $\tau(A)=2$ and the Zariski Decomposition of the divisor $-K_S-vA$ is given by:
{ {\allowdisplaybreaks\begin{align*}
&{\text{\bf a). }}&P(v)=
\begin{cases}
-K_S-vA-\frac{v}{4}(2B_2+3B_1+2C_1+D_1)\text{ if }v\in[0,1],\\
-K_S-vA-\frac{v}{4}(2B_2+3B_1+2C_1+D_1)-(v-1)B_3\text{ if }v\in[1,2].
\end{cases}\\
\text{}
&&N(v)=\begin{cases}
\frac{v}{4}(2B_2+3B_1+2C_1+D_1)\text{ if }v\in[0,1],\\
\frac{v}{4}(2B_2+3B_1+2C_1+D_1)+(v-1)B_3\text{ if }v\in[1,2].
\end{cases}\\
&{\text{\bf b). }}&P(v)=
\begin{cases}
-K_S-vA-\frac{v}{4}(5B_1+6C_1+4D_1+2E_1+3D_1')\text{ if }v\in[0,1],\\
-K_S-vA-\frac{v}{4}(5B_1+6C_1+4D_1+2E_1+3D_1')-(v-1)B_2\text{ if }v\in[1,2].
\end{cases}\\
\text{}
&&N(v)=\begin{cases}
\frac{v}{4}(5B_1+6C_1+4D_1+2E_1+3D_1')\text{ if }v\in[0,1],\\
\frac{v}{4}(5B_1+6C_1+4D_1+2E_1+3D_1')+(v-1)B_2\text{ if }v\in[1,2].
\end{cases}
\end{align*}}}
Moreover,
$$(P(v))^2=\begin{cases}
2 - \frac{3v^2}{4}\text{ if }v\in[0,1],\\
\frac{(v - 2)(v - 6)}{4}\text{ if }v\in[1,2].
\end{cases}P(v)\cdot A=\begin{cases}
\frac{3v}{4}\text{ if }v\in[0,1],\\
1-\frac{v}{4}\text{ if }v\in[1,2].
\end{cases}$$
In this case $\delta_P(S)=\frac{6}{7}\text{ if }P\in A\backslash B_1.$
\end{lemma}
\begin{proof}
In part a). the Zariski Decomposition  follows from 
 $-K_S-vA\sim_{\DR} (2-v)A + 2B_1 + 2C_1 + D_1+D_1' +B_2+B_3$.
 A similar statement holds in other parts.
We have
$S_S(A)=\frac{7}{6}.$
Thus, $\delta_P(S)\le \frac{6}{7}$ for $P\in A$. Note that for $P\in A\backslash (B_1\cup B_2)$ or if $P\in A\cap B_2$ we have:
$$h(v)\le \begin{cases} 
\frac{9v^2}{32}\text{ if }v\in[0,1],\\
 \frac{(4 - v) (7 v - 4)}{32}\text{ if }v\in[1,2].
\end{cases}
\text{ or }
h(v)\le \begin{cases} 
 \frac{21v^2}{32}\text{ if }v\in[0,1],\\
 \frac{ (4 - v) (3 v + 4)}{32}\text{ if }v\in[1,2].
\end{cases}$$
So 
$S(W_{\bullet,\bullet}^{A};P)\le\frac{7}{12}<\frac{7}{6}.$ or
$S(W_{\bullet,\bullet}^{A};P)\le\frac{7}{8}<\frac{7}{6}.$
Thus, $\delta_P(S)=\frac{6}{7}$ if $P\in A\backslash B_1$.
\end{proof}
\begin{lemma}\label{deg2-45-middleA6points}
Suppose $P$ belongs to a $(-2)$-curve $A$ and there exist $(-1)$-curves and $(-2)$-curves   which form the following dual graph:
\begin{figure}[h!]
    \centering
\includegraphics[width=6cm]{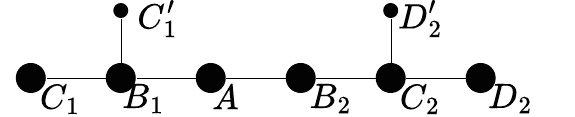}
    \caption{Dual graph: $(-K_S)^2=2$ and $\delta_P(S)=\frac{4}{5}$ with $\tau(A)=2$}
\end{figure}
\par 
Then $\tau(A)=2$ and the Zariski Decomposition of the divisor $-K_S-vA$ is given by:
{{\allowdisplaybreaks\begin{align*}
&&P(v)=\begin{cases}
-K_S-vA-\frac{v}{4}(3B_2+2C_2+D_2)-\frac{v}{3}(2B_1+C_1)\text{ if }v\in\big[0,\frac{3}{2}\big],\\
-K_S-vA-\frac{v}{4}(3B_2+2C_2+D_2)-(v-1)(2B_1+C_1)-(2v-3)C_1'\text{ if }v\in\big[\frac{3}{2},2\big].
\end{cases}\\
&&N(v)=\begin{cases}
\frac{v}{4}(3B_2+2C_2+D_2)+\frac{v}{3}(2B_1+C_1)\text{ if }v\in\big[0,\frac{3}{2}\big],\\
\frac{v}{4}(3B_2+2C_2+D_2)+(v-1)(2B_1+C_1)+(2v-3)C_1'\text{ if }v\in\big[\frac{3}{2},2\big].
\end{cases}
\end{align*}}}
Moreover, 
$$(P(v))^2=\begin{cases}2 - \frac{7v^2}{12}\text{ if }v\in\big[0,\frac{3}{2}\big],\\
\frac{(2-v)(10-3v)}{4}\text{ if }v\in\big[\frac{3}{2},2\big].
\end{cases}P(v)\cdot A=\begin{cases}\frac{7v}{12}\text{ if }v\in\big[0,\frac{3}{2}\big],\\
2-\frac{3v}{4}\text{ if }v\in\big[\frac{3}{2},2\big].
\end{cases}$$
In this case $\delta_P(S)=\frac{4}{5}\text{ if }P\in A$.
\end{lemma}
\begin{proof}
The Zariski Decomposition  follows from 
 $-K_S-vA\sim_{\DR} (2-v)A+2B_1+C_1+C_1'+2B_2+2C_2+D_2+D_2'$.
We have $S_S(A)=\frac{5}{4}$. Thus, $\delta_P(S)\le \frac{4}{5}$ for $P\in A$. Note that for $P\in A \backslash B_2$ or $P\in A \cap B_2$ we have:
$$h(v)\le \begin{cases} 
\frac{161v^2}{288}\text{ if }v\in\big[0,\frac{3}{2}\big],\\
\frac{(8 - 3 v) (13 v - 8)}{32}\text{ if }v\in\big[\frac{3}{2},2\big].
\end{cases}
\text{ or }
h(v)\le \begin{cases} 
\frac{175v^2}{288}\text{ if }v\in\big[0,\frac{3}{2}\big],\\
\frac{(8 -3 v ) (3 v + 8)}{32}\text{ if }v\in\big[\frac{3}{2},2\big].
\end{cases}$$
So  $S(W_{\bullet,\bullet}^{A};P)\le\frac{5}{4}$.
Thus, $\delta_P(S)=\frac{4}{5}$ if $P\in A$.
\end{proof}
\begin{lemma}\label{deg2-45-NOTmiidleA6points}
Suppose $P$ belongs to a $(-2)$-curve $A$ and there exist $(-1)$-curves and $(-2)$-curves   which form the following dual graph:
\begin{figure}[h!]
    \centering
\includegraphics[width=7cm]{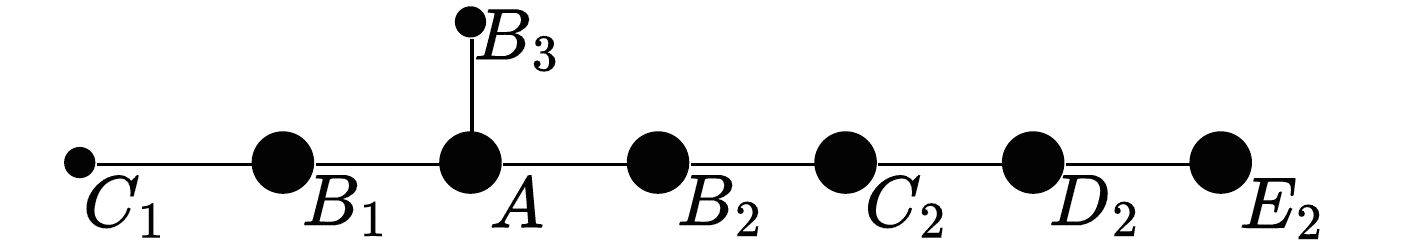}
    \caption{Dual graph: $(-K_S)^2=2$ and $\delta_P(S)=\frac{4}{5}$ with $\tau(A)=\frac{5}{2}$}
\end{figure}
\par  Then $\tau(A)=\frac{5}{2}$ and the Zariski Decomposition of the divisor $-K_S-vA$ is given by:
{{\allowdisplaybreaks\begin{align*}
\hspace*{-0.5cm}&&P(v)=\begin{cases}
-K_S-vA-\frac{v}{5}(4B_2+3C_2+2D_2+E_2)-\frac{v}{2}B_1\text{ if }v\in[0,1],\\
-K_S-vA-\frac{v}{5}(4B_2+3C_2+2D_2+E_2)-\frac{v}{2}B_1-(v-1)B_3\text{ if }v\in[1,2],\\
-K_S-vA-\frac{v}{5}(4B_2+3C_2+2D_2+E_2)-(v-1)(B_1+B_3)-(v-2)C_1\text{ if }v\in\big[2,\frac{5}{2}\big].
\end{cases}\\
\hspace*{-0.5cm}&&N(v)=\begin{cases}
\frac{v}{5}(4B_2+3C_2+2D_2+E_2)+\frac{v}{2}B_1\text{ if }v\in[0,2],\\
\frac{v}{5}(4B_2+3C_2+2D_2+E_2)+\frac{v}{2}B_1+(v-1)B_3\text{ if }v\in[1,2],\\
\frac{v}{5}(4B_2+3C_2+2D_2+E_2)+(v-1)(B_1+B_3)+(v-2)C_1\text{ if }v\in\big[2,\frac{5}{2}\big].
\end{cases}
\end{align*}}}
Moreover, 
$$(P(v))^2=\begin{cases}2-\frac{7v^2}{10}\text{ if }v\in[0,1],\\
3-2v+\frac{3v^2}{10} \text{ if }v\in[1,2],\\
\frac{(5-2v)^2}{5}\text{ if }v\in\big[2,\frac{5}{2}\big].
\end{cases}P(v)\cdot A=\begin{cases}
\frac{7v}{10}\text{ if }v\in[0,1],\\
1 - \frac{3v}{10}\text{ if }v\in[1,2],\\
2(1-\frac{2v}{5})\text{ if }v\in\big[2,\frac{5}{2}\big].
\end{cases}$$
In this case $\delta_P(S)=\frac{4}{5}\text{ if }P\in A\backslash B_2$.
\end{lemma}

\begin{proof}
The Zariski Decomposition  follows from 
 $$-K_S-vA\sim_{\DR} \Big(\frac{5}{2}-v\Big)A+\frac{1}{2}\Big(4B_2+3C_2+2D_2+E_2+3B_3+3B_1+C_1\Big).$$
We have
$S_S(A)=\frac{5}{4}$. Thus, $\delta_P(S)\le \frac{4}{5}$ for $P\in A$. Note that for $P\in A\backslash B_2$ we have:
$$h(v)\le 
\begin{cases}
\frac{119v^2}{200}\text{ if }v\in[0,1],\\
 \frac{ (10 - 3 v) (7 v + 10)}{200}\text{ if }v\in[1,2],\\
\frac{6v (5 -2 v)}{25}\text{ if }v\in\big[2,\frac{5}{2}\big].
\end{cases}$$
So  
$S(W_{\bullet,\bullet}^{A};P)\le\frac{53}{60}<\frac{5}{4}$. Thus, $\delta_P(S)=\frac{4}{5}$ if $ P\in A\backslash B_2$.
\end{proof}
\begin{lemma}\label{deg2-34-themiddleA5point}
Suppose $P$ belongs to a $(-2)$-curve $A$ and there exist $(-1)$-curves and $(-2)$-curves   which form the following dual graph:
\begin{figure}[h!]
    \centering
\includegraphics[width=12cm]{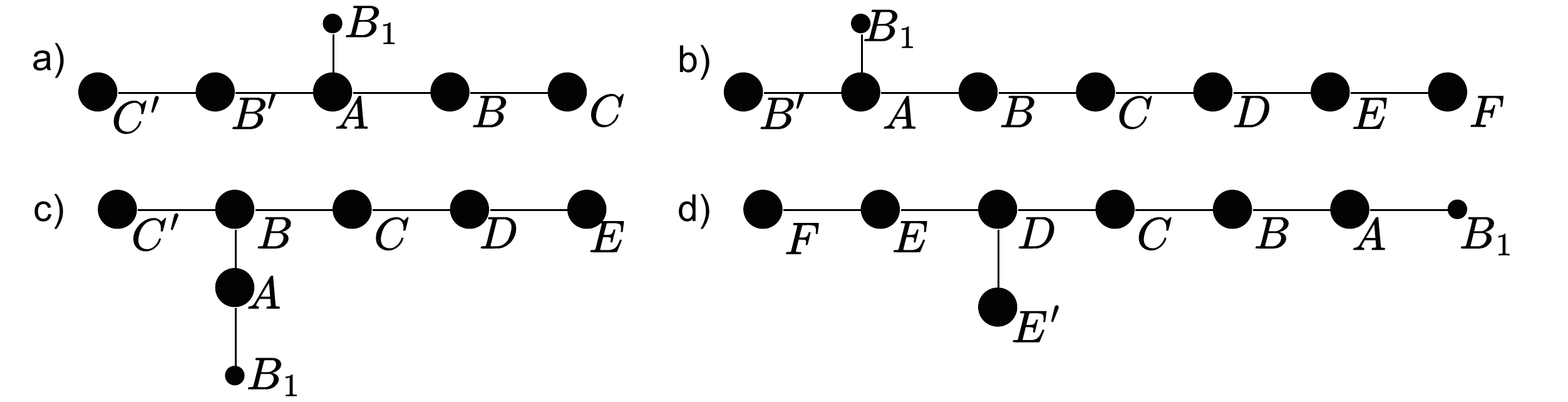}
    \caption{Dual graph: $(-K_S)^2=2$ and $\delta_P(S)=\frac{3}{4}$ with $\tau(A)=3$}
\end{figure}
\par Then $\tau(A)=3$ and the Zariski Decomposition of the divisor $-K_S-vA$ is given by:
{ {\allowdisplaybreaks\begin{align*}
&{\text{\bf a). }}&P(v)=\begin{cases}
-K_S-vA-\frac{v}{3}(C'+2B'+2B+C)\text{ if }v\in[0,1],\\
-K_S-vA-\frac{v}{3}(C'+2B'+2B+C)-(v-1)B_1\text{ if }v\in[1,3].
\end{cases}\\
&&N(v)=\begin{cases}
\frac{v}{3}(C'+2B'+2B+C)\text{ if }v\in[0,1],\\
\frac{v}{3}(C'+2B'+2B+C)+(v-1)B_1\text{ if }v\in[1,3].
\end{cases}\\
&{\text{\bf b). }}&P(v)=\begin{cases}
-K_S-vA-\frac{v}{6}(3B'+5B+4C+3D+2E+F)\text{ if }v\in[0,1],\\
-K_S-vA-\frac{v}{6}(3B'+5B+4C+3D+2E+F)-(v-1)B_1\text{ if }v\in[1,3].
\end{cases}\\
&&N(v)=\begin{cases}
\frac{v}{6}(3B'+5B+4C+3D+2E+F)\text{ if }v\in[0,1],\\
\frac{v}{6}(3B'+5B+4C+3D+2E+F)+(v-1)B_1\text{ if }v\in[1,3].
\end{cases}\\
&{\text{\bf c). }}&P(v)=\begin{cases}
-K_S-vA-\frac{v}{3}(2C'+4B+3C+2D+E)\text{ if }v\in[0,1],\\
-K_S-vA-\frac{v}{3}(2C'+4B+3C+2D+E)-(v-1)B_1\text{ if }v\in[1,3].
\end{cases}\\
&&N(v)=\begin{cases}
\frac{v}{3}(2C'+4B+3C+2D+E)\text{ if }v\in[0,1],\\
\frac{v}{3}(2C'+4B+3C+2D+E)+(v-1)B_1\text{ if }v\in[1,3].
\end{cases}\\
&{\text{\bf d). }}&P(v)=\begin{cases}
-K_S-vA-\frac{v}{3}(2F+4E+6D+5C+4B+3E')\text{ if }v\in[0,1],\\
-K_S-vA-\frac{v}{3}(2F+4E+6D+5C+4B+3E')-(v-1)B_1\text{ if }v\in[1,3].
\end{cases}\\
&&N(v)=\begin{cases}
\frac{v}{3}(2F+4E+6D+5C+4B+3E')\text{ if }v\in[0,1],\\
\frac{v}{3}(2F+4E+6D+5C+4B+3E')+(v-1)B_1\text{ if }v\in[1,3].
\end{cases}
\end{align*}}}
Moreover, 
$$(P(v))^2=\begin{cases}
2 - \frac{2v^2}{3}\text{ if }v\in[0,1],\\
\frac{(3-v)^2}{3}\text{ if }v\in[1,3].
\end{cases}P(v)\cdot A=\begin{cases}\frac{2v}{3}\text{ if }v\in[0,1],\\
1-\frac{v}{3}\text{ if }v\in[1,3].
\end{cases}$$
In this case $\delta_P(S)=\frac{3}{4}\text{ if }P\in A$.
\end{lemma}

\begin{proof}
In part a). the Zariski Decomposition  follows from 
 $-K_S-vA\sim_{\DR} (3-v)A+C + 2B + 2B' + C'+2B_1$.
A similar statement holds in other parts.
We have
$S_S(A)=\frac{4}{3}$. Thus, $\delta_P(S)\le \frac{3}{4}$ for $P\in A$. Note that if $P\in A \backslash (B\cup B')$ or if $P\in A \cap (B\cup B')$  we have:
$$h(v)\le \begin{cases} 
\frac{2v^2}{9}\text{ if }v\in[0,1],\\
 \frac{(3 - v) (5 v - 3)}{18}\text{ if }v\in[1,3].
\end{cases}
\text{ or }
h(v)\le \begin{cases} 
 \frac{2v^2}{3} \text{ if }v\in[0,1],\\
\frac{(3 - v) (v + 1)}{6}\text{ if }v\in[1,3].
\end{cases}$$
So we have 
$S(W_{\bullet,\bullet}^{A};P)\le\frac{2}{3}<\frac{4}{3}$ or
$S(W_{\bullet,\bullet}^{A};P)\le\frac{10}{9}<\frac{4}{3}$.
Thus, $\delta_P(S)=\frac{3}{4}$ if $P\in A$.
\end{proof}
\begin{lemma}\label{deg2-34-A7points}
Suppose $P$ belongs to a $(-2)$-curve $A$ and there exist $(-1)$-curves and $(-2)$-curves   which form the following dual graph:
\begin{figure}[h!]
    \centering
\includegraphics[width=13cm]{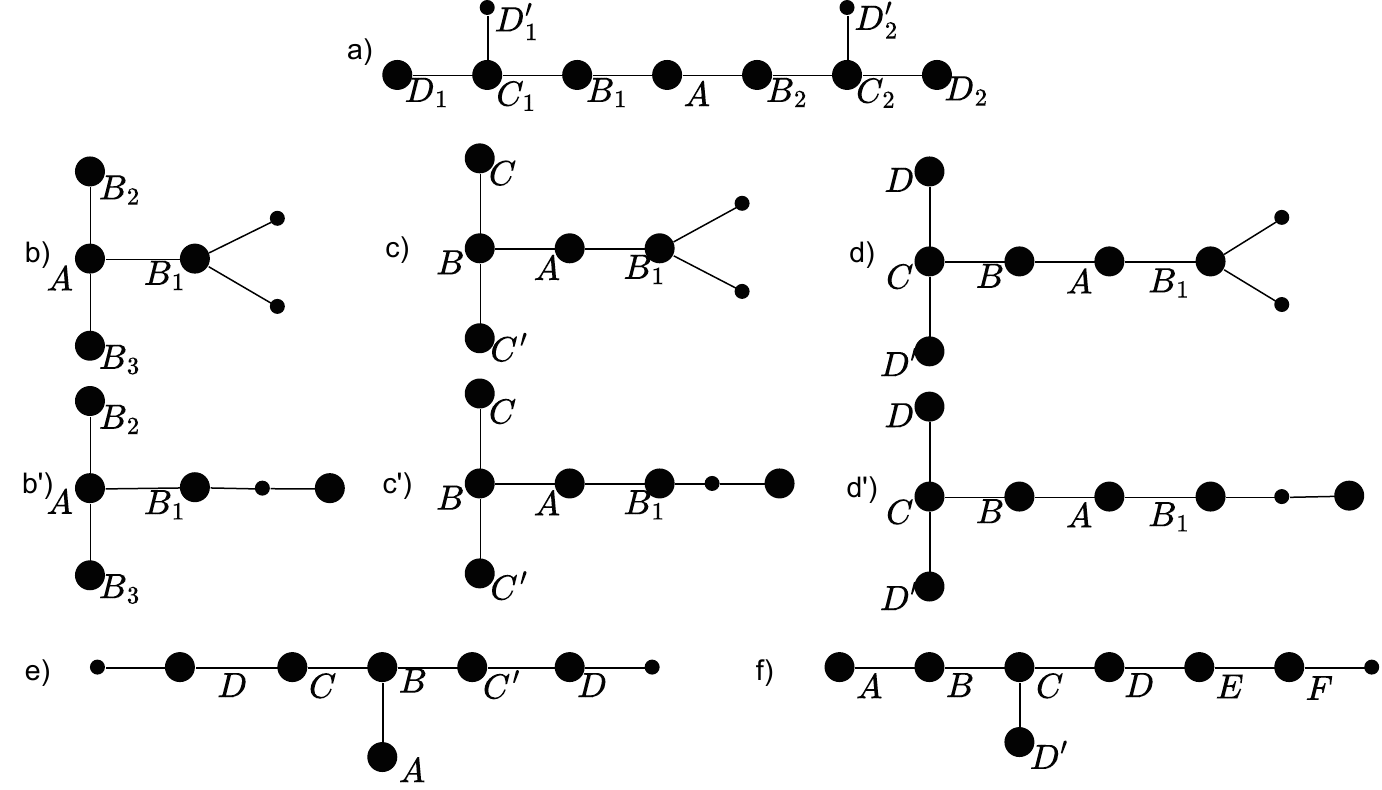}
    \caption{Dual graph: $(-K_S)^2=2$ and $\delta_P(S)=\frac{3}{4}$ with $\tau(A)=2$}
\end{figure}
\par 
Then $\tau(A)=2$ and the Zariski Decomposition of the divisor $-K_S-vA$ is given by:
{ {\allowdisplaybreaks\begin{align*}
&{\text{\bf a). }}&P(v)=-K_S-vA-\frac{v}{4}(D_1+2C_1+3B_1+3B_2+2C_2+D_2)\text{ if }v\in[0,2].\\&&N(v)=\frac{v}{4}(D_1+2C_1+3B_1+3B_2+2C_2+D_2)\text{ if }v\in[0,2].\\
&{\text{\bf b). }}&P(v)=-K_S-vA-\frac{v}{2}(B_1+B_2+B_3)\text{ if }v\in[0,2].\\&&N(v)=\frac{v}{2}(B_1+B_2+B_3)\text{ if }v\in[0,2].\\
&{\text{\bf c). }}&P(v)=-K_S-vA-\frac{v}{2}(B_1+2B+C+C')\text{ if }v\in[0,2].\\&&N(v)=\frac{v}{2}(B_1+2B+C+C')\text{ if }v\in[0,2].
\\
&{\text{\bf d). }}&P(v)=-K_S-vA-\frac{v}{2}(B_1+2B+2C+D+D')\text{ if }v\in[0,2].\\&&N(v)=\frac{v}{2}(B_1+2B+2C+D+D')\text{ if }v\in[0,2].\\
&{\text{\bf e). }}&P(v)=-K_S-vA-\frac{v}{2}(D+2C+3B+2C'+D')\text{ if }v\in[0,2].\\&&N(v)=\frac{v}{2}(D+2C+3B+2C'+D')\text{ if }v\in[0,2].\\
&{\text{\bf f). }}&P(v)=-K_S-vA-\frac{v}{2}(3B+4C+3D+2E+F+2D')\text{ if }v\in[0,2].\\&&N(v)=\frac{v}{2}(3B+4C+3D+2E+F+2D')\text{ if }v\in[0,2].
\end{align*}}}
Moreover, 
$$(P(v))^2=\frac{(2-v)(2+v)}{2}P(v)\cdot A= \frac{v}{2} \text{ if }v\in[0,2]. $$
In this case $\delta_P(S)=\frac{3}{4}\text{ if }P\in A\backslash B$.
\end{lemma}

\begin{proof}
In part a). the Zariski Decomposition  follows from 
 $-K_S-vA\sim_{\DR} (2-v)A+D_1'+D_1+2C_1+2B_1+2B_2+2C_2+D_2+D_2'$.
A similar statement holds in other parts.
We have $S_S(A)=\frac{4}{3}$. Thus, $\delta_P(S)\le \frac{3}{4}$ for $P\in A$. Note that for $P\in A\backslash B$ we have
$h(v) = \frac{v^2}{2}\text{ if }v\in[0,2].$
So  
$S(W_{\bullet,\bullet}^{A};P)=\frac{4}{3}$.
Thus, $\delta_P(S)=\frac{3}{4}$ if $P\in A\backslash B$.
\end{proof}
\begin{lemma}\label{deg2-34-onlyinA7points}
Suppose $P$ belongs to a $(-2)$-curve $A$ and there exist $(-1)$-curves and $(-2)$-curves   which form the following dual graph:
\begin{figure}[h!]
    \centering
\includegraphics[width=6cm]{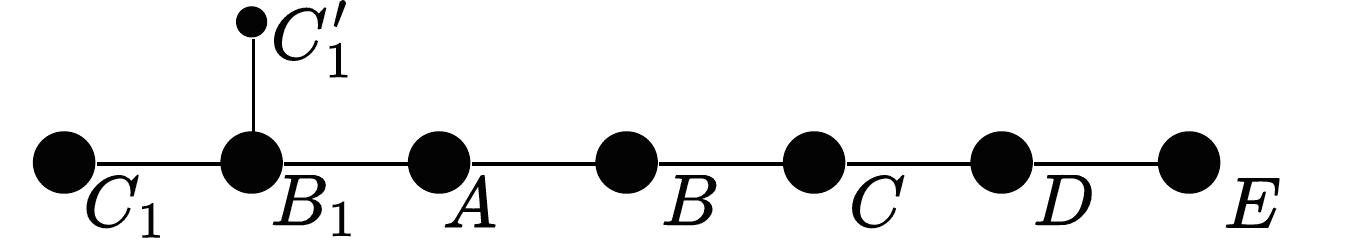}
    \caption{Dual graph: $(-K_S)^2=2$ and $\delta_P(S)=\frac{3}{4}$ with $\tau(A)=\frac{5}{2}$}
\end{figure}
\par
Then $\tau(A)=\frac{5}{2}$ and the Zariski Decomposition of the divisor $-K_S-vA$ is given by:
{{\allowdisplaybreaks\begin{align*}
&&P(v)=\begin{cases}
-K_S-vA-\frac{v}{5}(4B+3C+2D+E)-\frac{v}{3}(2B_1+C_1)\text{ if }v\in\big[0,\frac{3}{2}\big],\\
-K_S-vA-\frac{v}{5}(4B+3C+2D+E)-(v-1)(2B_1+C_1)-(2v-3)C_1'\text{ if }v\in\big[\frac{3}{2},\frac{5}{2}\big].
\end{cases}\\
&&N(v)=\begin{cases}
\frac{v}{5}(4B+3C+2D+E)+\frac{v}{3}(C_1+2B_1)\text{ if }v\in\big[0,\frac{3}{2}\big],\\
\frac{v}{5}(4B+3C+2D+E)+(v-1)(2B_1+C_1)+(2v-3)C_1'\text{ if }v\in\big[\frac{3}{2},\frac{5}{2}\big].
\end{cases}
\end{align*}}}
Moreover, 
$$(P(v))^2=\begin{cases}2 - \frac{8v^2}{15}\text{ if }v\in\big[0,\frac{3}{2}\big],\\
\frac{(5-2v)^2}{5}\text{ if }v\in\big[\frac{3}{2},\frac{5}{2}\big].
\end{cases}P(v)\cdot A=\begin{cases}\frac{8v}{15}\text{ if }v\in\big[0,\frac{3}{2}\big],\\
2(1-\frac{2v}{5})\text{ if }v\in\big[\frac{3}{2},\frac{5}{2}\big].
\end{cases}$$
In this case $\delta_P(S)=\frac{3}{4}\text{ if }P\in A\backslash B$.
\end{lemma}
\begin{proof}
The Zariski Decomposition  follows from 
 $$-K_S-vA\sim_{\DR} \Big(\frac{5}{2}-v\Big)A+\frac{1}{2}\Big(4B+3C+2D+E+6B_1+3C_1+4C_1'\Big).$$
We have
$S_S(A)=\frac{4}{3}$.
Thus, $\delta_P(S)\le \frac{3}{4}$ for $P\in A$. Note that for $P\in A\backslash B$ we have:
$$h(v)\le \begin{cases} 
 \frac{112 v^2}{225}\text{ if }v\in\big[0,\frac{3}{2}\big],\\
\frac{ 2 (5 - 2 v) (8 v - 5)}{25}\text{ if }v\in\big[\frac{3}{2},\frac{5}{2}\big].
\end{cases}$$
So we have 
$S(W_{\bullet,\bullet}^{A};P)\le\frac{4}{3}$.
Thus, $\delta_P(S)=\frac{3}{4}$ if $A\backslash B$.
\end{proof}
\begin{lemma}\label{deg2-35-D5points}
Suppose $P$ belongs to a $(-2)$-curve $A$ and there exist $(-1)$-curves and $(-2)$-curves   which form the following dual graph:
\begin{figure}[h!]
    \centering
\includegraphics[width=10cm]{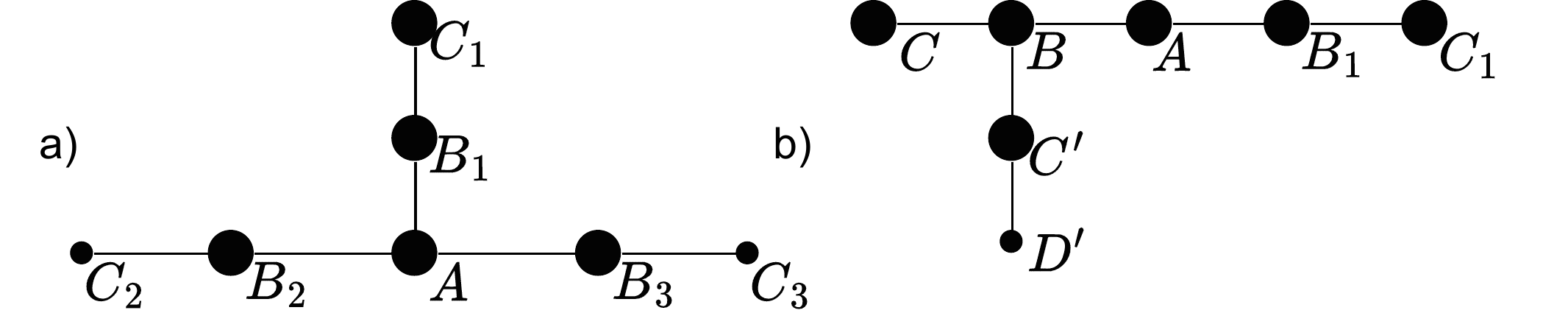}
    \caption{Dual graph: $(-K_S)^2=2$ and $\delta_P(S)=\frac{3}{5}$}
\end{figure}
\par 
Then $\tau(A)=3$ and the Zariski Decomposition of the divisor $-K_S-vA$ is given by:
{ {\allowdisplaybreaks\begin{align*}
\hspace*{-0.5cm}&{\text{\bf a). }}
&P(v)=\begin{cases}
-K_S-vA-\frac{v}{6}(2C_1+4B_1+3B_2+3B_2)\text{ if }v\in[0,2],\\
-K_S-vA-\frac{v}{3}(C_1+2B_1)-(v-1)(B_2+B_3)-(v-2)(C_2+C_3)\text{ if }v\in[2,3].
\end{cases}\\
\hspace*{-0.5cm}&&N(v)=\begin{cases}
\frac{v}{6}(2C_1+4B_1+3B_2+3B_2)\text{ if }v\in[0,2],\\
\frac{v}{3}(C_1+2B_1)+(v-1)(B_2+B_3)+(v-2)(C_2+C_3)\text{ if }v\in[2,3].
\end{cases}
\\
\hspace*{-0.5cm}&{\text{\bf b). }}&P(v)=
\begin{cases}
-K_S-vA-\frac{v}{6}(3C + 3C' + 6B + 4B_1 + 2C_1)\text{ if }v\in[0,2],\\
-K_S-vA-\frac{v}{3}(2B_1+C_1)-(v-1)(C+2B)-(2v-3)C' - (2v-4)D'\text{ if }v\in[2,3].
\end{cases}\\
\hspace*{-0.5cm}&&N(v)=\begin{cases}
\frac{v}{6}(3C + 3C' + 6B + 4B_1 + 2C_1)\text{ if }v\in[0,2],\\
\frac{v}{3}(2B_1+C_1)+(v-1)(C+2B)+ (2v-3)C' +(2v-4)D'\text{ if }v\in[2,3].
\end{cases}
\end{align*}}}
Moreover, 
$$(P(v))^2=\begin{cases}
2 - \frac{v^2}{3}\text{ if }v\in[0,2],\\
\frac{2(3-v)^2}{3}\text{ if }v\in[2,3].
\end{cases}P(v)\cdot A=\begin{cases}
\frac{v}{3}\text{ if }v\in[0,2],\\
2(1-\frac{v}{3})\text{ if }v\in[2,3].
\end{cases}$$
In this case $\delta_P(S)=\frac{3}{5}\text{ if }P\in A\backslash B$.
\end{lemma}
\begin{proof}
In part a). the Zariski Decomposition  follows from 
 $-K_S-vA\sim_{\DR} (3-v)A+2B_1+C_1+2B_2+C_2+2B_3+C_3$.
A similar statement holds in other parts. We have
$S_S(A)=\frac{5}{3}$. Thus, $\delta_P(S)\le \frac{3}{5}$ for $P\in A$. Note that for $P\in A \cap B_1$ or if $P\in A \backslash B$  we have:
$$h(v)\le \begin{cases} 
\frac{5v^2}{18}\text{ if }v\in[0,2],\\
\frac{2(3 - v)(v + 3)}{9}\text{ if }v\in[2,3].
\end{cases}
\text{ or }
h(v)\le \begin{cases} 
\frac{2v^2}{9}\text{ if }v\in[0,2],\\
\frac{4v(3 - v)}{9}\text{ if }v\in[2,3].
\end{cases}$$
So  
$S(W_{\bullet,\bullet}^{A};P)\le\frac{4}{3}<\frac{5}{3}$ or
$S(W_{\bullet,\bullet}^{A};P)\le\frac{10}{9}<\frac{5}{3}$. Thus, $\delta_P(S)=\frac{3}{5}$ if $P\in A \backslash B$.
\end{proof}
\begin{lemma}\label{deg2-47-points}
Suppose $P$ belongs to a $(-1)$-curve $A$ and there exist $(-1)$-curves and $(-2)$-curves   which form the following dual graph:
\begin{figure}[h!]
    \centering
\includegraphics[width=7cm]{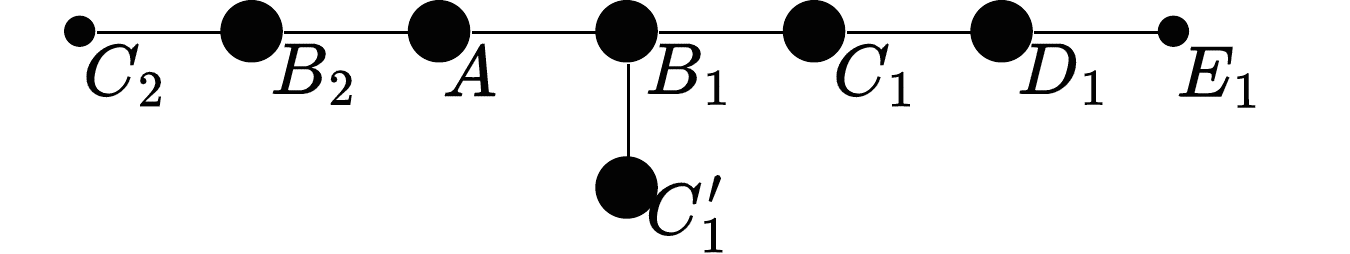}
    \caption{Dual graph: $(-K_S)^2=2$ and $\delta_P(S)=\frac{4}{7}$}
\end{figure}
\par
Then $\tau(A)=3$ and the Zariski Decomposition of the divisor $-K_S-vA$ is given by:
{\small {\allowdisplaybreaks\begin{align*}
\hspace*{-1cm}&&P(v)=\begin{cases}
-K_S-vA-\frac{v}{5}(6B_1+4C_1+2D_1+3C_1')-\frac{v}{2}B_2\text{ if }v\in[0,2],\\
-K_S-vA-\frac{v}{5}(6B_1+4C_1+2D_1+3C_1')-(v-1)B_2-(v-2)C_2\text{ if }v\in\big[2,\frac{5}{2}\big],\\
-K_S-vA-(v-1)(B_2+2B_1+C_1')-(v-2)C_2-(2v-3)C_1-(2v-4)D_1-(2v-5)E_1\text{ if }v\in\big[\frac{5}{2},3\big].
\end{cases}\\
\hspace*{-1cm}&&N(v)=\begin{cases}
\frac{v}{5}(6B_1+4C_1+2D_1+3C_1')+\frac{v}{2}B_2\text{ if }v\in[0,1],\\
\frac{v}{5}(6B_1+4C_1+2D_1+3C_1')+(v-1)B_2+(v-2)C_2\text{ if }v\in\big[2,\frac{5}{2}\big],\\
(v-1)(B_2+2B_1+C_1')+(v-2)C_2+(2v-3)C_1+(2v-4)D_1+(2v-5)E_1\text{ if }v\in\big[\frac{5}{2},3\big].
\end{cases}
\end{align*}}}
Moreover, 
$$(P(v))^2=\begin{cases}2-\frac{3v^2}{10}\text{ if }v\in[0,2],\\
4 - 2v + \frac{v^2}{5} \text{ if }v\in\big[2,\frac{5}{2}\big],\\
(3-v)^2\text{ if }v\in\big[\frac{5}{2},3\big].
\end{cases}P(v)\cdot A=\begin{cases}
\frac{3v}{10}\text{ if }v\in[0,2],\\
1 - \frac{v}{5}\text{ if }v\in\big[2,\frac{5}{2}\big],\\
3-v\text{ if }v\in\big[\frac{5}{2},3\big].
\end{cases}$$
In this case $\delta_P(S)=\frac{4}{7}\text{ if }P\in A\backslash B_1$.
\end{lemma}
\begin{proof}
The Zariski Decomposition  follows from 
 $-K_S-vA\sim_{\DR} (3-v)A+4B_1+3C_1+2D_1+E_1+2C_1'+2B_2+C_2$.
We have
$S_S(A)=\frac{7}{4}$. Thus, $\delta_P(S)\le \frac{4}{7}$ for $P\in A$. Note that for $P\in A\backslash B_1$ we have:
$$h(v)\le 
\begin{cases}
\frac{39v^2}{200}\text{ if }v\in[0,2],\\
\frac{(5 - v)( 9v - 5)}{50}\text{ if }v\in\big[2,\frac{5}{2}\big],\\
\frac{(3 - v)(1 + v)}{2}\text{ if }v\in\big[\frac{5}{2},3\big].
\end{cases}$$
So 
$S(W_{\bullet,\bullet}^{A};P)\le\frac{7}{6}<\frac{7}{4}$. Thus, $\delta_P(S)=\frac{4}{7}$ if $ P\in A\backslash B_1$.
\end{proof}
\begin{lemma}\label{deg2-916-points}
Suppose $P$ belongs to a $(-2)$-curve $A$ and there exist $(-1)$-curves and $(-2)$-curves   which form the following dual graph:
\begin{figure}[h!]
    \centering
\includegraphics[width=7cm]{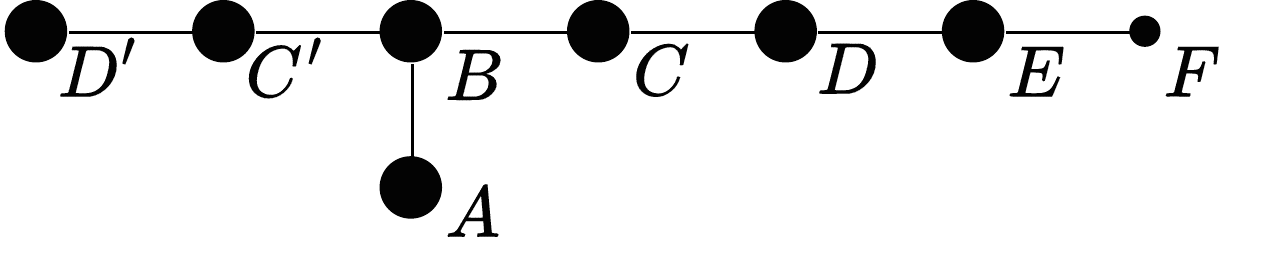}
    \caption{Dual graph: $(-K_S)^2=2$ and $\delta_P(S)=\frac{9}{16}$}
\end{figure}
\par
Then $\tau(A)=3$ and the Zariski Decomposition of the divisor $-K_S-vA$ is given by:
{\small{\allowdisplaybreaks\begin{align*}
\hspace*{-1cm}&&P(v)=\begin{cases}
-K_S-vA-\frac{v}{7}(4D'+8C'+12B+9C+6D+3E)\text{ if }v\in\big[0,\frac{7}{3}\big],\\
-K_S-vA-(v-1)(D'+2C'+3B)-(3v-4)C-(3v - 5)D-(3v - 6)E-(3v-7)F\text{ if }v\in\big[\frac{7}{3},3\big].
\end{cases}\\
\hspace*{-1cm}&&N(v)=\begin{cases}
\frac{v}{7}(4D'+8C'+12B+9C+6D+3E)\text{ if }v\in\big[0,\frac{7}{3}\big],\\
(v-1)(D'+2C'+3B)+(3v-4)C+(3v - 5)D+(3v - 6)E+(3v-7)F\text{ if }v\in\big[\frac{7}{3},3\big].
\end{cases}
\end{align*}}}
Moreover, 
$$(P(v))^2=\begin{cases}
2 - \frac{2v^2}{7}\text{ if }v\in\big[0,\frac{7}{3}\big],\\
(3-v)^2\text{ if }v\in\big[\frac{7}{3},3\big].
\end{cases}P(v)\cdot A=\begin{cases}
\frac{2v}{7}\text{ if }v\in\big[0,\frac{7}{3}\big],\\
3-v\text{ if }v\in\big[\frac{7}{3},3\big].
\end{cases}$$
In this case $\delta_P(S)=\frac{9}{16}\text{ if }P\in A\backslash B$.
\end{lemma}

\begin{proof}
The Zariski Decomposition  follows from 
 $-K_S-vA\sim_{\DR} (3-v)A+2D'+4C'+6B+5C+4D+3E+2F$.
We have
$S_S(A)=\frac{16}{9}$. Thus, $\delta_P(S)\le \frac{9}{16}$ for $P\in A$. Note that for $P\in A\backslash B$ we have:
$$h(v)\le \begin{cases} 
\frac{2v^2}{49}\text{ if }v\in\big[0,\frac{7}{3}\big],\\
\frac{(3 - v)^2}{2}\text{ if }v\in\big[\frac{7}{3},3\big].
\end{cases}$$
So 
$S(W_{\bullet,\bullet}^{A};P)\le\frac{2}{9}<\frac{16}{9}$. Thus, $\delta_P(S)=\frac{9}{16}$ if $P\in A\backslash B$.
\end{proof}
\begin{lemma}\label{deg2-12-D6points}
Suppose $P$ belongs to a $(-2)$-curve $A$ and there exist $(-1)$-curves and $(-2)$-curves   which form the following dual graph:
\begin{figure}[h!]
    \centering
\includegraphics[width=12.5cm]{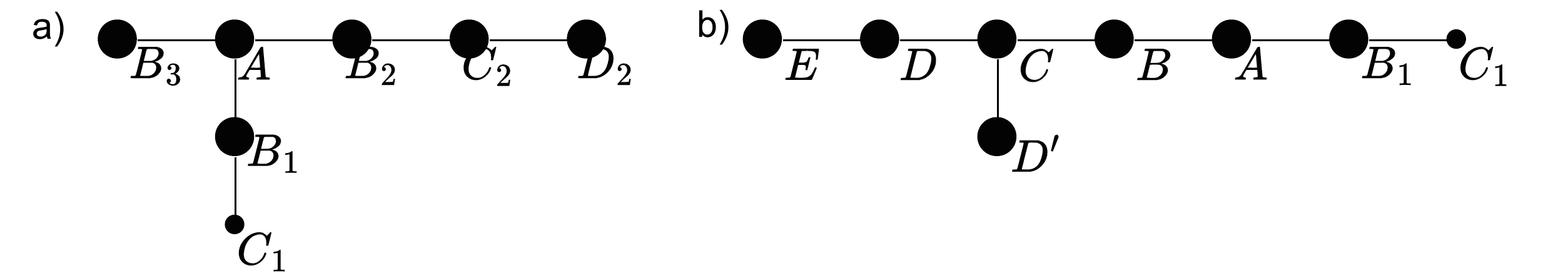}
    \caption{Dual graph: $(-K_S)^2=2$ and $\delta_P(S)=\frac{1}{2}$}
\end{figure}
\par 
Then $\tau(A)=4$ and the Zariski Decomposition of the divisor $-K_S-vA$ is given by:
{ {\allowdisplaybreaks\begin{align*}
&{\text{\bf a). }}&P(v)=\begin{cases}
-K_S-vA-\frac{v}{4}(3B_2+2C_2+D_2+2B_3+2B_1)\text{ if }v\in[0,2],\\
-K_S-vA-\frac{v}{4}(3B_2+2C_2+D_2+2B_3)-(v-1)B_1-(v-2)C_1\text{ if }v\in[2,4].
\end{cases}\\
&&N(v)=\begin{cases}
\frac{v}{4}(3B_2+2C_2+D_2+2B_3+2B_1)\text{ if }v\in[0,2],\\
\frac{v}{4}(3B_2+2C_2+D_2+2B_3)+(v-1)B_1+(v-2)C_1\text{ if }v\in[2,4].
\end{cases}\\
&{\text{\bf b). }}&P(v)=\begin{cases}
-K_S-vA-\frac{v}{4}(2E+4D+6C+5B+2B_1+3D')\text{ if }v\in[0,2],\\
-K_S-vA-\frac{v}{4}(2E+4D+6C+5B+3D')-(v-1)B_1-(v-2)C_1\text{ if }v\in[2,4].
\end{cases}\\
&&N(v)=\begin{cases}
\frac{v}{4}(2E+4D+6C+5B+2B_1+3D')\text{ if }v\in[0,2],\\
\frac{v}{4}(2E+4D+6C+5B+3D')+(v-1)B_1+(v-2)C_1\text{ if }v\in[2,4].
\end{cases}
\end{align*}}}
Moreover, 
$$(P(v))^2=\begin{cases}2 - \frac{v^2}{4}\text{ if }v\in[0,2],\\
\frac{(4-v)^2}{4}\text{ if }v\in[2,4].
\end{cases}P(v)\cdot A=\begin{cases}
\frac{v}{4}\text{ if }v\in[0,2],\\
1-\frac{v}{4}\text{ if }v\in[2,4].
\end{cases}$$
In this case $\delta_P(S)=\frac{1}{2}\text{ if }P\in A\backslash B$.
\end{lemma}
    \begin{proof}
In part a). the Zariski Decomposition  follows from 
 $-K_S-vA\sim_{\DR} (4-v)A+3B_2+2C_2+D_2+2B_3+3B_1+2C_1$.
A similar statement holds in other parts. We have
$S_S(A)=2$.
Thus, $\delta_P(S)\le \frac{1}{2}$ for $P\in E_3$. Note that if $P\in A \cap (B\cup B')$ or if $P\in A \backslash (B\cup B')$ we have:
$$h(v)\le \begin{cases} 
\frac{7v^2}{32}\text{ if }v\in[0,2],\\
\frac{(4 - v)(5v + 4)}{32}\text{ if }v\in[2,4].
\end{cases}
\text{ or }
h(v)\le \begin{cases} 
\frac{5v^2}{32}\text{ if }v\in[0,2],\\
\frac{(4 - v)(7v - 4)}{32}\text{ if }v\in[2,4].
\end{cases}
$$
So we have 
$S(W_{\bullet,\bullet}^{A};P)\le\frac{5}{3}<2$
or
$S(W_{\bullet,\bullet}^{A};P)\le\frac{4}{3}<2$.
Thus, $\delta_P(S)=\frac{1}{2}$ if $P\in A$.
\end{proof}
\begin{lemma}\label{deg2-37-points}
Suppose $P$ belongs to a $(-2)$-curve $A$ and there exist $(-1)$-curves and $(-2)$-curves   which form the following dual graph:
\begin{figure}[h!]
    \centering
\includegraphics[width=15cm]{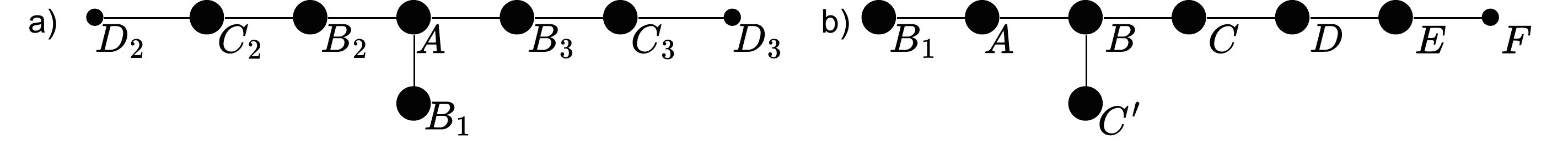}
    \caption{Dual graph: $(-K_S)^2=2$ and $\delta_P(S)=\frac{3}{7}$}
\end{figure}
\par 
Then $\tau(A)=4$ and the Zariski Decomposition of the divisor $-K_S-vA$ is given by:
{\small {\allowdisplaybreaks\begin{align*}
\hspace*{-0.5cm}&{\text{\bf a). }}&P(v)=\begin{cases}-K_S-vA-\frac{v}{6}(3B_1+2C_2+3B_2+3B_3+2C_3)\text{ if }v\in[0,3],\\
-K_S-vA-\frac{v}{2}B_1-(v-1)(B_2+B_3)-(v-2)(C_2+C_3)-(v-3)(D_2+D_3)\text{ if }v\in[3,4].
\end{cases}\\
\hspace*{-0.5cm}&&N(v)=\begin{cases}
\frac{v}{6}(3B_1+2C_2+3B_2+3B_3+2C_3)\text{ if }v\in[0,3],\\
\frac{v}{2}B_1+(v-1)(B_2+B_3)+(v-2)(C_2+C_3)+(v-3)(D_2+D_3)\text{ if }v\in[3,4].
\end{cases}\\
\hspace*{-0.5cm}&{\text{\bf b). }}&P(v)=\begin{cases}
-K_S-vA-\frac{v}{6}(3B_1+8B+6C+4D+2E+4C')\text{ if }v\in[0,3],\\
-K_S-vA-\frac{v}{2}B_1-(v-1)(2B+C')-(2v-3)C-(2v-4)D-(2v-5)E-(2v-6)F\text{ if }v\in[3,4].
\end{cases}\\
\hspace*{-0.5cm}&&N(v)=\begin{cases}
\frac{v}{6}(3B_1+8B+6C+4D+2E+4C')\text{ if }v\in[0,3],\\
\frac{v}{2}B_1+(v-1)(2B+C')+(2v-3)C+(2v-4)D+(2v-5)E+(2v-6)F\text{ if }v\in[3,4].
\end{cases}
\end{align*}}}
Moreover, 
$$(P(v))^2=\begin{cases}
2 - \frac{v^2}{6}\text{ if }v\in[0,3],\\
\frac{(4-v)^2}{2}\text{ if }v\in[3,4].
\end{cases}P(v)\cdot A=\begin{cases}
\frac{v}{6}\text{ if }v\in[0,3],\\
2-\frac{v}{2}\text{ if }v\in[3,4].
\end{cases}$$
In this case $\delta_P(S)=\frac{3}{7}\text{ if }P\in A\backslash B$.
\end{lemma}
\begin{proof}
In part a). the Zariski Decomposition  follows from 
 $-K_S-vA\sim_{\DR} (4-v)A+2B_1+D_2+2C_2+3B_2+3B_3+2C_3+D_3$.
A similar statement holds in other parts. We have
$S_S(A)=\frac{7}{3}$. Thus, $\delta_P(S)\le \frac{3}{7}$ for $P\in A$. Note that if  $P\in A\cap B_1$ or if  $P\in A\backslash (B_1\cup B)$ we have:
$$h(v)\le \begin{cases} 
\frac{7v^2}{72}\text{ if }v\in[0,3],\\
\frac{(4 - v)(v + 4)}{8}\text{ if }v\in[3,4].
\end{cases}
\text{ or }
h(v)\le \begin{cases} 
\frac{v^2}{8}\text{ if }v\in[0,3],\\
\frac{3(4 - v)}{8}\text{ if }v\in[3,4].
\end{cases}$$
So  $S(W_{\bullet,\bullet}^{A};P)\le\frac{4}{3}<\frac{7}{3}$ or 
$S(W_{\bullet,\bullet}^{A};P)\le\frac{4}{3}<\frac{7}{3}$. Thus, $\delta_P(S)=\frac{3}{7}$ if $P\in A \backslash B$.
\end{proof}
\begin{lemma}\label{deg2-38-points}
Suppose $P$ belongs to a $(-2)$-curve $A$ and there exist $(-1)$-curves and $(-2)$-curves   which form the following dual graph:
\begin{figure}[h!]
    \centering
\includegraphics[width=7cm]{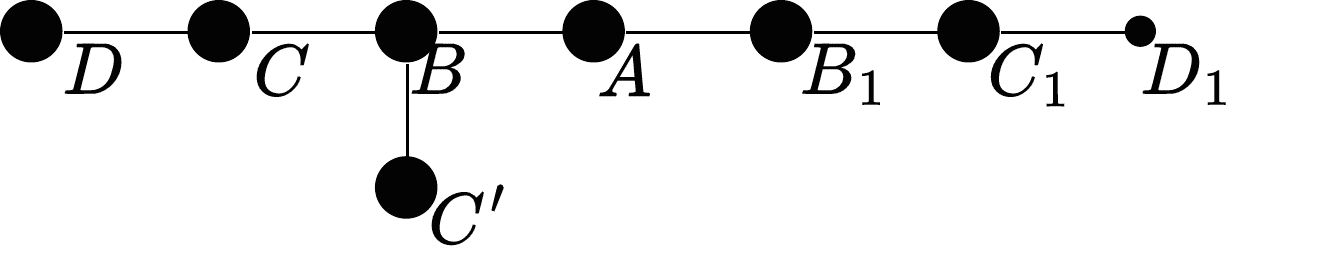}
    \caption{Dual graph: $(-K_S)^2=2$ and $\delta_P(S)=\frac{3}{8}$}
\end{figure}
\par
Then $\tau(A)=5$ and the Zariski Decomposition of the divisor $-K_S-vA$ is given by:
{{\allowdisplaybreaks\begin{align*}
\hspace*{-0.5cm}&&P(v)=\begin{cases}
-K_S-vA-\frac{v}{15}(6D+12C+18B+9C'+10B_1+5C_1)\text{ if }v\in[0,3],\\
-K_S-vA-\frac{v}{5}(2D+4C+6B+3C')-(v-1)B_1-(v-2)C_1-(v-3)D_1\text{ if }v\in[3,5].
\end{cases}\\
\hspace*{-0.5cm}&&N(v)=\begin{cases}
\frac{v}{15}(6D+12C+18B+9C'+10B_1+5C_1)\text{ if }v\in[0,3],\\
\frac{v}{5}(2D+4C+6B+3C')+(v-1)B_1+(v-2)C_1+(v-3)D_1\text{ if }v\in[3,5].
\end{cases}
\end{align*}}}
Moreover, 
$$(P(v))^2=\begin{cases}2 - \frac{2v^2}{15}\text{ if }v\in[0,3],\\
\frac{(5-v)^2}{5}\text{ if }v\in[3,5].
\end{cases}P(v)\cdot A=\begin{cases}
\frac{2v}{15}\text{ if }v\in[0,3],\\
1-\frac{v}{5}\text{ if }v\in[3,5].
\end{cases}$$
In this case $\delta_P(S)=\frac{3}{8}\text{ if }P\in A\backslash B$.
\end{lemma}

\begin{proof}
The Zariski Decomposition  follows from 
 $-K_S-vA\sim_{\DR} (5-v)A+2D+4C+6B+3C'+4B_1+3C_1+2D_1$.
We have
$S_S(A)=\frac{8}{3}$. Thus, $\delta_P(S)\le \frac{3}{8}$ for $P\in A$. Note that for  $P\in A\backslash B$ we have:
$$h(v)\le \begin{cases} 
\frac{49v^2}{450}\text{ if }v\in[0,3],\\
\frac{(5 - v)(9v - 5)}{50}\text{ if }v\in[3,5].
\end{cases}$$
So 
$S(W_{\bullet,\bullet}^{A};P)\le\frac{21}{10}<\frac{8}{3}$. Thus, $\delta_P(S)=\frac{3}{8}$ if $P\in  A\backslash B$.
\end{proof}
\begin{lemma}\label{deg2-310-points}
Suppose $P$ belongs to a $(-2)$-curve $A$ and there exist $(-1)$-curves and $(-2)$-curves   which form the following dual graph:
\begin{figure}[h!]
    \centering
\includegraphics[width=7cm]{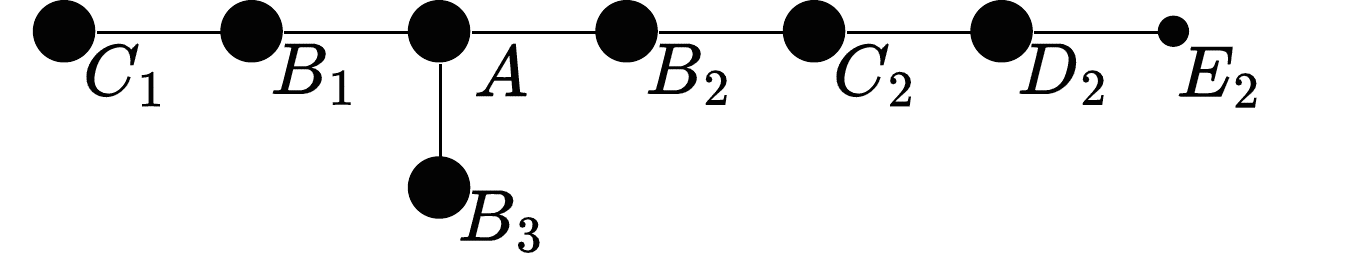}
    \caption{Dual graph: $(-K_S)^2=2$ and $\delta_P(S)=\frac{3}{10}$}
\end{figure}
\par
Then $\tau(A)=6$ and the Zariski Decomposition of the divisor $-K_S-vA$ is given by:
{{\allowdisplaybreaks\begin{align*}
\hspace*{-1cm}&&P(v)=\begin{cases}
-K_S-vA-\frac{v}{3}(C_1+2B_1)-\frac{v}{2}B_3-\frac{v}{4}(3B_2+2C_2+D_2)\text{ if }v\in[0,4],\\
-K_S-vA-\frac{v}{3}(C_1+2B_1)-\frac{v}{2}B_3-(v-1)B_2-(v-2)C_2-(v-3)D_2-(v-4)E_2\text{ if }v\in[4,6].
\end{cases}\\
\hspace*{-1cm}&&N(v)=\begin{cases}
\frac{v}{3}(C_1+2B_1)+\frac{v}{2}B_3+\frac{v}{4}(3B_2+2C_2+D_2)\text{ if }v\in[0,4],\\
\frac{v}{3}(C_1+2B_1)+\frac{v}{2}B_3+(v-1)B_2+(v-2)C_2+(v-3)D_2+(v-4)E_2\text{ if }v\in[4,6].
\end{cases}
\end{align*}}}
Moreover, 
$$(P(v))^2=\begin{cases}
2 - \frac{v^2}{12}\text{ if }v\in[0,4],\\
\frac{(6-v)^2}{6}\text{ if }v\in[4,6],
\end{cases}P(v)\cdot A=\begin{cases}
\frac{v}{12}\text{ if }v\in[0,4],\\
1-\frac{v}{6}\text{ if }v\in[4,6].
\end{cases}$$
In this case $\delta_P(S)=\frac{3}{10}\text{ if }P\in A$.
\end{lemma}
\begin{proof}
The Zariski Decomposition  follows from 
 $-K_S-vA\sim_{\DR} (6-v)A+2C_1+4B_1+5B_2+4C_2+3D_2+2E_2+3B_3$
We have
$S_S(A)=\frac{10}{3}$.
Thus, $\delta_P(S)\le \frac{3}{10}$ for $P\in A$. Note that if  $P\in A\cap (B_1\cup B_3)$ or if $P\in A\backslash(B_1\cup B_3)$ we have:
$$h(v)\le \begin{cases} 
\frac{17v^2}{288}\text{ if }v\in[0,4],\\
\frac{(6 - v)(7v + 6)}{72}\text{ if }v\in[4,6].
\end{cases}
\text{ or }
h(v)\le \begin{cases} 
\frac{19v^2}{288}\text{ if }v\in[0,4],\\
 \frac{(6 - v)(11v - 6)}{72}\text{ if }v\in[4,6].
\end{cases}$$
So we have 
$S(W_{\bullet,\bullet}^{A};P)\le\frac{7}{3}<\frac{10}{3}$
or
$S(W_{\bullet,\bullet}^{A};P)\le\frac{8}{3}<\frac{10}{3}$.
Thus, $\delta_P(S)=\frac{3}{10}$ if $P\in A$.
\end{proof}
\begin{lemma}\label{deg2-384209-nearA6points}
Suppose $P$ belongs to a $(-1)$-curve $A$ and there exist $(-1)$-curves and $(-2)$-curves   which form the following dual graph:
\begin{figure}[h!]
    \centering
\includegraphics[width=7cm]{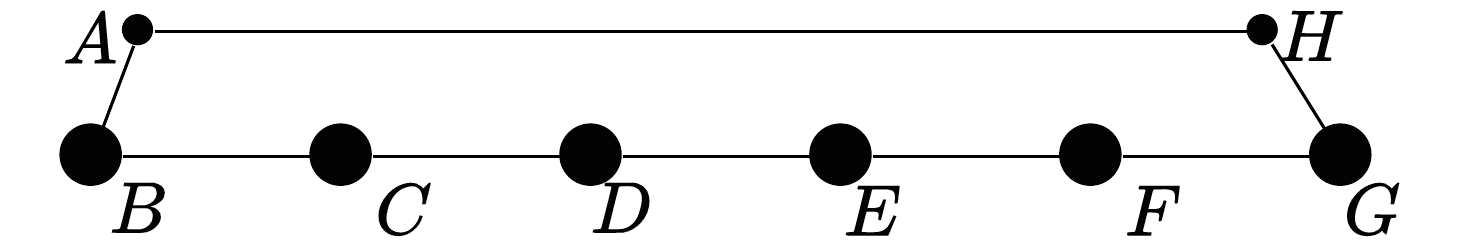}
    \caption{Dual graph: $(-K_S)^2=2$ and $\delta_P(S)\ge\frac{384}{209}$}
\end{figure}
\par
Then $\tau(A)=1$ and the Zariski Decomposition of the divisor $-K_S-vA$ is given by:
{\small {\allowdisplaybreaks\begin{align*}
\hspace*{-1.7cm}&&P(v)=\begin{cases}
-K_S-vA-\frac{v}{7}(6B+5C+4D+3E+2F+G)\text{ if }v\in\big[0,\frac{7}{8}\big],\\
-K_S-vA-(2v-1)B-(3v-2)C-(4v -3)D-
(5v-4)E-(6v-5)F-(7v-6)G-(8v-7)H\text{ if }v\in\big[\frac{7}{8},1\big].
\end{cases}\\
\text{}
\hspace*{-1.7cm}&&N(v)=\begin{cases}
\frac{v}{7}(6B+5C+4D+3E+2F+G)\text{ if }v\in\big[0,\frac{7}{8}\big],\\
(2v-1)B+(3v-2)C+(4v -3)D+
(5v-4)E+(6v-5)F+(7v-6)G+(8v-7)H\text{ if }v\in\big[\frac{7}{8},1\big].
\end{cases}
\end{align*}}}
Moreover, 
$$(P(v))^2=\begin{cases}
2 - 2v-\frac{v^2}{7}\text{ if }v\in\big[0,\frac{7}{8}\big],\\
9(1-v)^2\text{ if }v\in\big[\frac{7}{8},1\big].
\end{cases}P(v)\cdot A=\begin{cases}1+\frac{v}{7}\text{ if }v\in\big[0,\frac{7}{8}\big],\\
9(1-v)\text{ if }v\in\big[\frac{7}{8},1\big].
\end{cases}$$
In this case $\delta_P(S)\ge\frac{384}{209}\text{ if }P\in A\backslash B$.
\end{lemma}
\begin{proof}
The Zariski Decomposition follows from 
 $-K_S-vA\sim_{\DR} (1-v)A+B+C+D+E+F+G+H$.
 We have
$S_S(A)=\frac{23}{48}$. Thus, $\delta_P(S)\le \frac{48}{23}$ for $P\in A$. Note that for $P\in  A\backslash B$ we have:
$$h(v)\le \begin{cases} 
\frac{(v + 7)^2}{98}\text{ if }v\in\big[0,\frac{7}{8}\big],\\
 \frac{9 (7 v - 5) (1- v)}{9}\text{ if }v\in\big[\frac{7}{8},1\big].
\end{cases}$$
So 
$S(W_{\bullet,\bullet}^{A};P)\le\frac{209}{384}$.
Thus, $\delta_P(S)\ge\frac{384}{209}$ if $P\in A\backslash B$.
\end{proof}
\begin{lemma}\label{deg2-4927-nearA5}
Suppose $P$ belongs to a $(-1)$-curve $A$ and there exist $(-1)$-curves and $(-2)$-curves   which form the following dual graph:
\begin{figure}[h!]
    \centering
\includegraphics[width=6.5cm]{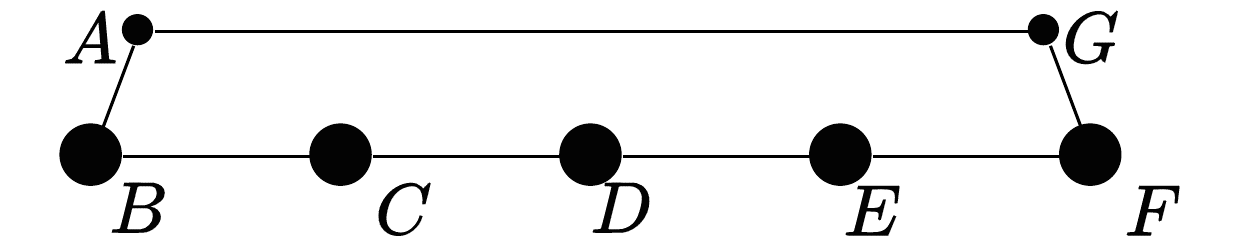}
    \caption{Dual graph: $(-K_S)^2=2$ and $\delta_P(S)\ge\frac{49}{27}$}
\end{figure}
\par Then $\tau(A)=1$ and the Zariski Decomposition of the divisor $-K_S-vA$ is given by:
{\small {\allowdisplaybreaks\begin{align*}
\hspace*{-0.7cm}&&P(v)=\begin{cases}
-K_S-vA-\frac{v}{6}(5B+4C+3D+2E+F)\text{ if }v\in\big[0,\frac{6}{7}\big],\\
-K_S-vA-(2v-1)B-(3v-2)C-(4v -3)D-
(5v-4)E-(6v-5)F-(7v-6)G\text{ if }v\in\big[\frac{6}{7},1\big].
\end{cases}\\
\text{}
\hspace*{-0.7cm}&&N(v)=\begin{cases}
\frac{v}{6}(5B+4C+3D+2E+F)\text{ if }v\in\big[0,\frac{6}{7}\big],\\
(2v-1)B+(3v-2)C+(4v -3)D+
(5v-4)E+(6v-5)F+(7v-6)G\text{ if }v\in\big[\frac{6}{7},1\big].
\end{cases}
\end{align*}}}
Moreover, 
$$(P(v))^2=\begin{cases}
2 - 2v-\frac{v^2}{6}\text{ if }v\in\big[0,\frac{6}{7}\big],\\
8(1-v)^2\text{ if }v\in\big[\frac{6}{7},1\big].
\end{cases}P(v)\cdot A=\begin{cases}1+\frac{v}{6}\text{ if }v\in\big[0,\frac{6}{7}\big],\\
8(1-v)\text{ if }v\in\big[\frac{6}{7},1\big].
\end{cases}$$
In this case $\delta_P(S)\ge\frac{49}{27}\text{ if }P\in A\backslash B.$
\end{lemma}
\begin{proof}
The Zariski Decomposition follows from 
 $-K_S-vA\sim_{\DR} (1-v)A+B+C+D+E+F+G$. We have
$S_S(A)=\frac{10}{21}.$
Thus, $\delta_P(S)\le \frac{21}{10}$ for $P\in A$. Note that for $P\in A\backslash B$ we have:
$$h(v)\le \begin{cases} 
\frac{(v + 6)^2}{72}\text{ if }v\in\big[0,\frac{6}{7}\big],\\
8(3 v - 2) (1-v)\text{ if }v\in\big[\frac{6}{7},1\big].
\end{cases}$$
So we have 
$S(W_{\bullet,\bullet}^{A};P)\le\frac{27}{49}$.
Thus, $\delta_P(S)\ge\frac{49}{27}$ if $P\in A\backslash B$.
\end{proof}
\begin{lemma}\label{deg2-216121-nearA4point}
Suppose $P$ belongs to a $(-1)$-curve $A$ and there exist $(-1)$-curves and $(-2)$-curves   which form the following dual graph:
\begin{figure}[h!]
    \centering
\includegraphics[width=5cm]{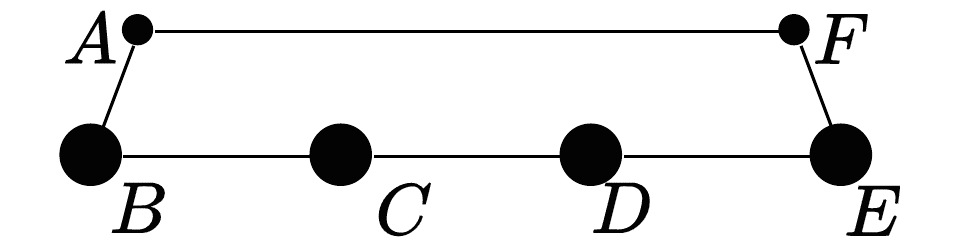}
    \caption{Dual graph: $(-K_S)^2=2$ and $\delta_P(S)\ge\frac{216}{121}$}
\end{figure}
\par Then $\tau(A)=1$ and the Zariski Decomposition of the divisor $-K_S-vA$ is given by:
{ {\allowdisplaybreaks\begin{align*}
\hspace*{-0.5cm}&&P(v)=\begin{cases}
-K_S-vA-\frac{v}{5}(4B+3C+2D+E)\text{ if }v\in\big[0,\frac{5}{6}\big],\\
-K_S-vA-(2v-1)B-(3v-2)C-(4v-3)D-(5v-4)E-(6v-5)F\text{ if }v\in\big[\frac{5}{6},1\big].
\end{cases}\\
\text{}
\hspace*{-0.5cm}&&N(v)=\begin{cases}
\frac{v}{5}(4B+3C+2D+E)\text{ if }v\in\big[0,\frac{5}{6}\big],\\
(2v-1)B+(3v-2)C+(4v-3)D+(5v-4)E+(6v-5)F\text{ if }v\in\big[\frac{5}{6},1\big].
\end{cases}
\end{align*}}}
Moreover, 
$$(P(v))^2=\begin{cases}2-2 v-\frac{v^2}{5}, \text{ if }v\in\big[0,\frac{5}{6}\big],\\
7(1-v)^2\text{ if }v\in\big[\frac{5}{6},1\big].
\end{cases}P(v)\cdot A=\begin{cases}1+\frac{v}{5}\text{ if }v\in\big[0,\frac{5}{6}\big],\\
7(1-v)\text{ if }v\in\big[\frac{5}{6},1\big].
\end{cases}$$
In this case $\delta_P(S)\ge\frac{216}{121}\text{ if }P\in A\backslash B.$
\end{lemma}
\begin{proof}
The Zariski Decomposition follows from 
 $-K_S-vA\sim_{\DR} (1-v)A+B+C+D+E+F$. We have
$S_S(A)=\frac{17}{36}.$
Thus, $\delta_P(S)\le \frac{36}{17}$ for $P\in A$. Note that for $P\in A\backslash B$ we have:
$$h(v)\le \begin{cases}
 \frac{(v + 5)^2}{50}\text{ if }v\in\big[0,\frac{5}{6}\big],\\
  \frac{7(5 v - 3) (1 -v )}{2}\text{ if }v\in\big[\frac{5}{6},1\big].
\end{cases}$$
So
$S(W_{\bullet,\bullet}^{A};P)\le\frac{121}{216}.$
Thus, $\delta_P(S)\ge\frac{216}{121}$ if $P\in  A\backslash B$.
\end{proof}
\begin{lemma}\label{deg2-7543-nearA3points}
Suppose $P$ belongs to a $(-2)$-curve $A$ and there exist $(-1)$-curves and $(-2)$-curves   which form the following dual graph:
\begin{figure}[h!]
    \centering
\includegraphics[width=3.5cm]{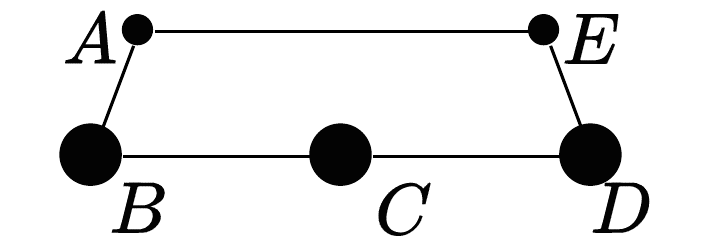}
    \caption{Dual graph: $(-K_S)^2=2$ and $\delta_P(S)\ge\frac{75}{43}$}
\end{figure}
\par Then $\tau(A)=1$ and the Zariski Decomposition of the divisor $-K_S-vA$ is given by:
{ 
{\allowdisplaybreaks\begin{align*}
&&P(v)=\begin{cases}
-K_S-vA-\frac{v}{4}(3B+2C+D)\text{ if }v\in\big[0,\frac{4}{5}\big],\\
-K_S-vA-(2v-1)B-(3v-2)C-(4v-3)D-(5v-4)E\text{ if }v\in\big[\frac{4}{5},1\big].
\end{cases}\\
&&N(v)=\begin{cases}
\frac{v}{4}(3B+2C+D)\text{ if }v\in\big[0,\frac{4}{5}\big],\\
(2v-1)B+(3v-2)C+(4v-3)D+(5v-4)E\text{ if }v\in\big[\frac{4}{5},1\big].
\end{cases}
\end{align*}}
}
Moreover, 
$$(P(v))^2=\begin{cases}
2-2 v-\frac{v^2}{4}, \text{ if }v\in\big[0,\frac{4}{5}\big],\\
6(1-v)^2\text{ if }v\in\big[\frac{4}{5},1\big].
\end{cases}P(v)\cdot A=\begin{cases}
1+\frac{v}{4}\text{ if }v\in\big[0,\frac{4}{5}\big],\\
6(1-v)\text{ if }v\in\big[\frac{4}{5},1\big].
\end{cases}$$
In this case $\delta_P(S)\ge\frac{75}{43}\text{ if }P\in A\backslash B.$
\end{lemma}

\begin{proof}
 The Zariski Decomposition follows from $-K_S-vA\sim_{\DR} (1-v)A+B+C+D+E$. We have
$S_S(A)=\frac{7}{15}.$ Thus, $\delta_P(S)\le \frac{15}{7}$ for $P\in A$. Note that for $P\in A\backslash B$ we have:
$$h(v)\le \begin{cases} \frac{(4+v)^2}{32} \text{ if }v\in\big[0,\frac{4}{5}\big],\\
6 (2 v - 1) (1-v)\text{ if }v\in\big[\frac{4}{5},1\big].
\end{cases}$$
So we have 
$S(W_{\bullet,\bullet}^{A};P)\le\frac{43}{75}.$
Thus, $\delta_P(S)\ge\frac{75}{43}$ if $A\backslash B $.
\end{proof}
\begin{lemma}\label{deg2-3219-nearA2points}
Suppose $P$ belongs to a $(-1)$-curve $A$ and there exist $(-1)$-curves and $(-2)$-curves   which form the following dual graph:
\begin{figure}[h!]
    \centering
\includegraphics[width=4cm]{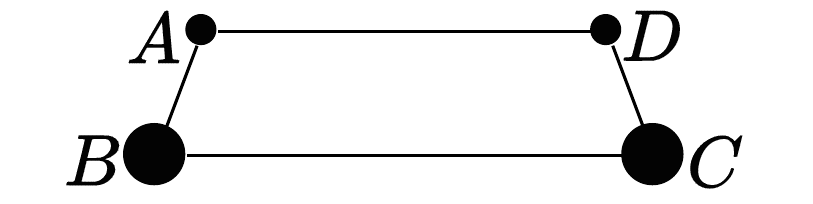}
    \caption{Dual graph: $(-K_S)^2=2$ and $\delta_P(S)\ge\frac{32}{19}$}
\end{figure}
\par Then $\tau(A)=1$ and the Zariski Decomposition of the divisor $-K_S-vA$ is given by:
{ {\allowdisplaybreaks\begin{align*}
&&P(v)=\begin{cases}-K_S-vA-\frac{v}{3}(2B+C)\text{ if }v\in\big[0,\frac{3}{4}\big],\\
-K_S-vA-(2v-1)B-(3v-2)C-(4v -3)D\text{ if }v\in\big[\frac{3}{4},1\big].
\end{cases}\\&&N(v)=\begin{cases}\frac{v}{3}(2B+C)\text{ if }v\in\big[0,\frac{3}{4}\big],\\
(2v-1)B+(3v-2)C+(4v -3)D\text{ if }v\in\big[\frac{3}{4},1\big].
\end{cases}
\end{align*}}
}
Moreover, 
$$(P(v))^2=\begin{cases}2 - 2v-\frac{v^2}{3}\text{ if }v\in\big[0,\frac{3}{4}\big],\\
5(1-v)^2\text{ if }v\in\big[\frac{3}{4},1\big].
\end{cases}P(v)\cdot A=\begin{cases}
1+\frac{v}{3}\text{ if }v\in\big[0,\frac{3}{4}\big],\\
5(1-v)\text{ if }v\in\big[\frac{3}{4},1\big].
\end{cases}$$
In this case $\delta_P(S)\ge\frac{32}{19}\text{ if }P\in A\backslash B.$
\end{lemma}

\begin{proof}
 The Zariski Decomposition follows from 
 $-K_S-vA\sim_{\DR} (1-v)A+B+C+D$. We have
$S_S(A)=\frac{11}{24}.$ Thus, $\delta_P(S)\le \frac{24}{11}$ for $P\in A$. Note that for $P\in A\backslash  B $ we have:
$$h(v)\le \begin{cases} \frac{(v+3)^2}{18} \text{ if }v\in\big[0,\frac{3}{4}\big],\\
 \frac{5(1-v)(3 v - 1)}{2} \text{ if }v\in\big[\frac{3}{4},1\big].
\end{cases}$$
So we have 
$S(W_{\bullet,\bullet}^{A};P)\le\frac{19}{32}.$
Thus, $\delta_P(S)\ge\frac{32}{19} $ if $P\in A\backslash B$.
\end{proof}
\begin{lemma}\label{deg2-2717-nearA1points}
Suppose $P$ belongs to a $(-1)$-curve $A$ and there exist $(-1)$-curves and $(-2)$-curves   form the following dual graph:
\begin{figure}[h!]
    \centering
\includegraphics[width=4cm]{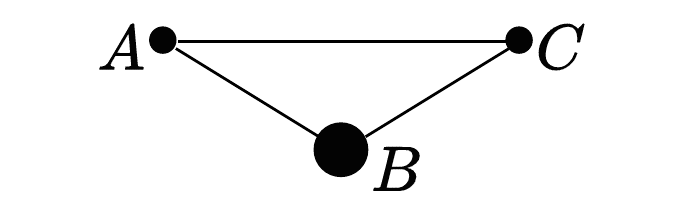}
    \caption{Dual graph: $(-K_S)^2=2$ and $\delta_P(S)\ge\frac{27}{17}$}
\end{figure}
\par Then $\tau(A)=\frac{3}{2}$ and the Zariski Decomposition of the divisor $-K_S-vA$ is given by:
{  \allowdisplaybreaks 
\begin{align*}
&&P(v)=\begin{cases}-K_S-vA-\frac{v}{2}B\text{ if }v\in\big[0,\frac{2}{3}\big],\\
-K_S-vA-(2v - 1)B-(3v - 2)C\text{ if }v\in\big[\frac{2}{3},1\big],
\end{cases}
\\&&N(v)=\begin{cases}\frac{v}{2}B\text{ if }v\in\big[0,\frac{2}{3}\big],\\
(2v - 1)B+(3v - 2)C\text{ if }v\in\big[\frac{2}{3},1\big].
\end{cases}
\end{align*}
}
Moreover, 
$$(P(v))^2=\begin{cases}2 - 2v - \frac{v^2}{2}\text{ if }v\in\big[0,\frac{2}{3}\big],\\
4(1-v)^2\text{ if }v\in\big[\frac{2}{3},1\big].
\end{cases}P(v)\cdot A=\begin{cases}1+\frac{v}{2}\text{ if }v\in\big[0,\frac{2}{3}\big],\\
4(1-v)\text{ if }v\in\big[\frac{2}{3},1\big].
\end{cases}$$
In this case: $\delta_P(S)>\frac{27}{17}\text{ if }P\in A\backslash B.$
\end{lemma}
\begin{proof}
 The Zariski Decomposition follows from $-K_S-vA\sim_{\DR} (1-v)A+B+C$. We have $S_S(A)=\frac{4}{9}.$ Thus, $\delta_P(S)\le \frac{9}{4}$ for $P\in A$. Note that for $P\in A\backslash B$ we have:
$$h(v)\le \begin{cases} \frac{ (v + 2)^2}{8} \text{ if }v\in\big[0,\frac{2}{3}\big],\\
4v(1-v)\text{ if }v\in\big[\frac{2}{3},1\big].
\end{cases}$$
So we have 
$S(W_{\bullet,\bullet}^{A};P)\le\frac{17}{27}.$ Thus, $\delta_P(S)\ge\frac{27}{17}$ if $P\in A\backslash B$.
\end{proof}

\subsection{Finding $\delta$-invariants for degree $2$}

Let $X$ be a singular del Pezzo surface of degree $2$ with and $S$ be a minimal resolution of $X$. Then there are several possible cases:
\begin{itemize}
    \item[I.] $X$ has an $\DA_1$ singularity and contains $44$ lines,
    \item[II.] $X$ has two $\DA_1$ singularities and contains $34$ lines. In this case, we let $E_1$ and $E_2$ be the exceptional divisors, $L_{12}$, $L_{12}'$ ,$L_{i,j}$ and $L_{i,j}'$ for $i\in\{1,2\}$, $j\in\{1,2,3,4\}$ be the lines on $S$,
    \item[III.] $X$ has  three $\DA_1$ singularities and contains $26$ lines. In this case, we let $E_i$ for $i\in\{1,2\}$ be the exceptional divisors, 
 $L_{1}$ and $L_{12}$ are the lines on $S$, 
    \item[IV.] $X$ has three $\DA_1$ singularities and contains $25$ lines. In this case, we let $E_1$, $E_2$ and $E_3$ be the exceptional divisors, $L_{123}$, $L_{i,j}$ and $L_{i,j}'$ for $i\in\{1,2,3\}$, $j\in\{1,2,3,4\}$ be the lines on $S$,
    \item[V.] $X$ has four  $\DA_1$ singularities and contains $20$ lines. In this case, we let $E_1$, $E_2$, $E_3$ and $E_4$ be the exceptional divisors, $L_{ij}$ and $L_{ij}'$ for $i,j\in\{1,2,3,4\}$ and $i<j$ be the lines on $S$,
     \item[VI.] $X$ has four  $\DA_1$ singularities and contains $19$ lines. In this case, we let $E_1$, $E_2$, $E_3$ and $E_4$ be the exceptional divisors, $L_{234}$, $L_{1,i}$ and $L_{1,i}'$ for $i\in\{1,2,3\}$, $L_{j,k}$ and $L_{j,k}'$ for $j\in\{2,3,4\}$, $k\in\{1,2\}$ be the lines on $S$,
      \item[VII.] $X$ has five  $\DA_1$ singularities and contains $14$ lines. In this case, we let $E_1$, $E_2$, $E_3$, $E_4$ and $E_5$ be the exceptional divisors, $L_{134}$, $L_{125}$ $L_{ij}$ and $L_{ij}'$ for $(i,j)\in\{(2,3),(2,4),(3,5),(4,5)\}$, $L_{1,k}$ and $L_{1,k}'$ for $k\in\{1,2\}$  be the lines on $S$,
      \item[VIII.] $X$ has six  $\DA_1$ singularities and contains $10$ lines. In this case, we let $E_1$, $E_2$, $E_3$, $E_4$, $E_5$ and $E_6$ be the exceptional divisors, $L_{246}$, $L_{136}$, $L_{235}$, $L_{145}$, $L_{4,5}$, $L_{ij}$ and $L_{ij}'$ for $(i,j)\in\{(1,2),(3,4),(5,6)\}$ be the lines on $S$,
    \item[IX.] $X$ has   $\DA_2$ singularity and contains $31$ lines. In this case, we let $E_1$ and $E_1'$ be the exceptional divisors, $L_{1,i}$ and $L_{1,i}'$ for $i\in\{1,..,6\}$ be the lines on $S$,
  \item[X.] $X$ has   $\DA_2$ and $\DA_1$ singularities and contains $20$ lines. In this case, we let $E_1$ and $E_1'$, $E_2$ be the exceptional divisors, $L_{1,i}$ and $L_{1,i}'$ for $i\in\{1,..,4\}$, $L_{2,j}$ and $L_{2,j}'$ for $j\in\{1,2\}$, $L_{12}$, $L_{12}'$ be the lines on $S$,
  \item[XI.] $X$ has   $\DA_2$ and two $\DA_1$ singularities and contains $18$ lines. In this case, we let $E_1$, $E_1'$, $E_2$ and $E_3$ be the exceptional divisors, $L_{12}$, $L_{12}'$, $L_{13}$, $L_{13}'$, $L_{23}$, $L_{23}'$, $L_{i,j}$ and $L_{i,j}'$ for $i,j\in\{1,2\}$ be the lines on $S$,
  \item[XII.] $X$ has   $\DA_2$ and three $\DA_1$ singularities and contains $13$ lines. In this case, we let $E_1$, $E_1'$, $E_2$, $E_3$ and $E_4$ be the exceptional divisors, $L_{234}$, $L_{12}$, $L_{12}'$, $L_{13}$, $L_{13}'$, $L_{14}$, $L_{14}'$, $L_{i,j}$ and $L_{i,j}'$ for $i\in\{2,3,4\}$, $j\in\{1,2\}$ be the lines on $S$,
  \item[XIII.] $X$ has two  $\DA_2$ singularities and contains $16$ lines. In this case, we let $E_1$, $E_1'$, $E_2$ and $E_2'$ be the exceptional divisors, $L_{12}$, $L_{12}''$, $L_{i,j}$  and $L_{i,j}'$ for  $i,j\in\{1,2\}$ be the lines on $S$,
  \item[XIV.] $X$ has two  $\DA_2$ and one $\DA_1$ singularities and contains $12$ lines. In this case, we let $E_1$, $E_1'$, $E_2$, $E_2'$ and $E_3$ be the exceptional divisors, $L_{12}$, $L_{12}''$,   $L_{i,1}$ and $L_{i,1}'$ for  $i\in\{1,2\}$ be the lines on $S$,
  \item[XV.] $X$ has three  $\DA_2$ singularities and contains $8$ lines. In this case, we let $E_1$, $E_1'$, $E_2$, $E_2'$ and $E_3$, $E_3'$ be the exceptional divisors, $L_{ij}$,  for  $i,j\in\{1,2,3\}$ and $i<j$ be the lines on $S$,
  \item[XVI.] $X$ has   $\DA_3$ singularity and contains $22$ lines. In this case, we let $E_1$, $E_1'$ and $E_2$ be the exceptional divisors, $L_{2,1}$, $L_{2,2}$, $L_{1,i}$ and  $L_{1,i}'$ for  $i\in\{1,2,3,4\}$ be the lines on $S$,
   \item[XVII.] $X$ has   $\DA_3$ and $\DA_1$ singularities and contains $16$ lines. In this case, we let $E_1$, $E_1'$, $E_2$ and $E_3$ be the exceptional divisors, $L_{2,1}$, $L_{2,2}$, $L_{13}$, $L_{13}'$, $L_{i,j}$ and  $L_{i,j}'$ for  $i\in\{1,3\}$, $j\in\{1,2\}$ be the lines on $S$,
   \item[XVIII.] $X$ has   $\DA_3$ and $\DA_1$ singularities and contains $15$ lines. In this case, we let $E_1$, $E_1'$, $E_2$ and $E_3$ be the exceptional divisors,  $L_{23}$,  $L_{1,i}$ and  $L_{1,i}'$ for  $i\in\{1,2,3,4\}$, $L_{3,j}$ and  $L_{3,j}'$ for $j\in\{1,2,3\}$ be the lines on $S$,
   \item[XIX.] $X$ has   $\DA_3$ and two $\DA_1$ singularities and contains $12$ lines. In this case, we let $E_1$, $E_1'$, $E_2$, $E_3$ and $E_4$ be the exceptional divisors, $L_{2,1}$, $L_{2,2}$,  $L_{ij}$ and  $L_{ij}'$ for  $(i,j)\in\{(1,3),(1,4),(3,4)\}$ be the lines on $S$,
   \item[XX.] $X$ has   $\DA_3$ and two $\DA_1$ singularities and contains $11$ lines. In this case, we let $E_1$, $E_1'$, $E_2$, $E_3$ and $E_4$ be the exceptional divisors, $L_{23}$,   $L_{ij}$ and  $L_{ij}'$ for  $(i,j)\in\{(1,4),(3,4)\}$, $L_{3,1}$, $L_{3,1}'$, $L_{1,k}$ and  $L_{1,k}'$ for  $k\in\{1,2\}$ be the lines on $S$,
   \item[XXI.] $X$ has   $\DA_3$ and three $\DA_1$ singularities and contains $8$ lines. In this case, we let $E_1$, $E_1'$, $E_2$, $E_3$ and $E_4$ be the exceptional divisors, $L_{345}$, $L_{25}$   $L_{ij}$ and  $L_{ij}'$ for  $(i,j)\in\{(1,4),(3,4)\}$, $L_{5,1}$, $L_{5,1}'$ be the lines on $S$,
   \item[XXII.] $X$ has   $\DA_3$ and  $\DA_2$ singularities and contains $10$ lines. In this case, we let $E_1$, $E_1'$, $E_2$, $E_3$ and $E_3'$ be the exceptional divisors, $L_{2,1}$, $L_{2,2}$, $L_{1,1}$, $L_{1,1}'$, $L_{13}$, $L_{13}''$  $L_{3,i}$ and  $L_{3,i}'$ for $i\in\{1,2\}$    be the lines on $S$,
   \item[XXIII.] $X$ has   $\DA_3$, $\DA_2$ and  $\DA_1$ singularities and contains $7$ lines. In this case, we let $E_1$, $E_1'$, $E_2$, $E_3$ and $E_3'$ be the exceptional divisors, $L_{2,1}$, $L_{2,2}$, $L_{1,1}$, $L_{1,1}'$, $L_{13}$, $L_{13}''$  $L_{3,i}$ and  $L_{3,i}'$ for $i\in\{1,2\}$    be the lines on $S$,
   \item[XXIV.] $X$ has two  $\DA_3$ singularities and contains $6$ lines. In this case, we let $E_1$, $E_1'$, $E_2$, $E_3$, $E_3'$ and $E_4$ be the exceptional divisors, $L_{2,1}$, $L_{2,2}$, $L_{4,1}$, $L_{4,1}'$, $L_{13}$, $L_{13}''$    be the lines on $S$
   \item[XXV.] $X$ has two  $\DA_3$ and on $\DA_1$ singularities and contains $4$ lines. In this case, we let $E_1$, $E_1'$, $E_2$, $E_3$, $E_3'$, $E_4$ and $E_5$ be the exceptional divisors, $L_{25}$, $L_{45}$, $L_{13}$, $L_{13}''$    be the lines on $S$,
   \item[XXVI.] $X$ has   $\DA_4$ singularity and contains $14$ lines. In this case, we let $E_1$, $E_1'$, $E_2$ and $E_2'$ be the exceptional divisors, $L_{2,1}$, $L_{2,1}'$, $L_{1,i}$ and   $L_{1,i}'$ for $i\in\{1,2,3\}$   be the lines on $S$,
   \item[XXVII.] $X$ has  $\DA_4$ and $\DA_1$ singularities and contains $10$ lines. In this case, we let $E_1$, $E_1'$, $E_2$, $E_2'$ and $E_3$ be the exceptional divisors, $L_{2,1}$, $L_{2,1}'$ $L_{1,i}$ and   $L_{1,i}'$ for $i\in\{1,2,3\}$   be the lines on $S$,
   \item[XXVIII.] $X$ has  $\DA_4$ and $\DA_2$ singularities and contains $6$ lines. In this case, we let $E_1$, $E_1'$, $E_2$, $E_2'$, $E_3$ and $E_3'$ be the exceptional divisors, $L_{13}$, $L_{13}''$, $L_{i,1}$ and   $L_{i,1}'$ for $i\in\{2,3\}$   be the lines on $S$,
   \item[XXIX.] $X$ has  $\DA_5$  singularity and contains $8$ lines. In this case, we let $L_{2,1}$, $L_{2,1}'$, $L_{1,i}$ and   $L_{1,i}'$ for $i\in\{1,2\}$   be the lines on $S$,
   \item[XXX.] $X$ has  $\DA_5$  singularity and contains $7$ lines. In this case, we let $E_1$, $E_1'$, $E_2$, $E_2'$ and $E_3$ be the exceptional divisors, $L_{3,1}$, $L_{1,i}$ and   $L_{1,i}'$ for $i\in\{1,2,3\}$   be the lines on $S$,
   \item[XXXI.] $X$ has  $\DA_5$ and $\DA_1$  singularities and contains $6$ lines. In this case, we let $E_1$, $E_1'$, $E_2$, $E_2'$, $E_3$ and $E_4$ be the exceptional divisors, $L_{3,1}$, $L_{1,i}$ and   $L_{1,i}'$ for $i\in\{1,2,3\}$   be the lines on $S$,
   \item[XXXII.] $X$ has  $\DA_5$ and $\DA_1$  singularities and contains $5$ lines. In this case, we let $E_1$, $E_1'$, $E_2$, $E_2'$, $E_3$ and $E_4$ be the exceptional divisors, $L_{3,1}$, $L_{1,i}$ and   $L_{1,i}'$ for $i\in\{1,2,3\}$   be the lines on $S$,
   \item[XXXIII.] $X$ has  $\DA_5$ and $\DA_2$  singularities and contains $3$ lines. In this case, we let $E_1$, $E_1'$, $E_2$, $E_2'$, $E_3$, $E_4$ and $E_4'$ be the exceptional divisors, $L_{14}$, $L_{14}'$ and $L_{3,1}$ be the lines on $S$,
   \item[XXXIV.] $X$ has  $\DA_6$   singularity and contains $4$ lines. In this case, we let $E_1$, $E_1'$, $E_2$, $E_2'$, $E_3$ and $E_3'$ be the exceptional divisors, $L_{i,1}$ and $L_{i,1}'$  for $i\in\{1,2\}$ be the lines on $S$,
   \item[XXXV.] $X$ has  $\DA_7$   singularity and contains $2$ lines. In this case, we let $E_1$, $E_1'$, $E_2$, $E_2'$, $E_3$, $E_3'$ and $E_4$ be the exceptional divisors, $L_{i,1}$ and $L_{i,1}'$  for $i\in\{1,2\}$ be the lines on $S$,
   \item[XXXVI.] $X$ has  $\mathbb{D}_4$   singularity and contains $14$ lines. In this case, we let $E_1$, $E_2$,  $E_3$, and $E$ be the exceptional divisors, $L_{i,j}$   for $i\in\{1,2,3\}$, $j\in\{1,2\}$ be the lines on $S$,
   \item[XXXVII.] $X$ has  $\mathbb{D}_4$   and $\DA_1$ singularities and contains $9$ lines. In this case, we let $E_1$, $E_1$,  $E_2$, $E_3$ and $E$ be the exceptional divisors, $L_{i,j}$ and $L_{i,j}'$ for $i\in\{1,3\}$, $j\in\{1,2\}$ be the lines on $S$,
   \item[XXXVIII.] $X$ has  $\mathbb{D}_4$   and two $\DA_1$ singularities and contains $6$ lines. In this case, we let $E_1$, $E_1'$,  $E_2$, $E_3$ and $E_3'$ and $E$ be the exceptional divisors, $L_{i}$ and $L_{i}'$ for $i\in\{2,3\}$ be the lines on $S$,
    \item[XXXIX.] $X$ has  $\mathbb{D}_4$   and three $\DA_1$ singularities and contains $4$ lines. In this case, we let $E_1$, $E_1'$,  $E_2$, $E_2'$, $E_3$, $E_3'$ and $E$ be the exceptional divisors, $L_{i}$ and $L_{i}'$ for $i\in\{1,2,3\}$ be the lines on $S$,
     \item[XL.] $X$ has  $\mathbb{D}_5$  singularity and contains $8$ lines. In this case, we let $E_1$, $E_1'$,  $E_2$, $E_3$,  and $E$ be the exceptional divisors, $L_{1}$ and $L_{1}'$, $L_{3,1}$ and $L_{3,2}$ be the lines on $S$,
     \item[XLI.] $X$ has  $\mathbb{D}_5$  and $\DA_1$ singularities and contains $5$ lines. In this case, we let $E_1$, $E_1'$,  $E_2$, $E_3$, $E_4$  and $E$ be the exceptional divisors, $L_{34}$, $L_{1}$ and $L_{1}'$, $L_{4,1}$ and $L_{4,1}'$ be the lines on $S$,
     \item[XLII.] $X$ has  $\mathbb{D}_6$  singularity and contains $3$ lines. In this case, we let $E_1$,   $E_2$, $E_3$, $E_4$, $E_5$  and $E$ be the exceptional divisors, $L_{34}$, $L_{5}$, $L_{5}'$ and $L$ be the lines on $S$,
     \item[XLIII.] $X$ has  $\mathbb{D}_6$  and $\DA_1$ singularities and contains $2$ lines. In this case, we let $E_1$,   $E_2$, $E_3$, $E_4$, $E_5$, $E_6$  and $E$ be the exceptional divisors, $L_{56}$ and $L$ be the lines on $S$,
     \item[XLIV.] $X$ has  $\mathbb{E}_6$  singularity and contains $4$ lines. In this case, we let $E_1$, $E_1'$,  $E_2$, $E_2'$, $E_3$ and $E$ be the exceptional divisors, $L_{1}$ and $L_1'$ be the lines on $S$,
     \item[XLV.] $X$ has  $\mathbb{E}_7$  singularity and contains $1$ line. In this case, we let $E_1$, $E_2$,  $E_3$, $E_4$, $E_5$, $E_6$ and $E$ be the exceptional divisors, $L_{6}$ be the line on $S$.
\end{itemize}
such that the dual graph of the $(-1)$-curves  and $(-2)$-curves or the dual graph of the $(-1)$-curves adjacent to a $(-2)$-curves if marked with $(\star)$ on $S$   is given the picture below. Then 
\begin{center}
\hspace*{0cm}\includegraphics[width=16cm]{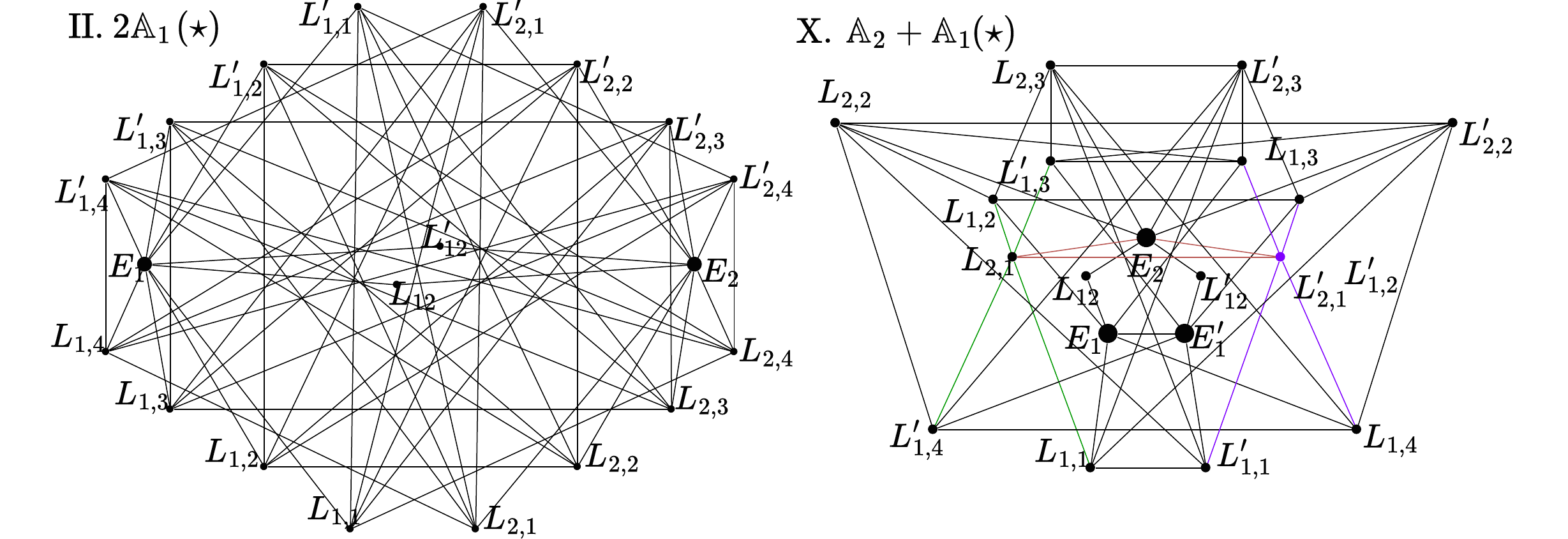}\\
\hspace*{0cm}\includegraphics[width=16cm]{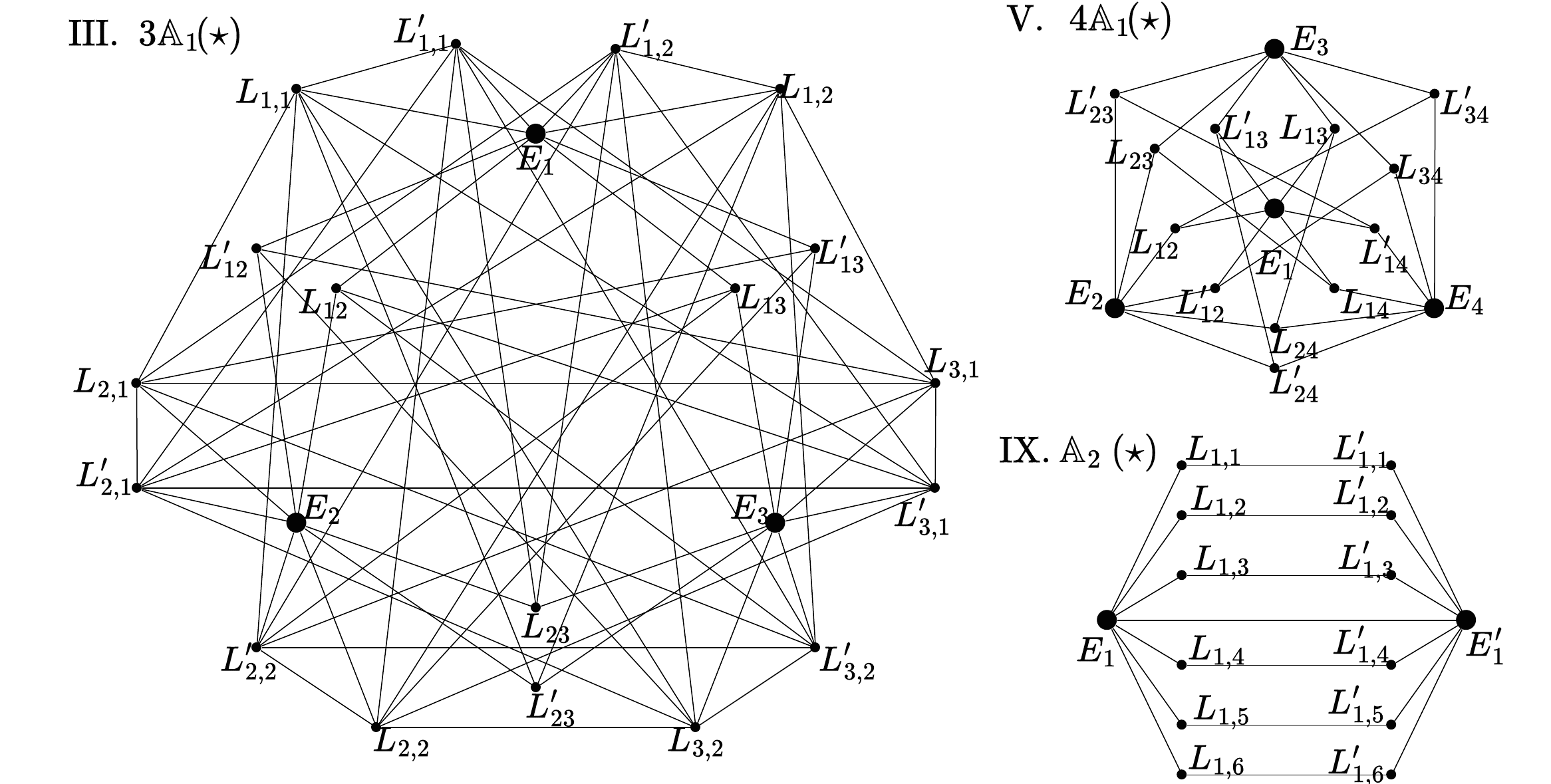}
\begin{figure}[h!]
    \centering
\hspace*{-0cm}\includegraphics[width=16cm]{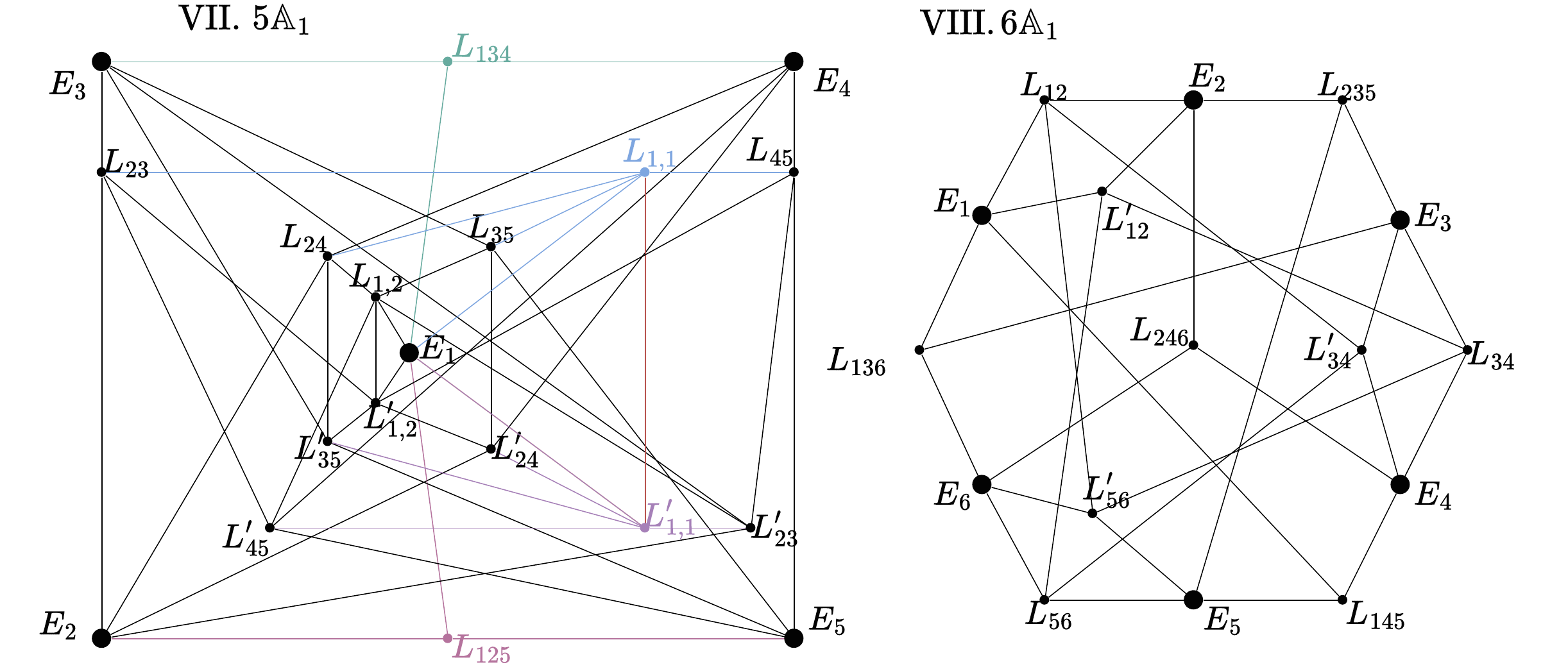}
\end{figure}
\newpage
\includegraphics[width=15cm]{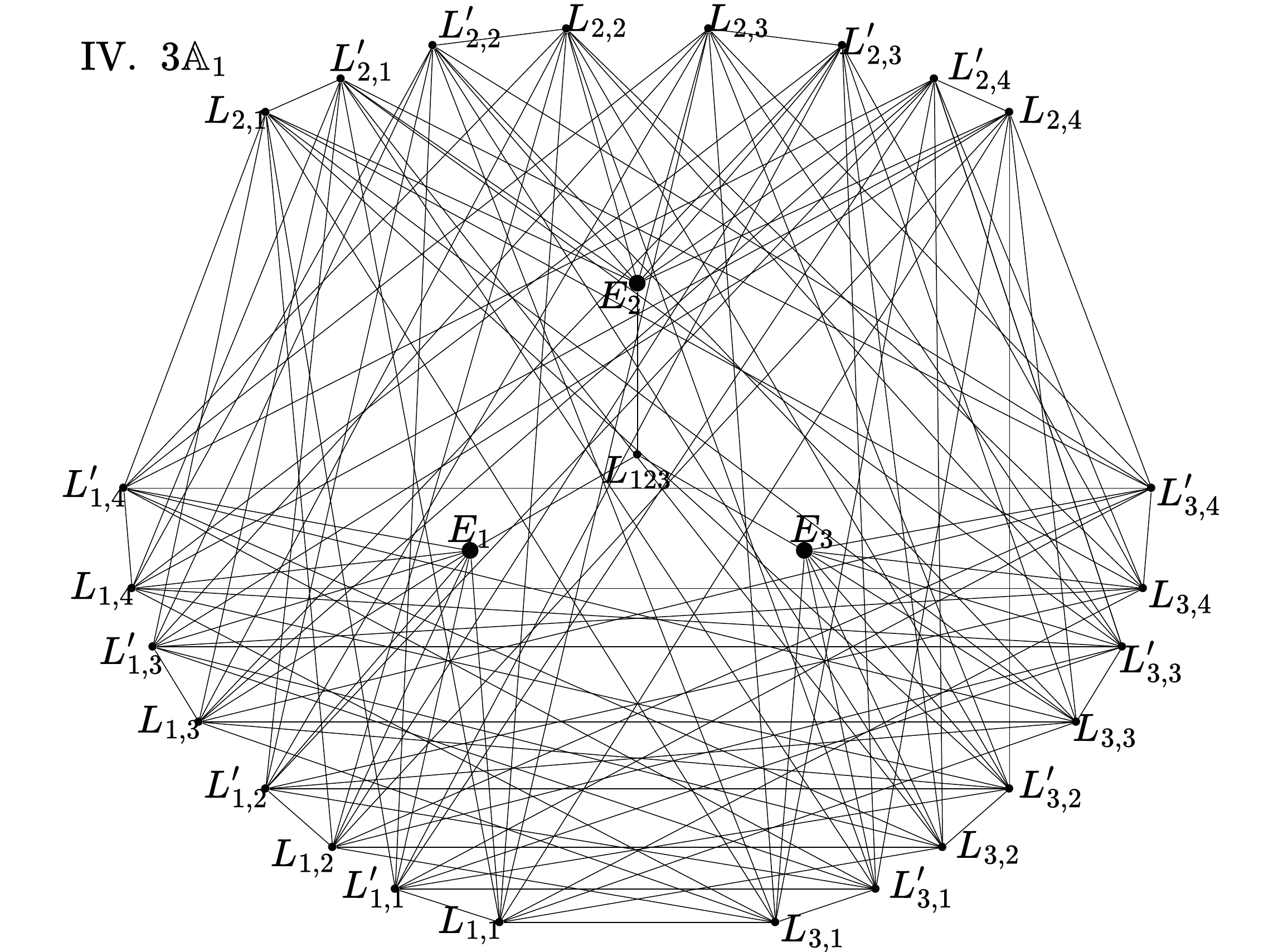}
\\$ $\\
\begin{figure}[h!]
    \centering
\hspace*{0cm}\includegraphics[width=17cm]{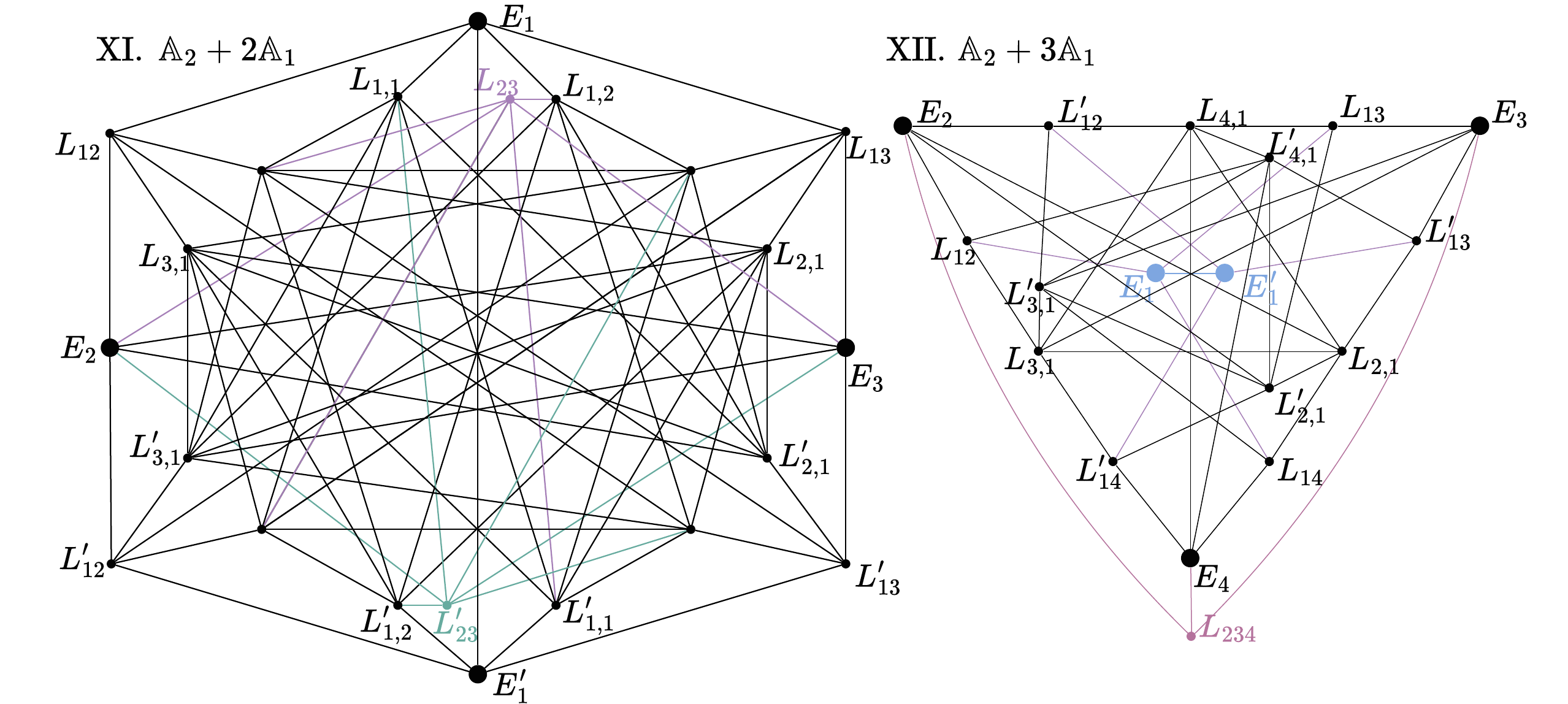}
\end{figure}
\newpage
\hspace*{0cm}\includegraphics[width=17cm]{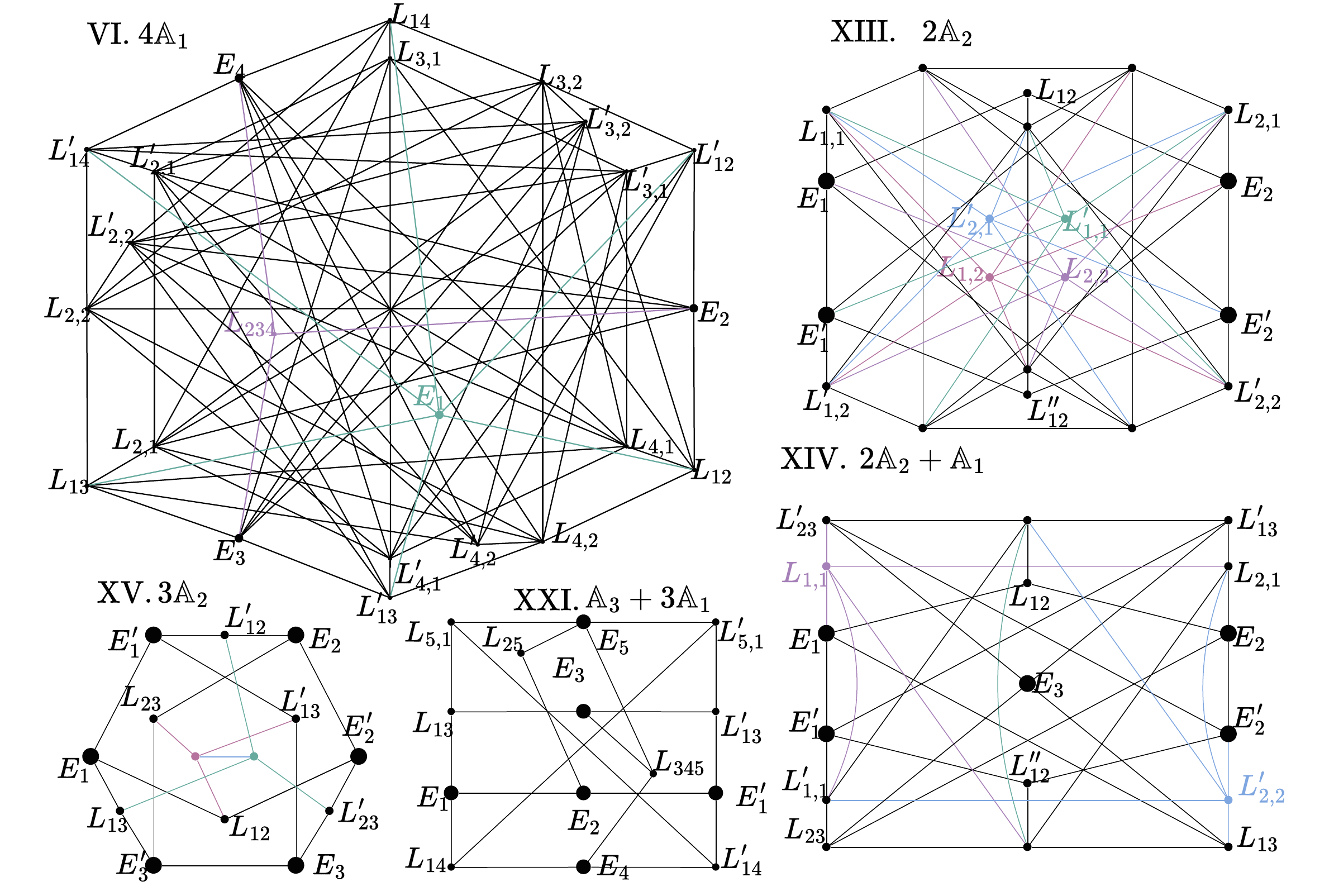}
\begin{figure}[h!]
    \centering
\hspace*{0cm}\includegraphics[width=17cm]{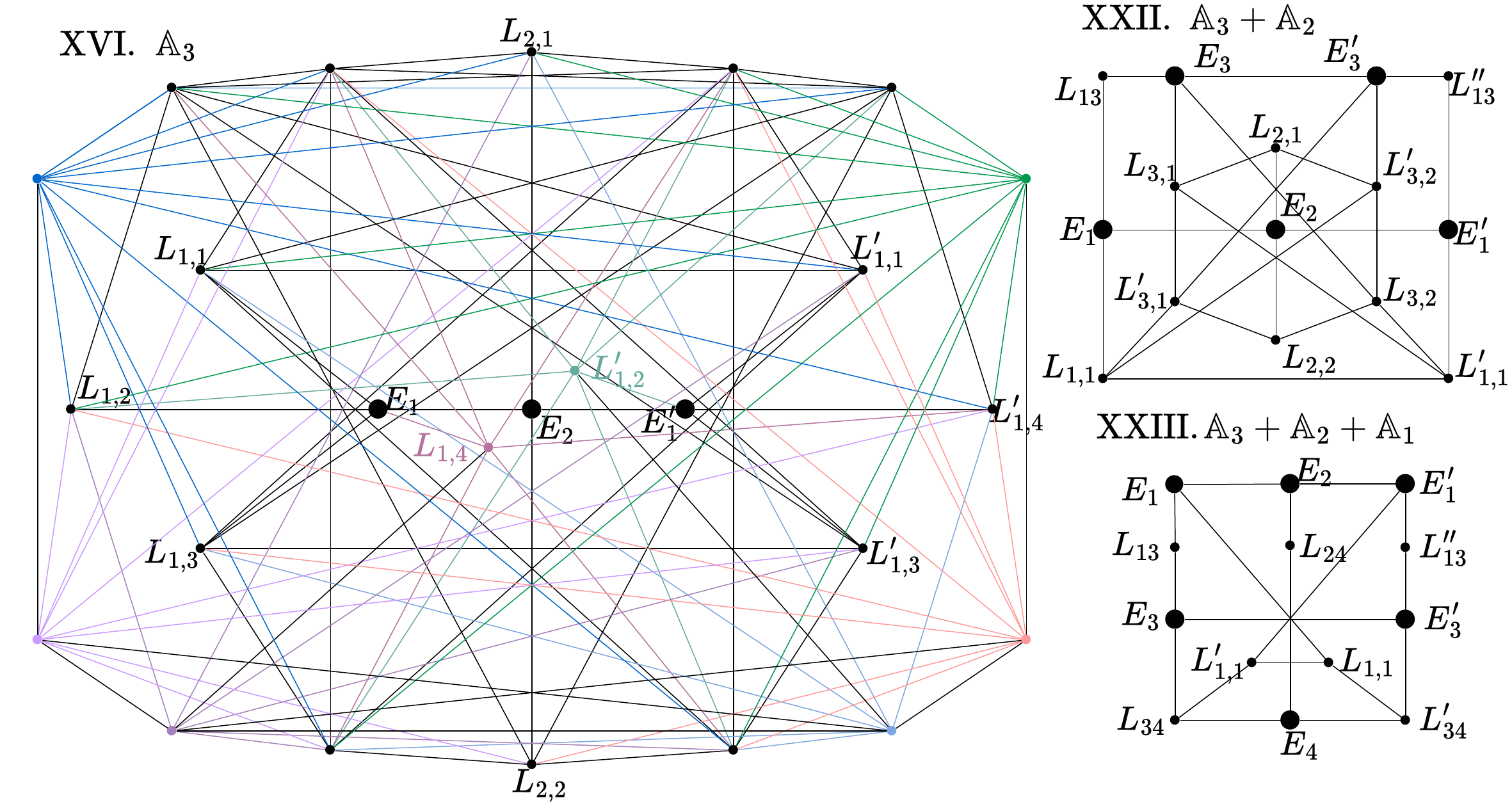}
\end{figure}
\newpage
\hspace*{0cm}\includegraphics[width=17cm]{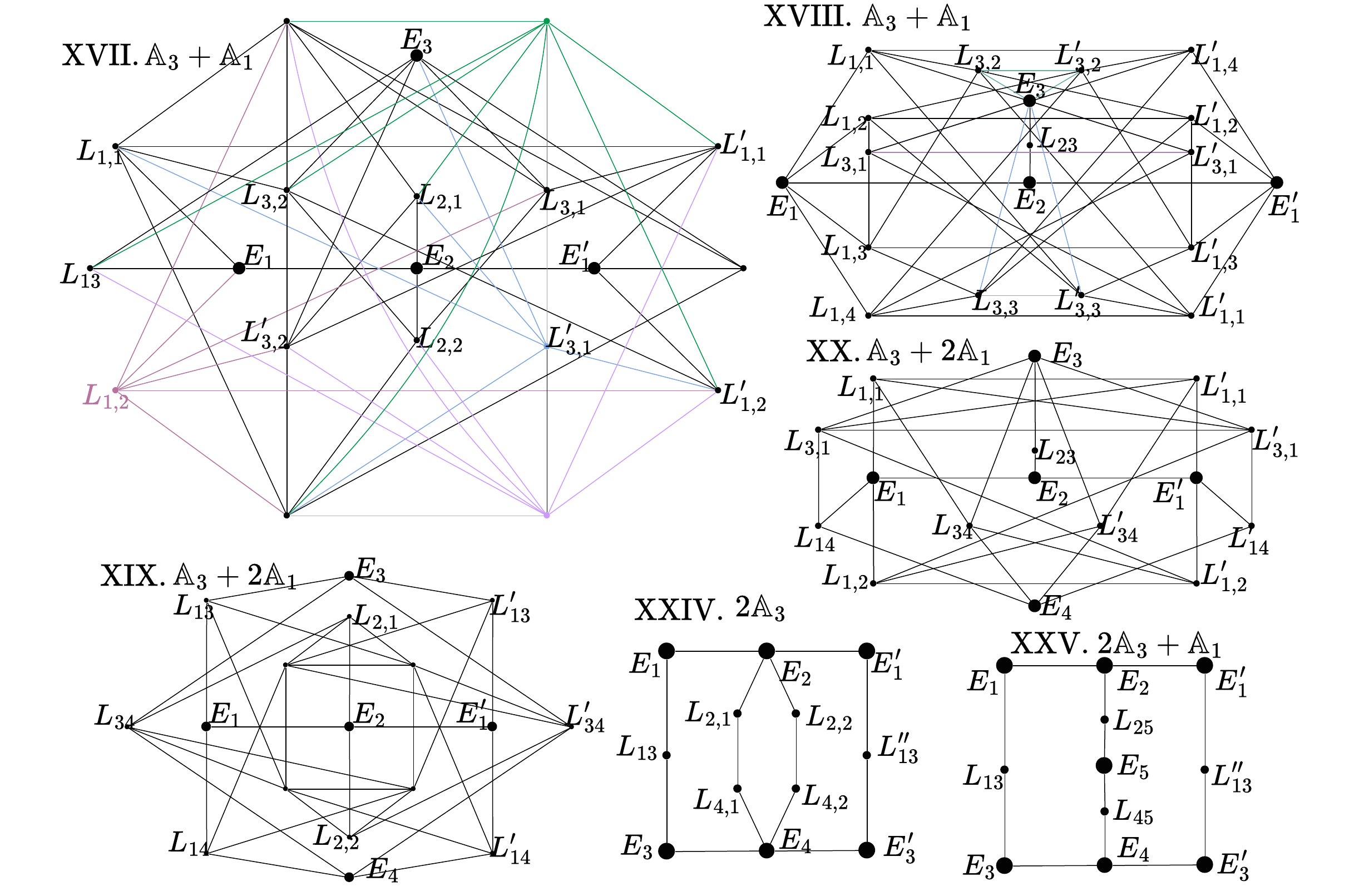}
\begin{figure}[h!]
    \centering
\hspace*{0cm}\includegraphics[width=17cm]{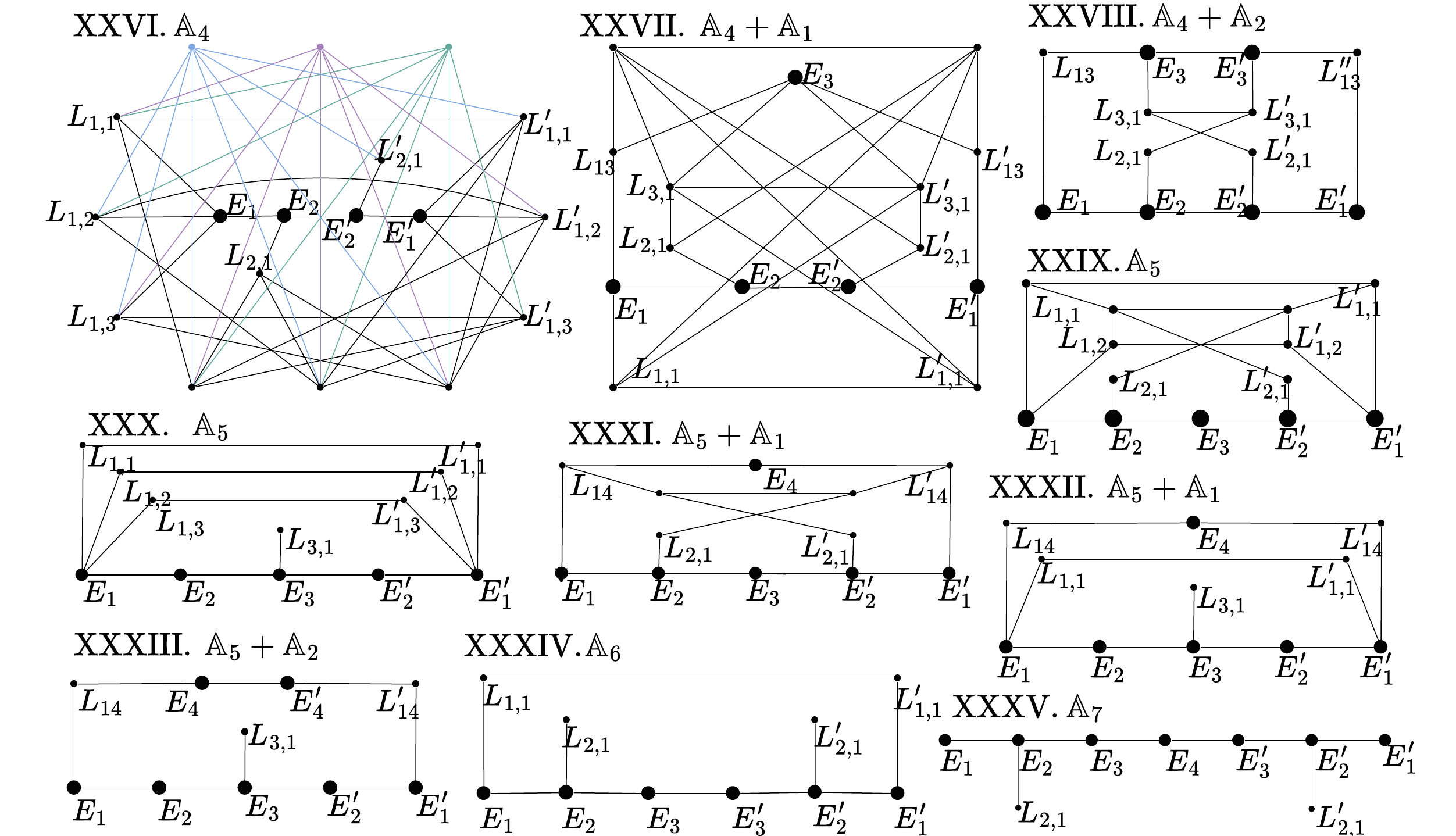}
\end{figure}
\newpage
\begin{figure}[h!]
    \centering
\hspace*{0cm}\includegraphics[width=17cm]{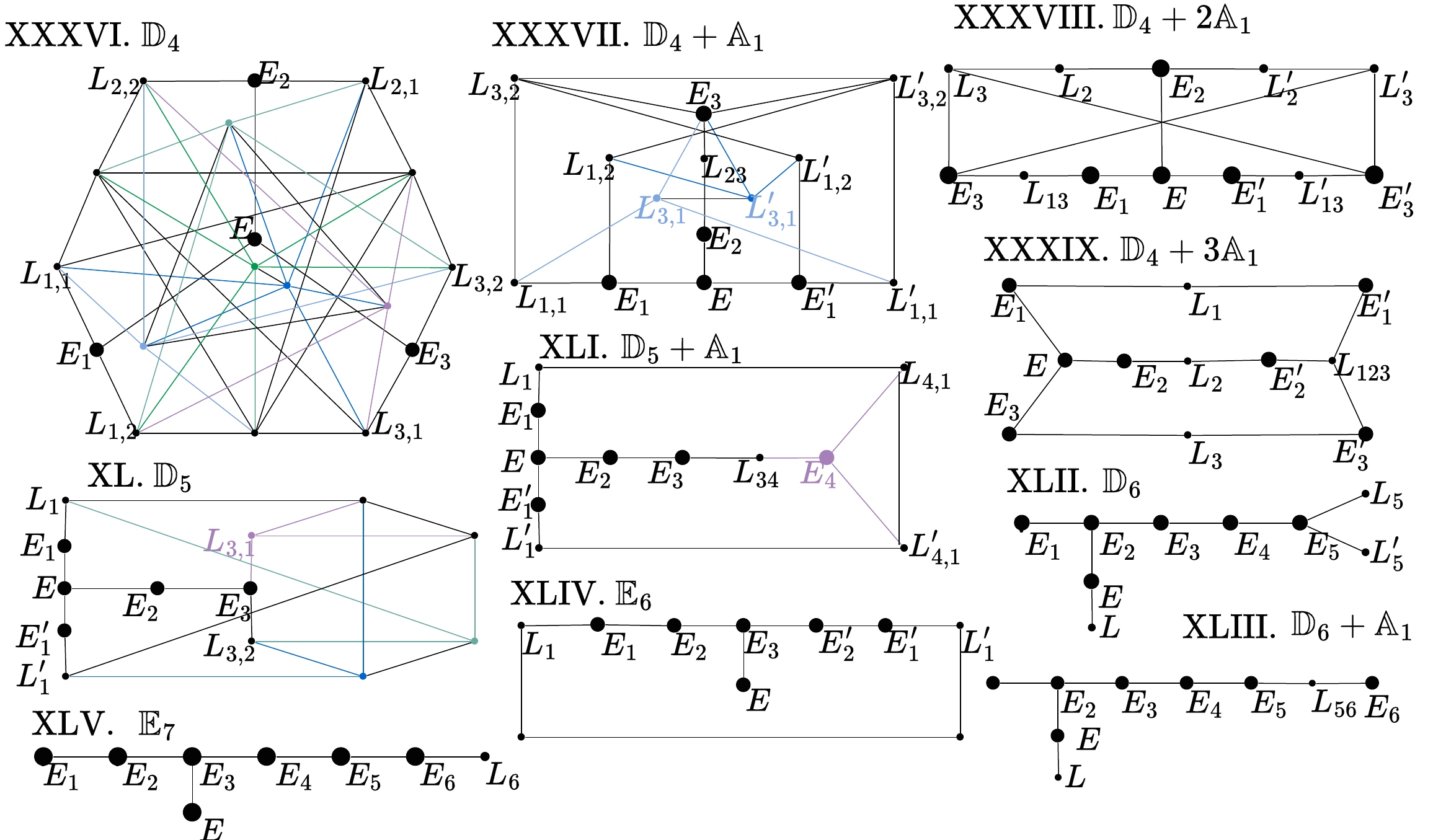}
\end{figure}
\end{center}
\newpage
\begin{itemize}
    \item[I.] One has $\delta(X)=\frac{3}{2}$.
    \item[II.] $\delta(X)=\frac{3}{2}$ since depending on the position of point $P\in S$ we have 
    \begin{table}[h!]
\hspace*{0.5cm}
\begin{tabular}{ | c || c | c | c | c |}
 \hline 
 $P$ & $E_1\cup E_2$ & $(L_{12}\cup L_{12}')\backslash (E_1\cup E_2)$ & $\mathbf{L}_2\backslash (E_1\cup E_2) $ & o/w \\\hline
$\delta_P(S)$ & $\frac{3}{2}$ & $2$ & $\ge\frac{27}{17}$ & $\ge \frac{9}{5}$\\\hline
    \end{tabular}\\
  \hspace*{0.5cm}  where $\mathbf{L}_2 := \bigcup_{i\in\{1,2\},\text{ }j\in\{1,2,3,4\}}\big(L_{i,j}\cup L_{i,j}'\big)$.
\caption{Local $\delta$-invariants: $(-K_S)^2=2$ and  $2\DA_1$ singularities}
\end{table}
 \item[III.] $\delta(X)=\frac{3}{2}$ since depending on the position of point $P\in S$ we have 
 \begin{table}[h!]
\hspace*{0.5cm}
\begin{tabular}{ | c || c | c | c | c |}
 \hline 
 $P$ & $E_1\cup E_2\cup E_3$ & $\mathbf{L}_3^{(1)}\backslash (E_1\cup E_2\cup E_3)$ & $\mathbf{L}_3^{(2)}\backslash (E_1\cup E_2\cup E_3) $ & o/w \\\hline
$\delta_P(S)$ & $\frac{3}{2}$ & $2$ & $\ge\frac{27}{17}$ & $\ge \frac{9}{5}$\\\hline
    \end{tabular}\\
  \hspace*{0.5cm} where $\mathbf{L}_3^{(1)}:=L_{12}\cup L_{13}\cup L_{23}\cup L_{12}'\cup L_{13}'\cup L_{23}' $, $\mathbf{L}_3^{(2)} :=\bigcup_{i\in\{1,2,3\},\text{ }j\in\{1,2\}}\big(L_{i,j}\cup L_{i,j}'\big)$.
\caption{Local $\delta$-invariants: $(-K_S)^2=2$ and  $3\DA_1$ singularities (26 lines)}
\end{table}
\item[IV.] $\delta(X)=\frac{3}{2}$ since depending on the position of point $P\in S$ we have 
\begin{table}[h!]
\hspace*{0.5cm}
\begin{tabular}{ | c || c | c | c | }
 \hline 
 $P$ & $E_1\cup E_2\cup E_3\cup L_{123}$ & $\mathbf{L}_4\backslash (E_1\cup E_2 \cup E_3) $ & o/w \\\hline
$\delta_P(S)$ & $\frac{3}{2}$  & $\ge\frac{27}{17}$ & $\ge \frac{9}{5}$\\\hline
    \end{tabular}\\
 \hspace*{0.5cm}   where $\mathbf{L}_4 := \bigcup_{i\in\{1,2,3\},\text{ }j\in\{1,2,3,4\}}\big(L_{i,j}\cup L_{i,j}'\big)\backslash (E_1\cup E_2\cup E_3)$.
\caption{Local $\delta$-invariants: $(-K_S)^2=2$ and  $3\DA_1$ singularities (25 lines)}
\end{table}
\item[V.] $\delta(X)=\frac{3}{2}$ since depending on the position of point $P\in S$ we have 
\begin{table}[h!]
\hspace*{0.5cm}
\begin{tabular}{ | c || c | c | c | }
 \hline 
 $P$ & $E_1\cup E_2\cup E_3\cup E_4$ & $\mathbf{L}_5\backslash (E_1\cup E_2\cup E_3\cup E_4) $ & o/w \\\hline
$\delta_P(S)$ & $\frac{3}{2}$  & $\ge\frac{27}{17}$ & $\ge \frac{9}{5}$\\\hline
    \end{tabular}\\
  \hspace*{0.5cm}  where $\mathbf{L}_5 := \bigcup_{i,j\in\{1,2,3,4\},\text{ }i<j}\big(L_{ij}\cup L_{ij}'\big)\backslash (E_1\cup E_2\cup E_3\cup E_4)$.
\caption{Local $\delta$-invariants: $(-K_S)^2=2$ and  $4\DA_1$ singularities (20 lines)}
\end{table}
\item[VI.] $\delta(X)=\frac{3}{2}$ since depending on the position of point $P\in S$ we have 
\begin{table}[h!]
\hspace*{0.5cm}
\begin{tabular}{ | c || c | c | c | c |}
 \hline 
 $P$ & $\mathbf{E}_6\cup L_{234}$ & $\bigcup_{i\in\{2,3,4\}}\big(L_{1i}\cup L_{1i}'\big)\backslash \mathbf{E}_6$ & $\bigcup_{i\in\{2,3,4\},\text{ }j\in\{1,2\}}\big(L_{i,j}\cup L_{i,j}'\big)\backslash \mathbf{E}_6 $ & o/w \\\hline
$\delta_P(S)$ & $\frac{3}{2}$ & $2$ & $\ge\frac{27}{17}$ & $\ge \frac{9}{5}$\\\hline
    \end{tabular}\\
  \hspace*{0.5cm}  where $\mathbf{E}_6:= E_1\cup E_2\cup E_3\cup E_4$.
\caption{Local $\delta$-invariants: $(-K_S)^2=2$ and  $4\DA_1$ singularities (19 lines)}
\end{table}
\item[VII.] $\delta(X)=\frac{3}{2}$ since depending on the position of point $P\in S$ we have 
\begin{table}[h!]
\hspace*{0.5cm}
\begin{tabular}{ | c || c | c | c | c |}
 \hline 
 $P$ & $\mathbf{E}_7\cup L_{134}\cup L_{125}$ & $\mathbf{L}_7\backslash \mathbf{E}_7$ & $L_{1,1}\cup L_{1,1}'\cup L_{1,2}\cup L_{1,2}'\backslash E_1 $ & o/w \\\hline
$\delta_P(S)$ & $\frac{3}{2}$ & $2$ & $\ge\frac{27}{17}$ & $\ge \frac{9}{5}$\\\hline
    \end{tabular}\\
   \hspace*{0.5cm} where $\mathbf{E}_7:=E_1\cup E_2\cup E_3\cup E_4\cup E_5$,  $\mathbf{L}_7:=\bigcup_{(i,j)\in\{(2,3),(2,4),(3,5),(4,5)\}}\big(L_{ij}\cup L_{ij}'\big)$.
\caption{Local $\delta$-invariants: $(-K_S)^2=2$ and  $5\DA_1$ singularities}
\end{table}
\newpage
\item[VIII.] $\delta(X)=\frac{3}{2}$ since depending on the position of point $P\in S$ we have 
\begin{table}[h!]
\hspace*{0.5cm}
\begin{tabular}{ | c || c | c | c | }
 \hline 
 $P$ & $\mathbf{E}_8\cup L_{246}\cup L_{136}\cup L_{235}\cup L_{145}$ & $\mathbf{L}_8\backslash \mathbf{E}_8$ & o/w \\\hline
$\delta_P(S)$ & $\frac{3}{2}$  & $2$ & $\ge \frac{9}{5}$\\\hline
    \end{tabular}\\
 \hspace*{0.5cm}   where $\mathbf{E}_8:= E_1\cup E_2\cup E_3\cup E_4\cup E_5\cup E_6$, $\mathbf{L}_8 := \bigcup_{(i,j)\in\{(1,2),(3,4),(5,6)\}}\big(L_{ij}\cup L_{ij}'\big)$.
\caption{Local $\delta$-invariants: $(-K_S)^2=2$ and  $6\DA_1$ singularities}
\end{table}
\item[IX.] $\delta(X)=\frac{6}{5}$ since depending on the position of point $P\in S$ we have 
\begin{table}[h!]
\hspace*{0.5cm}
\begin{tabular}{ | c || c | c | c | c | }
 \hline 
 $P$ & $E_1\cap E_1'$ & $(E_1\cup E_1')\backslash (E_1\cap E_1')$ & $\mathbf{L}_9\backslash (E_1\cup E_1') $ & o/w \\\hline
$\delta_P(S)$ & $\frac{6}{5}$ & $\frac{9}{7}$  & $\ge \frac{32}{19}$ & $\ge \frac{9}{5}$\\\hline
    \end{tabular}\\
 \hspace*{0.5cm}   $\mathbf{L}_9 := \bigcup_{i\in\{1,..,6\}}\big(L_{1,i}\cup L_{1,i}'\big)$.
\caption{Local $\delta$-invariants: $(-K_S)^2=2$ and  $\DA_2$ singularity}
\end{table}
\item[X.] $\delta(X)=\frac{6}{5}$ since depending on the position of point $P\in S$ we have 
\begin{table}[h!]
\hspace*{0.5cm}
\begin{tabular}{ | c || c | c | c | c | c | c | c | }
 \hline 
 $P$ & $\overline{\mathbf{E}}_{10}$ & $\mathbf{E}_{10} \backslash \overline{\mathbf{E}}_{10}$ & $E_2$ & $(L_{12}\cup L_{12}')\backslash (\mathbf{E}_{10}\cup E_2) $ & $\mathbf{L}_{10}^{(1)}\backslash (\mathbf{E}_{10}\cup E_2) $ & $\mathbf{L}_{10}^{(2)}\backslash \mathbf{E}_{10} $ & o/w \\\hline
$\delta_P(S)$ & $\frac{6}{5}$ & $\frac{9}{7}$ & $\frac{3}{2}$ & $\frac{15}{8}$ & $\ge \frac{27}{17}$ & $\ge \frac{32}{19}$ & $\ge \frac{9}{5}$\\\hline
    \end{tabular}\\
 \hspace*{0.5cm}  where $\mathbf{E}_{10}:= E_1\cup E_1'$,  $\overline{\mathbf{E}}_{10}:= E_1\cap E_1'$,  \\\hspace*{1.7cm} $\mathbf{L}_{10}^{(1)} := L_{2,1}\cup L_{2,1}'\cup L_{2,2}\cup L_{2,2}'$, $\mathbf{L}_{10}^{(2)} := \bigcup_{i\in\{1,..,4\}}\big(L_{1,i}\cup L_{1,i}'\big)$.
\caption{Local $\delta$-invariants: $(-K_S)^2=2$ and  $\DA_2\DA_1$ singularities}
\end{table}
\item[XI.] $\delta(X)=\frac{6}{5}$ since depending on the position of point $P\in S$ we have 
\begin{table}[h!]
\hspace*{0.5cm}
\begin{tabular}{ | c || c | c |c |  c | c | }
 \hline 
 $P$ & $E_1\cap E_1'$ & $\mathbf{E}_{11}^{(1)} \backslash (E_1\cap E_1')$ & $\mathbf{E}_{11}^{(2)}$ & $\mathbf{L}_{11}^{(0)}\backslash (\mathbf{E}_{11}\cup \mathbf{E}_{11}^{(2)}) $ & $(L_{23}\cup L_{23}')\backslash \mathbf{E}_{11}^{(2)} $ \\\hline
$\delta_P(S)$ & $\frac{6}{5}$ & $\frac{9}{7}$ & $\frac{3}{2}$ & $\frac{15}{8}$ & $2$ \\\hline
    \end{tabular}\\
\hspace*{0.6cm}\begin{tabular}{ | c || c | c | c | }
 \hline 
 $P$  &$( L_{2,1}\cup L_{2,1}'\cup L_{3,1}\cup L_{3,1}')\backslash (\mathbf{E}_{11}^{(1)}\cup \mathbf{E}_{11}^{(2)}) $ & $(L_{1,1}\cup L_{1,1}'\cup L_{1,2}\cup L_{1,2}')\backslash \mathbf{E}_{11} $ & o/w \\\hline
$\delta_P(S)$ &  $\ge \frac{27}{17}$ & $\ge \frac{32}{19}$ & $\ge \frac{9}{5}$\\\hline
    \end{tabular}\\
  \hspace*{0.5cm} where $\mathbf{E}_{11}^{(1)}:= E_1\cup E_1'$, $\mathbf{E}_{11}^{(2)}:= E_2\cup E_3$, $\mathbf{L}_{11}^{(0)}:= L_{12}\cup L_{12}'\cup L_{13}\cup L_{13}'$.
\caption{Local $\delta$-invariants: $(-K_S)^2=2$ and  $\DA_22\DA_1$ singularities}
\end{table}
\item[XII.] $\delta(X)=\frac{6}{5}$ since depending on the position of point $P\in S$ we have 
\begin{table}[h!]
\hspace*{0.5cm}
\begin{tabular}{ | c || c |  c | c | c | c | c | }
 \hline 
 $P$ & $E_1\cap E_1'$ & $\mathbf{E}_{12}^{(1)} \backslash (E_1\cap E_1')$ & $\mathbf{E}_{12}^{(2)}\cup L_{234}$ & $\mathbf{L}_{12}^{(0)} \backslash (\mathbf{E}_{12}^{(1)}\cup \mathbf{E}_{12}^{(2)}) $ & $\mathbf{L}_{12}^{(1)}\backslash \mathbf{E}_{12}^{(2)}$  & o/w \\\hline
$\delta_P(S)$ & $\frac{6}{5}$ & $\frac{9}{7}$ & $\frac{3}{2}$ & $\frac{15}{8}$ & $\ge \frac{27}{17}$ &  $\ge \frac{9}{5}$\\\hline
    \end{tabular}\\
 \hspace*{0.5cm}  where $\mathbf{E}_{12}^{(1)}:= E_1\cup E_1'$, $\mathbf{E}_{12}^{(2)}:= E_2\cup E_3\cup E_4$, $\mathbf{L}_{12}^{(0)} := L_{12}\cup L_{12}'\cup L_{13}\cup L_{13}'\cup L_{14}\cup L_{14}'$, \\
 \hspace*{1.7cm} $\mathbf{L}_{12}^{(1)} := \bigcup_{i\in\{2,3,4\},\text{ }j\in\{1,2\}}\big(L_{i,j}\cup L_{i,j}'\big)$.
\caption{Local $\delta$-invariants: $(-K_S)^2=2$ and  $\DA_23\DA_1$ singularities}
\end{table}
\newpage
\item[XIII.] $\delta(X)=\frac{6}{5}$ since depending on the position of point $P\in S$ we have 
\begin{table}[h!]
\hspace*{0.5cm}
\begin{tabular}{ | c || c | c | c | c |  c | }
 \hline 
 $P$ & $\mathbf{E}_{13}^{(0)}$ & $\mathbf{E}_{13}\backslash \mathbf{E}_{13}^{(0)}$ & $(L_{12}\cup L_{12}'')\backslash \mathbf{E}_{13}$ & $\bigcup_{i,j\in\{1,2\}}\big(L_{i,j}\cup L_{i,j}'\big)\backslash \mathbf{E}_{13}$ & o/w \\\hline
$\delta_P(S)$ & $\frac{6}{5}$ & $\frac{9}{7}$ & $\frac{12}{7}$  & $\ge \frac{32}{19}$ & $\ge \frac{9}{5}$\\\hline
    \end{tabular}\\
 \hspace*{0.5cm}  where $\mathbf{E}_{13}^{(0)}:= (E_1\cap E_1')\cup (E_2\cap E_2')$, $\mathbf{E}_{13}:= E_1\cup E_1'\cup E_2\cup E_2'$.
\caption{Local $\delta$-invariants: $(-K_S)^2=2$ and  $2\DA_2$ singularities}
\end{table}
\item[XIV.] $\delta(X)=\frac{6}{5}$ since depending on the position of point $P\in S$ we have 
\begin{table}[h!]
\hspace*{0.5cm}
\begin{tabular}{ | c || c | c | c | c | c | c |  c | }
 \hline 
 $P$ & $\mathbf{E}_{14}^{(0)}$ & $\mathbf{E}_{14}\backslash \mathbf{E}_{14}^{(0)}$ & $E_3$ & $(L_{12}\cup L_{12}'')\backslash \mathbf{E}_{14}$ & $\mathbf{L}_{14}^{(1)}\backslash (\mathbf{E}_{14} \cup E_3) $& $\mathbf{L}_{14}^{(2)} \backslash \mathbf{E}_{14}$ & o/w \\\hline
$\delta_P(S)$ & $\frac{6}{5}$ & $\frac{9}{7}$ & $\frac{3}{2}$ & $\frac{12}{7}$ & $\frac{15}{8}$ & $\ge \frac{32}{19}$ & $\ge \frac{9}{5}$\\\hline
    \end{tabular}\\
  \hspace*{0.5cm} where $\mathbf{E}_{14}^{(0)}:= (E_1\cap E_1')\cup (E_2\cap E_2')$, $\mathbf{E}_{14}:= E_1\cup E_1'\cup E_2\cup E_2'$, \\ 
   \hspace*{1.7cm} $\mathbf{L}_{14}^{(1)}:= L_{13}\cup L_{13}'\cup L_{23}\cup L_{23}'$, $\mathbf{L}_{14}^{(2)}:= L_{1,1}\cup L_{1,1}'\cup L_{2,1}\cup L_{2,1}' $.
\caption{Local $\delta$-invariants: $(-K_S)^2=2$ and  $2\DA_2\DA_1$ singularities}
\end{table}

\item[XV.] $\delta(X)=\frac{6}{5}$ since depending on the position of point $P\in S$ we have 
\begin{table}[h!]
\hspace*{0.5cm}
\begin{tabular}{ | c ||  c | c | c |  c | }
 \hline 
 $P$ & $\mathbf{E}_{15}^{(0)}$ & $\mathbf{E}_{15}\backslash \mathbf{E}_{15}^{(0)}$  & $\bigcup_{i,j\in\{1,2,3\}\text{, }i<j}\big(L_{ij}\cup L_{ij}'\big) \backslash \mathbf{E}_{15}$ & o/w \\\hline
$\delta_P(S)$ & $\frac{6}{5}$ & $\frac{9}{7}$ & $\frac{12}{7}$   & $\ge \frac{9}{5}$\\\hline
    \end{tabular}\\
  \hspace*{0.5cm} where $\mathbf{E}_{15}^{(0)}:= (E_1\cap E_1')\cup (E_2\cap E_2')\cup (E_3\cap E_3')$, $\mathbf{E}_{15}:= E_1\cup E_1'\cup E_2\cup E_2'\cup E_3\cup E_3'$.
\caption{Local $\delta$-invariants: $(-K_S)^2=2$ and  $3\DA_2$ singularities}
\end{table}
\item[XVI.] $\delta(X)=1$ since depending on the position of point $P\in S$ we have 
\begin{table}[h!]
\hspace*{0.5cm}
\begin{tabular}{ | c || c | c | c | c |  c | }
 \hline 
 $P$ & $E_2$ & $(E_1\cup E_1') \backslash E_2$ & $(L_{2,1}\cup L_{2,2}) \backslash E_2 $ & $\bigcup_{i\in\{1,2,3,4\}}\big(L_{1,i}\cup L_{1,i}'\big)\backslash (E_1\cup E_1') $ & o/w \\\hline
$\delta_P(S)$ & $1$ & $\frac{6}{5}$ & $2$  & $\ge \frac{75}{43}$ & $\ge \frac{9}{5}$\\\hline
    \end{tabular}
\caption{Local $\delta$-invariants: $(-K_S)^2=2$ and  $\DA_3$ singularity}
\end{table}
\item[XVII.] $\delta(X)=1$ since depending on the position of point $P\in S$ we have 
\begin{table}[h!]
\hspace*{0.5cm}
\begin{tabular}{ | c || c | c |c |  c | c | }
 \hline 
 $P$ & $E_2$ & $(E_1\cup E_1') \backslash E_2$ & $E_3$ & $(L_{13}\cup L_{13}')\backslash (E_1\cup E_1'\cup E_3)$ & $(L_{2,1}\cup L_{2,2}) \backslash E_2 $  \\\hline
$\delta_P(S)$ & $1$ & $\frac{6}{5}$ & $\frac{3}{2}$ & $\frac{9}{5}$ & $\frac{15}{8}$ \\\hline
    \end{tabular}\\
\hspace*{0.6cm}\begin{tabular}{ | c || c | c | c | }
 \hline 
 $P$  &$(L_{3,1}\cup L_{3,1}'\cup L_{3,2}\cup L_{3,2}')\backslash E_3$ & $(L_{1,1}\cup L_{1,1}'\cup L_{1,2}\cup L_{1,2}')\backslash (E_1\cup E_1') $ & o/w \\\hline
$\delta_P(S)$ &  $\ge \frac{27}{17}$ & $\ge \frac{75}{43}$ & $\ge \frac{9}{5}$\\\hline
    \end{tabular}
\caption{Local $\delta$-invariants: $(-K_S)^2=2$ and  $\DA_3\DA_1$ singularities (16 lines)}
\end{table}
\newpage
\item[XVIII.] $\delta(X)=1$ since depending on the position of point $P\in S$ we have 
\begin{table}[h!]
\hspace*{0.5cm}
\begin{tabular}{ | c || c | c |c |  c | c | c | }
 \hline 
 $P$ & $E_2$ & $(E_1\cup E_1') \backslash E_2$ & $E_3\cup L_{23}$  &$\mathbf{L}_{18}^{(1)}\backslash E_3$ & $\mathbf{L}_{18}^{(2)}\backslash (E_1\cup E_1') $ & o/w  \\\hline
$\delta_P(S)$ & $1$ & $\frac{6}{5}$ & $\frac{3}{2}$ & $\ge \frac{27}{17}$ & $\ge \frac{75}{43}$ & $\ge \frac{9}{5}$\\\hline
    \end{tabular}\\
 \hspace*{0.5cm}  where $\mathbf{L}_{18}^{(1)} := \bigcup_{i\in\{1,2,3\}}\big(L_{3,i}\cup L_{3,i}'\big)$, $\mathbf{L}_{18}^{(2)} := \bigcup_{j\in\{1,2,3,4\}}\big(L_{1,j}\cup L_{1,j}'\big)$.
\caption{Local $\delta$-invariants: $(-K_S)^2=2$ and  $\DA_3\DA_1$ singularities (15 lines)}
\end{table}
\item[XIX.] $\delta(X)=1$ since depending on the position of point $P\in S$ we have 
\begin{table}[h!]
\hspace*{0.5cm}
\begin{tabular}{ | c || c | c |c |  c | c | c | }
 \hline 
 $P$ & $E_2$ & $\mathbf{E}_{19}\backslash E_2$ & $E_3\cup E_4$  &$\mathbf{L}_{19}^{(1)}\backslash (\mathbf{E}_{19}\cup E_3\cup E_4)$ & $\mathbf{L}_{19}^{(2)}\backslash (E_2\cup E_3\cup E_4) $ & o/w  \\\hline
$\delta_P(S)$ & $1$ & $\frac{6}{5}$ & $\frac{3}{2}$ & $\frac{9}{5}$ & $2$ & $\ge \frac{9}{5}$\\\hline
    \end{tabular}\\
  \hspace*{0.5cm} where $\mathbf{E}_{19} := E_1\cup E_1'$, $\mathbf{L}_{19}^{(1)} := L_{13}\cup L_{13}'\cup L_{14}\cup L_{14}'$, $\mathbf{L}_{19}^{(2)} := L_{34}\cup L_{34}'\cup L_{2,1}\cup L_{2,2}$.
\caption{Local $\delta$-invariants: $(-K_S)^2=2$ and  $\DA_32\DA_1$ singularities (12 lines)}
\end{table}

\item[XX.] $\delta(X)=1$ since depending on the position of point $P\in S$ we have 
\begin{table}[h!]
\hspace*{0.5cm}
\begin{tabular}{ | c || c | c | c |  c |  }
 \hline 
 $P$ & $E_2$ & $(E_1\cup E_1') \backslash E_2$ & $(E_3\cup E_4\cup L_{23})\backslash E_2$ & $(L_{14}\cup L_{14}')\backslash (E_1\cup E_1'\cup E_4)$  \\\hline
$\delta_P(S)$ & $1$ & $\frac{6}{5}$ & $\frac{3}{2}$ & $\frac{9}{5}$ \\\hline
    \end{tabular}\\
\hspace*{0.6cm}\begin{tabular}{ | c || c | c | c |  c | }
 \hline 
 $P$ & $(L_{34}\cup L_{34}')\backslash (E_3\cup E_4)$  & $\mathbf{L}_{20}^{(1)}\backslash E_3$ & $\mathbf{L}_{20}^{(2)}\backslash (E_1\cup E_1') $ & o/w \\\hline
$\delta_P(S)$ & $2$ &  $\ge \frac{27}{17}$ & $\ge \frac{75}{43}$ & $\ge \frac{9}{5}$\\\hline
    \end{tabular}\\
 \hspace*{0.5cm}  where $\mathbf{L}_{20}^{(1)} := L_{3,1}\cup L_{3,1}$, $\mathbf{L}_{20}^{(2)} := \bigcup_{k\in \{1,2\}} \big(L_{1,k}\cup L_{1,k}'\big)$.
\caption{Local $\delta$-invariants: $(-K_S)^2=2$ and  $\DA_32\DA_1$ singularities (11 lines)}
\end{table}

\item[XXI.] $\delta(X)=1$ since depending on the position of point $P\in S$ we have 
\begin{table}[h!]
\hspace*{0.5cm}
\begin{tabular}{ | c || c | c |c |  c | c | c | }
 \hline 
 $P$ & $E_2$ & $\mathbf{E}_{21}^{(1)}\backslash E_2$ & $(\mathbf{E}_{21}^{(2)} \cup E_5\cup L_{345}\cup L_{25})\backslash E_2$  & $\mathbf{L}_{21}^{(1)}\backslash (\mathbf{E}_{21}^{(1)}\cup \mathbf{E}_{21}^{(2)})$ & $\mathbf{L}_{21}^{(2)}\backslash E_5 $ & o/w  \\\hline
$\delta_P(S)$ & $1$ & $\frac{6}{5}$ & $\frac{3}{2}$ & $\frac{9}{5}$ & $\ge\frac{27}{17}$ & $\ge \frac{9}{5}$\\\hline
    \end{tabular}\\
  \hspace*{0.5cm} where $\mathbf{E}_{21}^{(1)}:=E_1\cup E_1'$, $\mathbf{E}_{21}^{(2)}:=E_3\cup E_4$,  $\mathbf{L}_{21}^{(1)} := L_{13}\cup L_{13}'\cup L_{14}\cup L_{14}'$, $\mathbf{L}_{21}^{(2)} := L_{5,1}\cup L_{5,1}'$.
\caption{Local $\delta$-invariants: $(-K_S)^2=2$ and  $\DA_33\DA_1$ singularities}
\end{table}
\item[XXII.] $\delta(X)=1$ since depending on the position of point $P\in S$ we have 
\begin{table}[h!]
\hspace*{0.5cm}
\begin{tabular}{ | c || c | c |c |  c | c | }
 \hline 
 $P$ & $E_2$ & $(\mathbf{E}_{22}^{(0)}\cup E_1\cup E_1') \backslash E_2$ & $ (E_3\cup E_3')\backslash \mathbf{E}_{22}^{(0)}$ & $(L_{13}\cup L_{13}'')\backslash \mathbf{E}_{22}$ & $(L_{2,1}\cup L_{2,2}) \backslash E_2$  \\\hline
$\delta_P(S)$ & $1$ & $\frac{6}{5}$ & $\frac{9}{7}$ & $\frac{18}{11}$ & $2$ \\\hline
    \end{tabular}\par
\hspace*{0.6cm}\begin{tabular}{ | c || c | c | c | }
 \hline 
 $P$  & $( L_{3,1}\cup L_{3,1}'\cup L_{3,2}\cup L_{3,2}')\backslash \mathbf{E}_{22}$ & $(L_{1,1}\cup L_{1,1}')\backslash (E_1\cup E_1') $ & o/w \\\hline
$\delta_P(S)$ &  $\ge \frac{27}{17}$ & $\ge \frac{75}{43}$ & $\ge \frac{9}{5}$\\\hline
    \end{tabular}\\
  \hspace*{0.5cm} where $\mathbf{E}_{22}^{(0)}:= E_3\cap E_3'$, $\mathbf{E}_{22}:= E_1\cup E_1'\cup E_3\cup E_3'$.
\caption{Local $\delta$-invariants: $(-K_S)^2=2$ and  $\DA_3\DA_2$ singularities}
\end{table}
\newpage
\item[XXIII.] $\delta(X)=1$ since depending on the position of point $P\in S$ we have 
\begin{table}[h!]
\hspace*{0.5cm}
\begin{tabular}{ | c || c | c |c |  c |  c | }
 \hline 
 $P$ & $E_2$ & $(\mathbf{E}_{23}^{(0)}\cup E_1\cup E_1') \backslash E_2$ & $ (E_3\cup E_3')\backslash \mathbf{E}_{23}^{(0)}$ & $(L_{13}\cup L_{13}'')\backslash \mathbf{E}_{23}$ & $(E_4\cup L_{24}) \backslash E_2$  \\\hline
$\delta_P(S)$ & $1$ & $\frac{6}{5}$ & $\frac{9}{7}$ & $\frac{18}{11}$ & $\frac{3}{2}$ \\\hline
    \end{tabular}\\
\hspace*{0.6cm}\begin{tabular}{ | c || c | c | c | c |}
 \hline 
 $P$ & $ (L_{34}\cup L_{34}')\backslash (E_3\cup E_3'\cup E_4)$ & $\mathbf{L}_{23}\backslash \mathbf{E}_{23}$ & $(L_{1,1}\cup L_{1,1}')\backslash (E_1\cup E_1') $ & o/w \\\hline
$\delta_P(S)$ & $\frac{15}{8}$ &  $\ge \frac{27}{17}$ & $\ge \frac{75}{43}$ & $\ge \frac{9}{5}$\\\hline
    \end{tabular}\\
 \hspace*{0.5cm}  where $\mathbf{E}_{23}^{(0)}:= E_3\cap E_3'$, $\mathbf{E}_{23}^{(1)}:= E_1\cup E_1'\cup E_3\cup E_3'$,   $\mathbf{L}_{23} := L_{3,1}\cup L_{3,1}'\cup L_{3,2}\cup L_{3,2}'$.
\caption{Local $\delta$-invariants: $(-K_S)^2=2$ and  $\DA_3\DA_2\DA_1$ singularities}
\end{table}
\item[XXIV.] $\delta(X)=1$ since depending on the position of point $P\in S$ we have 
\begin{table}[h!]
\hspace*{0.5cm}
\begin{tabular}{ | c || c | c |c |  c | c | }
 \hline 
 $P$ & $E_2\cup E_4$ & $\mathbf{E}_{24} \backslash (E_2\cup E_4)$ & $(L_{13}\cup L_{13}'')\backslash \mathbf{E}_{24}$  & $\mathbf{L}_{24}\backslash (E_2\cup E_4)$ & o/w  \\\hline
$\delta_P(S)$ & $1$ & $\frac{6}{5}$ & $\frac{3}{2}$ & $2$ & $\ge \frac{9}{5}$\\\hline
    \end{tabular}\\
 \hspace*{0.5cm}  where $\mathbf{E}_{24}:=E_1\cup E_1'\cup E_3\cup E_3'$,  $\mathbf{L}_{24} :=L_{2,1}\cup L_{2,2}\cup L_{4,1}\cup L_{4,2}$.
\caption{Local $\delta$-invariants: $(-K_S)^2=2$ and  $2\DA_3$ singularities}
\end{table}

\item[XXV.] $\delta(X)=1$ since depending on the position of point $P\in S$ we have 
\begin{table}[h!]
\hspace*{0.5cm}
\begin{tabular}{ | c || c | c |c |  c |}
 \hline 
 $P$ & $E_2\cup E_4$ & $\mathbf{E}_{25} \backslash (E_2\cup E_4)$ & $(L_{13}\cup L_{13}''\cup E_5\cup L_{25}\cup L_{45})\backslash \mathbf{E}_{25}$   & o/w  \\\hline
$\delta_P(S)$ & $1$ & $\frac{6}{5}$ & $\frac{3}{2}$  & $\ge \frac{9}{5}$\\\hline
    \end{tabular}\\
  \hspace*{0.5cm} where $\mathbf{E}_{25}:=E_1\cup E_1'\cup E_3\cup E_3'$.
\caption{Local $\delta$-invariants: $(-K_S)^2=2$ and  $2\DA_3\DA_1$ singularities}
\end{table}
\item[XXVI.] $\delta(X)=\frac{12}{13}$ since depending on the position of point $P\in S$ we have 
\begin{table}[h!]
\hspace*{0.5cm}
\begin{tabular}{ | c || c | c |c |  c | c | }
 \hline 
 $P$ & $ E_2\cup E_2'$ & $( E_1\cup E_1') \backslash ( E_2\cup E_2')$ & $(L_{2,1}\cup L_{2,1}')\backslash ( E_2\cup E_2')$  & $\bigcup_{i\in\{1,2,3\}} \big(L_{1,i}\cup L_{1,i}'\big)$ & o/w  \\\hline
$\delta_P(S)$ & $\frac{12}{13}$ & $\frac{36}{31}$ & $\frac{24}{13}$ & $\ge\frac{216}{121}$  & $\ge \frac{9}{5}$\\\hline
    \end{tabular}
\caption{Local $\delta$-invariants: $(-K_S)^2=2$ and  $\DA_4$ singularity}
\end{table}
\item[XXVII.] $\delta(X)=\frac{12}{13}$ since depending on the position of point $P\in S$ we have 
\begin{table}[h!]
\hspace*{0.5cm}
\begin{tabular}{ | c || c | c |c |  c | c | }
 \hline 
 $P$ & $ E_2\cup E_2'$ & $( E_1\cup E_1') \backslash ( E_2\cup E_2')$ & $E_3$ &$(L_{2,1}\cup L_{2,1}')\backslash ( E_2\cup E_2')$   \\\hline
$\delta_P(S)$ & $\frac{12}{13}$ & $\frac{36}{31}$ & $\frac{3}{2}$ & $\frac{24}{13}$  \\\hline
    \end{tabular}
\\
\hspace*{0.6cm}\begin{tabular}{ | c || c | c |c |c |  }
 \hline 
 $P$ & $(L_{13}\cup L_{13}') \backslash (E_1\cup E_1'\cup E_3)$ 
 & $ (L_{3,1}\cup L_{3,1}')\backslash E_3$ & $(L_{1,1}\cup L_{1,1}')\backslash (E_1\cup E_1')$ & o/w   \\\hline
$\delta_P(S)$ & $\frac{72}{41}$ & $\ge\frac{27}{17}$ & $\ge\frac{216}{21}$ & $\ge\frac{9}{5}$ \\\hline
    \end{tabular}
\caption{Local $\delta$-invariants: $(-K_S)^2=2$ and  $\DA_4\DA_1$ singularities}
\end{table}
\newpage
\item[XXVIII.] $\delta(X)=\frac{12}{13}$ since depending on the position of point $P\in S$ we have 
\begin{table}[h!]
\hspace*{0.5cm}
\begin{tabular}{ | c || c | c |c |  c |  }
 \hline 
 $P$ & $ E_2\cup E_2'$ & $( E_1\cup E_1') \backslash ( E_2\cup E_2')$ & $E_3\cap E_3'$ & $ (E_3\cup E_3)\backslash(E_3\cap E_3) $   \\\hline
$\delta_P(S)$ & $\frac{12}{13}$ & $\frac{36}{31}$ & $\frac{6}{5}$ & $\frac{9}{7}$ \\\hline
    \end{tabular}
\\
\hspace*{0.6cm}\begin{tabular}{ | c || c | c |c |c |c | c |  }
 \hline 
 $P$  & $(L_{13}\cup L_{13}') \backslash (E_1\cup E_1'\cup E_3\cup E_3')$ 
 &$(L_{2,1}\cup L_{2,1}')\backslash ( E_2\cup E_2')$   & $ (L_{3,1}\cup L_{3,1}')\backslash E_3$  & o/w   \\\hline
$\delta_P(S)$ & $\frac{36}{23}$ & $\frac{24}{13}$ & $\ge\frac{32}{19}$  & $\ge\frac{9}{5}$ \\\hline
    \end{tabular}
\caption{Local $\delta$-invariants: $(-K_S)^2=2$ and  $\DA_4\DA_2$ singularities}
\end{table}
\item[XXIX.] $\delta(X)=\frac{6}{7}$ since depending on the position of point $P\in S$ we have 
\begin{table}[h!]
\hspace*{0.5cm}
\begin{tabular}{ | c || c | c |c |  c | c | }
 \hline 
 $P$ & $ E_2\cup E_2'\cup E_3$ & $\mathbf{E}_{29}^{(1)}\backslash \mathbf{E}_{29}^{(2)}$ & $(L_{2,1}\cup L_{2,1}')\backslash \mathbf{E}_{29}^{(2)}$  & $\bigcup_{i\in\{1,2\}}\big(L_{1,i}\cup L_{1,i}'\big)\backslash \mathbf{E}_{29}^{(1)}$ & o/w  \\\hline
$\delta_P(S)$ & $\frac{6}{7}$ & $\frac{8}{7}$ & $\frac{12}{7}$ & $\ge\frac{49}{27}$  & $\ge \frac{9}{5}$\\\hline
    \end{tabular}
    \\\hspace*{0.5cm} where $\mathbf{E}_{29}^{(1)}:=E_1\cup E_1'$, $\mathbf{E}_{29}^{(2)}:=E_2\cup E_2'$.
\caption{Local $\delta$-invariants: $(-K_S)^2=2$ and  $\DA_5$ singularity (8 lines)}
\end{table}

\item[XXX.] $\delta(X)=\frac{3}{4}$ since depending on the position of point $P\in S$ we have 
\begin{table}[h!]
\hspace*{0.5cm}
\begin{tabular}{ | c || c | c |c |  c | c | c | }
 \hline 
 $P$ & $ E_3$ & $\mathbf{E}_{30}^{(2)}\backslash E_3$ & $\mathbf{E}_{30}^{(1)}\backslash \mathbf{E}_{30}^{(2)}$ & $ L_{3,1} \backslash E_3$ & $\bigcup_{i\in\{1,2,3\}}\big(L_{1,i}\cup L_{1,i}'\big)\backslash \mathbf{E}_{30}^{(1)}$ & o/w  \\\hline
$\delta_P(S)$ & $\frac{3}{4}$ & $\frac{9}{10}$ & $\frac{9}{8}$ & $\frac{3}{2}$ & $\ge\frac{49}{27}$  & $\ge \frac{9}{5}$\\\hline
    \end{tabular}
    \\\hspace*{0.5cm} where $\mathbf{E}_{30}^{(1)}:=E_1\cup E_1'$, $\mathbf{E}_{30}^{(2)}:=E_2\cup E_2'$.
\caption{Local $\delta$-invariants: $(-K_S)^2=2$ and  $\DA_5$ singularity (7 lines)}
\end{table}
\item[XXXI.] $\delta(X)=\frac{6}{7}$ since depending on the position of point $P\in S$ we have 
\begin{table}[h!]
\hspace*{0.5cm}
\begin{tabular}{ | c || c | c |c |  c | c | }
 \hline 
 $P$ & $ E_2\cup E_2'\cup E_3$ & $\mathbf{E}_{31}^{(1)}\backslash \mathbf{E}_{31}^{(2)}$ & $E_4$ & $(L_{2,1}\cup L_{2,1}'\cup L_{14}\cup L_{14}')\backslash (\mathbf{E}_{31}^{(1)}\cup\mathbf{E}_{31}^{(2)}\cup E_4)$   & o/w  \\\hline
$\delta_P(S)$ & $\frac{6}{7}$ & $\frac{8}{7}$ & $\frac{3}{2}$ & $\frac{12}{7}$ & $\ge \frac{9}{5}$\\\hline
    \end{tabular}
    \\\hspace*{0.5cm} where $\mathbf{E}_{31}^{(1)}:=E_1\cup E_1'$, $\mathbf{E}_{31}^{(2)}:=E_2\cup E_2'$.
\caption{Local $\delta$-invariants: $(-K_S)^2=2$ and  $\DA_5\DA_1$ singularities (6 lines)}
\end{table}
\item[XXXII.] $\delta(X)=\frac{3}{4}$ since depending on the position of point $P\in S$ we have 
\begin{table}[h!]
\hspace*{0.5cm}
\begin{tabular}{ | c || c | c |c |c |  c | c | c | }
 \hline 
 $P$ & $ E_3$ & $\mathbf{E}_{32}^{(2)}\backslash E_3$ & $\mathbf{E}_{32}^{(1)}\backslash \mathbf{E}_{32}^{(2)}$ & $ (E_4\cup L_{3,1}) \backslash E_3$ & $\mathbf{L}_{32}^{(1)} \backslash (\mathbf{E}_{32}^{(1)}\cup E_4) $& $\mathbf{L}_{32}^{(2)} \backslash \mathbf{E}_{32}^{(1)}$ & o/w  \\\hline
$\delta_P(S)$ & $\frac{3}{4}$ & $\frac{9}{10}$ & $\frac{9}{8}$ & $\frac{3}{2}$ & $\frac{45}{26}$ & $\ge\frac{49}{27}$  & $\ge \frac{9}{5}$\\\hline
    \end{tabular}
    \\\hspace*{0.5cm} where $\mathbf{E}_{32}^{(1)}:=E_1\cup E_1'$, $\mathbf{E}_{32}^{(2)}:=E_2\cup E_2'$,
    $\mathbf{L}_{32}^{(1)} := L_{14}\cup L_{14}'$, $\mathbf{L}_{32}^{(2)} := L_{1,1}\cup L_{1,1}$.
\caption{Local $\delta$-invariants: $(-K_S)^2=2$ and  $\DA_5\DA_1$ singularities (5 lines)}
\end{table}
\newpage
\item[XXXIII.] $\delta(X)=\frac{3}{4}$ since depending on the position of point $P\in S$ we have 
\begin{table}[h!]
\hspace*{0.5cm}
\begin{tabular}{ | c || c | c |c |c |  c | c | c | }
 \hline 
 $P$ & $ E_3$ & $\mathbf{E}_{33}^{(2)}\backslash E_3$ & $\mathbf{E}_{33}^{(1)}\backslash \mathbf{E}_{33}^{(2)}$ & $\mathbf{E}_{33}^{(0)}$ & $\mathbf{E}_{33}^{(4)}\backslash \mathbf{E}_{33}^{(0)} $& $\mathbf{L}_{33}\backslash (\mathbf{E}_{33}^{(1)}\cup E_3\cup \mathbf{E}_{33}^{(4)}) $ & o/w  \\\hline
$\delta_P(S)$ & $\frac{3}{4}$ & $\frac{9}{10}$ & $\frac{9}{8}$ & $\frac{6}{5}$ & $\frac{9}{7}$ & $\frac{3}{2}$  & $\ge \frac{9}{5}$\\\hline
    \end{tabular}
    \\\hspace*{0.5cm} where $\mathbf{E}_{33}^{(0)}:=E_4\cap E_4'$, $\mathbf{E}_{33}^{(i)}:=E_i\cup E_i'$ ($i\in\{1,2,4\}$), $\mathbf{L}_{33} := L_{3,1}\cup L_{14}\cup L_{14}'$.
\caption{Local $\delta$-invariants: $(-K_S)^2=2$ and  $\DA_5\DA_2$ singularities}
\end{table}
\item[XXXIV.] $\delta(X)=\frac{4}{5}$ since depending on the position of point $P\in S$ we have 
\begin{table}[h!]
\hspace*{0.5cm}
\begin{tabular}{ | c || c | c |c |  c | c | }
 \hline 
 $P$ & $\mathbf{E}_{34}^{(2)}\cup \mathbf{E}_{34}^{(3)}$ & $\mathbf{E}_{34}^{(1)}\backslash \mathbf{E}_{34}^{(2)}$ & $(L_{2,1}\cup L_{2,1}')\backslash \mathbf{E}_{34}^{(2)}$ &  $(L_{1,1}\cup L_{1,1}')\backslash \mathbf{E}_{34}^{(1)}$ & o/w  \\\hline
$\delta_P(S)$ & $\frac{4}{5}$ & $\frac{60}{53}$ & $\frac{60}{37}$ & $\ge\frac{384}{209}$   & $\ge \frac{9}{5}$\\\hline
    \end{tabular}
    \\\hspace*{0.5cm} where $\mathbf{E}_{34}^{(1)}:=E_1\cup E_1'$, $\mathbf{E}_{34}^{(2)}:=E_2\cup E_2'$, $\mathbf{E}_{34}^{(3)}:=E_3\cup E_3'$.
\caption{Local $\delta$-invariants: $(-K_S)^2=2$ and  $\DA_6$ singularity}
\end{table}

\item[XXXV.] $\delta(X)=\frac{3}{4}$ since depending on the position of point $P\in S$ we have 
\begin{table}[h!]
\hspace*{0.5cm}
\begin{tabular}{ | c || c | c |c |  c |  }
 \hline 
 $P$ & $\mathbf{E}_{35}^{(2)}\cup \mathbf{E}_{35}^{(3)}\cup E_4$ & $\mathbf{E}_{35}^{(1)}\backslash \mathbf{E}_{35}^{(2)}$ & $(L_{2,1}\cup L_{2,1}')\backslash \mathbf{E}_{35}^{(2)}$ & o/w  \\\hline
$\delta_P(S)$ & $\frac{3}{4}$ & $\frac{9}{8}$ & $\frac{3}{2}$    & $\ge \frac{9}{5}$\\\hline
    \end{tabular}
    \\\hspace*{0.5cm} where $\mathbf{E}_{35}^{(1)}:=E_1\cup E_1'$, $\mathbf{E}_{35}^{(2)}:=E_2\cup E_2'$, $\mathbf{E}_{35}^{(3)}:=E_3\cup E_3'$.
\caption{Local $\delta$-invariants: $(-K_S)^2=2$ and  $\DA_7$ singularity}
\end{table}
\item[XXXVI.] $\delta(X)=\frac{3}{4}$ since depending on the position of point $P\in S$ we have 
\begin{table}[h!]
\hspace*{0.5cm}
\begin{tabular}{ | c || c | c |c |  c |  }
 \hline 
 $P$ & $E$ & $ (E_1\cup E_2\cup E_3) \backslash E$ & $\bigcup_{i\in\{1,2,3\},j\in\{1,2\}} L_{i,j} \backslash (E_1\cup E_2\cup E_3)$ & o/w  \\\hline
$\delta_P(S)$ & $\frac{3}{4}$ & $1$ & $2$    & $\ge \frac{9}{5}$\\\hline
    \end{tabular}
\caption{Local $\delta$-invariants: $(-K_S)^2=2$ and  $\mathbb{D}_4$ singularity}
\end{table}
\item[XXXVII.] $\delta(X)=\frac{3}{4}$ since depending on the position of point $P\in S$ we have 
\begin{table}[h!]
\hspace*{0.5cm}
\begin{tabular}{ | c || c | c |c |  c |c |  c |   }
 \hline 
 $P$ & $E$ & $ (\mathbf{E}_{37}^{(1)}\cup E_2) \backslash E$ & $(E_3\cup L_{23})\backslash E_2 $ & $\mathbf{L}_{37}^{(1)} \backslash \mathbf{E}_{37}^{(1)}$ & $ \mathbf{L}_{37}^{(2)}\backslash E_3$ & o/w  \\\hline
$\delta_P(S)$ & $\frac{3}{4}$ & $1$ & $\frac{3}{2}$ & $2$  & $\ge \frac{27}{17}$  & $\ge \frac{9}{5}$\\\hline
    \end{tabular}
    \\\hspace*{0.5cm}where $\mathbf{E}_{37}^{(1)}:=E_1\cup E_1'$, $\mathbf{L}_{37}^{(1)} := \bigcup_{i\in\{1,2\}} \big(L_{1,i}\cup L_{1,i}'\big)$, $\mathbf{L}_{37}^{(2)} := L_{3,1}\cup L_{3,1}'\cup L_{3,2}\cup L_{3,2}'$.
\caption{Local $\delta$-invariants: $(-K_S)^2=2$ and  $\mathbb{D}_4\DA_1$ singularities}
\end{table}
\newpage
\item[XXXVIII.] $\delta(X)=\frac{3}{4}$ since depending on the position of point $P\in S$ we have 
\begin{table}[h!]
\hspace*{0.5cm}
\begin{tabular}{ | c || c | c |c | c |  c |   }
 \hline 
 $P$ & $E$ & $ (\mathbf{E}_{38}^{(1)}\cup E_2) \backslash E$ & $(\mathbf{E}_{38}^{(3)}\cup L_{13}\cup L_{13}')\backslash \mathbf{E}_{38}^{(1)} $ & $ \mathbf{L}_{38}\backslash (E_2\cup \mathbf{E}_{38}^{(3)})$ & o/w  \\\hline
$\delta_P(S)$ & $\frac{3}{4}$ & $1$ & $\frac{3}{2}$ & $2$    & $\ge \frac{9}{5}$\\\hline
    \end{tabular}
    \\\hspace*{0.5cm} where $\mathbf{E}_{38}^{(1)}:=E_1\cup E_1'$, $\mathbf{E}_{38}^{(3)}:=E_3\cup E_3'$, $\mathbf{L}_{38} := L_2\cup L_2'\cup L_3\cup L_3'$.
\caption{Local $\delta$-invariants: $(-K_S)^2=2$ and  $\mathbb{D}_42\DA_1$ singularities}
\end{table}
\item[XXXIX.] $\delta(X)=\frac{3}{4}$ since depending on the position of point $P\in S$ we have 
\begin{table}[h!]
\hspace*{0.5cm}
\begin{tabular}{ | c || c | c |c | c |  c |   }
 \hline 
 $P$ & $E$ & $ \mathbf{E}_{39} \backslash E$ & $(\mathbf{E}_{39}'\cup L_{1}\cup L_2\cup L_3\cup L_{123})\backslash \mathbf{E}_{39} $  & o/w  \\\hline
$\delta_P(S)$ & $\frac{3}{4}$ & $1$ & $\frac{3}{2}$    & $\ge \frac{9}{5}$\\\hline
    \end{tabular}
    \\\hspace*{0.5cm} where $\mathbf{E}_{39}:=E_1\cup E_2\cup E_3$, $\mathbf{E}_{39}':=E_1'\cup E_2'\cup E_3'$.
\caption{Local $\delta$-invariants: $(-K_S)^2=2$ and  $\mathbb{D}_43\DA_1$ singularities}
\end{table}
\item[XL.] $\delta(X)=\frac{3}{5}$ since depending on the position of point $P\in S$ we have 
\begin{table}[h!]
\hspace*{0.5cm}
\begin{tabular}{ | c || c | c |c | c |  c | c | c |  c |   }
 \hline 
 $P$ & $E$ & $ E_2 \backslash E$ & $\mathbf{E}_{40} \backslash E$ & $  E_3 \backslash E_2$  & $(L_1\cup L_1') \backslash \mathbf{E}_{40}$ & $(L_{3,1}\cup L_{3,2})\backslash E_3$ & o/w  \\\hline
$\delta_P(S)$ & $\frac{3}{5}$ & $\frac{3}{4}$ & $\frac{9}{10}$& $1$ & $\frac{9}{5}$ & $2$    & $\ge \frac{9}{5}$\\\hline
    \end{tabular}
    \\\hspace*{0.5cm} where $\mathbf{E}_{40}:=E_1\cup E_1'$. 
\caption{Local $\delta$-invariants: $(-K_S)^2=2$ and  $\mathbb{D}_5$ singularity}
\end{table}

\item[XLI.] $\delta(X)=\frac{3}{5}$ since depending on the position of point $P\in S$ we have 
\begin{table}[h!]
\hspace*{0.5cm}
\begin{tabular}{ | c || c | c |c | c |  c | c | c |  c | c |   }
 \hline 
 $P$ & $E$ & $ E_2 \backslash E$ & $\mathbf{E}_{41} \backslash E$ & $E_3 \backslash E_2$ & $(E_4\cup L_{34})\backslash E_3$ & $\mathbf{L}_{41} \backslash \mathbf{E}_{41}$ & $(L_{4,1}\cup L_{4,1}')\backslash E_4$ & o/w  \\\hline
$\delta_P(S)$ & $\frac{3}{5}$ & $\frac{3}{4}$ & $\frac{9}{10}$ & $1$ & $\frac{3}{2}$  & $\frac{9}{5}$ & $\ge \frac{27}{17}$    & $\ge \frac{9}{5}$\\\hline
    \end{tabular}
    \\\hspace*{0.5cm} where $\mathbf{E}_{41}:=E_1\cup E_1'$, $\mathbf{L}_{41}:=L_1\cup L_1'$. 
\caption{Local $\delta$-invariants: $(-K_S)^2=2$ and  $\mathbb{D}_5\DA_1$ singularities}
\end{table}
\item[XLII.] $\delta(X)=\frac{1}{2}$ since depending on the position of point $P\in S$ we have 
\begin{table}[h!]
\hspace*{0.5cm}
\begin{tabular}{ | c || c | c |c | c |  c | c | c |c | c |  c |   }
 \hline 
 $P$ & $E_2$ & $ E_3 \backslash E_2$ & $ (E\cup E_4) \backslash (E_2\cup E_3) $ & $ E_1 \backslash E_2$    & $ E_4 \backslash E_5$ & $ L \backslash E$ & $ (L_5\cup L_5') \backslash E_5$ & o/w  \\\hline
$\delta_P(S)$ & $\frac{1}{2}$ & $\frac{3}{5}$ & $\frac{3}{4}$& $\frac{6}{7}$ & $1$ & $\frac{3}{2}$ & $2$   & $\ge \frac{9}{5}$\\\hline
    \end{tabular}
\caption{Local $\delta$-invariants: $(-K_S)^2=2$ and  $\mathbb{D}_6$ singularity}
\end{table}
\item[XLIII.] $\delta(X)=\frac{1}{2}$ since depending on the position of point $P\in S$ we have 
\begin{table}[h!]
\hspace*{0.5cm}
\begin{tabular}{ | c || c | c |c |   c | c | c |c | c |  c |   }
 \hline 
 $P$ & $E_2$ & $ E_3 \backslash E_2$ & $ (E\cup E_4) \backslash (E_2\cup E_3) $ & $ E_1 \backslash E_2$    & $ E_4 \backslash E_5$ & $ (L\cup E_6\cup L_{56}) \backslash E_5$  & o/w  \\\hline
$\delta_P(S)$ & $\frac{1}{2}$ & $\frac{3}{5}$ & $\frac{3}{4}$& $\frac{6}{7}$ & $1$ & $\frac{3}{2}$  & $\ge \frac{9}{5}$\\\hline
    \end{tabular}
\caption{Local $\delta$-invariants: $(-K_S)^2=2$ and  $\mathbb{D}_6\DA_1$ singularities}
\end{table}
\newpage
\item[XLIV.] $\delta(X)=\frac{3}{7}$ since depending on the position of point $P\in S$ we have 
\begin{table}[h!]
\hspace*{0.5cm}
\begin{tabular}{ | c || c | c |c |   c | c | c |    }
 \hline 
 $P$ & $E_3$ & $ (E_2\cup E_2')\backslash E_3$ & $E \backslash E_3$ & $(E_1\cup E_1')\backslash (E_2\cup E_2')$ & $(L_1\cup L_1')\backslash (E_1\cup E_1')$ & o/w  \\\hline
$\delta_P(S)$ & $\frac{3}{7}$ & $\frac{4}{7}$ & $\frac{3}{4}$ & $\frac{6}{7}$  & $\frac{12}{7}$     & $\ge \frac{9}{5}$\\\hline
    \end{tabular}
\caption{Local $\delta$-invariants: $(-K_S)^2=2$ and  $\mathbb{E}_6$ singularity}
\end{table}
\item[XLV.] $\delta(X)=\frac{3}{10}$ since depending on the position of point $P\in S$ we have 
\begin{table}[h!]
\hspace*{0.5cm}
\begin{tabular}{ | c || c | c |c | c |  c | c | c |c | c |  c |   }
 \hline 
 $P$ & $E_3$ & $ E_4 \backslash E_3$  & $ E_2 \backslash E_3$ &  $ E_5 \backslash E_4$    & $ E \backslash E_3$ & $(E_1\cup E_6) \backslash (E_2\cup E_5)$ & $ L_6\backslash E_6$ & o/w  \\\hline
$\delta_P(S)$ & $\frac{3}{10}$ & $\frac{3}{8}$ & $\frac{3}{7}$& $\frac{1}{2}$ & $\frac{9}{16}$ & $\frac{3}{4}$ & $\frac{3}{2}$   & $\ge \frac{9}{5}$\\\hline
    \end{tabular}
\caption{Local $\delta$-invariants: $(-K_S)^2=2$ and  $\mathbb{E}_7$ singularity}
\end{table}
\end{itemize}

\begin{proof}
We prove each case separately using lemmas from the previous section.
\begin{itemize}
    \item[I.] If $P$ is a point on the unique $(-2)$-curve the assertion follows from the Lemma \ref{deg2-32-A1points}.
If $P$ on all the curves adjacent to the unique  $(-2)$-curve, the assertion follows from Lemma \ref{deg2-2717-nearA1points}.
Otherwise, the assertion follows from Lemma \ref{deg2-genpoint}.
    \item[II.] If $P\in  E_1\cup E_2$, the assertion follows from Lemma \ref{deg2-32-A1points}.
  If $P\in  (L_{12}\cup L_{12}')\backslash (E_1\cup E_2)$, the assertion follows from Lemma \ref{deg2-2-near2A1points}  [a).].
  If $P\in \bigcup_{i\in\{1,2\},\text{ }j\in\{1,2,3,4\}}\big(L_{i,j}\cup L_{i,j}'\big)\backslash (E_1\cup E_2)$, the assertion follows from Lemma \ref{deg2-2717-nearA1points}. Otherwise, the assertion follows from Lemma \ref{deg2-genpoint}.
  \item[III.] If $P\in  E_1\cup E_2\cup E_3$, the assertion follows from Lemma \ref{deg2-32-A1points}.
  If $P\in   (L_{12}\cup L_{13}\cup L_{23}\cup L_{12}'\cup L_{13}'\cup L_{23}')\backslash (E_1\cup E_2\cup E_3)$, the assertion follows from Lemma \ref{deg2-2-near2A1points} [a).].
  If $P\in \bigcup_{i\in\{1,2,3\},\text{ }j\in\{1,2\}}\big(L_{i,j}\cup L_{i,j}'\big)\backslash (E_1\cup E_2\cup E_3)$, the assertion follows from Lemma \ref{deg2-2717-nearA1points}.
Otherwise, the assertion follows from Lemma \ref{deg2-genpoint}.
\item[IV.] If $P\in  E_1\cup E_2\cup E_3 $, the assertion follows from Lemma \ref{deg2-32-A1points}.
  If $P\in  L_{123}\backslash ( E_1\cup E_2\cup E_3)$, the assertion follows from Lemma \ref{deg2-32-near3A1points} [a).].
  If $P\in  \bigcup_{i\in\{1,2,3\},\text{ }j\in\{1,2,3,4\}}\big(L_{i,j}\cup L_{i,j}'\big)\backslash (E_1\cup E_2\cup E_3)$, the assertion follows from Lemma \ref{deg2-2717-nearA1points}.
Otherwise, the assertion follows from Lemma \ref{deg2-genpoint}.
\item[V.] If $P\in  E_1\cup E_2\cup E_3\cup E_4$, the assertion follows from Lemma \ref{deg2-32-A1points}.
  If $P\in  \bigcup_{i,j\in\{1,2,3,4\},\text{ }i<j}\big(L_{ij}\cup L_{ij}'\big)\backslash (E_1\cup E_2\cup E_3\cup E_4) \bigcup_{i,j\in\{1,2,3,4\},\text{ }i<j}\big(L_{ij}\cup L_{ij}'\big)\backslash (E_1\cup E_2\cup E_3\cup E_4)$, the assertion follows from Lemma \ref{deg2-2-near2A1points}  [a).].  Otherwise, the assertion follows from Lemma \ref{deg2-genpoint}.
  \item[VI.] If $P\in  E_1\cup E_2\cup E_3\cup E_4 $, the assertion follows from Lemma \ref{deg2-32-A1points}.
  If $P\in  L_{234} \backslash ( E_1\cup E_2\cup E_3\cup E_4 )$, the assertion follows from Lemma \ref{deg2-32-near3A1points}[a).].
  If $ P\in \bigcup_{i\in\{2,3,4\},\text{ }i<j}\big(L_{1i}\cup L_{1i}'\big)\backslash (E_1\cup E_2\cup E_3\cup E_4) $, the assertion follows from Lemma \ref{deg2-2-near2A1points}  [a).].
  If $P\in \bigcup_{i\in\{2,3,4\},\text{ }j\in\{1,2\}}\big(L_{i,j}\cup L_{i,j}'\big)\backslash (E_2\cup E_3\cup E_4)$, the assertion follows from Lemma \ref{deg2-2717-nearA1points}.
   Otherwise, the assertion follows from Lemma \ref{deg2-genpoint}.
   \item[VII.] If $P\in E_1\cup E_2\cup E_3\cup E_4\cup E_5 $, the assertion follows from Lemma \ref{deg2-32-A1points}.
  If $P\in  (L_{134}\cup L_{125})\backslash (E_1\cup E_2\cup E_3\cup E_4\cup E_5)$, the assertion follows from Lemma \ref{deg2-32-near3A1points} [a).].
  If $P\in \bigcup_{(i,j)\in\{(2,3),(2,4),(3,5),(4,5)\}}\big(L_{ij}\cup L_{ij}'\big)\backslash (E_1\cup E_2\cup E_3\cup E_4\cup E_5)$, the assertion follows from Lemma \ref{deg2-2-near2A1points}  [a).].
   If $P\in  (L_{1,1}\cup L_{1,1}'\cup L_{1,2}\cup L_{1,2}')\backslash E_1$, the assertion follows from Lemma \ref{deg2-2717-nearA1points}.
   Otherwise, the assertion follows from Lemma \ref{deg2-genpoint}.
     \item[VIII.] If $P\in E_1\cup E_2\cup E_3\cup E_4\cup E_5\cup E_6 $, the assertion follows from Lemma \ref{deg2-32-A1points}.
  If $P\in  (L_{136}\cup L_{235}\cup L_{145})\backslash ( E_1\cup E_2\cup E_3\cup E_4\cup E_5\cup E_6)$, the assertion follows from Lemma \ref{deg2-32-near3A1points} [a).].
  If $P\in\bigcup_{(i,j)\in\{(1,2),(3,4),(5,6)\}}\big(L_{ij}\cup L_{ij}'\big)\backslash (E_1\cup E_2\cup E_3\cup E_4\cup E_5\cup E_6)$, the assertion follows from Lemma \ref{deg2-2-near2A1points}  [a).].
Otherwise, the assertion follows from Lemma \ref{deg2-genpoint}.
 \item[IX.] If $P\in E_1\cap E_1 $, the assertion follows from Lemma \ref{deg2-65-A2points}.
  If $P\in  (E_1\cup E_1')\backslash (E_1\cap E_1')$, the assertion follows from Lemma \ref{deg2-97-A2points} [a).].
  If $P\in \bigcup_{i\in\{1,..,6\}}\big(L_{1,i}\cup L_{1,i}'\big)\backslash (E_1\cup E_1')$, the assertion follows from Lemma \ref{deg2-3219-nearA2points}.
Otherwise, the assertion follows from Lemma \ref{deg2-genpoint}.
\item[X.] If $P=E_1\cap E_1'$, the assertion follows from Lemma \ref{deg2-65-A2points}.
  If $ P\in  (E_1\cup E_1')\backslash (E_1\cap E_1')$, the assertion follows from Lemma \ref{deg2-97-A2points}  [b).].
  If $P\in \bigcup_{i\in\{1,..,4\}}\big(L_{1,i}\cup L_{1,i}'\big)\backslash (E_1\cup E_1')$, the assertion follows from Lemma \ref{deg2-3219-nearA2points}.
  If $P\in  E_2$, the assertion follows from Lemma \ref{deg2-32-A1points}.
  If $P\in  (L_{12}\cup L_{12}')\backslash (E_1\cup E_1'\cup E_2)$, the assertion follows from Lemma \ref{deg2-158-nearA1A2points} [a).].
  If $P\in (L_{2,1}\cup L_{2,1}'\cup L_{2,2}\cup L_{2,2}')\backslash (E_1\cup E_1'\cup E_2)$, the assertion follows from Lemma \ref{deg2-2717-nearA1points}.
   Otherwise, the assertion follows from Lemma \ref{deg2-genpoint}.
   \item[XI.] If $P=E_1\cap E_1'$, the assertion follows from Lemma \ref{deg2-65-A2points}.
  If $ P\in(E_1\cup E_1')\backslash (E_1\cap E_1')$, the assertion follows from Lemma \ref{deg2-97-A2points}  [c).].
  If $P\in E_2\cup E_3$, the assertion follows from Lemma \ref{deg2-32-A1points}.
  If $P\in  (L_{12}\cup L_{12}'\cup L_{13}\cup L_{13}')\backslash (E_1\cup E_1'\cup E_2\cup E_3)$, the assertion follows from Lemma \ref{deg2-158-nearA1A2points} [b).].
  If $P\in  (L_{23}\cup L_{23}')\backslash (E_2\cup E_3)$.  By Lemma \ref{deg2-2-near2A1points}  [a).].
  If $P\in (L_{2,1}\cup L_{2,1}'\cup L_{3,1}\cup L_{2,2}')\backslash (E_2\cup E_3)$, the assertion follows from Lemma \ref{deg2-2717-nearA1points}.
  If $P\in  (L_{1,1}\cup L_{1,1}'\cup L_{1,2}\cup L_{1,2}')\backslash (E_1\cup E_1')$, the assertion follows from Lemma \ref{deg2-3219-nearA2points}.
   Otherwise, the assertion follows from Lemma \ref{deg2-genpoint}.
   \item[XII.] If $P=E_1\cap E_1'$, the assertion follows from Lemma \ref{deg2-65-A2points}.
  If $ P\in  (E_1\cup E_1')\backslash (E_1\cap E_1')$, the assertion follows from Lemma \ref{deg2-97-A2points}  [d).].
  If $P\in E_2\cup E_3\cup E_4$, the assertion follows from Lemma \ref{deg2-32-A1points}.
  If $P\in  L_{234} \backslash (E_2\cup E_3\cup E_4)$.  By Lemma \ref{deg2-32-near3A1points} [a).].
  If $P\in  (L_{12}\cup L_{12}'\cup L_{13}\cup L_{13}'\cup L_{14}\cup L_{14}')\backslash (E_1\cup E_1'\cup E_2\cup E_3\cup E_4)$, the assertion follows from Lemma \ref{deg2-158-nearA1A2points} [c).].
  If $P\in\bigcup_{i\in\{2,3,4\},\text{ }j\in\{1,2\}}\big(L_{i,j}\cup L_{i,j}'\big)\backslash ( E_2\cup E_3\cup E_4)$, the assertion follows from Lemma \ref{deg2-2717-nearA1points}.
   Otherwise, the assertion follows from Lemma \ref{deg2-genpoint}.
    \item[XIII.]  If $P\in(E_1\cap E_1')\cup (E_2\cap E_2')$, the assertion follows from Lemma \ref{deg2-65-A2points}.
  If $ P\in (E_1\cup E_1'\cup E_2\cup E_2')\backslash \big((E_1\cap E_1')\cup (E_2\cap E_2')\big)$, the assertion follows from Lemma \ref{deg2-97-A2points}  [e).].
  If $P\in (L_{12}\cup L_{12}'') \backslash  (E_1\cup E_1'\cup E_2\cup E_2') $, the assertion follows from Lemma \ref{deg2-127-near2A2point} [a).].
  If $P\in\bigcup_{i,j\in\{1,2\}}\big(L_{i,j}\cup L_{i,j}'\big)\backslash ( E_1\cup E_1'\cup E_2\cup E_2')$, the assertion follows from Lemma \ref{deg2-3219-nearA2points}.
   Otherwise, the assertion follows from Lemma \ref{deg2-genpoint}.
   \item[XIV.] If $P\in (E_1\cap E_1')\cup (E_2\cap E_2') $, the assertion follows from Lemma \ref{deg2-65-A2points}.
  If $ P\in  (E_1\cup E_1'\cup E_2\cup E_2')\backslash \big((E_1\cap E_1')\cup (E_2\cap E_2')\big)$, the assertion follows from Lemma \ref{deg2-97-A2points}  [f).].
  If $P\in   E_3$, the assertion follows from Lemma \ref{deg2-32-A1points}.
  If $P\in (L_{12}\cup L_{12}'') \backslash  (E_1\cup E_1'\cup E_2\cup E_2')$, the assertion follows from Lemma \ref{deg2-127-near2A2point} [a).].
  If $P\in (L_{13}\cup L_{13}'\cup L_{23}\cup L_{23}')\backslash (E_1\cup E_1'\cup E_2\cup E_2'\cup E_3)$, the assertion follows from Lemma \ref{deg2-158-nearA1A2points} [d).].
  If $ P\in (L_{1,1}\cup L_{1,1}'\cup L_{2,1}\cup L_{2,1}')\backslash  (E_1\cup E_1'\cup E_2\cup E_2')$.  
   Otherwise, the assertion follows from Lemma \ref{deg2-genpoint}.
   \item[XV.] If $P\in   (E_1\cap E_1')\cup (E_2\cap E_2')\cup (E_3\cap E_3') $, the assertion follows from Lemma \ref{deg2-65-A2points}.
  If $ P\in   (E_1\cup E_1'\cup E_2\cup E_2'\cup E_3\cup E_3')\backslash \big((E_1\cap E_1')\cup (E_2\cap E_2')\cup (E_3\cap E_3')\big)$, the assertion follows from Lemma \ref{deg2-97-A2points}  [g).].
  If $P\in\bigcup_{i,j\in\{1,2,3\}\text{, }i<j}\big(L_{ij}\cup L_{ij}'\big))$, the assertion follows from Lemma \ref{deg2-127-near2A2point} [a).].
Otherwise, the assertion follows from Lemma \ref{deg2-genpoint}.
\item[XVI.] If $P\in E_2 $, the assertion follows from Lemma \ref{deg2-1-middleA3} [a).].
  If $ P\in  (E_1\cup E_1') \backslash E_2$, the assertion follows from Lemma \ref{deg2-65-A3points} [a).].
  If $P\in (L_{2,1}\cup L_{2,2}) \backslash E_2$, the assertion follows from Lemma \ref{deg2-2-near2A1points} [b).].
  If $P\in\bigcup_{i\in\{1,2,3,4\}}\big(L_{1,i}\cup L_{1,i}'\big)$, the assertion follows from Lemma \ref{deg2-7543-nearA3points}.
   Otherwise, the assertion follows from Lemma \ref{deg2-genpoint}.
    \item[XVII.] If $P\in E_2 $, the assertion follows from Lemma \ref{deg2-1-middleA3}.
  If $ P\in (E_1\cup E_1') \backslash E_2$, the assertion follows from Lemma \ref{deg2-65-A3points} [b).].
  If $P\in  E_3$, the assertion follows from Lemma \ref{deg2-32-A1points}.
  If $P\in  (L_{13}\cup L_{13}')\backslash (E_1\cup E_1'\cup E_3)$, the assertion follows from Lemma \ref{deg2-95-nearA1A3} [a).].
  If $P\in  (L_{2,1}\cup L_{2,2}) \backslash E_2$, the assertion follows from Lemma \ref{deg2-2-near2A1points} [b).].
  If $P\in  (L_{3,1}\cup L_{3,1}'\cup L_{3,2}\cup L_{3,2}')\backslash E_3 $, the assertion follows from Lemma \ref{deg2-2717-nearA1points}.
  If $P\in (L_{1,1}\cup L_{1,1}'\cup L_{1,2}\cup L_{1,2}')\backslash (E_1\cup E_1')$, the assertion follows from Lemma \ref{deg2-7543-nearA3points}.
   Otherwise, the assertion follows from Lemma \ref{deg2-genpoint}.
   \item[XVIII.] If $P\in E_2 $, the assertion follows from Lemma \ref{deg2-1-middleA3} [b).].
  If $ P\in (E_1\cup E_1') \backslash E_2$, the assertion follows from Lemma \ref{deg2-65-A3points} [a).].
  If $P\in E_3$, the assertion follows from Lemma \ref{deg2-32-A1points}.
  If $P\in L_{23} \backslash (E_2\cup E_3)$, the assertion follows from Lemma \ref{deg2-32-nearA1middleA3} [b).].
  If $P\in  \bigcup_{i\in\{1,2,3\}}\big(L_{3,i}\cup L_{3,i}'\big)\backslash E_3 $, the assertion follows from Lemma \ref{deg2-2717-nearA1points}.
  If $P\in \bigcup_{j\in\{1,2,3,4\}}\big(L_{1,j}\cup L_{1,j}'\big)\backslash E_3  $, the assertion follows from Lemma \ref{deg2-7543-nearA3points}.
   Otherwise, the assertion follows from Lemma \ref{deg2-genpoint}.
   \item[XIX.] If $P\in E_2 $, the assertion follows from Lemma \ref{deg2-1-middleA3} [a).].
  If $ P\in (E_1\cup E_1') \backslash E_2$, the assertion follows from Lemma \ref{deg2-65-A3points} [c).].
  If $P\in E_3\cup E_4$, the assertion follows from Lemma \ref{deg2-32-A1points}.
  If $P\in  (L_{13}\cup L_{13}'\cup L_{14}\cup L_{14}')\backslash (E_1\cup E_1'\cup E_3\cup E_4)$, the assertion follows from Lemma \ref{deg2-95-nearA1A3} [b).].
  If $P\in (L_{2,1}\cup L_{2,2}) \backslash E_2$, the assertion follows from Lemma \ref{deg2-2-near2A1points} [b).].
  If $P\in (L_{34}\cup L_{34}) \backslash (E_3\cup E_4)$, the assertion follows from Lemma \ref{deg2-2-near2A1points} [a).].
   Otherwise, the assertion follows from Lemma \ref{deg2-genpoint}.
   \item[XX.] If $P\in E_2$, the assertion follows from Lemma \ref{deg2-1-middleA3}.
  If $ P\in  (E_1\cup E_1') \backslash E_2$, the assertion follows from Lemma \ref{deg2-65-A3points} [b).].
  If $P\in E_3\cup E_4$, the assertion follows from Lemma \ref{deg2-32-A1points}.
  If $P\in L_{23} \backslash (E_2\cup E_3)$, the assertion follows from Lemma \ref{deg2-32-nearA1middleA3} [b).].
  If $P\in   (L_{14}\cup L_{14}')\backslash (E_1\cup E_1'\cup E_4)$, the assertion follows from Lemma \ref{deg2-95-nearA1A3} [b).].
  If $P\in  (L_{34}\cup L_{34}')\backslash (E_3\cup E_4)$, the assertion follows from Lemma \ref{deg2-2-near2A1points} [a).].
  If $P\in (L_{3,1}\cup L_{3,1})\backslash E_3 $, the assertion follows from Lemma \ref{deg2-2717-nearA1points}.
  If $P\in  \bigcup_{k\in \{1,2\}} \big(L_{1,k}\cup L_{1,k}'\big) \backslash(E_1\cup E_1')$, the assertion follows from Lemma \ref{deg2-7543-nearA3points}.
   Otherwise, the assertion follows from Lemma \ref{deg2-genpoint}.
   \item[XXI.] If $P\in E_2$, the assertion follows from Lemma \ref{deg2-1-middleA3} [b).].
  If $ P\in  (E_1\cup E_1') \backslash E_2$, the assertion follows from Lemma \ref{deg2-65-A3points} [c).].
  If $P\in E_3\cup E_4\cup E_5$, the assertion follows from Lemma \ref{deg2-32-A1points}.
  If $P\in L_{345}\backslash (E_3\cup E_4\cup E_5)$, the assertion follows from Lemma \ref{deg2-32-nearA1middleA3} [a).].
  If $P\in L_{25}\backslash (E_2\cup E_5)$, the assertion follows from Lemma \ref{deg2-32-nearA1middleA3} [b).].
  If $P\in (L_{13}\cup L_{13}'\cup L_{14}\cup L_{14}') \backslash (E_1\cup E_1'\cup E_3\cup E_4)$, the assertion follows from Lemma \ref{deg2-95-nearA1A3} [b).].
  If $P\in (L_{5,1}\cup L_{5,1}') \backslash E_5$, the assertion follows from Lemma \ref{deg2-2717-nearA1points}.
   Otherwise, the assertion follows from Lemma \ref{deg2-genpoint}.
   \item[XXII.]  If $P\in E_2$, the assertion follows from Lemma \ref{deg2-1-middleA3} [a).].
  If $ P\in (E_1\cup E_1') \backslash E_2$, the assertion follows from Lemma \ref{deg2-65-A3points} [d).].
  If $P= E_3\cap E_3'$, the assertion follows from Lemma \ref{deg2-65-A2points}.
  If $ P\in  (E_3\cup E_3')\backslash (E_3\cap E_3')$, the assertion follows from Lemma \ref{deg2-97-A2points}  [h).].
  If $P\in (L_{13}\cup L_{13}'')\backslash(E_1\cup E_1'\cup E_3\cup E_3')$, the assertion follows from Lemma  \ref{deg2-1811-nearA2A3}.
  If $P\in(L_{2,1}\cup L_{2,2}) \backslash E_2$, the assertion follows from Lemma \ref{deg2-2-near2A1points} [b).].
  If $P\in  (L_{3,1}\cup L_{3,1}'\cup L_{3,2}\cup L_{3,2}')\backslash (E_3\cup E_3')$.  
  If $P\in (L_{1,1}\cup L_{1,1}') \backslash(E_1\cup E_1')$, the assertion follows from Lemma \ref{deg2-7543-nearA3points}.
   Otherwise, the assertion follows from Lemma \ref{deg2-genpoint}.
   \item[XXIII.] If $P\in E_2$, the assertion follows from Lemma \ref{deg2-1-middleA3} [b).].
  If $ P\in (E_1\cup E_1') \backslash E_2$, the assertion follows from Lemma \ref{deg2-65-A3points} [d).].
  If $P= E_3\cap E_3'$, the assertion follows from Lemma \ref{deg2-65-A2points}.
  If $ P\in(E_3\cup E_3')\backslash (E_3\cap E_3')$, the assertion follows from Lemma \ref{deg2-97-A2points}  [i).].
  If $P\in (L_{13}\cup L_{13}'')\backslash(E_1\cup E_1'\cup E_3\cup E_3')$.  
  If $P\in  E_4$, the assertion follows from Lemma \ref{deg2-32-A1points}.
  If $P\in L_{24} \backslash (E_2\cup E_4)$, the assertion follows from Lemma \ref{deg2-32-nearA1middleA3} [b).].
  If $P\in (L_{34}\cup L_{34}')\backslash (E_3\cup E_3'\cup E_4)$, the assertion follows from Lemma \ref{deg2-158-nearA1A2points} [e).].
  If $P\in  (L_{1,1}\cup L_{1,1}')\backslash(E_1\cup E_1')$, the assertion follows from Lemma \ref{deg2-7543-nearA3points}.
   Otherwise, the assertion follows from Lemma \ref{deg2-genpoint}.
   \item[XXIV.] If $P\in E_2\cup E_4 $, the assertion follows from Lemma \ref{deg2-1-middleA3} [a).].
  If $ P\in   (E_1\cup E_1'\cup E_3\cup E_3') \backslash (E_2\cup E_4) $, the assertion follows from Lemma \ref{deg2-65-A3points} [e).].
  If $P\in (L_{13}\cup L_{13}'')\backslash (E_1\cup E_1'\cup E_3\cup E_3')$, the assertion follows from Lemma \ref{deg2-32-nearA1middleA3} [c).].
  If $P\in (L_{2,1}\cup L_{2,2}\cup L_{4,1}\cup L_{4,2} ) \backslash (E_2\cup E_4)$, the assertion follows from Lemma \ref{deg2-2-near2A1points}.
   Otherwise, the assertion follows from Lemma \ref{deg2-genpoint}.
   \item[XXV.] If $P\in _2\cup E_4$, the assertion follows from Lemma \ref{deg2-1-middleA3} [b).].
  If $ P\in  (E_1\cup E_1'\cup E_3\cup E_3') \backslash (E_2\cup E_4) $, the assertion follows from Lemma \ref{deg2-65-A3points} [e).].
  If $P\in  E_5$, the assertion follows from Lemma \ref{deg2-32-A1points}.
  If $P\in (L_{25}\cup L_{45}) \backslash (E_1\cup E_1'\cup E_3\cup E_3'\cup E_5)$, the assertion follows from Lemma \ref{deg2-32-nearA1middleA3} [b).].
  If $P\in (L_{13}\cup L_{13}''\cup E_5)\backslash (E_1\cup E_1'\cup E_3\cup E_3')$, the assertion follows from Lemma \ref{deg2-32-nearA1middleA3} [c).].
   Otherwise, the assertion follows from Lemma \ref{deg2-genpoint}.
   \item[XXVI.] If $P\in  E_2\cup E_2'$, the assertion follows from Lemma \ref{deg2-1213-A4points}.
  If $ P\in ( E_1\cup E_1') \backslash ( E_2\cup E_2')$, the assertion follows from Lemma \ref{deg2-3631-A4points} [a).].
  If $P\in  (L_{2,1}\cup L_{2,1}')\backslash ( E_2\cup E_2')$, the assertion follows from Lemma \ref{deg2-2413-nearA4points} [a).].
  If $P\in\bigcup_{i\in\{1,2,3\}} \big(L_{1,i}\cup L_{1,i}'\big)$, the assertion follows from Lemma \ref{deg2-216121-nearA4point}.
   Otherwise, the assertion follows from Lemma \ref{deg2-genpoint}.
   \item[XXVII.] If $P\in E_2\cup E_2'$, the assertion follows from Lemma \ref{deg2-1213-A4points}.
  If $ P\in ( E_1\cup E_1') \backslash ( E_2\cup E_2')$, the assertion follows from Lemma \ref{deg2-3631-A4points} [b).].
  If $P\in E_3$, the assertion follows from Lemma \ref{deg2-32-A1points}.
  If $P\in  (L_{2,1}\cup L_{2,1}')\backslash ( E_2\cup E_2')$, the assertion follows from Lemma \ref{deg2-2413-nearA4points} [b).].
  If $P\in (L_{13}\cup L_{13}') \backslash (E_1\cup E_1'\cup E_3)$, the assertion follows from Lemma \ref{deg2-7241-nearA1A4}.
  If $P\in(L_{3,1}\cup L_{3,1}')\backslash E_3$, the assertion follows from Lemma \ref{deg2-2717-nearA1points}.
  If $P\in  (L_{1,1}\cup L_{1,1}')\backslash (E_1\cup E_1')$, the assertion follows from Lemma \ref{deg2-216121-nearA4point}.
   Otherwise, the assertion follows from Lemma \ref{deg2-genpoint}.
   \item[XXVIII.] If $P\in  E_2\cup E_2'$, the assertion follows from Lemma \ref{deg2-1213-A4points}.
  If $ P\in  ( E_1\cup E_1') \backslash ( E_2\cup E_2')$, the assertion follows from Lemma \ref{deg2-3631-A4points} [c).].
  If $P= E_3\cap E_3'$, the assertion follows from Lemma \ref{deg2-65-A2points}.
  If $ P\in (E_3\cup E_3)\backslash(E_3\cap E_3)$, the assertion follows from Lemma \ref{deg2-97-A2points}  [j).].
  If $P\in  (L_{13}\cup L_{13}'')\backslash ( E_1\cup E_1'\cup E_3\cup E_3)$, the assertion follows from Lemma \ref{deg2-3623-nearA2A4}.
  If $P\in (L_{2,1}\cup L_{2,1}')\backslash (E_2\cup E_2')$, the assertion follows from Lemma \ref{deg2-2413-nearA4points} [c).].
  If $P\in (L_{3,1}\cup L_{3,1}')\backslash (E_3\cup E_3')$.  
   Otherwise, the assertion follows from Lemma \ref{deg2-genpoint}.
   \item[XXIX.] If $P\in E_3$, the assertion follows from Lemma \ref{deg2-67-themiddleA5points} [a).].
  If $ P\in   (E_2\cup E_2') \backslash E_3 $, the assertion follows from Lemma \ref{deg2-67-A5points} [a).].
  If $P\in (E_1\cup E_1')\backslash (E_2\cup E_2')$, the assertion follows from Lemma \ref{deg2-87-A5points} [a).].
  If $P\in  (L_{2,1}\cup L_{2,1}')\backslash (E_2\cup E_2')$, the assertion follows from Lemma \ref{deg2-127-nearA5} [b).].
  If $P\in \bigcup_{i\in\{1,2\}}\big(L_{1,i}\cup L_{1,i}'\big)\backslash (E_1\cup E_1')$, the assertion follows from Lemma \ref{deg2-4927-nearA5}.
   Otherwise, the assertion follows from Lemma \ref{deg2-genpoint}.
   \item[XXX.] If $P\in E_3$, the assertion follows from Lemma \ref{deg2-34-themiddleA5point} [a).].
  If $ P\in  (E_2\cup E_2') \backslash E_3$, the assertion follows from Lemma \ref{deg2-910-A5points}.
  If $P\in  (E_1\cup E_1') \backslash (E_2\cup E_2')$, the assertion follows from Lemma \ref{deg2-98-A5points} [a).].
  If $P\in L_{3,1} \backslash E_3$, the assertion follows from Lemma \ref{deg2-32-nearA1middleA3} [d).].
  If $P\in \bigcup_{i\in\{1,2,3\}}\big(L_{1,i}\cup L_{1,i}'\big)$, the assertion follows from Lemma \ref{deg2-4927-nearA5}.
   Otherwise, the assertion follows from Lemma \ref{deg2-genpoint}.
   \item[XXXI.] If $P\in E_3$, the assertion follows from Lemma \ref{deg2-67-themiddleA5points} [a).].
  If $ P\in  (E_2\cup E_2')\backslash E_3$, the assertion follows from Lemma \ref{deg2-67-A5points}  [a).].
  If $P\in (E_1\cup E_1')\backslash (E_2\cup E_2')$, the assertion follows from Lemma \ref{deg2-87-A5points} [b).].
  If $P\in E_4$, the assertion follows from Lemma \ref{deg2-32-A1points}.
  If $P\in (L_{2,1}\cup L_{2,1}')\backslash ( E_2\cup E_2')$, the assertion follows from Lemma \ref{deg2-127-nearA5} [b).].
  If $P\in (L_{14}\cup L_{14}')\backslash (E_1\cup E_1'\cup E_4)$, the assertion follows from Lemma \ref{deg2-127-near2A2point} [c).].
   Otherwise, the assertion follows from Lemma \ref{deg2-genpoint}.
   \item[XXXII.] If $P\in E_3$, the assertion follows from Lemma \ref{deg2-34-themiddleA5point} [a).].
  If $ P\in (E_2\cup E_2') \backslash E_3 $, the assertion follows from Lemma \ref{deg2-910-A5points}.
  If $P\in  (E_1\cup E_1') \backslash (E_2\cup E_2')$, the assertion follows from Lemma \ref{deg2-98-A5points} [b).].
  If $P\in E_4$, the assertion follows from Lemma \ref{deg2-32-A1points}.
  If $P\in L_{3,1} \backslash E_3$, the assertion follows from Lemma \ref{deg2-32-nearA1middleA3} [d).].
  If $P\in (L_{14}\cup L_{14}')\backslash (E_1\cup E_1'\cup E_4)$, the assertion follows from Lemma \ref{deg2-4526-nearA1A5}.
  If $P\in (L_{1,1}\cup L_{1,1})\backslash(E_1\cup E_1')$, the assertion follows from Lemma \ref{deg2-4927-nearA5}.
   Otherwise, the assertion follows from Lemma \ref{deg2-genpoint}.
    \item[XXXIII.] If $P\in E_3$, the assertion follows from Lemma \ref{deg2-34-themiddleA5point} [a).].
  If $ P\in  (E_2\cup E_2') \backslash E_3 $, the assertion follows from Lemma \ref{deg2-910-A5points}.
  If $P\in (E_1\cup E_1') \backslash (E_2\cup E_2')$, the assertion follows from Lemma \ref{deg2-98-A5points} [c).].
  If $P=E_4\cap E_4$, the assertion follows from Lemma \ref{deg2-65-A2points}
  If $ P\in  (E_4\cup E_4')\backslash (E_4\cap E_4')$, the assertion follows from Lemma \ref{deg2-97-A2points}  [k).].
  If $P\in L_{3,1} \backslash E_3$, the assertion follows from Lemma \ref{deg2-32-nearA1middleA3} [d).].
  If $P\in (L_{14}\cup L_{14}') \backslash (E_1\cup E_1'\cup E_4\cup E_4')$, the assertion follows from Lemma \ref{deg2-32-nearA1middleA3} [e).].
   Otherwise, the assertion follows from Lemma \ref{deg2-genpoint}.
    \item[XXXIV.] If $P\in E_3\cup E_3'$, the assertion follows from Lemma \ref{deg2-45-middleA6points}.
  If $ P\in (E_2\cup E_2')\backslash (E_3\cup E_3')$, the assertion follows from Lemma \ref{deg2-45-NOTmiidleA6points}.
  If $P\in  (E_1\cup E_1') \backslash (E_2\cup E_2')$, the assertion follows from Lemma \ref{deg2-6053-A6points}.
  If $P\in  (L_{2,1}\cup L_{2,1}') \backslash (E_2\cup E_2')$, the assertion follows from Lemma \ref{deg2-6037-nearA6points}.
 If $P\in (L_{1,1}\cup L_{1,1}')\backslash (E_1\cup E_1')$, the assertion follows from Lemma \ref{deg2-384209-nearA6points}.
   Otherwise, the assertion follows from Lemma \ref{deg2-genpoint}.
   \item[XXXV.] If $P\in E_4$, the assertion follows from Lemma \ref{deg2-34-A7points} [a).].
  If $P\in (E_3\cup E_3')\backslash E_4$, the assertion follows from Lemma \ref{deg2-34-onlyinA7points}.
 If $P\in  (E_2\cup E_2')\backslash (E_3\cup E_3')$.   By Lemma \ref{deg2-34-themiddleA5point} [b).].
 If $P\in (E_1\cup E_1') \backslash (E_2\cup E_2')$.   By Lemma \ref{deg2-98-A7points}.
  If $P\in  (L_{2,1}\cup L_{2,1}')\backslash (E_2\cup E_2')$, the assertion follows from Lemma \ref{deg2-32-nearA1middleA3} [f).].
   Otherwise, the assertion follows from Lemma \ref{deg2-genpoint}.
   \item[XXXVI.] If $P\in E$, the assertion follows from Lemma \ref{deg2-34-A7points} [b).].
  If $ P\in (E_1\cup E_2\cup E_3) \backslash E$, the assertion follows from Lemma \ref{deg2-1-middleA3} [c).].
  If $P\in \bigcup_{i\in\{1,2,3\},j\in\{1,2\}} L_{i,j}  \backslash (E_1\cup E_2\cup E_3)$, the assertion follows from Lemma \ref{deg2-2-near2A1points} [c).].
Otherwise, the assertion follows from Lemma \ref{deg2-genpoint}.
 \item[XXXVII.] If $P\in E$, the assertion follows from Lemma \ref{deg2-34-A7points} [b).].
  If $P\in (E_1\cup E_1') \backslash E $, the assertion follows from Lemma \ref{deg2-1-middleA3} [c).].
  If $P\in E_2\backslash E$, the assertion follows from Lemma \ref{deg2-1-middleA3} [d).].
  If $P\in E_3$, the assertion follows from Lemma \ref{deg2-32-A1points}.
  If $P\in L_{23}\backslash (E_2\cup E_3)$, the assertion follows from Lemma \ref{deg2-32-near3A1points} [f).].
  If $P\in  \bigcup_{i\in\{1,2\}} \big(L_{1,i}\cup L_{1,i}'\big) \backslash (E_1\cup E_1')$, the assertion follows from Lemma \ref{deg2-2-near2A1points} [c).].
  If $P\in (L_{3,1}\cup L_{3,1}'\cup L_{3,2}\cup L_{3,2}')\backslash E_3$, the assertion follows from Lemma \ref{deg2-2717-nearA1points}.
   Otherwise, the assertion follows from Lemma \ref{deg2-genpoint}.
    \item[XXXVIII.] If $P\in E$, the assertion follows from Lemma \ref{deg2-34-A7points} [b).].
  If $P\in E_2\backslash E$, the assertion follows from Lemma \ref{deg2-1-middleA3} [c).].
  If $P\in (E_1\cup E_1')\backslash E$, the assertion follows from Lemma \ref{deg2-1-middleA3} [d).].
  If $P\in E_3\cup E_3'$, the assertion follows from Lemma \ref{deg2-32-A1points}.
  If $P\in (L_{13}\cup L_{13}'')\backslash( E_1\cup E_1'\cup E_3\cup E_3')$, the assertion follows from Lemma \ref{deg2-32-near3A1points} [f).].
  If $P\in (L_2\cup L_2')\backslash  E_2$, the assertion follows from Lemma \ref{deg2-2-near2A1points} [c).].
  If $P\in (L_3\cup L_{3}')\backslash  ( E_3\cup E_3')$, the assertion follows from Lemma \ref{deg2-2-near2A1points} [a).].
   Otherwise, the assertion follows from Lemma \ref{deg2-genpoint}.
   \item[XXXIX.] If $P\in E$, the assertion follows from Lemma \ref{deg2-34-A7points} [b).].
  If $P\in (E_1\cup E_2\cup E_3) \backslash E$, the assertion follows from Lemma \ref{deg2-1-middleA3} [d).].
  If $P\in E_1'\cup E_2'\cup E_3'$, the assertion follows from Lemma \ref{deg2-32-A1points}.
  If $P\in (L_{1}\cup L_2\cup L_3)\backslash( E_1\cup E_2\cup E_3\cup E_1'\cup E_2'\cup E_3')$, the assertion follows from Lemma \ref{deg2-32-near3A1points} [f).].
  If $P\in L_{123} \backslash(  E_1'\cup E_2'\cup E_3')$, the assertion follows from Lemma \ref{deg2-32-near3A1points} [a).].
   Otherwise, the assertion follows from Lemma \ref{deg2-genpoint}.
   \item[XL.] If $P\in E$, the assertion follows from Lemma \ref{deg2-35-D5points} [a).].
  If $P\in E_2 \backslash E$, the assertion follows from Lemma \ref{deg2-34-A7points} [c).].
  If $P\in  (E_1\cup E_1') \backslash E$, the assertion follows from Lemma \ref{deg2-910-D5point}.
  If $P\in E_3\backslash E_2$, the assertion follows from Lemma \ref{deg2-1-middleA3} [e).].
  If $P\in (L_1\cup L_1') \backslash (E_1\cup E_1')$, the assertion follows from Lemma \ref{deg2-95-nearA1A3} [c).].
  If $P\in (L_{3,1}\cup L_{3,2})\backslash E_3 $, the assertion follows from Lemma \ref{deg2-2-near2A1points} [e).].
   Otherwise, the assertion follows from Lemma \ref{deg2-genpoint}.
   \item[XLI.] If $P\in E$, the assertion follows from Lemma \ref{deg2-35-D5points} [a).].
  If $P\in  E_2 \backslash E$, the assertion follows from Lemma \ref{deg2-34-A7points} [c).].
  If $P\in  (E_1\cup E_1') \backslash E$, the assertion follows from Lemma \ref{deg2-910-D5point}.
  If $P\in E_3\backslash E_2$, the assertion follows from Lemma \ref{deg2-1-middleA3} [f).].
  If $P\in E_4$, the assertion follows from Lemma \ref{deg2-32-A1points}.
  If $P\in L_{34} \backslash (E_3\cup E_4)$, the assertion follows from Lemma \ref{deg2-32-near3A1points} [g).].
  If $P\in(L_1\cup L_1') \backslash (E_1\cup E_1')$, the assertion follows from Lemma \ref{deg2-95-nearA1A3} [d).].
  If $P\in  (L_{4,1}\cup L_{4,1}')\backslash E_4$, the assertion follows from Lemma \ref{deg2-2717-nearA1points}.
   Otherwise, the assertion follows from Lemma \ref{deg2-genpoint}.
   \item[XLII.]  If $P\in E_2$, the assertion follows from Lemma \ref{deg2-12-D6points}.
  If $P\in E_3 \backslash E_2$, the assertion follows from Lemma \ref{deg2-35-D5points} [b).].
  If $P\in E \backslash E_2$, the assertion follows from Lemma \ref{deg2-34-themiddleA5point} [c).].
  If $P\in E_4 \backslash E_3$, the assertion follows from Lemma \ref{deg2-34-A7points} [d).].
  If $P\in E_1 \backslash E_2$, the assertion follows from Lemma \ref{deg2-67-themiddleA5points} [b).].
  If $P\in L \backslash E$, the assertion follows from Lemma \ref{deg2-32-nearA1middleA3} [h).].
  If $P\in E_5 \backslash E_4$, the assertion follows from Lemma \ref{deg2-1-middleA3} [g).].
  If $P\in  (L_5\cup L_5') \backslash E_5$, the assertion follows from Lemma \ref{deg2-2-near2A1points} [e).].
   Otherwise, the assertion follows from Lemma \ref{deg2-genpoint}.
   \item[XLIII.] If $P\in E$, the assertion follows from Lemma \ref{deg2-12-D6points} [a).].
  If $P\in E_4 \backslash E_3$, the assertion follows from Lemma \ref{deg2-35-D5points} [b).].
  If $P\in E \backslash E_2$, the assertion follows from Lemma \ref{deg2-34-themiddleA5point}.
  If $P\in E_4 \backslash E_2$, the assertion follows from Lemma \ref{deg2-34-A7points}.
  If $P\in E_1 \backslash E_2$, the assertion follows from Lemma \ref{deg2-67-themiddleA5points} [b).].
  If $P\in L\backslash E$, the assertion follows from Lemma \ref{deg2-32-nearA1middleA3} [h).].
  If $P\in E_5 \backslash E_4$, the assertion follows from Lemma \ref{deg2-1-middleA3} [h).].
  If $P\in E_6$, the assertion follows from Lemma \ref{deg2-32-A1points}.
  If $P\in L_{56}\backslash (E_5\cup E_6)$, the assertion follows from Lemma \ref{deg2-32-near3A1points} [i).].
   Otherwise, the assertion follows from Lemma \ref{deg2-genpoint}.
   \item[XLIV.] If $P\in E$, the assertion follows from Lemma \ref{deg2-37-points} [a).].
  If $P\in E\backslash E_3$, the assertion follows from Lemma \ref{deg2-34-A7points} [e).].
  If $P\in(E_2\cup E_2')\backslash E_3$, the assertion follows from Lemma \ref{deg2-47-points}.
  If $ P\in (E_1\cup E_1')\backslash (E_2\cup E_2') $, the assertion follows from Lemma \ref{deg2-67-A5points}  [b).].
  If $P\in (L_1\cup L_1')\backslash (E_1\cup E_1')$, the assertion follows from Lemma \ref{deg2-127-nearA5} [d).].
   Otherwise, the assertion follows from Lemma \ref{deg2-genpoint}.
   \item[XLV.] If $P\in E_3$, the assertion follows from Lemma \ref{deg2-310-points}.
  If $P\in E_4\backslash E_3$, the assertion follows from Lemma \ref{deg2-38-points}.
  If $P\in E\backslash E_3$, the assertion follows from Lemma \ref{deg2-916-points}.
  If $P\in E_2\backslash E_3$, the assertion follows from Lemma \ref{deg2-37-points} [b).].
  If $P\in E_5\backslash E_4$, the assertion follows from Lemma \ref{deg2-12-D6points} [b).].
  If $P\in E_1\backslash E_2$, the assertion follows from Lemma \ref{deg2-34-A7points} [f).].
  If $P\in E_6 \backslash E_5$, the assertion follows from Lemma \ref{deg2-34-themiddleA5point} [d).].
  If $P\in L_6\backslash E_6$, the assertion follows from Lemma \ref{deg2-32-near3A1points} [j).].
   Otherwise, the assertion follows from Lemma \ref{deg2-genpoint}.
\end{itemize}
\end{proof}

 \printbibliography
\end{document}